\documentclass[a4paper,
11pt]{book}
\usepackage[all]{xy}
\usepackage{imakeidx}
\makeindex
\newcommand{\relative}{{}\xspace}  
  
\usepackage[intlimits]{amsmath}
\usepackage{amsthm, amsfonts, amssymb,xspace,stackrel}

\usepackage{geometry}

\usepackage{fancyhdr}
\usepackage{appendix}

\usepackage{eucal}
\usepackage{mathrsfs}

\numberwithin{section}{chapter}
\numberwithin{equation}{section}

\usepackage{amscd}

\usepackage{calrsfs}
\DeclareMathAlphabet{\pazocal}{OMS}{zplm}{m}{n}

\newcommand{\Ib}{\pazocal{I}}

\theoremstyle{plain}

\newtheorem{theorem}       [subsection]    	{Theorem} 
\newtheorem{thm}		[subsection]		{Theorem}
\newtheorem{lemma}  [subsection]		{Lemma}
\newtheorem{lem}    [subsection]		{Lemma}
\newtheorem{cor}    [subsection]		{Corollary}
\newtheorem{prop}   [subsection]		{Proposition}

\theoremstyle{definition}

\newtheorem{dfn}    	[subsection]		{Definition}

\newtheorem{ex}     	[subsection]		{Example}
\newtheorem{example}	[subsection]		{Example}

\newtheorem{cond2}     [subsection]{Condition}
\theoremstyle{remark}
\newtheorem{remark}      [subsection]		{Remark}
\newtheorem{rem}      [subsection]		{Remark}
\newtheorem{con}     	[subsection]		{Construction}

\newcommand{\labeleq}[2]{\addtocounter{subsection}{1}\begin{equation*}\tag{\thesubsection}\label{#1} {#2}\end{equation*}}
		
\let\realequation\equation
\def\equation{\setcounter{equation}{\arabic{subsection}}%
   \refstepcounter{subsection}%
   \realequation}

\newcommand{\untwist}{\tau}

\newcommand{\N}             	{\ensuremath{\mathbb{N}}}
\newcommand{\R}             	{\ensuremath{\mathbb{R}}}

\newcommand{\T}            	{\ensuremath{\mathbb{T}}}
\newcommand{\Z}             {\ensuremath{\mathbb{Z}}}
\newcommand{\Q}             {\ensuremath{\mathbb{Q}}}
\newcommand{\C}             {\ensuremath{\mathbb{C}}}

\newcommand{\G}{{\ensuremath{G}}}
\newcommand{\h}{{\ensuremath{H}}}
\newcommand{\subgroup}{{\ensuremath{H}}}

\newcommand{\Or}[1]{{\bf {O}}_{#1}}
\newcommand{\obindelta}[1]{[{#1}]}
\newcommand{\Gor}[1]{{\bf {G}}_{#1}}
\newcommand{\B}          {\ensuremath{{\bf{B}}}}
\newcommand{\catOr}          {\ensuremath{{\bf{O}}}}

\newcommand{\catDelta}{\ensuremath{{\bf{\Delta}}}}

\newcommand{\inj}        {{\ensuremath{{\rm-inj}}}}
\newcommand{\cof}        {{\ensuremath{{\rm-cof}}}}
\newcommand{\cell}        {{\ensuremath{{\rm-cell}}}}

\newcommand{\catA}          {\ensuremath{\mathcal A}}
\newcommand{\catB}          {\ensuremath{\mathcal B}}
\newcommand{\catC}          {\ensuremath{\mathcal C}}
\newcommand{\catD}          {\ensuremath{\mathcal D}}

\newcommand{\catset}        {\mathrm{Set}} 
\newcommand{\catFin}        {\mathrm{Fin}} 
\newcommand{\catcAlg}{\mathrm{cAlg}}
\newcommand{\catRAlg}       {R_{\mathrm{Alg}}} 
\newcommand{\catRcalg}      {R_{\mathrm{CAlg}}} 
\newcommand{\catRmod}       {R_{\mathrm{Mod}}} 

\newcommand{\catTop}        {\ensuremath{{\catT\!\textit{op}}}}          

\newcommand{\catT}						{{{\bf{\CMcal{T}}}}}

\newcommand{\catU}						{{{\bf{\CMcal{U}}}}}

\newcommand{\catGT}{{\ensuremath{\G\catT}}}
\newcommand{\catGspaces}[1]{{\ensuremath{#1\catT}}}
\newcommand{\catHT}{{\ensuremath{\subgroup\catT}}}

\newcommand{\catTG}{{{\ensuremath\catT}_\G}}

\newcommand{\catO}			  		{{\ensuremath{{\CMcal{O}}}}}

\newcommand{\catOJ}		  		{{\ensuremath{{\CMcal{O}}_{J}}}}
\newcommand{\catOj}		  		{{\ensuremath{{\CMcal{O}}_{J}}}}
\newcommand{\catOE}{\catJE}

\newcommand{\catOT}					{{\ensuremath{\bf{\CMcal{OT}}}}}

\newcommand{\catOS}					{{\bf{O}\CMcal{T}}} 
\newcommand{\catOGS}{\catOS\!\!_G}
\newcommand{\catGOS}					{{\ensuremath{\G\catOS}}}

\newcommand{\catM}           {\ensuremath{\mathcal M}}

\newcommand{\catGC}          {\ensuremath{\G{\mathcal C}}}
\newcommand{\catCG}          {\ensuremath{\mathcal C}_{\G}}

\newcommand{\catGset}			  {\ensuremath{{{\G}\catset}}}

\DeclareMathOperator{\FF}      	 		  {{\CMcal{F}}}
\DeclareMathOperator{\GG}							{{\CMcal{G}}}
\DeclareMathOperator{\FM}							{{\ensuremath{\mathbb{M}}}}

\newcommand{\Comm}{{\ensuremath{\mathbb{E}}}}
\newcommand{\Ass}{{\ensuremath{\mathbb{A}}}}

\newcommand{\Fr}	[2]{{\ensuremath{\CMcal{F}^{#1}_{#2}}}}
\newcommand{\Gr}	[2]{{\ensuremath{\CMcal{G}^{#1}_{#2}}}}

\newcommand{\I}              {\ensuremath{\mathbb{I}}}
\newcommand{\II}              {\ensuremath{\mathbb{J}}}

\newcommand{\Sp}             {\ensuremath{\mathbb{S}}}

\renewcommand{\smash}{{\ensuremath{\wedge}}}

\newcommand{\dsi}     [2]    {{#2}^{\!\oplus{{#1}}}}

\newcommand{\pri}     [2]    {\prod_{#1}\!#2}

\newcommand{\faktor}[2]{{#1}/{#2}}
\newcommand{\PhiNH}{\Phi^N}

\newcommand{\pullback}      {\ar@{}[dr]|(.3)*{\txt{\LARGE$\lrcorner$}}}
\newcommand{\pushout}       {\ar@{}[ul]|(.3)*{\txt{\LARGE$\ulcorner$}}}
\newcommand 				{\unskew}{\entrymodifiers={!!<0pt,.6ex>+}}

\DeclareMathOperator*{\colim}         {colim}
\DeclareMathOperator*{\hocolim}         {hocolim}
\DeclareMathOperator*{\holim}         {holim}

\DeclareMathOperator{\op}         {op}
\DeclareMathOperator{\sk}         {sk}

\DeclareMathOperator{\id}         	 {id}
\DeclareMathOperator{\proj}         	 {pr}

\DeclareMathOperator{\Cyl}           {Cyl}
\DeclareMathOperator{\inc}         	 {inc}	
\DeclareMathOperator{\Stab}          {Stab}

\DeclareMathOperator{\Ob}						 {Ob}	
\DeclareMathOperator{\ev}      	 		 {ev}
\DeclareMathOperator{\twist}   	 		 {twist}
\DeclareMathOperator{\Map}           {{\bf Map}}
\DeclareMathOperator{\THH}           {{T\mkern-1mu H\mkern-1mu H}}
\DeclareMathOperator{\TC}           {{T\mkern-1mu C}}
\DeclareMathOperator{\Hom}           {{\bf Hom}}

\newenvironment{indentpar}[1]%
{\begin{list}{}%
         {\setlength{\leftmargin}{#1}}%
         \item[]%
}
{\end{list}}

\newcommand{\catI}{\ensuremath{{\CMcal{I}}}}  					
\newcommand{\catL}{\ensuremath{{\CMcal{L}}}}  	
\newcommand{\catLr}{{{\bf{L}}}} 
  	
\newcommand{\fml}{\ensuremath{{\mathscr{F}}}}  	
\newcommand{\gml}{\mathscr{G}}
  	
\newcommand{\hml}{\ensuremath{{\mathscr{H}}}}  	
\newcommand{\nml}{\ensuremath{{\mathscr{F}[N]}}}
\newcommand{\All}{\ensuremath{{\mathscr{A}\hspace{-1pt}\ell\ell}}}
\newcommand{\AI}{\ensuremath{{\mathscr{A\hspace{-4pt}I}}}}
\newcommand{\AIplusV}[1]{\AI\hspace{-4pt}_+\hspace{-2pt}^{#1}}
\newcommand{\AIplus}{\AI\hspace{-4pt}_+}
\newcommand{\Qll}{\ensuremath{{\mathscr{Q}}}}

\newcommand{\catTri}{\ensuremath{{\star}}}

\newcommand{\pr}     [1]    {\prod\limits_{#1}}
\newcommand{\ds}     [2]    {{#2}^{\!\oplus{{#1}}}}

\newcommand{\fev}[2]{\ensuremath{{\rm \ev}^{#1}_{#2}}}

\newcommand{\catJJ}{\ensuremath{{\catO_{\!\J}}}}
\newcommand{\catJE}{\ensuremath{{\catO_{\!\E}}}}

\newcommand{\Lan}   		 [2] {{\rm{Lan}}_{#1}{#2}}
\newcommand{\catW}{\ensuremath{{\bf\mathcal{W}}}}

\newcommand{\X}                {X}   
\newcommand{\J}								{J}
 \newcommand				{\E}								{E}
 
 \newcommand{\Fix}            {\ensuremath{\mathrm{Fix}}}
\newcommand{\Aut}            {\ensuremath{\mathrm{Aut}}}
 \newcommand{\End}            {\ensuremath{\mathcal{E}}}

\newcommand{\FU}{\ensuremath{\mathbb{U}}} 

\newcommand{\FP}             {\ensuremath{\mathbb{P}}}

\newcommand{\Ind}            {\ensuremath{\mathrm{Ind}}}	
\newcommand{\D}{{\bf{D}}}

\newcommand{\reg}            {\ensuremath{\mathrm{reg}}}
\newcommand{\catV}{\ensuremath{{\bf\mathcal{V}}}}

\newcommand{\xto}{\xrightarrow}

\newcommand{\cC}{\mathcal C} 
\newcommand{\bo}{\mathcal I} 
\newcommand{\Fin}{\mathrm{Fin}} 
\newcommand{\Ens}{\mathrm{Set}} 
\newcommand{\skFin}{\overline{\Fin}}
\newcommand{\bN}{\mathbf N}
\newcommand{\bn}{\mathbf n}
\newcommand{\bm}{\mathbf m}
\newcommand{\comC}{\mathbf{Com}_{\cC}}
\newcommand{\internhom}{\hom_{\catOS}}
\newcommand{\PP}{b}
\newcommand{\cat}{\mathrm{Cat}}
\newcommand{\universeU}{\mathcal U}

\newcommand{\eg}{e.g.,\xspace}

\newcommand{\ie}{i.e.,\xspace}
\newcommand{\wrt}{with respect to\xspace}
\usepackage{epsfig}
\usepackage{color}
\newcommand{\smsh}{\smash}
\newcommand{\defas}{:=}
\newcommand{\naively}{free\xspace}

\newcommand{\aml}{\mathscr{A}}
\newcommand{\bml}{\mathscr{B}}
\newcommand{\GL}{\mathrm{GL}}
\newcommand{\hur}{Hurewicz\xspace}

\newcommand{\orfi}{orthogonally final\xspace}
\newcommand{\siap}{has a simplicial approximation\xspace}
\newcommand{\sa}[1]{\mathrm{sa}_{#1}}
\newcommand{\SpGreg}{\Sp_\reg}
\newcommand{\Indreg}{\Ind_\reg}

\usepackage[ps2pdf]{hyperref}
\title{Equivariant Structure on Smash Powers}
\author{Morten Brun \\
  Department of Mathematics, University of Bergen, 5008 Bergen, Norway\\
  \href{mailto:morten.brun@math.uib.no}{morten.brun@math.uib.no}
\and
Bj\o rn Ian Dundas \\
Department of Mathematics, University of Bergen, 5008 Bergen, Norway\\
  \href{mailto:dundas@math.uib.no}{dundas@math.uib.no}
\and 
Martin Stolz\\
\href{mailto:stolz.bonn@googlemail.com}{stolz.bonn@googlemail.com}
}
\begin{document}

\maketitle
 \fancyhead{}
 \fancyhead[RO]{\rightmark}
 \fancyhead[LE]{\leftmark}
 \fancyfoot{}
 \fancyfoot[LE,RO]{\thepage}

\tableofcontents
\clearpage
%
\chapter*{Preface}

\section*{Introduction}

Let $G$ be a compact Lie group.  Below we prove theorems demonstrating to what extent commutative $G$-ring spectra have functorial ``diagonals'' and how these results about equivariant $G$-ring spectra interact.  Why should the reader care?

Stable homotopy occupies a sweet spot between geometry, combinatorics and algebra.  This is particularly apparent when taking symmetries into account.
Spectra have equivariant features originating in geometry and in algebra, and playing these features against each other has proven to be extremely fruitful, also when answering questions that themselves originate purely in either geometry or algebra.

The smash product of spectra is historically a latecomer and there was rightly much enthusiasm when it finally was constructed on the point-set level in the 1990's.  In many functorial models it is a surprisingly uncomplicated construction, but any way you look at it, the smash product has combinatorial roots encapsulating symmetries of the indexing category, which at the very least includes finite sets.  This becomes particularly exciting when coupled with a commutative (or $E_\infty$) algebra structure.  Some consequences of this structure were well known in classical algebra, but appeared as separate entities not obviously having any connections.

An instructive example can be seen as follows.  If $X$ is a set, you can consider the action of the group $C_2$ of order two swapping the two factors in the product $X\times X$.  The diagonal $\Delta\colon X\to X\times X$ is $C_2$-equivariant (with the trivial action on $X$), with image the fixed points $\Delta X=\{(x,x)\mid x\in X\}$.  Considering the associated free abelian group $\Z[X\times X]\cong\Z[X]\otimes\Z[X]$ we see that there are more fixed points than just the diagonal ones; any element $\sum_{(x,y)}a_{x,y}(x,y)$ with $a_{x,y}=a_{y,x}$ is a fixed point, and the fixed points may be identified with free group on the orbits $\Z[X\times X/C_2]$.  However, for an arbitrary abelian group there need not exist any linear ``diagonal homomorphism'' $M\to M\otimes M$.

 In the dual world of algebraic geometry, the analogy is the $C_2$-action on the tensor product $A\otimes A$ of a commutative ring $A$ and the diagonal is given by the multiplication $A\otimes A\to A$.  The dual of the fixed points is the orbits in the category of commutative rings and a quick check reveals that the multiplication is exactly the map to the orbits.  However, the also fixed points for commutative rings are important in many situations, for instance, a divided power structure is given by a certain lift of the multiplication to maps $(A^{\otimes n})^{\Sigma_n}\to A$. 

For spectra $E$ there does not exist a natural ``diagonal'' $E\to E\smsh E$.  Just as for free abelian groups, diagonals do exist for suspension spectra $\Sigma^\infty X$, but in contrast with the algebraic situation, the diagonal for spaces has significant remnants in spectra whose consequences are felt widely.  This text is all about this phenomenon in a guise which we have called the ``geometric diagonal''.  

The fixed points of the natural action of the group $C_2$ of order two on $E\smsh E$ contains a lot of information, both of an algebraic and of a combinatorial nature.  For instance, there is an important map 
$$(E\smsh E)^{C_2}\to E$$
(we hide all fibrant and cofibrant replacements in this introduction), which we in this introduction will call the {\em restriction map}.  

What follows is in many ways an enormous elaboration on the intricacies of this map.    The restriction map is {\bf not} some sort of multiplication map: $E$ was not required to have a ring structure, but if $E$ {\em does} have a ring structure, the restriction map is behind a very important construction in commutative algebra.  The algebraic shadow of this map is closely connected with what is known as Witt-vectors which appears prominently in algebraic number theory and in lifts of structure from finite to infinite characteristic.

As a concrete example, if $E=H\Z/2$ is the mod-$2$ Eilenberg-Mac~Lane spectrum, then $E\smsh E$ represents the ``dual Steenrod algebra'' and on path components the map $(E\smsh E)^{C_2}\to E$ is the surjection $\Z/4\to\Z/2$.  Furthermore, the fixed points of the $2^n$-fold smash $(E^{\smsh 2^n})^{C_{2^n}}$ has path components $\Z/2^n$: letting $n$ go to infinity we get a tower converging to the $2$-adic numbers and have gone from characteristic $2$ to infinite characteristic.  At the next level, the fixed points of the Klein group $C_2\times C_2$ on the fourfold smash of $E=H\Z/2$  goes further.  Not only does the divisibility by $2$ increase, but the $C_2\times C_2$-fixed points also capture a shadow of Bott periodicity --- and so leave the realm of algebra and see a bit more of the complexity of the sphere spectrum.  

This is the beginning of a phenomenon in stable homotopy theory, transcending the notion of characteristic into an entire ``chromatic tower'', with the sphere spectrum sitting at the top. What has not always been so clear is the strong equivariant nature of chromatic homotopy theory  -- in the above toy example reflected in that the complexity (in this case the rank) of the symmetries corresponds to the height of the chromatic information.  

In hindsight, the success of B\"okstedt, Hsiang and Madsen's topological cyclic homology, $\TC$, partially relies on this phenomenon for the case of cyclic groups.  They build their theory on B\"okstedt's topological Hochschild homology, $\THH$, which was an exceedingly clever way to encode smash powers functorially -- even before the smash product existed.  More precisely, $\THH(A)$ is the realization of a simplicial spectrum, which in degree $q$ has as many smash factors of $A$ as there are simplices in the standard simplicial circle, and then uses the structure of the ring spectrum $A$ to make face and degeneracy maps.  Now, the cyclic action on all the smash powers ensure that $\THH$ is a $\T^1$-spectrum and the restriction maps ensure that one has  control over the equivariant structure with respect to the finite cyclic groups.

What B\"okstedt, Hesselholt and Madsen observe, first for $A=H\Z/p$ and $A=H\Z$ and later for more complex situation (with crowning achievement being the case of local number fields in \cite{HMlocal}) is that the fixed points of $\THH(A)$ conspire to shift the chromatic behavior and they see this in connection with the Lichtenbaum-Quillen conjecture for algebraic K-theory.  Later, Ausoni and Rognes state this more generally in the form of their {\em red-shift conjecture} and take it one step further.  

The red-shift conjecture has proven to be elusive, partially because it predicts a behavior for algebraic K-theory which was not underpinned by even a conjectural understanding of what underlying mechanism might be responsible.  This lack of foundations has been somewhat ameliorated, but we have still a long way to go.  The point of the current text is that we feel that we are closer to understanding the corresponding behavior in stable equivariant homotopy theory (with the link given by variants of the cyclotomic trace), and that the phenomenon is more likely to be understood at this level.

Another example where the equivariant structure of smash powers has turned out to capture crucial information can be found in Hill, Hopkins and Ravenel's solution of the Kervaire invariant one problem \cite{HHR}, where they had to study a certain four-fold smash power with a twisted action.  The reader will notice some convergent evolution regarding the relation of \cite{HHR} and the current text, whose first version appeared in fully independent form in the last author's 2011 PhD thesis.  The subsequent lengthy editing process has not disrupted this independence, but by now we must probably concede priority.  We will point out a few differences and similarities later on.

We propose that a fuller understanding of the functorial properties of smash powers, also indexed by spaces, is needed. Smash powers can be indexed functorially over sets, $S\mapsto\bigwedge_SA$, precisely when $A$ is a commutative ring spectrum, and in that case we can even index over arbitrary spaces (ignoring the equivariant structure, modelling $\THH$ as a smash power $\bigwedge_{\T^1}A$ (aka ``tensor'', $\T^1\otimes A$) is old and appears already in \cite{MSV}; the reason $\THH$ can be applied to non-commutative ring spectra is that the circle $\T^1$ has a cyclic model).  Some steps towards an equivariant understanding were taken in \cite{BCD} and \cite{CDD}, but these were based on B\"okstedt's (brilliant) hack for representing smash products, which has several technical disadvantages when trying to extract more delicate structure.  

The present text aims at understanding the full equivariant structure for smash powers over free $G$-spaces, for $G$ a compact Lie group, directly from its definition. This approach has the great advantage of allowing categorical access to the construction.  However, from a homotopical algebra point of view it is well known that {\em commutative} ring spectra are quite elusive, especially in an equivariant context.  We deal with this by keeping the categorical foundations smaller and -- we believe -- more transparent than what is usual, but provide it with a model structure with an inordinately liberal view on cofibrancy.  This results in a delicate navigation between the Scylla of destroying homotopy invariance and the Charybdis of not having sufficient power to perform induction arguments.   It is of course Quillen-equivalent to the usual stable model structure \cite{MM} and in view of the relationship with Shipley's ``convenient'' structure we call this model structure the $\Sp$-model structure.

The nontrivial extension we propose as a prototypical proponent of red-shift is as follows: let $\alpha\colon\T^n\to\T^n$ be an isogeny of the $n$-dimensional torus, \ie a surjective group homomorphism $\T^n\to\T^n$ with finite kernel.  For a commutative ring spectrum $A$ (for instance $A=H\Z/2$) we construct a ``restriction map''
 $$\xymatrix{(\bigwedge_{\T^n}A)^{\ker(\alpha)}\ar[r]&\bigwedge_{\T^n}A }$$
of commutative $\T^n$ ring spectra which is compatible as $\alpha$ varies and which has been shown to exhibit extensions sufficiently  intricate to explain some of the phenomena predicted by red-shift (see \eg \cite{Veen}).  Actually, this generalizes to a class of compact Lie groups $G$ with a finite normal subgroup $N$ being ``\orfi and having simplicial approximations'', resulting in a map
$$\xymatrix{(\bigwedge_{X}A)^{N}\ar[r]&\bigwedge_{X_N}A }$$
from the fixed points of the smash indexed by a free cofibrant $G$-space $X$ to the smash indexed over the orbits $X_N$ (the target could also be described as the $N$-orbits of the $X$-fold smash in the category of commutative ring spectra).  The ``cyclotomic structure'' promoted in much of the literature is nothing but the fact that if $\alpha$ is an isogeny then $\T^n\cong\T^n/{\ker{\alpha}}$.

One way to explain this state of affairs is that there is a well understood ``isotropy separation'' fiber sequence
$$\xymatrix{(\bigwedge_{X}A)_{hN}\ar[r]^{\text{norm}}&(\bigwedge_{X}A)^{N}\ar[r]&\Phi^N\bigwedge_{X}A},$$
where $\Phi^N$ signifies the so-called geometric fixed point construction.

  An important part of our contribution is to provide a \emph{natural} ``geometric diagonal'' $\Delta_X$ between $\bigwedge_{X_N}A$ and $\Phi^N\bigwedge_{X}A $ and show that the geometric diagonal is an isomorphism 
$$\xymatrix{{\Delta_X}\colon\bigwedge_{X_N}A\ar[r]^\cong&\Phi^N\bigwedge_{X}A}$$ 
for all $\Sp$-cofibrant commutative ring spectra $A$ (see Theorem~\ref{generalliegroup}).

Being $\Sp$-cofibrant is a very weak requirement (and, of course, all stable homotopy classes contain $\Sp$-cofibrant ones) --  much weaker than the cofibrancy conditions usually considered in the literature.  We restrict ourselves to free (cofibrant) $G$-spaces $X$ partially out of convenience and partially out of necessity.  It is not hard to construct maps of $G$-spaces where the geometric diagonal screws things badly up; however things are well behaved if the map is what is called {\em isovariant} which is automatically satisfied in the free setting.  This was discovered by the second author in a late revision, and we have refrained from incorporating this generalization of our results.

Another way of looking at the geometric diagonal is obtained by observing that the source $\bigwedge_{X_N}A$ is naturally isomorphic to the $N$-orbits of $\bigwedge_{X}A$ in the category of commutative orthogonal ring spectra (coproduct powers preserve colimits), and so exposes the geometric $N$-fixed point functor as an orbit construction.

Regardless, the fact that the geometric diagonal is an \emph{isomorphism} (as opposed to some sort of weak equivalence) is a very handy technical fact, allowing us a functorial inverse  automatically inheriting a trail of good properties on the nose.

The full naturality of the geometric diagonal is a key ingredient: it is not to hard writing up a candidate for a ``geometric diagonal'', but the book-keeping and the choice of appropriate cells to do the induction over are non-trivial.  For instance, if you let $G$ be finite and $X$ a \emph{fixed} finite free $G$-set, we can dispense of the ring structure on $A$.  Then, for sufficiently ``free'' orthogonal spectra $A$, just by looking at the definitions we find that a trivial application of the Yoneda Lemma shows that $\bigwedge_{X_N}A$ and $\Phi^N\bigwedge_{X}A $ are isomorphic - a fact that has been known for a long time and is here recorded as Proposition~\ref{geomfree}.  However, Definition~\ref{definitionofgeomfixedpoints} exposes this isomorphism as a particular case of a map which we use significant time on developing the naturality of with respect to $G$, $A$ and -- in case $A$ is a commutative ring spectrum -- $X$.  Even so, there are still non-trivial hurdles to be overcome in order to pass from finite discrete $G$ and $X$ to the most general situation.

Part of what makes this delicate is that none of the previously understood model structures on equivariant orthogonal spectra do much of a job in this setting and we need the full force of our new structure.

In his thesis, Kro \cite{Kro} studied the non-commutative case for $X$ the circle, in which case the classical stable structure is adequate.  Another difference between the commutative and associative case is that -- as we've already commented on -- in the commutative setting the full functoriality of our construction in $G$ and $X$ is a feature which requires a lot of careful attention.  In the non-commutative setting this is neither true nor essential, but in our setup it is indispensable.

\section*{Organization}
In Chapter~\ref{app.equivarianthtpy} we collect the results about unstable equivariant
spaces that we are going to need. In particular we recall Illman's
Triangulation Theorem and, inspired by Shipley \cite{Shconv}, we provide 
mixed model structures on \(G\)-spaces for \(G\) a compact Lie group
 depending on pairs of families of subgroups of \(G\).
   
In Chapter~\ref{ch:eos} we present mixed model structures on the category of
equivariant orthogonal \(G\)-spectra. We follow the by now standard way of passing
from so called level model structures on orthogonal spectra to stable
model structures in Chapter~\ref{ch:ss}. By focusing more on semi-free orthogonal spectra than
on free ones we gain flexibility in the choice of level model
structures.

We observe that there is a bijection between the set of
isomorphism classes of \(n\)-dimensional orthogonal representation of \(G\) and
the set of conjugacy classes of subgroups \(P\) of \(G \times \Or
{\R^n}\) with the property that the projection to the first factor
induces an embedding of \(P\) in \(G\). This observation leads us to
work with compatible  families of subgroups of the groups \(G \times \Or
{\R^n}\) instead of universes of \(G\)-representations, which in turn
leads to level model structures on orthogonal \(G\)-spectra different from the
ones usually obtained from universes.
 
In Chapter~\ref{ch:filtering} we review some properties of fixed
point spectra and introduced the concept of an \emph{orthogonally final subgroups}. We prove that the cases we are interested in are all orthogonally final, but embarrassingly enough we don't know of any examples that are \emph{not} orthogonally final. The importance of a subgroup $N\subseteq G$ being orthogonally final is that the geometric $N$-fixed points preserve both induction and restriction, as explained in Proposition~\ref{3810} and Proposition~\ref{geomfixedrestrict}.  

In Chapter~\ref{ch:finsmashcom} we study the equivariant properties of \emph{finite} smash powers (categorical
tensors) of a commutative orthogonal ring-spectrum $A$ indexed on a space $X$. 
 The main result of this chapter,
Theorem~\ref{relativeversion}, gives a detailed equivariant description of the cells present in smash powers. It is both
used in the 
construction of model structures on the category of commutative
orthogonal \(G\)-ring-spectra and in the identification of geometric
fixed points of  a smash-power as a smash-power~\ref{geomScofib}.
Here our choice of model structure is seen to be crucial in the sense that we retain full equivariant control over the smash powers.  We end the chapter by giving the foundations necessary for moving from the finite case to more general topological spaces.

In the final Chapter~\ref{ch:smashpow} 
we start by exploring how the functoriality of the smash powers implies that many of our results for finite smash powers continue to hold when the indexation is over spaces.  Most importantly for the equivariant study of iterated topological Hochschild homology, for $A$ an $\Sp$-cofibrant commutative orthogonal ring spectrum, Theorem~\ref{gocompact} states that the geometric diagonal $\bigwedge_{X_N}A\to\Phi^N\bigwedge_{X}A$ is an isomorphism when $N$ is the kernel of an isogeny of a torus $\T^n$ and $X$ is a free cofibrant $\T^n$-space.  This is generalized to other compact Lie groups in Theorem~\ref{generalliegroup}.  The question regarding what properties of the sphere spectrum are used to prove these results is addressed in Section~\ref{thm:geodiagoverR} where a parallel result for algebras over a commutative orthogonal ring spectrum $R$ with certain properties is discussed.

The Appendix~\ref{ch:cat} is used to fix some notation, recall some 
facts about topological model
categories 
and provide a result about assembling model structures. This
result is used in the construction of level model structures on the
category of orthogonal \(G\)-spectra. 

\section*{Notation}
We use the following notation:
\begin{enumerate}
\item $\N=\{0,1,\dots\}$, $\Z=\{\dots,-1,0,1,2\dots\}$, $\Q$, $\R$ and $\C$ are the sets of natural, integer, rational, real and complex numbers, with the usual algebraic structure.\index{N@$\N$}\index{Z@$\Z$}\index{Q@$\Q$}\index{R@$\R$}\index{C@$\C$}
\item For $n\geq 0$, $\R^n$ is \(n\)-dimensional Euclidean space (with the dot product). For $1\leq i\leq n$, $e_i\in\R^n$ forms the standard basis. We choose \(\R^{m+n}\) together with the canonical inclusions of \(R^m\) onto the first coordinates and of \(\R^n\) as the last coordinates 
as the direct sum \(\R^m \oplus \R^n\).
\item The phrase ``let \(V\) be a Euclidean space'' means ``let \(V = \R^n\) for some \(n \ge 0\)''.
\item The one-point compactification \(S^V\)\index{SV@$S^V$, $V$-sphere} of a real inner product
  space $V$ with its induced action by the orthogonal group $\Or{V}$
  is denoted $S^V$, and is referred to as the {\em
    $V$-sphere}\index{Vsphere@"$V$-sphere}. It is pointed at the added point $\infty$. Given \(x \in V\) we
  consider \(x\) as an element \(x \in S^V\) through the inclusion of
  \(V\) in its one-point compactification \(S^V\). Given \(x,y \in
  S^V\) we allow ourselves to write \(x + y\) with the convention that
  \(\infty + x = \infty = x + \infty\).  Whenever the phrase ``$x\in S^W$ is in the orthogonal complement of $V$ in $W$'' is applied to $V=W$, it will mean ``$x\in\{0,\infty\}$''.
\item The category $\catU$\index{U@$\catU$, compactly generated weak Hausdorff spaces} is the category of compactly generated weak
  Hausdorff spaces and continuous functions (see \eg \cite{St}), and $\catT$\index{T@$\catT=*/\catU$} is the category of based spaces in $\catU$. Unless we explicitly state otherwise, a \emph{space} is an object in $\catT$ and maps between spaces are assumed to be continuous (and base point preserving).  If $X\in\catU$ is an unbased space, then $X_+\in\catT$ is the space obtained by adding a disjoint base point.
\item When we use the symbol \(G\) to denote a group it will always be a compact Lie group. Subgroups of compact Lie groups are always assumed to be closed.  

\end{enumerate}

\thispagestyle{empty}
\newpage\thispagestyle{empty}
\thispagestyle{empty}

\newcommand{\setsminus}{-}
\chapter{Unstable equivariant homotopy theory} \label{app.equivarianthtpy}
In this section we will recall results from (unstable) equivariant
homotopy theory. We begin with a recollection on model structures on
$\G$-spaces and will continue with some consequences of the results of
Illman \cite{Ill}. We work in the pointed setting $\catT$ (the
category of based, compactly generated, weak Hausdorff
spaces) as this is the more important
case for us, but all results could be stated in the unbased category
$\catU$ as well. We refer to the objects of \(\catT\) as spaces.
\section{\G -Spaces}
\label{G-spaces_section}
Let $\G$ be a compact Lie group, and let $\catGT$ be the category of
spaces with a continuous action of $\G$. Considering \(G\) as a one
object category, \(\catGT\)\index{GT@$\catGT$, category of $G$-spaces} is the category of continuous functors
from \(G\) to \(\catT\). 
  \begin{dfn}\index{genuine cofibration}\label{defineIG}
    The set
$$I_G\index{IG@\(I_G\)}= \{(S^{n-1} \times G/H)_+ \subseteq (D^{n} \times G/H)_+\mid n
\ge 0, H\subseteq G\text{ closed subgroup}\}$$ 
is called the set of {\em generating genuine cofibrations}.
An {\em $I_G$-cell complex} is a (possibly transfinite) composition of coproducts of pushouts of maps in $I_G$.  
A map of \(G\)-spaces is
a {\em genuine cofibration}\index{genuine!cofibration}\index{cofibration!genuine} if it is a retract of an \(I_G\)-cell
complex.  A $G$-space $Z$ is {\em genuinely cofibrant}\index{genuinely cofibration}\index{cofibrant!genuine} if the map from the one-point set is a genuine cofibration.
  \end{dfn}
The genuine cofibrations will be the cofibrations in the model structure of Definition~\ref{genuineGTstructure}.  There is a competing structure, occasionally called the {\em Borel structure}\index{Borel model structure}; there the cofibrations are the genuine cofibrations where only the trivial subgroup $0=H\subseteq G$ is allowed for the cells.  We may in the future refer to these as {\em free cofibrations of $G$-spaces}.\index{free cofibrant $G$-space} 
Limits and colimits in $\catGT$ are formed in $\catT$ and then given
the induced $\G$-action. We say that \(G\) {\em acts freely}\index{act freely}\index{free action} on a $G$-space
\(Z\) if the stabilizer group $\Stab_z=\{g\in G\,|\,gz=z \}$\index{Stabz@$\Stab_z$, stabilizer subgroup}  of every point $z\in Z$ except the base--point is trivial (since we are considering based spaces, the base point will always be a fixed point).
\begin{dfn}\label{fixPointDef}
Given a surjective group homomorphism \(p \colon G \to G'\)
  with kernel \(N\), the continuous functor $p^* \colon G'\catT
  \rightarrow \catGT$ equipping a \(G'\) space with the $\G$-action
  through \(p\) has both a left and a right adjoint.
The left adjoint \[(-)_N\colon \catGT \rightarrow G'\catT,\]
assigns to a $\G$-space $X$ its \(N\)-\emph{orbit space}\index{orbit
  space} $X_N$, and the right adjoint  
\[(-)^N\colon \catGT \rightarrow G'\catT,\]
assigns the subspace $X^N$ of $X$ of \emph{$N$-fixed points}\index{fixed point space}.
\end{dfn}
The fact that they are adjoints, implies in particular that $(-)_N$
preserves colimits and $(-)^N$ preserves limits, but even more is
true (cf.~\cite[III.1.6]{MM}, or \cite[Proposition 1.2]{Malkiewich}): 
\begin{lemma}\label{fixedprop}
\label{H-fixed}
In the situation of Definition~\ref{fixPointDef} the functor $(-)^N$ preserves coproducts, pushouts of diagrams one leg of which is a closed inclusion, and colimits along sequences of closed inclusions. For $X$ and $Y$ in $\catGT$, we have $(X \smash Y)^N = X^N \smash Y^N$.
\end{lemma}
Recall that 
a \hur cofibration
is a map \(i \colon A \to X\) in
\(\catGT\) so that the canonical map \(Mi = X \coprod_{A} A \wedge I_+
\to X \wedge I_+\)\index{Mi@$Mi = X \coprod_{A} A \wedge I_+$, mapping cylinder} has a retract.
The following technicality proves to be helpful in several places:
\begin{lemma}\label{fixprescof}
In the situation of Definition~\ref{fixPointDef} the fixed point functors $(-)^N$ preserve Hurewicz cofibrations.
\end{lemma}
\begin{proof}
  By Lemma~\ref{H-fixed} there are homeomorphisms \((X \wedge I_+)^N \cong
  X^N \wedge I_+\) and \((Mi)^N \cong M(i^N)\), so that \((Mi)^N \to
  (X \wedge I_+)^N\) has a retraction.
\end{proof}

Given a subgroup \(H\) of \(G\), the left adjoint of the restriction $i^*\colon \catGT \rightarrow \catHT$ is given by inducing up: 
\begin{dfn}\label{inducingup}
Let \(H\) be a closed subgroup of \(G\). For an $H$-Space $Y \in
\catHT$ the smash product $\G_+ \smash Y$ has 
an action of $G \times H$, given by the formula
\((g,h)(a \wedge y) = gah^{-1} \wedge hy\) for \((g,h) \in
G \times H\) and \(a \wedge y \in G_+ \smash Y\). The \emph{induced $G$-space}\index{induced $G$-space} is the $H$-orbit $G$-space \[G_+ \smash_H Y=(G_+\smsh Y)_H.\]\end{dfn}
While the inducing up functor is generally not symmetric monoidal, there is an important compatibility property with the smash product of spaces which one checks by inspection:
\begin{lemma}\label{inducingsmashcomp}
 If $X$ is an $H$-space and $Y$ a $G$-space there is a natural $G$-equivariant homeomorphism
\[(G_+ \smash_H X)\smash Y \cong G_+ \smash_H (X \smash i^*Y),\qquad([g,x],y) \leftrightarrow [g,(x,g^{-1}y)].\]
\end{lemma}
\begin{lemma}\label{lem:inducedfixpoints}
  Let $N\subseteq H\subseteq G$ be inclusions of closed subgroups with $N$ normal in $G$.  If $X$ is an $H$-space then the natural map
$$G/N_+\smsh_{H/N}X^N\to[G_+\smsh_H X]^N$$
adjoint to the $N$-fixed points of the $H$-map $X\to G_+\smsh_HX$ sending $x\in X$ to $[1,x]$ is an isomorphism.
\end{lemma}
\begin{proof}
  Sending $[g,x]\in[G_+\smsh_HX]^N$ to $[[g],x]\in G/N_+\smsh_{H/N}X^N$ is a well defined inverse: since $[g,x]$ is an $N$-fixed point, $N\subseteq H$ and $N$ normal, we have for any $n\in N$ that 
$$[g,x]=[(gng^{-1})g,x]=[gn,x]=[g,nx]$$ 
and so $x$ is an $N$-fixed point.
\end{proof}

\subsubsection{Orbits and Fixed Points for semi-direct Products}\label{subsectsemidirprod}
Actions of groups that are semi-direct products
appear in various places throughout our work.
We will recall a few elementary properties, before investigating the more complicated interactions of the orbit and fixed point functors that play a role in computing the fixed points of smash powers. We restate the definition, in order to fix some notation:
\begin{dfn}\label{defsemidirectprod}
 Let $(G,\cdot,e)$ be a compact Lie group acting on another
 compact Lie group $(O,\odot,E)\in \catU$ via a group homomorphism $\phi\colon G
 \rightarrow \Aut(O)$. 
 The \emph{semi-direct product $G \ltimes (O,\phi)$}\index{semi-direct product}\index{GOphi@$G \ltimes (O,\phi)$, the semi-direct product} is the
 product space $G \times O$ equipped with the multiplication defined
 by 
 \begin{eqnarray*}
   (G \ltimes (O,\phi))\times (G \ltimes (O,\phi))& \rightarrow &(G \ltimes (O,\phi))\\
   (g,A),(h,B)&\mapsto&(g\cdot h,A\odot \phi(g)(B)).
 \end{eqnarray*}
 When \(\phi\) is implicit in the context we write \(G \ltimes O\)\index{GO@$G \ltimes O$, the semi-direct product}
 instead of \(G \ltimes (O,\phi)\).
\end{dfn}
For the rest of this section \(G\) and \(O\) are compact Lie groups
with a group homomorphism \(\phi \colon G \to \Aut(O)\).
\begin{remark}\label{semidirectactionrem}
  Specifying an action of the semi-direct product \(G \ltimes O\) on a space
  \(Z\) is the same as a giving actions of \(G\) and \(O\) on \(Z\)
  such that the map \(O \to \catT(Z,Z)\) defining the action of \(O\) on
  \(Z\) is a
  \(G\)-map.   
\end{remark}
\begin{remark}\label{shearingsemi}
  If \(G\) acts on \(O\) through inner automorphisms, that is, if
  there is a homomorphism \(\psi \colon A \to O\) such that
  \(\phi(g)(A) = \psi(g) \odot A \odot \psi(g^{-1})\), then we write \(G
  \ltimes_{\psi} O\)\index{GpsiO@$G \ltimes_{\psi} O$, the semi-direct product for inner actions} instead of \(G \ltimes (O,\phi)\). In this
  situation then there is
  an isomorphism \(G \ltimes_{\psi} O  \to G \times O\) of topological
  groups taking \((g,A)\) to
  \((g,A \odot \psi(g))\). 
\end{remark}
It is usually cumbersome to explicitly write the signs
``$\cdot$'' and ``$\odot$'' for the multiplications of $G$ and $O$, so we often omit them. The following properties are elementary:
\begin{lemma}
\begin{enumerate}
 \item Mapping $A \in O$ to $(e,A) \in G \ltimes O$ embeds $O$ as a (closed) normal subgroup.
\item Mapping $g\in G$ to $(g,E)\in G \ltimes O$ embeds $G$ as a closed subgroup.
\item For $A \in O$ and $g\in G$, the following elements of the semi-direct product are equal:
\[(e,\phi(g)(A)) = (g,E)(e,A)(g,E)^{-1}\]
\item $G \times \{E\}$ is normal in $G \ltimes O$ if and only if $\phi$ is
  trivial. In this situation the semi-direct product is actually the
  direct product. 
\item The projection $\proj_1\colon G \ltimes O \rightarrow G$ to the first factor is a group homomorphism.
\end{enumerate}
\end{lemma}
Motivated by the first two points in the lemma, we identify \(A \in
O\) with \((e,A) \in G \ltimes O\) and \(g \in G\) with \((g,E) \in G
\ltimes O\). Under this identification, the third point of the lemma reads $gAg^{-1}=\phi(g)(A)$.
\begin{lemma}\label{secfactinj}
  Let \(Z\) be an \(G \ltimes O\)-space. Given \(z \in Z\), the \(O\)-orbit
  \(Oz\) is free if and only if the composition
  \begin{displaymath}
    \Stab_z \subseteq G \ltimes O \xto {\proj_1} G
  \end{displaymath}
   from the stabilizer subgroup with respect to the $G\ltimes O$-action is injective.
\end{lemma}
\begin{proof}
  The \(O\)-orbit \(Oz\) is free if and only if the stabilizer
  subgroup \(\Stab_z^O\) of \(O\) is trivial. Under the embedding of
  \(O\) in \(G \ltimes O\), this stabilizer subgroup corresponds to
  the kernel
  \(\Stab_z \cap (\{e\} \ltimes O) = \Stab_z \cap \proj_1^{-1}(e)\) of the
  composite homomorphism  in question.
\end{proof}
For any space $Z$ with an action of the semi-direct product, its orbit space $Z_O$ inherits an action of $G$.
We want to investigate in how far taking such $O$-orbits commutes with taking fixed points with respect to subgroups of $G$.
\begin{prop}\label{fact}
  Let \(Z\) be an \(G \ltimes O\)-space.
  Assume that $O$ acts freely on $Z$.\\
Then the canonical map from the quotient of the fixed points into the fixed points of the quotient \labeleq{factmap}{({Z^{G}})_{O^{G}} \rightarrow \left({Z}_{O}\right)^{G}} is injective.\\
\end{prop}
\begin{proof}
Assume that two non-base points $[z_1]$ and $[z_2]$ in $({Z^G})_{O^G}$
map to the same element in the target, \ie $z_1 = A z_2$, for some $A$
in $O$. Then for any $g$ in $G$ we have that 
\[Az_2 = z_1 = gz_1 = g( A z_2) = (\phi(g) A) (g z_2) = (\phi(g) A) z_2.\]
Since the $O$-action on $Z$ was free, this implies that $A =
\phi(g) A$ for all $g \in G$. Thus $A \in O^G$ and $[z_1]= [z_2]$.
\end{proof}
Surjectivity can not be guaranteed in this generality, as the following example illustrates:

\begin{ex}\label{factmapnon}
  Let $Z=S(\C)_+$ where \(S(\C)\)\index{SC@$S(\C)$, the unit circle}
  is the unit circle in the complex
plane. Let $O = \Z/4$ and $G = \Z/2$ such that the action of the 
non trivial element in $G$ maps an element of $O$ to its inverse. In
particular, $G \ltimes O$ is the dihedral group $D_4$\index{D4@$D_4$,
  dihedral group}.  Then $O$ acts freely on $S(\C)$ through rotations by $90$ degrees, $G$ acts by complex conjugation, and one checks that the actions compatibly fit together into an action of $D_4$. We take a closer look at source and target of the map \({({Z^{G}})_{O^{G}} \rightarrow \left({Z}_{O}\right)^{G}}\).

On one hand, $O^G$ is the subgroup of self-inverse elements $\Z/2 \subset
\Z/4$, \ie generated by the rotation by $180$ degrees.
The $G$-fixed
point space $(S(\C))^G=\{-1,1\}$ consists of two points which are in the same
$O^G$-orbit, \ie the source of \({({Z^{G}})_{O^{G}} \rightarrow
  \left({Z}_{O}\right)^{G}}\) contains only one non-base point.

On
the other hand, taking orbits first, we see that ${Z}_{O}$ is
isomorphic to $S(\C)$, with the action of $G$ again given by complex
conjugation, \ie target of \({({Z^{G}})_{O^{G}} \rightarrow
  \left({Z}_{O}\right)^{G}}\) consists of two non-base points, such
that the map can not be surjective. 
\end{ex}
The intuition behind the failure in surjectivity is that there are
``diagonal'' copies of $G$ in $G \ltimes O$, and points with isotropy
type of such a diagonal copy contribute to the target of
\({({Z^{G}})_{O^{G}} \rightarrow \left({Z}_{O}\right)^{G}}\) but not to the source. Motivated by this, we will give
a sufficient condition for the surjectivity of
\({({Z^{G}})_{O^{G}} \rightarrow \left({Z}_{O}\right)^{G}}\). First, consider the simple case of only one $G
\ltimes O$-orbit, and take a closer look at the target space: 
\begin{lemma}
Let \(Z\) be a \(G \ltimes O\)-space consisting  of a
single $O$-free $G \ltimes O$-orbit, or in other words (by Lemma~\ref{secfactinj}), $Z= \faktor{G \ltimes O}{P}_+$
for some closed subgroup $P$ of $G \ltimes O$ with the projection
$\proj_1: P \rightarrow G$ to the first factor injective. Then
$({Z}_{O})^G$ contains at most two element, and it contains two elements if and only if the projection $\proj_1: P \rightarrow G$ is an isomorphism.
\end{lemma}
\begin{proof}
 Since the projection $\proj_1: P \rightarrow G$ is
 injective we get an isomorphism of $G$-spaces ${Z}_{O} \cong
 \faktor{G}{\proj_1{P}}_+$ and so has exactly
 two $G$-fixed point if and only if $\proj_1(P)=G$, and one
 \(G\)-fixed point otherwise.
\end{proof}On the other hand, the following elementary fact gives a characterization of the source:
\begin{lem}
Let $H$ be a topological group with subgroups $P$ and $G$. Then the space $\left(\faktor{H}{P}_+\right)^G$ is non-trivial, if and only if $G$ is subconjugate to $P$.\\
More precisely, $\left(\faktor{H}{P}\right)^G$ is a quotient of the subspace of all those elements $h \in H$, that conjugate $G$ into $P$.
\end{lem}
\begin{proof}
 Let $hP$ be a point in the orbit space. Then $hP$ is $G$-fixed, if and only if for all $g\in G$ we have $ghP = hP$, or equivalently $h^{-1}gh \in P$.
\end{proof}
\begin{prop}\label{crucialcommutefixandorbits}
Let $Z$ be a genuinely cofibrant $G \ltimes O$-space. Suppose \(O\) acts
freely on \(Z\). Then the map
\({({Z^{G}})_{O^{G}} \rightarrow \left({Z}_{O}\right)^{G}}\) is an isomorphism if
for every stabilizer subgroup $P$ of an orbit appearing in the
cell-decomposition of $Z$ the following two statements are logically
equivalent:
\begin{enumerate}
\item the
projection to the first factor induces an isomorphism
$\proj_1(P)\cong G$
\item $G \ltimes \{E\} \subset G \ltimes O$ is
subconjugate to $P$. 
\end{enumerate}
\end{prop}
\begin{proof}
 Note that both taking fixed points and taking orbits preserves the cell-complex construction by Lemma~\ref{fixedprop}. Hence the natural map \({({Z^{G}})_{O^{G}} \rightarrow \left({Z}_{O}\right)^{G}}\) induces a natural isomorphism of cell diagrams, hence an isomorphism on the transfinite composition.
\end{proof}
\begin{cor}\label{factisochar}
Suppose \(G \ltimes O\) has the property that for every
subgroup \(P\) of \(G \ltimes O\) the
projection to the first factor induces an isomorphism
$\proj_1(P)\cong G$ if and only if $G \ltimes \{E\} \subset G \ltimes O$ is
subconjugate to $P$. 
Then for any genuinely cofibrant $G \ltimes O$-space the natural map \({({Z^{G}})_{O^{G}} \rightarrow \left({Z}_{O}\right)^{G}}\) is an isomorphism.
\end{cor}
The main example we want to apply Corollary~\ref{factisochar} to is
the following: 
\begin{ex}\label{mainex}
Let \(V\) be a Euclidean space with automorphism group \(\Or V\),
$G$ a discrete group, $X$ a discrete free $G$-space and 
\(\varphi\colon G \to \Or V\) a representation of \(G\) (in which case it is customary to abuse language and say that ``$V$ is an orthogonal $G$-representation''). 
Then \(G\) acts on \(\Or V\) through the composition 
\(G \xto \varphi \Or V \to \Aut(\Or V)\) of
\(\varphi\) and the conjugation homomorphism letting \(\Or V\) act on
itself by inner automorphisms. If $Q\subseteq\Or V$ is a $G$-invariant subgroup, let  $O=\pri{X}{Q}$, with $G$ acting by conjugation: \ie \(g \in G\) acts on
\(\{M_x\}_{x \in X} \in O\) by taking it to 
\(\{\varphi(g)M_{g^{-1}x}\varphi(g^{-1})\}_{x\in X}\). The
conditions of Corollary~\ref{factisochar} are satisfied by the
semi-direct product $G \ltimes \pri{X}{Q}$:
\\ 

\noindent Indeed, consider a subgroup $P\subseteq G \ltimes \pri{X}{Q}$ with $\proj_1(P) = G$
and let 
\begin{eqnarray*}
  \psi\colon\,\,\,\, G&\rightarrow&P \subset G
  \ltimes  \prod_{X}{Q} \\g &\mapsto &(g,\{A^g_x\}_{x\in X})
\end{eqnarray*}
 be the inverse of $\proj_1$. 
Given $(g,\{A_x^g\}) \in
P$, the fact that $\psi$ is a group homomorphism, amounts to 
 the formula 
\labeleq{formulas}{A^g_x \varphi(g)A^h_{g^{-1}x} \varphi(g^{-1}) =
  A^{gh}_x\,\,\,\,\,\,\,\,\,\,\forall g,h \in G,\,\,x\in X} 
Now chose a system of representatives $R$ for the $G$-orbits in
$X$. Let $B=\{B_x\} \in  \pri{X}{Q}$ be the element given by 
\[B_x = B_{hr} \defas{} \varphi(h) A_r^{h^{-1}} \varphi(h^{-1}),\] 
where $x=hr$ is the unique presentation of $x$ with $h \in G$ and $r
\in R$. 
The formal definition of multiplication in 
{$G \ltimes \pri{X}{Q}$}
gives: 
\begin{displaymath}
  (e,\{B_x\})(g,\{A_x^g\}) = (g,\{B_xA_x^g\})
\end{displaymath}
and
\begin{displaymath}
  (g,E)(e,\{B_x\}) = (g,\{\varphi(g)B_{g^{-1}x} \varphi(g^{-1})\}).
\end{displaymath}
If \(x = hr\), then \(g^{-1}x = g^{-1}h r\) and \(B_{g^{-1}x} =
\varphi(g^{-1}h)A_r^{h^{-1}g}\varphi(h^{-1}g)\). By equation~\ref{formulas} we
have \(A_r^{h^{-1}g} = A_r^{h^{-1}}\varphi(h^{-1})A^g_{hr}\varphi(h)\) so
\begin{displaymath}
  \varphi(g)B_{g^{-1}x} \varphi(g^{-1}) =
  \varphi(h)A_r^{h^{-1}g}\varphi(h^{-1})   =
  \varphi(h)A_r^{h^{-1}}\varphi(h^{-1})A^g_{hr}   = 
  B_x A^g_x.
\end{displaymath}
Hence $P$ is subconjugate to $G \ltimes \{E\}$ via $B\in O$. Because
\(\proj_1\) is surjective this implies that \(P\) is conjugate to \(G \ltimes \{E\}\) and therefore taking $O$-orbits commutes with taking $G$-fixed points.
\end{ex}
\begin{prop}\label{regularfix}
{Let $G$ be a discrete group, $X$ a free discrete $G$-space, $V$
a finite dimensional $G$-representation and $Q\subseteq\Or V$ a $G$-invariant subgroup. Let as in Example~\ref{mainex}
$O=\pri{X}{Q}$, with $G$ acting by conjugation. }

For every \(O\)-free genuinely cofibrant $G \ltimes O$-space $Z$ and for every
subgroup \(H\) of \(G\), taking $O$-orbits commutes with taking
$H$-fixed points in the sense that the canonical map
\[{({Z^{H}})_{O^{H}}} \rightarrow \left[{Z}_{O}\right]^{H},\]
is an isomorphism.
\end{prop}
\begin{proof}
Note that for subgroups $H \subset G$, any free $G$-set is a free
$H$-set, and any genuinely cofibrant 
{$G \ltimes O$-space
is also genuinely cofibrant as a $H \ltimes O$-space}
 by
Corollary~\ref{illsubgroup}. We can then apply Corollary~\ref{factisochar} with
the help of Example~\ref{mainex} for all choices of $H$. 
\end{proof}
\section{Illman's Triangulation Theorem}\label{illtriangsect}
In several places we will need to check that certain spaces are genuinely cofibrant in the sense of Definition~\ref{defineIG}.
By the general theory
for model categories it 
will usually suffice to understand the class of $I_\G$-cell complexes
in $\catGT$ (cf.~Definition~\ref{defineIG} and~\ref{cellularMaps}). For the rest
of the section, we restrict to  
the case of $\G$ a compact Lie group. For convenience, we will recall
the statements and the relevant definition from \cite{Ill} before we
give some corollaries. 
\begin{dfn}
\label{G-isotropy}
Let $X$ be a $\G$-space. Given an orbit $[x] \in X_\G$, define the 
\emph{$\G$-isotropy type} of $[x]$ as the conjugacy class of the stabilizer subgroup $\Stab_x=\{g\in G\,|\,gx=x\}$ of $\G$.

Since the stabilizer subgroups of elements in the same orbit are all conjugate, the $G$-isotropy type indeed only depends on the element $[x] \in X_\G$. 
\end{dfn}
The following results concern \(I_G\)-cell complexes in the sense of
Definition~\ref{defineIG} and~\ref{cellularMaps}.
\begin{theorem}[{\cite[5.5, 6.1,
    7.1]{Ill}}] \label{illmantriang}\label{illcomplex}\label{illtriangmfd} 
Let $X$ be space with an action of a compact Lie group \(G\) and a
triangulation $t\colon K \rightarrow X_G$ of the orbit space $X_G$,
such that the $\G$-isotropy type is constant on open simplices, \ie
for each open simplex $\mathring{s}$ of $K$ the $\G$-isotropy type is
constant on $t(\mathring{s}) \subset X_G$. Then $X$ admits a structure
of a $\G$-equivariant $CW$-complex. In particular $X$ is an $I_\G$-cell complex.

Furthermore, if $M$ is a smooth $\G$-manifold (with or without boundary), then the orbit space $M_G$ {\em does} admit a triangulation such that the $\G$-isotropy type is constant on open simplices, and consequently $M$ admits a structure of a $\G$-equivariant $CW$-complex.
\end{theorem}

\begin{remark}\label{spheresareCW}
  In particular, for every representation \(V\) of a compact Lie group \(G\), its one-point compactification \(S^V\) is a \(G\)-CW-complex.    
\end{remark}
\begin{cor}\label{illsubgroup}
Let $\G$ be a compact Lie group and $\h$ a closed subgroup. Then \(G\) is an \(H\)-CW complex and 
an $I_G$-cell complex is an $I_H$-cell complex.
\end{cor}
\begin{proof}
By induction on the cell structure, and since smashing with a space
preserves colimits, it suffices to show that for any closed subgroup
$K \subset \G$ the orbit space $\faktor{\G}{K}$ is an $I_\h$-cell
complex. However, $\faktor{\G}{K}$ is a smooth \(G\)-manifold, and
thus also a smooth \(H\)-manifold. Theorem~\ref{illtriangmfd} now implies that 
$\faktor{\G}{K}$ is an \(I_H\)-cell complex.
\end{proof}
The following result is elementary:
\begin{lemma}\label{isotropytypeofcof}
  Let \(Y\) be a \(G\)-retract of an \(I_G\)-cell complex \(X\). The cells of \(X\) are of the form
  \(D^n \times G/L\), where \(L\) is the isotropy group of an element of
  \(X\). Conversely, if \(L\) is the isotropy group of an element of
  \(Y\), then there is at least one cell of \(X\) isomorphic to a cell
  of the form \(D^n \times G/L\).
\end{lemma}
\begin{cor}\label{illproducts}
Let $\h$ and $K$ closed subgroups of $\G$. The product
$\faktor{\G}{\h} \times \faktor{\G}{K}$ is again an $I_\G$-cell
complex, and the only orbit types that appear are of the form $\faktor{\G}{L}$, with $L$ subconjugate to both $\h$ and $K$.
\end{cor}
\begin{proof} 
Note that $\faktor{\G}{\h} \times \faktor{\G}{K}$ is isomorphic to $\faktor{(\G \times \G)}{(\h \times K)}$ and embed $\G$ into $\G \times \G$ as the diagonal (closed) subgroup, so that Corollary~\ref{illsubgroup} gives that $\faktor{\G}{\h} \times \faktor{\G}{K}$ is an $I_G$-cell complex. 
For the statement about orbit types, we check what kind of stabilizer subgroups can appear in the product $\faktor{\G}{\h} \times \faktor{\G}{K}$ and use Lemma~\ref{isotropytypeofcof}. In fact, if $L$ is the stabilizer of $([g_1],[g_2])$, then we have that $Lg_1 \subset g_1\h$ or equivalently that $L$ is subconjugate to $\h$. The analogous argument for $K$ finishes the proof.
\end{proof}

\section{Mixed Model Structures}

In this section \(G\) is a compact Lie group.  When referring to an unspecified model structure on $\catT$, the default is the Quillen model structure with Serre fibrations and weak equivalences. 
We want to introduce model structures on the category of $G$-spaces where the weak equivalences are defined to be the maps that remain weak equivalences after forming fixed point spaces with respect to subgroups in a given family of subgroups of $G$, {\em but} where the cofibrations are determined by some {\em other} family of subgroups of $G$.  The fibrations will have to find a compromise and we refer informally to this situation as a {\em mixed model structure}.

\begin{dfn}\label{subsectfamilies}
  A set $\aml$ of closed subgroups of $G$ is a (closed) \emph{family}\index{closed family}\index{family} if it is closed under taking subgroups and conjugates. 
\end{dfn}

It will be convenient to have some homotopical properties of Hurewicz
cofibrations (\ref{dfn:h-cof}) at hand when working with topological model categories.
The following lemma sums up what we will use in the $G$-equivariant context
(a map $X\to Y$ of $G$-spaces is said to be an $\aml$-equivalence if for all $H\in\aml$ the map of fixed points $X^H\to Y^H$ is a weak equivalence of spaces, cf.~Definition~\ref{def:amlmodel} below): 
\begin{lemma} Let $\aml$ be a family of subgroups of $G$. \label{propsfml}
\begin{enumerate}
\item sequential colimits of $\aml$-equivalences along Hurewicz cofibrations are $\aml$-equivalences
\item pushouts of $\aml$-equivalences along Hurewicz cofibrations between well based $G$-spaces are $\aml$-equivalences
\item the cube lemma~\ref{gen.cube} holds for $\aml$-equivalences and Hurewicz cofibrations between well-based $G$-spaces
\item the cube lemma holds for $G$-homotopy equivalences and Hurewicz cofibrations.
\end{enumerate}
\end{lemma}
\begin{proof}
All points follow directly from the analogous statement about weak
equivalences in $\catT$ summarized in Corollary~\ref{hcofcubespaces}, and Lemma~\ref{fixedprop}. 
\end{proof}
 
\begin{dfn}
  If $\aml$ is a family of subgroups of $G$, then $\aml\catT$ is the functor category from the {\em orbit category} associated with $\aml$, which we by abuse of notation also call $\aml$:
  \ie if $K,H$ are in the family $\aml$, then the space of morphisms from $K$ to $H$ is  
  $$\aml(K,H) = G\catT(G/H_+,G/K_+)$$
  (so that the endomorphisms $\aml(H,H)$ can be identified with the Weyl group $W_{\!G}H=N_GH/H$ and $H\mapsto G/H_+$ gives an equivalence to a full subcategory of $\catGT$).
\end{dfn}

Note that for the family $\{e\}$ containing only the trivial group we have a natural identification $\{e\}\catT=G\catT$ and we write $\mathrm{ev}_e\colon\aml\catT\to G\catT$ for the map induced by the inclusion $\{e\}\subseteq\aml$ sending $X\in\aml\catT$ to the $G$-space $\mathrm{ev}_eX=X(e)$.

For a general subgroup $H\subseteq G$, the inclusion of the full category containing only $H$ induces a map $\aml\catT\to W_GH\catT$, but we will use $\mathrm{ev}_H\colon\aml\catT\to\catT$ generically to refer to the composite with map forgetting the $W_GH$-action

Consider the evaluations
$\xymatrix{\catT&\ar[l]_{\mathrm{ev}_H}
  \aml\catT\ar[r]^{\mathrm{ev}_e}&G\catT}$ for varying $H\in\aml$ and their left and right adjoints, some of which are given names in the following diagram
$$\xymatrix{\catT\ar@<2ex>[rr]^{\Fr{\aml}{H}}\ar@<-2ex>[rr]_{}&&
  \ar[ll]_{\mathrm{ev}_H}\aml\catT\ar[rr]^{\mathrm{ev}_e}&&G\catT,\ar@<2ex>[ll]_{\mathrm{fix}}\ar@<-2ex>[ll]
  }$$
  where for convenience we notice that $\Fr{\aml}{H}Y(K)=(G/H)^K_+\smsh Y=\aml(H,K)\smsh Y$ and $(\mathrm{fix}Z)(K)=Z^K$.

  \begin{dfn}\label{def:amlmodel}
    The \emph{projective model structure}\index{projective model structure}\index{model structure!projective} on $\aml\catT$ is the model structure fetched (\eg by Theorem~\ref{enrhirsch}) via the pairs $(\Fr{\aml}{H},\mathrm{ev}_H)_{H\in\aml}$, \ie where a map $X\to Y$ in $\aml\catT$ is a weak equivalence (fibration) if for all $H\in \aml$ the map $X(H)\to Y(H)$ is a weak equivalence (Serre fibration) and the generating (acyclic) cofibrations are the maps $\Fr{\aml}{H}i=\aml(H,-)\smsh i$ for $i$ a generating (acyclic) cofibration in $\catT$.

    The \emph{$\aml$-model structure}\index{Amodel structure@$\aml$-model structure on $\catGT$}\index{model structure!$\aml$-} on $\catGT$ is likewise the model structure obtained via the adjoint pair $$\xymatrix{\catT\ar@<1ex>[rr]^{\Fr{}{H}=\mathrm{ev}_H\Fr{\aml}{H}}&&G\catT\ar@<1ex>[ll]^{\mathrm{fix}^H}}$$
  for $H\in\aml$,
  \ie where a map $X\to Y$ of $G$-spaces is an $\aml$-equivalence ($\aml$-fibration) if for all $H\in\aml$ the map $X^H\to Y^H$ is a weak equivalence (Serre fibration).  The sets of generating (acyclic) $\aml$-cofibrations are
  \index{IA@\(I_{\aml}\)} \index{JA@\(J_{\aml}\)} 
\begin{align*}
  I_{\aml}&=\{G/H_+ \smsh i\mid H\in\aml, i\text{ a generating cofibration for }\catT\}\\
  J_{\aml}&=\{G/H_+ \smsh j\mid H\in\aml, j\text{ a generating acyclic cofibration for }\catT\},
\end{align*}
  
  \end{dfn}

  For the rest of this section we fix an inclusion \(i\colon \bml \subseteq \aml\) of families of subgroups of \(G\) and consider the induced functor $i^*\colon\aml\catT\to\bml\catT$ and the left and right adjoints
  $$\xymatrix{
  \aml\catT\ar[rr]^{i^*}&&\bml\catT.\ar@<2ex>[ll]\ar@<-2ex>[ll]_L
  }$$
We use this adjunction to ``mix'' $\aml$- and $\bml$-model structures, which we subsequently pull back to $\catGT$ via the following observation (the second part is essentially Elmendorf's theorem which follows since the generating sets of (acyclic) cofibrations are preserved, a version of which can be found in \cite[Theorem III.1.8]{MM}).
\begin{lemma}
The induced functor \(i^*\colon \aml\catT \to \bml\catT\)  is a right Quillen functor with respect to the {\em projective model structures}\index{projective level model structure}\index{model structure!projective level}\index{level!projective model structure}.

 The fixed point functor $\mathrm{fix}\colon G\catT\to\aml\catT$ is a right Quillen equivalence from the $\aml$-model structure to the projective model structure.
\end{lemma}

The following annoying nomenclature is too common to avoid. 
\begin{dfn}\label{genuineGTstructure}
   The {\em genuine model structure}\index{genuine!model structure}\index{model structure!genuine} on \(\catGT\) is the \(\All\)-model
   structure, where is \(\All\) the family of {\em all} closed subgroups of \(G\).
   
 Also, $I_{\All}$ coincides with $I_G$ of Definition~\ref{defineIG}.
\end{dfn}
In particular, a map $f\colon X\to Y$ in $G\catT$ is a genuine equivalence (resp.~genuine fibration) if for all closed subgroups $H\subseteq G$ the map of fixed points $f^H\colon X^H\to Y^H$ is a weak equivalence (resp.~fibration).

It is not uncommon to hear ``na\"\i ve'' or ``Borel'' when people are talking about the model structure on $\catGT$ with respect to the trivial family $\{e\}$ (or in other words, the projective structure on $\catGT$ considered as a functor category).  In that case, $I_{\{e\}}$ will give rise to the $G$-free cofibrations.

\begin{dfn}\label{univFspace}
  The cofibrant replacement of \(S^0\) in the \(\aml\)-model
  structure is denoted \(E\aml_+\).\index{EA@\(E\aml_+\), cofibrant replacement of \(S^0\)}  For the trivial family, the free $G$-space $E\{e\}_+$ is usually denoted $EG_+$.
\end{dfn}


Now that we have model structures given by the families $\bml$ and $\aml$ separately it's time to start mixing them.  We do this first in the orbit category set-up and then translate it to $G$-spaces.

\begin{dfn}\label{def:sH}\label{definethesetS}
  For \(H \in \aml \setsminus \bml\), we let \(l_H \colon C_H \to
i^*\aml(H,-) = \aml(H,i(-))\) be the cofibrant replacement of
\(i^* \aml(H,-)\) in the projective level model structure on
\(\bml\catT\), and we let \(\lambda_H \colon LC_H \to \aml(H,-)\)
be the adjoint morphism of \(l_H\) in \(\aml\catT\). Using the
mapping cylinder like in the proof of Ken Brown's Lemma \cite[Lemma 7.7.1]{Hir}, we factor  
\(\lambda_H\) as the composition 
$$\xymatrix{LC_H\,\, \ar@{>->}[r]^{s_H}& 
M_H 
\ar[r]^-{r_H}&\aml(H,-)}$$ 
of a cofibration \(s_H\)\index{sH@$s_H$} and the
left inverse \(r_H\)\index{rH@$r_H$} of an acyclic cofibration in \(\aml\catT\).
We call the $\aml$-cofibration $s_H$ the \emph{$(\bml,\aml)$-localizer}\index{BAlocalizer@$(\bml,\aml)$-localizer} associated with $H$, and for the rest of the section we let $S$ denote the set of $(\bml,\aml)$-localizers:
$$S=\{s_H\mid H\in\aml\setsminus\bml\}.$$
\end{dfn}
\begin{remark}
  In the interesting case $\bml\subseteq\aml=\All$ this can be seen quite concretely in $\catGT$: for a closed subgroup $H\subseteq G$ the cofibrant replacement $l_H$ of $\aml(H,i-)$ corresponds to the $\bml$-cofibrant replacement
  $$(G\times_HE\bml)_+\to G/H_+$$
  which we factor as a genuine cofibration -- the ``$(\bml,\All)$-localizer'' -- followed by a retraction of an acyclic genuine cofibration.
\end{remark}
 
\begin{dfn}
A map $f$ in $\aml\catT$ is called a {\em \(\bml\)-equivalence} if $i^*f$ is a weak equivalence in $\bml\catT$.
  A map \(g \colon z \to w\) in
  \(\aml\catT\) is an {\em \(S\)-fibration}\index{Sfibration@\(S\)-fibration}\index{fibration!\(S\)-} if it is a fibration in
  \(\aml\catT\) with the property that for every $(\bml,\aml)$-localizer \(s \colon
  a \to b\) of \(S\), the square 
  \begin{displaymath}
    \xymatrix{
      \aml\catT(b,z) \ar[r]^{s_*} \ar[d]_{g^*} & 
      \aml\catT(a,z) \ar[d]^{g_*}\\
      \aml\catT(b,w) \ar[r]^{s^*} 
       & \aml\catT(a,w)
    }
  \end{displaymath}
  is a homotopy pullback square in \(\catT\).
\end{dfn}
Note that every $(\bml,\aml)$-localizer
is a
  \(\bml\)-equivalence.
\begin{prop}
  \label{result_about_S-fibrations}
  If \(g \colon z \to w\) is both a \(\bml\)-equivalence in
  \(\aml\catT\) and an \(S\)-fibration, then it is a weak
  equivalence in the projective level model structure on
  \(\aml\catT\). 
\end{prop}
\begin{proof}
  It suffices to show that if \(H \in \aml \setsminus \bml\), then the
  upper horizontal maps in the diagram
  \begin{displaymath}
    \xymatrix{
      \aml\catT(\aml(H,-),z) \ar[r]^{g_*} \ar[d]_{\cong} & 
      \aml\catT(\aml(H,-),w)  \ar[d]^{\cong}\\
      z(H) \ar[r]^{g} 
       & w(H),
    }
  \end{displaymath}
  where the vertical maps are Yoneda isomorphisms, is a weak equivalence in
  \(\catT\). Since \(g\) is a fibration in \(\aml\catT\) and a
  \(\bml\)-equivalence, the map \(i^*g\) is an acyclic fibration in
  \(\bml\catT\), so \(\aml\catT(LC_H,g) \cong \bml\catT(C_H,i^*g)\)
  is an acyclic fibration in \(\catT\).

  Since \(g\) is an \(S\)-fibration, for every \(s_H \colon LC_H \to
  M_H\) in \(S\), the square 
  \begin{displaymath}
    \xymatrix{
      \aml\catT(M_H,z) \ar[rr]^{\aml\catT(s_H,z)} \ar[d]_{\aml\catT(M_H,g)} &&
      \aml\catT(LC_H,z) \ar[d]^{\aml\catT(LC_H,g)} \\
      \aml\catT(M_H,w) \ar[rr]^{\aml\catT(s_H,w)} &&
      \aml\catT(LC_H,w)
    }
  \end{displaymath}
  is a homotopy pullback. Since \(r_H\) is the left inverse of an acyclic
  cofibration in \(\aml\catT\), also the square
  \begin{displaymath}
    \xymatrix{
      \aml\catT(\aml(H,-),z) \ar[rr]^{\aml\catT(\lambda_H,z)} \ar[d]_{\aml\catT(\aml\catT(H,-),g)} &&
      \aml\catT(LC_H,z) \ar[d]^{\aml\catT(LC_H,g)} \\
      \aml\catT(\aml(H,-),w) \ar[rr]^{\aml\catT(\lambda_H,w)} &&
      \aml\catT(LC_H,w)
    }
  \end{displaymath}
  is a homotopy pullback. Since \(i^* g\) is an acyclic fibration, the
  right hand vertical map \(\aml\catT(LC_H,g) \cong
  \bml\catT(C_H,i^*g)\) a weak equivalence. Thus also the left hand
  vertical map \(\aml\catT(\aml(H,-),g) \cong g(H)\) is a weak
  equivalence in \(\catT\).  
\end{proof}
\begin{dfn}\label{def:B-model}
  The {\em \(\bml\)-model structure}\index{Bmodel structure@\(\bml\)-model structure on $\aml\catT$}\index{model structure!\(\bml\)} on \(\aml\catT\) is the
  Bousfield localization of the projective level model structure on
  \(\aml\catT\) with respect to the set \(S\) of $(\bml,\aml)$-localizers of
  Definition~\ref{definethesetS}. We write \(J_S\) for 
  the generating set of acyclic cofibrations given by the union of the set  
  of generating acyclic cofibrations for the projective level model
  structure on \(\aml\catT\) and the set consisting of
  pushout products (see Definition~\ref{pushoutproduct}) of the
  form \(i \Box s\) for \(i \in I\) a generating cofibration for
  \(\catT\) and \(s \in S\). 
\end{dfn}
  Note that the cofibrations for the \(\bml\)-model structure on \(\aml\catT\)
  are the original cofibrations of the projective model structure on
  \(\aml\catT\) and that this model 
  structure is proper and cellular (right properness follows since if  $f\colon X\to Y$ is a $\bml$-fibration in $\aml\catT$, then its restriction $i^*f\in\bml\catT$ is a fibration, guaranteeing that base change along $f$ preserves $\bml$-equivalences).  In a chart, we have that for $f\colon X\to Y\in\aml\catT$ the following implications hold
  $$\xymatrixrowsep{.5\baselineskip}\xymatrix{
    i^*f\text{ is a we}&\Leftrightarrow &f\text{ is a $\bml$-equivalence}&\Leftarrow&f\text{ is a we}\\
    i^*f\text{ is a fib}&\Leftarrow&f\text{ is a $\bml$-fibration}&\Rightarrow&f\text{ is a fib}\\
    i^*f\text{ is a cof}&\Rightarrow&f\text{ is a $\bml$-cofibration}&\Leftrightarrow&f\text{ is a cof}
    }$$

  Finally, as before, we transport this structure along the adjoint pair 
  $$\xymatrix{\aml\catT\ar@<.5ex>[r]^{\mathrm{ev}_e} &G\catT\ar@<.5ex>[l]^{\mathrm{fix}},}\qquad \mathrm{ev}_eY=Y(e), \quad \mathrm{fix}X(H)=X^H$$
  to our desired mixed model structure on $G$-spaces.
  
\begin{dfn}\label{definemixedgenacyccof}
  For \(\bml \subseteq \aml\) an inclusion of non-empty families of subgroups of
  \(G\), let $K=\mathrm{ev_e}S$ be the evaluation at $e$ of the set $S$ of $(\bml,\aml)$-localizers and let 
  $$J_{\bml,\aml}\index{JBA@\(J_{\bml,\aml}\)}=J_{\aml}\cup\{i\Box k\mid k\in K, i\text{ a generating cofibration for }\catT\},
  $$
  whose elements we refer to as generating acyclic $(\bml,\aml)$-cofibrations.
\end{dfn}

\begin{thm}
\label{GFmixedmodelstr}\label{defineivandjv}
  Let \(\bml \subseteq \aml\) be non-empty families of subgroups of
  \(G\). The set \(I_{\aml}\) of Definition~\ref{def:amlmodel} is a set of generating cofibrations and the set
  \(J_{\bml,\aml}\) of Definition~\ref{definemixedgenacyccof} is a set of generating acyclic cofibrations for
  a proper cellular monoidal model structure on \(\catGT\).  

In this structure, a map $f\colon X\to Y$ in $G\catT$ is 
\begin{enumerate}
\item a weak equivalence if it is a $\bml$-equivalence, \ie if the induced map on $H$-fixed points $f^H\colon X^H\to Y^H$ is a weak equivalence for all $H\in\bml$, 
\item a cofibration if it is an $I_\aml$-cofibration and 
\item a fibration if it has the right lifting property with respect to maps that are both a weak equivalence and a cofibration.
\end{enumerate}
\end{thm}
\begin{dfn}
  \label{bamixedmodelstr}
  The {\em \((\bml,\aml)\)-model structure}\index{BA-model structure@\((\bml,\aml)\)-model structure}\index{model structure!\((\bml,\aml)\)} on the category
  \(G\catT\) of pointed \(G\)-spaces is the model structure specified
  in Theorem~\ref{GFmixedmodelstr}.
\end{dfn}
 Again we have the implications for $f\colon X\to Y\in \catGT$
  $$\xymatrixrowsep{.5\baselineskip}\xymatrix{
    f\text{ is a $\bml$-eq}&\Leftrightarrow &f\text{ is a $(\bml,\aml)$-equivalence}&\Leftarrow&f\text{ is an $\aml$-eq}\\
     f\text{ is a $\bml$-fib}&\Leftarrow&f\text{ is a $(\bml,\aml)$-fibration}&\Rightarrow&f\text{ is an $\aml$-fib}\\
    f\text{ is a $\bml$-cof}&\Rightarrow&f\text{ is a $(\bml,\aml)$-cofibration}&\Leftrightarrow&f\text{ is an $\aml$-cof}
    }$$
In the particular case where \(\bml
= \aml\), the sets $J_{(\aml,\aml)}$ and $J_\aml$ are equal, and the
$\aml$- and $(\aml,\aml)$-structures coincide.
\begin{lem}\label{mixedchangeofgroups}
  Let \(\bml \subseteq \aml\) be non-empty families of subgroups of a
  compact Lie group \(G\) and let \(i \colon H \to G\) be the
  inclusion of a subgroup. If we let \(i^*\bml\) be the family of
  subgroups \(P \cap H\) of \(H\) obtained by intersecting members
  \(P\) of \(\bml\) with \(H\) and likewise let \(i^*\aml\) be the
  family of intersections members of \(\aml\) with \(H\), then the
  forgetful functor \(i_H^* \colon G\catT \to H\catT\) is a right
  Quillen functor of mixed model structures with respect to
  \((\bml,\aml)\) and \((i^*\bml,i^*\aml)\) respectively. 
\end{lem}

\chapter{Equivariant Orthogonal Spectra}
\label{ch:eos}
In this chapter we set up a basic framework for equivariant orthogonal spectra.  Our presentation is guided by the needs of the commutative ring setting.  This explains why we use the ``semi-free'' spectra as our building blocks.  Since we want to end up in a situation which is Quillen equivalent to the usual stable structure, this results in the partial decoupling in the ``mixed model structures'' of the cofibration structure and the levelwise equivalences.   
In Chapter~\ref{ch:ss} we use this structure to obtain a convenient stable model structure on orthogonal spectra. 

Readers accustomed to the classical setup will also notice that the usual annoying worry about ``universes'' is downplayed in our presentation.  There is one fixed underlying category and to determine the model structures we focus on families of subgroups of the product of our Lie group and an orthogonal group.  These families naturally control our choice of model structure and are partially responsible for the mentioned decoupling of cofibrations and level equivalences.

\section{Orthogonal spectra as functors}\label{sect.orthspecfun}

In this section we give a short review of one definition of orthogonal spectra as functors. Later, when we work with model structures, this functoriality is convenient. An orthogonal spectrum is defined to be a functor from a \(\catT\)-category \(\catOr\) to spaces.

The term ``inner product space'' in the following is shorthand for ``finite dimensional real vector space with a non-degenerate inner product'' and the terms ``linear isometric embedding'' and ``isometry'' will be used interchangeably.  If $W$ is an inner product space, then $S^W$ is the one-point compactification of $W$, pointed at $\infty$.

\begin{dfn}\label{thecategoryoflinearembeddings}
  The \(\catU\)-category \(\catL\)\index{L@\(\catL\), linear isometric embeddings} of linear isometric embeddings has as objects the class of inner product spaces, and as morphism spaces the spaces of isometries. 
\end{dfn}
The following is essentially \cite[Definition I.5.9]{MM}.
\begin{dfn}\label{dfn.catObig}  
  The \(\catT\)-category \(\catO\)\index{O@\(\catO\)} has as objects the class of inner product spaces. Given inner product spaces \(V\) and \(W\), the morphism space \(\catO(V,W)\) is the subspace of \(\catL(V,W)_+ \wedge S^W\) consisting of points of the form \(f \wedge x\) for \(f \colon V \to W\) a linear isometric embedding and \(x\) contained in the subspace of \(S^W\) given by the one-point compactification of the orthogonal complement of \(f(V)\) in \(W\). The composition \((f,x) \circ (g,y)\) of two non-trivial morphisms \(f \wedge x \in \catO(V,W)\) and \(g \wedge y \in \catO(U,V)\) is \(fg \wedge (x + fy)\in\catO(U,W)\).

  The category \(\catO\) has a symmetric monoidal product
  \begin{eqnarray*}
    \catO(V,W) \times \catO(V',W') & \xto \oplus &\catO(V \oplus V',W \oplus W') \\
    (f \wedge x ,f' \wedge x') & \mapsto &(f \oplus f') \wedge (x+x').
  \end{eqnarray*}
\end{dfn}

For $V,W$ are inner product spaces of equal dimension the canonical inclusion $\catL(V,W)_+\subseteq \catO(V,W)$ is an isomorphism and if $\dim V>\dim W$ then $\catO(V,W)=*$.

As mentioned, our setup is largely independent of the classical notion of universes.  On a technical level, the untwisting map we now define and state some properties of translates between aspects of the cofibrations and weak equivalences.
\begin{dfn}\label{untwistingmap}
  The {\em untwisting map}\index{untwisting map}
  \begin{displaymath}
    \tau_{V,W} \colon \catO(V,W) \wedge S^V \to \catL(V,W)_+ \wedge S^W 
  \end{displaymath}
  takes \(f \wedge x \wedge y \in \catO(V,W) \wedge S^V\), where \(f \in \catL(V,W)\), \(x\) is in the orthogonal complement of \(f(V)\) in \(W\) and \(y \in V\), to \(f \wedge (x + f(y)) \in \catL(V,W)_+ \wedge S^W\).     
\end{dfn}
\begin{lemma}\label{catoruntwisting}
 The untwisting map
\[\untwist_{V,W} \colon \catO(V,W) \wedge S^V \xto \cong \catL(V,W)_+ \wedge S^W\] 
  is a homeomorphism. 
\end{lemma}
\begin{lemma}\label{magicaboutuntwisting}
  For all morphisms \(\varphi \colon V'\to V\in \catL\) and \(\psi \colon W \to W'\in\catO\) the following diagram commutes:
  \[
    \xymatrix{
      \catO(V,W) \wedge S^V \ar[r]^-{\untwist_{V,W}} \ar[d]^-{\psi_*} & 
      \catL(V,W) \wedge S^W \ar[r]^-{\varphi^*} &
      \catL(V',W) \wedge S^W \ar[r]^-{\untwist_{V',W}^{-1}} &
      \catO(V',W) \wedge S^{V'} \ar[d]^-{\psi_*} \\ 
      \catO(V,W') \wedge S^V \ar[r]^-{\untwist_{V,W'}} & 
      \catL(V,W') \wedge S^{W'} \ar[r]^-{\varphi^*} &
      \catL(V',W') \wedge S^{W'} \ar[r]^-{\untwist_{V',W'}^{-1}} &
      \catO(V',W') \wedge S^{V'}.
      }
  \]
\end{lemma}
\begin{dfn}\label{dfn.catO}\label{dfn.catL}
  The symmetric monoidal category $\catOr$\index{O@$\catOr$} is the full subcategory of
  \(\catO\) with objects the Euclidean spaces \(\R^n\) for \(n \in \N\) and, likewise, the symmetric monoidal category $\catLr$\index{L@$\catLr$} is the full
  subcategory of \(\catL\) with objects the Euclidean spaces \(\R^n\) for \(n \in \N\).
  Given an object \(V\) of \(\catOr\), we write
  \(\Or V=\catL(V,V)\) for its automorphism group, so that $(\Or V)_+=\Or{}(V,V)$.

  If $f:V\to W$ is an isometry we get an injective group homomorphism $\Or f:\Or V\to \Or W$ such that $f_*=f^*\Or f:\Or V \to \catL(V,W)$.  Explicitly, if $A\in\Or V$, $v\in V$ and $w$ in the orthogonal complement of $f$ in $W$, we have that $\Or f(A)(f(v)+w)=f(Av)+w$.  In this way $V\mapsto\Or V$ becomes a functor from the category of isometries to Lie groups.
\end{dfn}
\begin{remark}
  The standard basis gives an isomorphism \(\Or {\R^n} \cong
  \Or n\) between the group of isometric automorphisms of \(\R^n\) and
  the group \(\Or n\) of orthogonal \(n \times n\)-matrices.  
  For
  \(m \le n\), the morphism space \(\catLr(\R^m,\R^n)\) is homeomorphic
  to the factor group \(\Or {\R^n}/\Or {\R^{n-m}}\).
\end{remark}

\begin{dfn}\label{thecategoryoforthogonalspectra}
  An {\em orthogonal spectrum}\index{orthogonal!spectrum} \(X\) is a \(\catT\)-functor \(X \colon \catOr \to \catT\). Morphisms of orthogonal spectra are natural transformations. We write \(\catOS\)\index{OT@\(\catOS\), orthogonal spectra} for the category of orthogonal spectra.
\end{dfn}
Given an orthogonal spectrum \(X\) and an object \(V\) of \(\catOr\), we write \(X_V\)\index{XV@\(X_V\)} for the value of the functor \(X\) at \(V\). Let us stress that in this paper, an orthogonal spectrum is only defined on the inner product spaces \(\R^n\) with the dot product as inner product. Choosing an equivalence \(\catO\to \catOr\) of categories we can extend an orthogonal spectrum to a functor defined on \(\catO\). A more canonical way of extending orthogonal spectra to \(\catO\) is by letting \(X(V) = \catL(\R^n,V)_+ \wedge_{\Or{\R^n}} X_{\R^n}\) when \(V\) is an \(n\)-dimensional inner product space.  We have chosen only to evaluate orthogonal spectra on coordinate spaces since this leads us to be more explicit about morphisms in certain constructions. 

\begin{example}[The sphere spectrum]\label{thespherespectrum}
  The {\em sphere spectrum}\index{sphere!spectrum} \(\Sp \colon \catOr \to \catT\) is the representable functor \(\Sp = \catOr(0,-)\)\index{S@$\Sp$, sphere spectrum} with \(\Sp_V = \catOr(0,V) \cong S^V\).
\end{example}

Since the category \(\catOS\) is a category of \(\catT\)-functors into \(\catT\) it is enriched and tensored and cotensored over \(\catT\). The tensor \(K \wedge X\) of a space \(K\) and an orthogonal spectrum \(X\) is the functor given as the composition
\[
  \catOr \xrightarrow X \catT \xrightarrow{K \wedge -} \catT.
\]
The cotensor \(F(K,X)\) is the  
functor given as the composition
\[
  \catOr \xrightarrow X \catT \xrightarrow{F(K, -)} \catT,
\]
where \(F(K,-)\) takes a space \(L\) to the space \(F(K,L)\) of functions from \(K\) to \(L\). Applying the adjunction isomorphism for the smash product and function space objectwise we obtain adjunction isomorphisms
\[
  \catT(K, \catOS(X,Y)) \cong \catOS(K \wedge X,Y) \cong \catOS(X,F(K,Y))
\]  
that are natural in the space \(K\) and in the orthogonal spectra \(X\) and \(Y\).
\begin{example}
  The tensor \(K \wedge \Sp\) of a space \(K\) with the sphere
  spectrum called is the \emph{suspension spectrum}\index{suspension spectrum} of \(K\).
  The Yoneda homeomorphism \(\catOS(\Sp,X) \cong X_0\) and the above adjunction isomorphism for the tensor give a natural homeomorphism
\[
  \catT(K,X_0) \cong \catT(K,\catOS(\Sp,X)) \cong \catOS(K \wedge \Sp,X).
\]
Thus the functor 
\(\Sigma^{\infty} \defas - \wedge \Sp \colon \catT \to \catOS\)\index{Sigmainfty@$\Sigma^\infty$, suspension spectrum functor} 
is left adjoint to the evaluation functor \(\ev_0 \colon \catOS \to \catT\) with \(\ev_0(X) = X_0\)\index{ev0@$\ev_0$}.
\end{example}

Since the category of orthogonal spectra is the category of \(\catT\)-functors from a symmetric monoidal \(\catT\)-category to \(\catT\), it can be equipped with a closed symmetric monoidal structure.
This is discussed in detail in $\cite[\S 21]{MMSS}$ based on the more general statement of \cite[\S 3,4]{D}. 

Let \(X\) and \(Y\) be orthogonal spectra. Their \emph{external smash product}\index{external smash product, $\bar\smash$}\index{smash product!external, $\bar\smash$} is the functor
\begin{eqnarray*}X \bar\smash Y\colon\,\,\, \catOr\times\catOr& \rightarrow& \catT.\\
(V,W) &\mapsto & X_V \smash Y_{W}
\end{eqnarray*}
The \emph{(internal) smash product}\index{smash product!(internal), $\smash$} of \(X\) and \(Y\) is the $\catT$-enriched left Kan extension 
\[X \smash Y \defas{} \Lan{\oplus}{(X \bar\smash Y) \colon \catOr
  \rightarrow \catT}\]
of \(X \bar \smash Y\) along the monoidal product \(\oplus \colon
\catOr \times \catOr \to \catOr\).
In particular, by \cite[4.25]{Kel}, we can write the smash product as the coend
\[(X \smash Y)_U = \int^{(V,W)\in\catOr\times\catOr}\catOr(V\oplus
W,U)\smash X_V \smash Y_{W}.\]
The internal hom-object \(\internhom(X,Y)\)\index{homot@$\internhom$} is given by
\[\internhom(X,Y)_V = \catOS(X,Y(V \oplus -)).\]  
The fact that the definition given here indeed gives a closed symmetric monoidal structure on \(\catOS\) can be checked by applying the enriched Kan-extension to the coherence diagrams for $\catOr$, using the fact that the Kan-extension is natural in all its inputs, together with the fact that $\catT$ itself is closed symmetric monoidal.

\section{Equivariant Orthogonal Spectra}
\label{sec-equivariantorthogonal}

In this section we construct model structures on orthogonal spectra with action of a compact Lie group \(G\). For the rest of this chapter \(G\) will be a fixed but arbitrary compact Lie group.

\begin{dfn}\label{naiveorthgspec}
  An {\em orthogonal \(G\)-spectrum}\index{orthogonal!$G$-spectrum} \(X\) is an orthogonal spectrum with continuous action of \(G\), that is, with a continuous monoid homomorphism \(G \to \catOS(X,X)\). Morphisms of orthogonal \(G\)-spectra are required to respect the action of \(G\). We write \(\catGOS\)\index{GOS@\(\catGOS\), orthogonal $G$-spectra} for the category of orthogonal \(G\)-spectra.
\end{dfn}

Similar ways of looking at \(\G\)-spectra have come up before, for example in the context of \(\Gamma\)-spaces \cite{Shimikawa} and quite generally in \cite{DORI}. 

\begin{ex}[The Sphere \(G\)-Spectrum]\label{trivactiononspheresp}\index{sphere!G-spectrum@$G$-spectrum}
  Since the automorphism group of \(0\) in \(\catOr\) is trivial, \(G\) can only act trivially on the sphere spectrum \(\Sp = \catOr(0,-)\). 
\end{ex}
\begin{ex}[Free orthogonal \(G\)-spectra]\label{freeGorthsp}
  Let \(K\) be a pointed \(G\)-space. Given an object \(W\) of
  \(\catOr\) and a homomorphism \(\varphi \colon G \to \Or W\), the {\em
    free orthogonal \(G\)-spectrum}\index{free orthogonal \(G\)-spectrum} on \(\varphi\) and \(K\) is the
  orthogonal \(G\)-spectrum
  \(\Fr{}W K = \catOr(W,-) \wedge K\),\index{FWK@$\Fr{}W K$, free orthogonal \(G\)-spectrum} where \(G\) acts diagonally on 
  \((\Fr{}W K)_V = \catOr(W,V) \wedge K\). 
\end{ex}

\begin{ex}
  \label{ex:freewithreresentationindex}
  By definition our orthogonal $G$-spectra are continuous functors $X$ from $\Or{}$ to $G\catT$.  Hence, if $K$ is a topological group acting on $V=\R^n$, then $X_V$ will have a $G\times K$-action (it is a $K$-space in $G\catT$).  Especially when $K=G$ this invites to abuse of notation since we can consider $X_V$ as a $G$-space via the diagonal homomorphism $G\to G\times G$.  This is unfortunately occasionally confusing and we apologize for following the tradition in this respect but try to be clear on the matter.  Later on we'll have some use for this perspective allowing us to systematically evaluate at $G$-representations, but for now we limit ourselves to the following cases.
  \begin{itemize}
  \item Recall from Example~\ref{trivactiononspheresp} that \(G\) must act trivially on the sphere spectrum \(\Sp\). However, given an orthogonal $G$-representation $V$ (\ie a pair consisting of a real inner product space $V$ and a homomorphism \(\varphi \colon G \to \Or V\)), then the \(G\)-space \(\Sp_V= S^V\)
  is the one-point compactification \(S^V\) of the representation \(V\).
\item More generally, the free functor of Definition~\ref{freeGorthsp} extends as follows.  Let \(K\) be a pointed \(G\)-space and let $V$ and $W$ be orthogonal $G$-representations.  Then the \(G\)-space \((\Fr{}W K)_V = \catOr(W,V)
  \wedge K\) has diagonal \(G\)-action of the \(G \times G\)-action obtained
  from the action of \(G\) on 
  \(K\) and the action of \(G\) on \(\catOr(W,V)\)
  by conjugation.  In Lemma~\ref{lem:twistandshout} below, we will see that this construction can be replaced by ``semi-free'' spectra indexed on Euclidean spaces (\ie there is no longer any need for an action on the indexing spaces).
  \end{itemize}
\end{ex}

\begin{remark}Considering \(G\) as a topological category with one object, the category \(\catGOS\) is the category of  \(\catT\)-functors from \(G_+\) to \(\catOS\). Since the category \(\catOS\) is enriched, tensored and cotensored over \(\catT\), the category \(\catGOS\) is enriched, tensored and cotensored over \(G\catT\). Moreover, since \(\catOS\) is closed symmetric monoidal, so is \(\catGOS\).

  Explicitly, the enrichment $\catOGS$\index{OTG@$\catOGS$} of $\catOS$ in \(G\catT\) is given by considering the space \(\catOS(X,Y)\) of morphisms between two orthogonal \(G\)-spectra as a \(G\)-space $\catOGS(X,Y)$ of morphisms in a $G\catT$-category $\catOGS$ by letting \(G\) act by conjugation, that is, \(g \in G\) takes \(f \in \catOS(X,Y)\) to the composition
\[
  X \xto{g^{-1}} X \xto f Y \xto g Y.
\]
The \(G\)-action on the smash product $X \smash Y$ is through the
diagonal embedding of \(G\) in \(G \times G\). The internal function
spectrum \(\hom_{\catGOS}(X,Y)\)\index{homot@$\hom_{\catGOS}$} is given by the underlying internal function spectrum
\(\internhom(X,Y)\) with \(G\)-action by conjugation as above.

For \(X\) an orthogonal spectrum, the adjunction isomorphisms for the tensor \(\smash\) and for the suspension spectrum $\Sigma^\infty X$ give natural isomorphisms
\[\catT(\G_+, \catOS(X,X))\cong \catOS(\G_+ \smash X, X) \cong \catOS(\Sigma^\infty \G_+ \smash X, X).\]
  The action map $\mu\colon G_+ \rightarrow \catOS(X,X)$ is adjoint
  to a map $\bar\mu\colon \Sigma^\infty \G_+ \smash X \rightarrow
  X$. The fact that $\mu$ is a map of monoids exactly translates to
  $X$ being a module over the orthogonal ring spectrum
  $\Sp_{[G]}\defas{}\Sigma^\infty G_+$ via $\bar\mu$. Then morphisms
  in $\catOGS$ correspond to \(\Sp\)-module morphisms between
  $\Sp_{[G]}$-modules, whereas morphisms in $\catGOS$ correspond to
  \(\Sp_{[G]}\)-module maps. 
\end{remark}

 By naturality the \(\catT\)-isomorphisms for tensors and cotensors in \(\catOS\) over \(\catT\) can be considered as $\catGT$-natural isomorphisms
\[\catTG(D,\catOGS(X,Y)) \cong \catOGS(D \smash X, Y) \cong \catOGS(X, F(D,Y))\]
for tensors $\smash$ and cotensors $F$ in \(\catOGS\) over \(\catT_G\). Analogously to Example~\ref{GTtensor}, taking the $\G$-fixed points of the spaces above yields
\[\catGT(D,\catOGS(X,Y)) \cong \catGOS(D \smash X, Y) \cong \catGOS(X, F(D,Y)).\]

\begin{dfn}
  Let \(V\) be a Euclidean space and let \(Y \in \catGOS\) be an orthogonal
  \(G\)-spectrum.
  \begin{enumerate}
  \item The restriction of \(Y\) to
    the full subcategory of $\catO$ with the single object \(V\) (which we identify with $(\Or V)_+$) is \(Y_V \in G\Or{V}\catT\),
    and we consider it as an object \(Y_V\) of \((G \times \Or V)\catT\).
  \item If \(\varphi \colon G \to \Or V\) is a continuous homomorphism,
    then through the homomorphism \(i_{\varphi} \colon G \to G \times \Or V\)
    with \(i_\varphi(g) = (g , \varphi(g))\) we obtain the \(G\)-space
    \(i_\varphi^* Y_V\). If the action $\phi$ is clear from the context we omit \(i_\varphi^*\) from the
    notation and refer to \(i_\varphi^* Y_V\) as {\em the \(G\)-space \(Y_V\)}.
  \end{enumerate}
\end{dfn}

\begin{remark}
  Given any universe \(\catU\) of \(G\)-representations in the sense of \cite[Definition II.1.1]{MM}, the category of orthogonal \(G\)-spectra considered here is equivalent to the category in 
\cite[Definition II.2.1]{MM} of 
spectra defined on representations isomorphic to finite dimensional \(G\)-invariant subspaces of \(\catU\), as explained by Mandell and May in \cite[Theorem V.1.5]{MM}. 
\end{remark}

\section{Semi-Free Equivariant Spectra} \label{subsecteqsemifree}

\begin{dfn}\label{semifreespectra}
  Let \(V\) be an object of \(\catOr\). The functor \(\Gr{} V \colon
  \Or V \catT \to \catOS\) is the \(\catT\)-functor taking  a pointed
  \(\Or V\)-space \(K\) to the {\em semi-free orthogonal spectrum} \(\Gr{}{V} K\) with 
\[(\Gr{}V K)_W = \catOr(V,W) \wedge_{\Or V} K\] \index{GVK@$\Gr{}V K$, semi-free orthogonal spectrum}
and with functoriality in \(W\) induced by composition in \(\catOr\).
\end{dfn}
\begin{remark}
  \label{remark:freevssemifree}
  The natural isomorphism
$$\Gr{}V({\Or V}_+\smsh C)=\catOr(V,-)\smsh_{\Or V}({\Or V}_+\smsh C)\cong \catOr(V,-)\smsh C=\Fr{} V C$$
for $C$ any space gives a correspondence between semi-free spectra and the free spectra from Example~\ref{freeGorthsp}.

More generally, if $V$ is a Euclidean space, $U$ an inner product space, $P\subseteq\catO_V$ a subgroup and $M$ a $P$-space, composition induces an isomorphism
$$\gamma\colon\Gr{}V(\catO(U,V)\smsh_PM)=\catOr(V,-)\smsh_{\Or V}(\catO(U,V)\smsh_PM)\cong \catO(U,-)\smsh_PM,$$
and in particular, when $U=V$ an isomorphism
$$\gamma\colon\Gr{}V({\Or V}_+\smsh_PM)\cong \catOr(V,-)\smsh_PM,$$
however, the naturality of the latter isomorphism is not as one is prone to guess in concrete situations.
To spell this out, assume $\alpha\colon U'\to U$ is an isometry, $P$ a subgroup of $\catO_{U}$ and let $P^\alpha=\alpha^{-1} P\alpha\subseteq\catO_{U'}$.  If $M$ is a $P$-space, $M'$ is a $P^\alpha$-space and $\beta\colon M\to M'$ satisfies $\beta(p\cdot m)=(\alpha^{-1}p\alpha)\cdot\beta(m)$, then
$$
\xymatrix{
 \Gr{}V(\catO(U,V)\smsh_PM)\ar[d]^{\Gr{}V{(\alpha^*\smsh\beta)}}_\cong\ar@{=}[r]&
   \catOr(V,-)\smsh_{\Or V}(\catO(U,V)\smsh_PM)\ar[d]_\cong^{\id\smsh(\alpha^*\smsh\beta)}\ar[r]^-\gamma_-\cong&
   \catO(U,-)\smsh_PM\ar[d]_\cong^{\alpha^*\smsh\beta}\\
  \Gr{}V(\catO(U',V)\smsh_{P^\alpha}M')\ar@{=}[r]&
  \catOr(V,-)\smsh_{\Or V}(\catO(U',V)\smsh_{P^\alpha}M')\ar[r]^-\gamma_-\cong&\catO(U',-)\smsh_{P^\alpha}M'
}
$$
commutes.  When $V=U$ and $\alpha$ is an isomorphism it is known to cause distress that the $V$ in $\Gr{}V$ is not touched by $\alpha$ whereas the $U$ in $\catO(U,-)$ is.  It is also part of the conspiration towards indexing by trivial representations only.
\end{remark}

The following result holds since \(\Gr{}V\) is an explicit construction of the left Kan extension of \(K \colon \Or V \to \catT\) along the inclusion \(\catOr_V \to \catOr\).
\begin{lemma}\label{lemma:semifreeadjoint}
  The \(\catT\)-functor \(\Gr{} V \colon \Or V \catT \to \catOS\) is left adjoint to the evaluation \(\catT\)-functor \(\ev'_V \colon \catOS \to  \Or V \catT\)\index{evv@$\ev'_V$} taking an orthogonal spectrum \(X\) to the \(\Or V\)-space \(X_V\).
\end{lemma}

Let \(\catOr'\)\index{O@$\catOr'$} be the subcategory of \(\catOr\) with \(\catOr'(V,V) =
\Or V\) and with \(\catOr'(V,W)\) equal to the one point space for \(V
\ne W\). The restriction
functor \(\ev' \colon \catOS \to \catOr'\catT\) has a left adjoint
\(\Gr{}{}\) with \(\Gr{}{}Y = \bigvee_{V} \Gr{} V Y_V\) for \(Y
\in \catOr'\catT\). Given an orthogonal spectrum \(X\), we write \(c_X
\colon \Gr{}{} \ev' X \to X\) for the counit of this adjunction.

The following lemma reduces questions regarding functors preserving coequalizers to semi-free spectra.
\begin{lem}\label{coequalizeroffree}
  For every orthogonal spectrum \(X\), the diagram
  \begin{displaymath}
    \xymatrix{
      \Gr{}{} \ev' \Gr{}{} \ev' X \ar@<-.9ex>[r]_-{\Gr{}{}\ev' c_X}
      \ar@<.9ex>[r]^-{c_{\Gr{}{}\ev' X} } &
      \Gr{}{} \ev' X \ar[r]^-{c_X} & X
    }
  \end{displaymath}
  is a coequalizer diagram in \(\catOS\).
\end{lem}
\begin{proof}
  After applying \(\ev'\) the diagram becomes a split coequalizer
  diagram in \(\catOr' \catT\). 
\end{proof}

One of the most important properties of semi-free spectra is that it is easy to calculate their smash products with other spectra and in particular with each other. The following proposition makes this precise, and generalizes Lemma \cite[1.8]{MMSS}. It is analogous to Lemma \cite[I.4.5-6]{Sub} in the case of symmetric spectra. 
\begin{prop}\label{semi-free-smash}\label{Grandsymmetry}\label{Grtwisttwo}
  Let \(V\) and \(W\) be Euclidean spaces, $K$ an $\Or{V}$-space and $L$ an $\Or{W}$-space.
  Then the dual Yoneda lemma induces an isomorphism
  $$\catOr(V\oplus W,-)\smsh_{\Or V\times\Or W}K\smsh L\cong \int^{V',W'\in\catOr}\catOr(V'\oplus W',-)\smsh\catOr(V,V')\smsh_{\Or V}K\smsh\catOr(W,W')\smsh_{\Or W}L$$
    (the inverse is given by composition) and hence - using the coend formula for the smash and the isomorphism $\gamma$ of Remark~\ref{remark:freevssemifree} - a natural isomorphism
\[y\colon\Gr{}{V \oplus W}({\Or{V \oplus W}}_+ \smash_{\Or{V}\times\Or{W}} K \smash L) \xto \cong
  \Gr{}{V}K \smash \Gr{}{W}L.\]
If $\tau_{V,W}\colon V\oplus W\cong W\oplus V$, $\tau_{K,L}\colon K\smsh L\cong L\smsh K$ and $\tau_{\Gr{}VK,\Gr{}WL}\colon \Gr{}VK\smsh\Gr{}WL\cong \Gr{}WL\smsh \Gr{}VK$\index{tau@$\tau$, twist isomorphism}  are the permutations, then
$$\xymatrix{
  \Gr{}{V \oplus V}({\Or{V \oplus V}}_+ \smash_{\Or{V}\times\Or{V}} K \smash K)  \ar[r]^-y_-\cong\ar[d]^{\Gr{}{V\oplus V}(\tau_{V,V}^*\smsh\tau_{K,K})}_\cong&
  \Gr{}{V}K \smash \Gr{}{V}K\ar[d]^{\tau_{\Gr{}VK,\Gr{}VK}}_\cong\\
  \Gr{}{V \oplus V}({\Or{V \oplus V}}_+ \smash_{\Or{V}\times\Or{V}} K \smash K)  \ar[r]^-y_-\cong&
  \Gr{}{V}K \smash \Gr{}{V}K
}$$
commutes.

\end{prop}
\begin{proof}
  The only thing that needs explanation is the commutativity of the diagram, which follows by setting $V=W$ and $K=L$ in the less useful diagram
$$\xymatrix{
  \Gr{}{V \oplus V}(\catOr(V\oplus W,V\oplus V) \smash_{\Or{V}\times\Or{W}} K \smash L)
  \ar[r]^-y_-\cong\ar[d]^{\Gr{}{V\oplus V}(\tau_{V,W}^*\smsh\tau_{K,L})}_\cong&
  \Gr{}{V}K \smash \Gr{}{W}L\ar[d]^{\tau_{\Gr{}VK,\Gr{}WL}}_\cong\\
  \Gr{}{V \oplus V}(\catOr(W\oplus V,V\oplus V)\smash_{\Or{W}\times\Or{V}} L \smash K)  \ar[r]^-y_-\cong&
  \Gr{}{W}L \smash \Gr{}{V}K.
}$$
  The commutativity of this diagram follows by the functoriality of $\gamma$ explained in Remark~\ref{remark:freevssemifree} (with $V$ replaced by $V\oplus V$, $\alpha$ by $\tau_{V,W}$, $\beta$ by $\tau_{K,L}$ and $P$ by $\Or V\times\Or W$) and the fact that
  $$\xymatrix{
    \catOr(V\oplus W,-)\smsh_{\Or V\times\Or W}K\smsh L\ar[rr]^-{\text{dual Yoneda}}_-\cong
    \ar[d]^{\tau_{V,W}^*\smsh\tau_{K,L}}_\cong&& \Gr{}VK\smsh\Gr{}WL\ar[d]^{\tau_{\Gr{}VK,\Gr{}WL}}_\cong\\
    \catOr(W\oplus V,-)\smsh_{\Or W\times\Or V}L\smsh K\ar[rr]^-{\text{dual Yoneda}}_-\cong&& \Gr{}WL\smsh\Gr{}VK
  }$$
  commutes which is seen explicitly by inspection upon replacing the smash with their coend formulations.
\end{proof}
\begin{remark}
  Note that the left vertical map in the diagram in Lemma~\ref{Grtwisttwo} is of the form
\(\Gr{}{V \oplus V}(\tau^* \wedge \tau)\): the two copies of \(V\) in
\(\Gr{}{V \oplus V}\) are {\em not} permuted.
\end{remark}

Let \(X\) be a finite discrete \(G\)-set, let \(V\) be a Euclidean
space and let \(K\) be an \(\Or V\)-space. Then \(G\) acts on
\(K^{\wedge X}\) and on \(\prod_x \Or V\) by permuting factors. The
group \(G \times \Or {V^{\oplus X}}\) acts on the space \({\Or
{V^{\oplus X}}}_+ \wedge_{\prod_x \Or V} K^{\wedge X}\) by letting
\(G\) act on \(\Or {V^{\oplus X}}\) by right multiplication of
permutation of summands and letting \(\Or {V^{\oplus X}}\) act on
itself by left multiplication. Considering \(V^{\oplus X}\) as an
Euclidean space, we obtain an orthogonal \(G\)-spectrum 
\(\Gr{}{V^{\oplus X}} ({\Or
{V^{\oplus X}}}_+ \wedge_{\prod_x \Or V} K^{\wedge X})\). Notice that
\(G\) acts trivially on the Euclidean space \(V^{\oplus X}\) in 
\(\Gr{}{V^{\oplus X}}\).

Since permutations of a finite set \(X\) are generated by transpositions,
iterated application of Corollary~\ref{Grtwisttwo} gives:
\begin{prop}\label{semifreesmashpower}
 Let $X$ be a finite discrete $G$-set, \(V\) a Euclidean space and
 \(K\) an \(\Or V\)-space. Then the dual Yoneda lemma as in Proposition~\ref{semi-free-smash} induces a natural isomorphism of $G$-spectra 
\[
y\colon\Gr{}{V^{\oplus X}} ({{\Or {V^{\oplus X}}}_+}
  \wedge_{\prod_X \Or V} K^{\smash X})
  \xto \cong
(\Gr{}{V}K)^{\smash X} 
,\]
where $G$ acts on  $K^{\smsh X}$ and
$(\Gr{}{V}K)^{\smash X}$ by permuting factors and on $\Or{V^{\oplus X}}$ by precomposition with the permutation of summands in $V^{\oplus X}$, but the $G$-action on the index in $\Gr{}{V^{\oplus X}}$ is trivial.
\end{prop}
\begin{prop}\label{semi-free-smash-with-gen}
Let \(E\) be an orthogonal spectrum and let \(K\) be an \(\Or{V}\)-space. The structure map
\(
  K \smash E_{W} \cong (\Gr{}{V}K)_V \smash E_{W} \to (\Gr{}{V}K \smash E)_{V \oplus W}  
\)
for the smash-product induces a natural isomorphism  
\[{\Or{V \oplus W}}_+ \smash_{\Or{V}\times \Or{W}} K \smash E_{W} \cong (\Gr{}{V}K \smash E)_{V \oplus W} \]
of \(\Or{V \oplus W}\)-spaces. If the dimension of \(V'\) is smaller than the dimension of \(V\), then 
\((\Gr{}{V}K \smash E)_{V'}\) is a one-point space.
\end{prop}
\begin{proof}
  Considered as functors of \(E\) both sides commutes with
  colimits. By Lemma~\ref{coequalizeroffree} we can represent \(E\) as a
  colimit of semi-free spectra so we can reduce to the case where \(E =
  \Gr{}{U}L\) is semi-free. In this case the asserted isomorphism is a
  consequence of Proposition~\ref{semi-free-smash} since 
\(
  {\Or{V \oplus W}}_+ \smash_{\Or{W}} \catOr (U,W)\)
is isomorphic to \(\catOr (V\oplus U,V\oplus W)
\).
\end{proof}

Often actions will occur through
the semi-direct product \(G \ltimes_{\varphi} \Or V\) of the inner
automorphism of a representation \(\varphi \colon G \to \Or
V\). Recall that \(G \ltimes_{\varphi} \Or
V\) has \(G \times \Or V\) as underlying set and multiplication given
by the rule
\begin{displaymath}
  (g,\alpha )(h, \beta) = (gh, \alpha \varphi(g) \beta \varphi(g^{-1})).
\end{displaymath}
\begin{dfn}\label{twistedsemifree}
  Let \(V\) be a Euclidean space, let \(\varphi \colon G \to \Or V\)
  be a
  group homomorphism and let \(f_\varphi \colon G \times \Or V \to G
  \ltimes_{\varphi} \Or V\) be 
  the isomorphism \(f_\varphi(g,\alpha) = (g,\alpha\varphi(g^{-1}))\).
 For $K$ a $G\ltimes_{\varphi} \Or V$-space we define the
  semi-free orthogonal \(G\)-spectrum \(\Gr{\varphi}V K\)\index{GphiVK@$\Gr{\varphi}V K$, semi-free orthogonal $G$-spectrum} by
  \begin{displaymath}
    \Gr{\varphi}V K = \Gr{}V f^*_\varphi K.
  \end{displaymath}
\end{dfn}
Note that there is an isomorphism of \(G \times \Or V\)-spaces of the form
\begin{displaymath}
  f^*_\varphi K \cong \catOr(V,V) \wedge_{\Or V} K, \quad k \mapsto \id_V
  \wedge k,
\end{displaymath}
where \(G \times \Or V\) acts on \(\catOr(V,V) \wedge_{\Or V} K\) by
the rule \((g,\alpha)\cdot (\gamma \wedge k) = \alpha \gamma
\varphi(g^{-1}) \wedge gk\) (where we have written $gk$ for $(g,1)\cdot k$).
\begin{lemma}\label{groupsemifree}
Let \(V\) be a Euclidean space and let \(K\) be a \(G \times \Or
V\)-space. The unit of adjunction $K\cong \fev{'}{V}\Gr{}{V}K$ gives an isomorphism
$$\catOS(\Gr{}{V}K,\Gr{}{V}K)\cong \catGspaces{\Or{V}}(K,\fev{'}{V}\Gr{}{V}K)\cong \catGspaces{\Or{V}}(K,K)
$$
of topological monoids. Since \(G\) acts on \(K\), this isomorphism
provides an action of \(G\) on \(\Gr{}V K\). The levelwise orbit spectrum
$\left[\Gr{}{V}K\right]_G$ is naturally isomorphic to $\Gr{}{V}(K_G)$.
\end{lemma}
\begin{proof}
To check the statement about orbits we note that on levels $W$ we get the isomorphisms
\[[\Gr{}{V}K_W]_G = [\catO(V,W)\smash_{\Or{V}}K]_G \cong [\catO(V,W)\smash K]_{\Or{V}\times G} \cong \catO(V,W)\smash_{\Or{V}}(K_G).\]
\end{proof}

\section{Families of Representations}\label{subsectuniverses}
Traditionally a universe for a compact Lie group $G$ is used to define the various stable
model structures of orthogonal \(G\)-spectra. However we will find it
convenient to have the flexibility of choosing compatible universes
consisting of an indexing category of \(H\)-representations for each
subgroup \(H\) of \(G\). The \(G\)-typical families of representations
introduced in this section form a way to encode this.  Also, this strategy helps remove some of the traditional complications due to a dependence on a choice of universe and the need to change it.
 
Let $H$ be a subgroup of \(G\) and consider a representation of $H$, that is, a Euclidean space $V$ together with a homomorphism $\varphi \colon H \to \Or V$. As explained below, the subgroup \(H\) and the homomorphism \(\varphi\) can be recovered from the subgroup \(P_V\) of \(G \times \Or V\) consisting of pairs \((h,\varphi(h))\) for \(h\) in \(H\).

Given a Euclidean space \(V\) and a subgroup \(P\) of \(G \times \Or V\), let \(\proj_1 \colon P \to G\) be the composition of  the inclusion \(P \subseteq G \times \Or V\) and the projection \(G\times \Or V \to G\).
\begin{dfn}\label{representationgivenbyP}
  Let \(V\) be a Euclidean space and let \(P\subseteq G \times \Or V\) be a subgroup such that \(\proj_1 \colon P \to G\) is injective with image \(H = \proj_1(P)\). The \(H\)-representation \(V(P) = (V,\varphi)\)\index{VP@$V(P)$, representation associated to $P$} has underlying vector space \(V\) and \(H\)-action \(\varphi \colon H \to \Or V\) given as the composition
\[
  H \xrightarrow{\proj_1^{-1}} P \subseteq G \times \Or V \xrightarrow{\proj_2} \Or V.
\]
\end{dfn}
We will typically refer to this situation as an {\em $H$-representation of the form $V(P)$}.  In particular, when referring to a $G$-representation of the form $V(P)$ it is understood that the projection $\proj_1\colon P\to G$ is an isomorphism.

The following lemma is trivial but important.
  \begin{lemma}
    \label{lem:twistandshout}
    Let $H\subseteq G$ be a closed subgroup.  Let $V$ be a $G$-representation with underlying space $\R^n$ and action
 $\varphi\colon G\to\Or V$.  Consider $G/H\times\Or V$ as a $G\times\Or{\R{}^n}$-space with $(g_1,A_1)\in G\times\Or{\R{}^n}$ acting by sending $(gH,A)\in G/H\times\Or V$ to
$(g_1g,A_1A\varphi(g_1^{-1}))$.

Let $P=\Gamma_{\varphi|_H}=\{(h,\varphi(h))\in H\times\Or{\R{}^n}\}$ be the graph of $\varphi|_H$ considered as a subgroup of $G\times\Or{\R{}^n}$.  Then the map 
\labeleq{ovpgrep}{
  (G\times\Or{\R{}^n})/\Gamma_{\varphi|_H}\to G/H\times\Or V,\qquad (g,A)\Gamma_{\varphi|_H}\, \mapsto \,(gH,A\varphi(g^{-1}))} 
is an isomorphism of $G\times\Or{\R{}^n}$-spaces.
The map 
$$\catOr(V,-)\smsh G/H_+\to \catOr(\R{}^n,-)\smsh_{\Or{\R{}^n}}(G/H\times\Or V)_+$$
sending $(f,gH)$ to $(f,(gH,1))$ is an isomorphism of $G$-spectra,
where $\gamma\in G$ acts on the source by sending $(f,gH)$ to $(f\varphi(\gamma^{-1}),\gamma gH)$ and on the target by sending $(f,gH,A)=(fA,gH,1)$ to $(f,\gamma gH,A\varphi(\gamma^{-1}))$.

Hence the free spectrum $\Fr{}V(G/H_+)$ of Example~\ref{ex:freewithreresentationindex}
indexed by the $G$-representation $(V,\varphi)$ is isomorphic to the semi-free spectrum $\Gr{}{\R{}^n}(((G\times\Or{\R{}^n})/\Gamma_{\varphi|_H})_+)$
$$\Fr{}V(G/H_+)\cong \Gr{}{\R{}^n}(((G\times\Or{\R{}^n})/\Gamma_{\varphi|_H})_+).
$$
  \end{lemma}
  The reader should notice that as a consequence the free spectrum $\Fr{}V(G/H_+)$ only depends on $\varphi$'s restriction to $H$.
  
\begin{dfn}\label{dfn:AI}\label{dfn:Qll}
  If $V$ is a Euclidean space, let $\AI^V$\index{AI@$\AI$} be the set of closed subgroups \(P\) of \(G \times \Or V\) with the property that \(\proj_1 \colon P \to G\) is injective.

  The subset  $\Qll^V\subset\AI^V$\index{QV@$\Qll^V$} consists of the subgroups  $Q\subseteq G\times \Or V$ such that there is a subgroup $P\subseteq G\times \Or V$ with $Q\subseteq P$ having the property that that the projection $P\subseteq G\times \Or V\to G$ is an isomorphism (``$P$ is a $G$-representation'').
\end{dfn}
  We abbreviate $\AI=(\AI^V)_{V}$ and $\Qll=(\Qll)_V$ when $V$ varies over Euclidean spaces.
The name ``$\AI$'' is chosen as a mnemonic for ``all injective'' and ``$\Qll$'' since it is the basis for the so-called $q$-model structure.

\begin{lemma}
  For each inner product space \(V\), the assignment \(P \mapsto V(P)\) is a bijection from $\AI^V$ to the set of representations of closed subgroups of \(G\) with \(V\) as underlying vector space.
\end{lemma}
\begin{proof}
  Let \(B\) be the set of pairs \((H,\varphi)\) where \(H\) is a subgroup of \(G\) and \(\varphi \colon H \to \Or V\) is a continuous homomorphism. There is an obvious bijection between \(B\) and the set of representations \((V,\varphi)\) of any subgroup of \(G\) with underlying vector space \(V\).

  Let \(A\) be the set of closed subgroups \(P\) of \(G \times \Or V\) with the property that \(\proj_1 \colon P \to G\) is injective.
  Given \((H,\varphi) \in B\), the inclusion \(i \colon H \to G\) together with \(\varphi \colon H \to \Or V\) gives an injective homomorphism \((i,\varphi) \colon H \to G \times \Or V\), and thus the pair \((H,\varphi)\) yields a closed subgroup, namely the image \(P = (i,\varphi)(H)\) of \(G \times \Or V\) with the property that \(\proj_1 \colon P \to G\) is injective. This gives a function \(f \colon B \to A\).

Conversely, if \(P \in A\), we obtain a pair \((H,\varphi) \in B\) by letting \(H = \proj_1(P)\) and letting \(\varphi\) be the composition of \(\proj_1^{-1} \colon H \to P\) and the homomorphism \(P \to \Or V\) obtained from the projection \(G \times \Or V \xrightarrow{\proj_2} \Or V\). This gives a function \(g \colon A \to B\), and \(f\) and \(g\) are inverse bijections.
\end{proof}

\begin{dfn}\label{dfn:sumofrepns}
  Given Euclidean spaces \(V\) and \(W\) and subgroups \(P \subseteq G
  \times \Or V\) and \(Q \subseteq G \times \Or W\) in $\AI$, we define
  \(P \oplus Q \subseteq G \times \Or {V \oplus W}\) to be the pullback in
$$\xymatrix{P\oplus Q\ar[d]\ar[rrr]^{\subseteq}&&&G\times\Or{V\oplus W}\ar[d]^{\text{diagonal}\times\id}\\
P\times Q\ar[r]^-{\subseteq}&G\times\Or V\times G\times\Or W\ar[rr]^-{\text{shuffle and }\oplus}&&G\times G\times\Or{V\oplus W}.}
$$
\end{dfn}
Expressed in terms of the associated representations: \(V(P \oplus Q)\) is the
representation \(V(P) \oplus V(Q)\) of the intersection \(\proj_1(P) \cap \proj_1(Q)\).

\begin{dfn}\label{gtypicalrepresentations}
  A {\em \(G\)-typical family of representations}\index{Gtypicalfamily@$G$-typical family} consists  of a sequence \(\hml = (\hml^V)_{V}\)\index{H@$\hml$, a $G$-typical
    family} indexed over the Euclidean spaces, such that for all Euclidean spaces \(V\) and \(W\):
  \begin{enumerate}
  \item \(\hml^V\) is a family of subgroups of \(G \times \Or V\)
  \item $\hml^V\subseteq\AI^V$.
  \item There is a non-zero \(U\) such that \(\hml^U\) contains \(G
    \times \{\id_U\}\).
  \item For every \(P\) in \(\hml^V\), there exists a \(Q \in \AI^W\)
    and an
  \(R \in \hml^{V \oplus W}\) with \(\proj_1(Q) = \proj_1(P)\) and
  \(\proj_1(R) = G\) such that 
  \(P \oplus Q = (\proj_1(P) \times \Or {V \oplus W}) \cap R\).
\end{enumerate}
\end{dfn}
In the above definition, condition (i) allows us to use standard
methods from equi\-variant homotopy theory. Condition (ii) gives a
firm connection to representations of subgroups of \(G\): let $H$ be the image of the injection $\proj_1\colon P\to G$. Using the inverse of the induced isomorphism $P\cong H$ we get a representation $\phi\colon H\cong P\subseteq G\times\Or V\to \Or V$ (where the last map is the second projection) and $P$ is simply the graph 
$$P=\Gamma\phi=\{(h,\phi(h))\mid h\in H\}\subseteq G\times\Or V.$$ 

Condition
(iii) is needed for the stable model structure on orthogonal
\(G\)-spectra to be stable in the sense of model categories
\cite[Definition 7.1.1]{H}.

Condition (iv) is of a more technical
nature related to cofinality of index categories in homotopy colimits
used in a construction of fibrant replacement, see in particular Remark~\ref{rem:compMM} and the proof of Corollary~\ref{cor:acyfiblevelvsstable}.  In terms of representations it says that if $P$ is the graph of the representation $\psi\colon H\to\Or V$, then there exists a representation $\phi\colon H\to\Or W$ such that the sum $\psi\oplus\phi\colon H\to\Or{V\oplus W}$ is the restriction of a representation $G\to\Or{V\oplus W}$ (whose graph is) in $\hml^{V\oplus W}$

Given \(g \in G\) and a subgroup \(H\) of \(G\) we write \(c_g \colon H \to gHg^{-1}\)\index{cg@$c_g$, conjugation isomorphism} for the conjugation isomorphism with \(c_g(h) = ghg^{-1}\) (and $c_{g^{-1}}$ for its inverse). 

\begin{prop}\label{Gtypicalrestrictions}
  Let \(\hml\) be a \(G\)-typical family of representations. The set
  \(\catV\) of isomorphism classes of representations of the form
  \(V(P)\) for \(P\) in \(\hml\) is closed under conjugation and
  restriction, that is, if \(g \in G\) and \(i \colon K \to gHg^{-1}\)
  is an inclusion of subgroups of \(G\) and if an \(H\)-representation
  \(W\) is in \(\catV\), then the \(K\)-representation
  \(i^*{c_{g^{-1}}^*W}\) is in \(\catV\). 
\end{prop}
\begin{proof}
  Choose \(P \in \hml\) with an isomorphism \(W \cong V(P)\). The
  representation \(c_{g^{-1}}^* W\) of \(gHg^{-1}\) is then
  isomorphic to \(c_{g^{-1}}^* V(P) \cong V(c_g(P))\), where
  \(c_g(P)\) consists of the elements of the form
  \((g,\id)p(g^{-1},\id)\) for \(p \in P\). The
  subgroup \(c_gP\) of \(G \times \Or V\) is in \(\hml\) because
  \(\hml\) is closed under conjugation. 
  Since \(i \colon K \to gHg^{-1}\) is an inclusion of subgroups of
  \(G\), we have \(i^*V(c_gP) = V(Q)\) for the subgroup \(Q =
  \proj_1^{-1}(K)\) of \(c_gP\).  
\end{proof}
 Note that if \(P \in \hml\) has \(\proj_1(P)=H\) and if \(Q =
 (g,\alpha)P(g,\alpha)^{-1}\) for some \(\alpha \in \Or V\) and \(g
 \in G\), then \(\alpha\) is an isomorphism \(\alpha\colon V(P) \to
 c_g^*V(Q)\) of \(H\)-representations.
 
 \begin{dfn}\label{dfn:retractofrepn}
  Consider a $P \in\AI^V$ and a $Q \in\AI^W$.  We say that $Q$ 
is a
\emph{retract}\index{retract of elements in $G$-typical families} of $P$ if there is an isometry $f\colon W\to V$ such that
 
\[\xymatrix{
    &P\ar[r]^{\subseteq}\ar[d]^\cong&G \times \Or V\ar[d]^{{\proj}_1}\ar[r]^{{\proj}_2}&\Or V\ar[d]_{f^*}\\
K\ar[r]^{\subseteq}&H\ar[r]^{\subseteq}&G&\catL(W,V)\\
Q\ar[rr]^{\subseteq}\ar[u]^\cong&&G \times \Or W\ar[u]_{{\proj}_1}\ar[r]^{{\proj}_2}&\Or W\ar[u]^{f_*}}\]
commutes (start in $K$ and end up in $\catL(W,V)$).
\end{dfn}
\begin{remark}
  This is an instance where classical terminology is a bit easier to digest.  If we think of the $P$ and $Q$ above as representations $\phi\colon H\to \Or V$ and $\psi\colon K\to\Or W$, we see that $(W,\psi)$ is a ``retract'' of $(V,\phi)$ if and only if $K$ is a subgroup of $H$ and there is a $K$-linear isometry $f\colon W\to V$, with $K$ acting on $V$ by restricting $\phi$ to $K$, \ie the diagram
  $$\xymatrix{H\ar[r]^\phi&\Or V\ar[d]_{f^*}\\
    &\catL(W,V)\\
    K\ar[uu]^{\subseteq}\ar[r]^\psi&\Or W\ar[u]^{f_*}
  }$$
  commutes.
\end{remark}


In the next example, if $z$ is a real or complex number of length $1$ we write $z\cdot$ for the element in $\Or\R$ or $\Or\C$ given by multiplication by $z$
\begin{ex}
  Let $G=\T^1$ and $Q=\{(1,1\cdot),(-1,-1\cdot)\}\subseteq \T^1\times \Or {\R}$, corresponding to the sign representation $C_2\to\Or\R$. The fact that the sign representation on $\R$ extends to a unitary representation of $\T^1$ on $\C$ shows that $Q$ is a retract of an element in $\Qll^{\C}$.  In more detail, let $f\colon\R\subseteq\C$ be the standard inclusion and
  $P=\{(z,z\cdot)\mid z\in \T^1\}\subseteq G\times\Or\C$.  Then
\[\xymatrix{&P\ar[r]^{\subseteq}\ar[d]^\cong&G \times \Or\C\ar[d]\ar[r]&\Or \C\ar[d]_{f^*}\\
C_2\ar[r]^{\subseteq}&G\ar[r]^{=}&G&\catL(\R,\C)\\
Q\ar[rr]^{\subseteq}\ar[u]^\cong&&G \times \Or\R\ar[u]\ar[r]&\Or \R\ar[u]^{f_*}}\]
commutes ($\pm1\in C_2$ is sent to $\pm f$ either way), showing that $Q$ is a retract of $P\in\Qll^\C$.  

On the other hand, since for any $g\in \catL(\R,V)$ the homomorphism $\Or g\colon \Or\R\to\Or V$ is an isomorphism of path components, 
we get that 
\[\xymatrix{&P\ar[r]^{\subseteq}\ar[d]^\cong&G \times \Or V\ar[d]\ar[r]&\Or V\\
K\ar[r]^{\subseteq}&H\ar[r]^{\subseteq}&G&\\
 Q\ar[rr]^{\subseteq}\ar[u]^\cong&&G \times \Or \R\ar[u]\ar[r]&\Or \R\ar[uu]^{\Or g}}\]
commutes for no connected $P$.
\end{ex}

\begin{dfn}\label{saturatedgtypicalrepresentations}
  We say that \(G\)-typical family of
    representations \(\hml\) is {\em {closed under retracts}}\index{closed under retracts} if it has the
  property that if $P\in\hml^V$ and $Q\in\AI^W$ is a retract of $P$ (for some inner product spaces $V$ and $W$), then $Q$ is in $\hml^W$. 
\end{dfn}
Obviously,  $\AI$ is closed under retracts.

\begin{dfn}\label{def:closureoffamily}
  Let $\hml$ be a $G$-typical family of representations.  
  The {\em closure}\index{closure!of typical families} \(\overline \hml\)  of \(\hml\) is the subset of $\AI$ consisting of retracts of elements in $\hml$.
\end{dfn}

\begin{lemma}
  The closure $\overline\hml$ of a $G$-typical family of representations $\hml$ is the smallest closed  $G$-typical family of representations containing $\hml$.
\end{lemma}
\begin{proof}
  Since retracts of retracts are retracts, we only need to show that $\overline\hml$ is a $G$-typical family of representations.   In Definition~\ref{gtypicalrepresentations}, as applied to $\overline\hml$, condition (i)-(iii) are immediate.  As to condition (iv), assume (the graph of) the representation $\phi'\colon H'\to \Or{V'}$ is a retract of an element in $\hml$, \ie assume there is an isometry $f\colon V'\to V$ and a representation $\phi\colon H\to\Or V$ (whose graph is) in $\hml^V$ such that $H'\subseteq H$ and such that $f$ is $H'$-linear.  Then, by condition~(iv) for $\hml$ there is a representation $\psi\colon H\to\Or W$ and a representation $\alpha\colon G\to \Or{V\oplus W}$ in $\hml^{V\oplus W}\subseteq\overline\hml^{V\oplus W}$ so that $\alpha|_H=\phi\oplus\psi$.  Since $f$ is $H'$-linear we thus get that $f$ composed with the inclusion $V\subseteq V\oplus W$ is $H'$-linear.
\end{proof}

\begin{lemma}\label{lem:closedretracthmlisbig}
  If a \(G\)-typical family of representations \(\hml\) is closed under retracts, then $\hml^0$ consists of all subgroups of $G\times\Or 0$.
\end{lemma}
\begin{proof}
  By Definition~\ref{gtypicalrepresentations} (iii), there is a Euclidean space $U$ such that $G\times\{\id_U\}\in\hml^U$.  The zero representation is a retract of $U$, and hence $G\times\{\id_0\}$ (and all its subgroups) is in $\hml^0$.
\end{proof}

\begin{example}\label{ex:familiesofGreps}
  \begin{enumerate}
  \item For an inner product space $V$, recall from  Definition~\ref{dfn:AI} the ``All Injective'' family $\AI^V$ consisting of all subgroups \(P\) of \(G \times \Or V\) with \(\proj_1 \colon P \to G\) injective.  Letting $V$ vary, the $\AI^V$ assemble to the maximal \(G\)-typical family 
$\AI=\{\AI^V\}_V$\index{AI@$\AI$}.

Every representation of any subgroup of \(G\) is in this family. This \(G\)-typical family of representations is {closed under retracts}.
\item The {\em positive} part $\AIplus$\index{AIplus@$\AIplus$} of $\AI$ has $\AIplusV 0=\emptyset$ and $\AIplusV V=\AI^V$ for $V\neq 0$.  

The positive part is not closed under retracts, but is the family we will use in conjunction with the ``convenient'' $\Sp$-model structure on commutative orthogonal ring spectra.
\item 
  Let $V$ be an inner product space.  Recall from Definition~\ref{dfn:Qll} that an element in $\Qll^V$ is a subgroup  $Q\subseteq G\times \Or V$ such that there is a subgroup $P\subseteq G\times \Or V$ with $Q\subseteq P$ and such that the projection $P\subseteq G\times \Or V\to G$ is an isomorphism.

In other words, the family $\Qll^V$ consists of the $P$ in $\AI^V$ such that $P$ is the graph $\Gamma_{\varphi|_H}$ of the restriction of a homomorphism $\varphi\colon G\to\Or V$ to a closed subgroup $H\subseteq G$.

Note that $\Qll$ generally is smaller than $\AI$: If for instance $G=\Z/4\Z$ and $V=\R$ (so that $\Or V=\{\pm1\}$), we can let $P=\{(0,1),(2,-1)\}\in\AI^V$ be the graph of the nontrivial action of $2\Z/4\Z$ on $\R$ which can not be extended to $\Z/4\Z$.

This is the family which gives rise to the model structures that typically are called ``$q$-model structures'' ($q$ for Quillen). The content of Lemma~\ref{lem:twistandshout} is that the cofibrations we can build out of free spectra of the form $\Fr{}V(G/H_+)$ (with $V$ the representation corresponding to $\varphi$) correspond exactly to the ones we build out of semi-free spectra indexed over trivial representations of $(G\times\catO_{\mathbf R^n})/P$ for $P\in\Qll^{\mathbf R^n}$.
\item The {\em positive} part $\Qll_+$\index{qplus@$\Qll_+$} of $\Qll$ consists of the graphs $\Gamma_{\varphi|_H}\in\Qll^V$ such that there is a $0\neq v\in V$ with  $v=\varphi(g)v$  for all $g\in G$.

Again, the contents of Lemma~\ref{lem:twistandshout} is -- in a language to be established -- that the ``positive $q$-cofibrations are $\Sp$-cofibration''.  The conditions for the positive part $\Qll_+$ are chosen for compatibility with the classical setup.  Note that for $G=\T^1$ and $P$ the graph of the rotation action of $\T^1$ on $\C$, we have that $\C^{\T^1}=0$, so that $P$ is in $\Qll\cap\AIplus$, but not in $\Qll_+$.
\item
    The argument we gave for $\AI$ also shows that there is a
    minimal \(G\)-typical family of representations which is closed
    under retracts consisting of all subgroups of \(G \cong G \times
    \{\id\}\) considered as subgroups of \(G \times \Or V\). The
    representations in this \(G\)-typical family of representations
    are the trivial representations. 
  \end{enumerate}
\end{example}

Crucial to our applications, by \cite[III.4.5]{BtD85} we have that $\AI$ is the closure of the family $\Qll$:
\begin{lemma}
  Let $G$ be a compact Lie group and $Q\in\AI^W$.  Then $Q$ is an retract of an element of $\Qll$.  Hence, $\overline{\Qll}=\AI$. 
\end{lemma}

\section{Level Model Structures}\label{sec:levelmodel}
In this section \(\hml\) is a fixed but arbitrary \(G\)-typical family
of representations. Given \(V\) in \(\catOr\), we consider the
\(\hml^V\)-model structure (as in Definition~\ref{def:amlmodel}
on \((G \times \Or V) \catT\).
\begin{dfn}\label{hmlmonoidal}
An {\em \(\hml\)-model structure}\index{H model structure@$\hml$-model structure}\index{model structure!$\hml$-} \(\catM\)
consists of 
a sequence \(\catM = (\catM_V)_{V \in \catLr}\) of cofibrantly generated model structures, where each \(\catM_V\) is a \(G\catT\)-model structure on the category \((G \times \Or V) \catT\), considering \((G \times \Or V) \catT\) as enriched and tensored over \(G\catT\) via the isomorphism \((G \times \Or V) \catT \cong \Or V(G \catT)\) together with choices $I_V$ and $J_V$ of sets of generating cofibrations and acyclic cofibrations of $\catM_V$. This data
  is required to satisfy the following four conditions:
  \begin{enumerate}
  \item All cofibrations in \(\catM_V\) are Hurewicz cofibrations (\ref{dfn:h-cof}).
  \item Given Euclidean spaces \(V\) and \(W\) and a generating acyclic cofibration \(j \in J_V\) the map \((\Gr{}{V}j)_W\) is a weak equivalence in \(\catM_W\).
  \item The class of weak equivalences in \(\catM_V\) is the class of \(\hml^V\)-equivalences.
  \item For every \(V \in \catOr\), the class of cofibrations in the 
    \(\hml^V\)-model structure on \((G \times \Or V) \catT\)
    is contained in the class of cofibrations in \(\catM_V\).
  \end{enumerate}
\end{dfn}
In the above definition condition (i) is of a technical kind
used to ensure that homotopy groups have associated cofibration
sequences. Condition (ii) is needed to obtain a level
\(\catM\)-model structure on the category of orthogonal spectra.
Condition (iv) is used to identify the stably fibrant
orthogonal spectra as a kind of \(\Omega\)-spectra. Together with
the other conditions, condition
(iii) determines the stable \(\catM\)-model structure on
orthogonal spectra. Condition (iv) implies that every fibration
\(f\) in
\(\catM_V\) is an \(\hml^V\)-fibration in the sense that \(f^P\)
is a fibration for every \(P\) in \(\hml\).

Below we use the pushout product \(\square\) from Definition~\ref{pushoutproduct}.
    
\begin{dfn}\label{generalizedpushoutproductaxiom}
  Let \(\catM\) be an \(\hml\)-model structure. With \(I_V\) the set of generating cofibrations for \(\catM_V\), we say that \(\catM\) {\em satisfies the pushout product axiom}\index{pushout product!axiom} if it is so that given \(i \in I_V\) and \(j \in I_W\), the map \({\Or {V \oplus W}}_+ \wedge_{\Or V \times \Or W} i \square j\) is a cofibration in \(\catM_{V \oplus W}\), and if in addition one of the former maps is a weak equivalence, so is the latter. 
\end{dfn}

\begin{dfn}\label{classdfnlevel}
Let \(\catM\) be an \(\hml\)-model structure and let $f\colon{} X \rightarrow Y$ be a morphism of orthogonal $G$-spectra.
\begin{indentpar}{0cm}
\begin{enumerate}
\item $f$ is a {\em level \(\catM\)-equivalence}\index{level!\(\catM\)-equivalence, fibration, cofibration} if $f_V$ is an \(\catM_V\)-equivalence for all \(V\) in \(\catLr\).
\item $f$ is a \emph{level \(\catM\)-fibration} if  $f_V$ is a fibration in \(\catM_V\) for all $V$ in \(\catLr\).
\item $f$ is an \emph{\(\catM\)-cofibration} if it satisfies the left lifting property with respect to all maps that are both level \(\catM\)-equivalences and level \(\catM\)-fibrations.
\end{enumerate}
\end{indentpar}
\end{dfn}
\begin{dfn}\label{generatorsforMorthogonalspectra}
  Given an \(\hml\)-model structure \(\catM\), we write 
  $$\Gr{}{}I_{\catM} = \bigcup_V \Gr{}{}I_V\index{GIM@$\Gr{}{}I_{\catM}$}
  \quad\text{and}\quad
    \Gr{}{}J_{\catM} = \bigcup_V \Gr{}{}J_V,\index{GJM@$\Gr{}{}J_{\catM}$}
  $$
  where $\Gr{}{}I_{V}$\index{GIV@$\Gr{}{}I_{V}$} is the set of all maps of the form $\Gr{}{V}i$ for $i$ in the set $I_{V}$ of generating cofibrations for \(\catM_V\) and similarly for $J_V$ the set of generating acyclic cofibrations for \(\catM_V\).\index{GJV@$\Gr{}{}J_{V}$}
\end{dfn}
\begin{thm}\label{catMlevel}
If \(\catM\) is an \(\hml\)-model structure, then the \(\catM\)-cofibrations, level \(\catM\)-equivalences and level \(\catM\)-fibrations give a left proper cofibrantly generated \(G\)-topological  model structure on the category $\catGOS$ of orthogonal \(G\)-spectra. The set \(\Gr{}{} I_{\catM}\) is a set of generating cofibrations, and the set \(\Gr{}{} J_{\catM}\) is a set of generating acyclic cofibrations for this model structure.
\end{thm}
We postpone the proof till the end of the section.

\begin{dfn}
  The model structure of Theorem~\ref{catMlevel} is the {\em level \(\catM\)-model structure}\index{level! \(\catM\)-model structure}\index{model structure!level \(\catM\)-} on \(\catGOS\).
\end{dfn}

\begin{thm}
  If an \(\hml\)-model structure \(\catM\) satisfies the
  pushout product axiom, and \(\hml^0\) contains all subgroups of
  \(G \times \Or 0\), then  the  level \(\catM\)-model structure on \(\catGOS\) is monoidal.
\end{thm}
\begin{proof}
  Since \(\hml^0\) contains all subgroups of
  \(G \times \Or 0\) and the zero sphere \(S^0\) is cofibrant in
  \(\catM_0\), the unit \(\Sp = \Gr{} 0 S^0\) for the monoidal product
  is cofibrant. Thus we only need to verify the pushout product axiom.
  It is a direct consequence of the natural
  isomorphism \[\Gr{}{V}i \square \Gr{}{W}j = \Gr{}{V \oplus W} ({\Or
    {V \oplus W}}_+ \wedge_{\Or V \times  \Or W} i \square j)\]
  and Definition~\ref{generalizedpushoutproductaxiom}.
\end{proof}

We also will need the following result.
\begin{lemma}
Any orthogonal $G$-spectrum is small with respect to levelwise inclusions, in particular with respect to all Hurewicz-cofibrations or levelwise Hurewicz-cofibrations.
\end{lemma}
\begin{proof} 
Let $X$ be an orthogonal spectrum, then by \cite[2.4.1]{H} for
each $n$, there exists a cardinal $\kappa_n$ such that both
$X_n$ and all $\catO(n,m)\smash X$ are $\kappa_n$-small with respect
to inclusions.
Let $\kappa$ be a cardinal with $\kappa > \kappa_n$ for $n \geq 0$ and
$\kappa > \aleph_0$. 
We claim that $X$ is $\kappa$-small. So let \[Y^0 \rightarrow Y^1 \rightarrow \ldots \rightarrow Y^\beta \rightarrow \ldots\] be a $\lambda$-sequence of levelwise inclusions for some $\kappa$-filtered ordinal $\lambda$. In particular $\lambda$ is $\kappa_n$-filtered for $n \geq 0$. We want to show that the map
\[\colim_{\beta < \lambda}\catOS(X,Y^\beta) \rightarrow \catOS(X,\colim_{\beta<\lambda} Y^\beta)\] is an isomorphism. For injectivity consider two elements on the left represented by $f\colon{} X \rightarrow Y^\alpha$ and $g: X \rightarrow Y^\beta$ that are mapped to the same morphism $h: X \rightarrow Y$ on the right side, where $Y$ denotes the colimit. That implies that for each $n$, $f_n$ and $g_n$ induce the same map $X_n \rightarrow Y_n$, which implies that there is a $\gamma_n<\kappa_n$ such that the composites $X_n \stackrel{f}{\rightarrow} Y^\alpha_n \rightarrow Y^{\gamma_n}_n$ and $X_n \stackrel{g}{\rightarrow} Y^\beta_n \rightarrow Y^{\gamma_n}_n$ are equal. Hence for $\gamma \defas{} \sup \gamma_n < \lambda$, $f$ and $g$ already induced the same element $X \rightarrow Y^\gamma$ as desired. For surjectivity let $f\colon{} X \rightarrow Y$ be a map to the colimit. As before we get an ordinal $\gamma$, such that all the maps $f_n: X_n \rightarrow Y_n$ and $f_m \circ \sigma, \sigma \circ f_n : \catO(n,m)\smash X_n \rightarrow Y_m$ factor through $Y^\gamma$. Then for each pair $(n,m)$, there is a cardinal $\gamma < \delta_{n,m}< \lambda$ such that $\sigma \circ f_n$ and $f_m \circ \sigma$ are factor through the same map $f_{n,m}: \catO(n,m)\smash X_n \rightarrow Y^{\delta_{n,m}}_m$ by the argument for injectivity discussed before. Hence there is a factorization of maps of spectra $X\rightarrow Y^{\delta} \rightarrow Y$ for $\delta \defas{} \sup \delta_{n,m} < \lambda$ as desired.
\end{proof}
\begin{rem}\label{hcofsmalrem}
Hence any set $I$ of maps in $G\catOS$ that consists of cofibrations has the property that domains of maps in $I$ are small with respect to $I$-cell. The same holds for $\catA$ a category with a faithful forgetful functor to $G\catOS$ and $I$ a set of maps in $\catA$ satisfying the cofibration hypothesis~\ref{cofhyp}.
\end{rem}
{}

\begin{proof}[Proof (of Theorem~\ref{catMlevel}).]
We use the Assembling Theorem~\ref{puzzling}. Proposition~\ref{assemblenriched} shows that the resulting structure is $G$-topological. 
To apply Theorem~\ref{puzzling} we need to know that the maps in $\Gr{}{}J_{V}$ are actually level equivalences. However, if \(j \in J_V\), then by part (ii) of Definition~\ref{hmlmonoidal} \((\Gr{}{V}j)_W\) is a weak equivalence in \(\catM_W\).

In order to see that the level \(\catM\)-model structure on \(\catGOS\) is left proper we note that (i) of Definition~\ref{hmlmonoidal} implies that if \(f \colon A \to X\) is a level \(\catM\)-cofibration in \(\catGOS\), then \(f_V \colon A_V \to X_V\) is a Hurewicz cofibration for all \(V\). Now left properness is a consequence of part (iii) of Definition~\ref{hmlmonoidal}.
\end{proof}

\section{Mixing Pairs}
\label{sec:mixparis}

We now introduce the concept of a {\em mixing pair},
where we allow the family $\gml$ controlling the cofibrations to be different from a possibly smaller family $\hml$ controlling the weak equivalences (leaving to the fibrations to deal with the inevitable family conflict).  This will turn out to be advantageous at a certain point in our discussion of fixed points under different groups, because it will allow us to keep the old cofibrations even if we have to change the family controlling the weak equivalences.

\begin{dfn}\label{mixing}
  A {\em \(G\)-mixing pair}\index{Gmixing pair@$G$-mixing pair} \((\hml,\gml)\) consists of
  \begin{itemize}
  \item a sequence
  \(\gml = (\gml^V)_{V \in \catLr}\) of families \(\gml^V\) of subgroups of \(G \times \Or V\) and
\item a \(G\)-typical family of representations (c.f.~Definition~\ref{gtypicalrepresentations}) \(\hml\)
  \end{itemize}
 satisfying the following conditions
  for all Euclidean spaces \(V\) and \(W\):
  \begin{enumerate}
  \item \(\hml^V\subseteq\gml^V\)
  \item If $P\in\AI^V$ and $Q\in\AI^W$ with $\proj_1P=\proj_1Q$ is such that the sum (as in Definition~\ref{dfn:sumofrepns}) $P\oplus Q$ is in $\hml^{V\oplus W}$, then $P\in\hml^V$ or $P\notin\gml^V$.
  \item \(\gml\) is closed under sum in the sense that if \(P \in \gml^V\) and \(Q \in \gml^W\), then the isotropy groups of elements of the \(G \times \Or{V\oplus W}\)-space  
\[{\Or {V \oplus W}} \times _{\Or V \times \Or W} ((G \times \Or V)/P \times (G \times \Or W)/ Q)\]
  are in \(\gml^{V \oplus W}\). Here the group \(G\) acts on \((G \times \Or V)/P \times (G \times \Or W)/ Q\) through the diagonal embedding \(G \to G \times G\).
  \end{enumerate}
\end{dfn}
Note that condition~(ii) is \emph{some} sort of ``closure under summands'' in that it says that summands of things in $\hml$ are either in $\hml$ -- or they are not even in $\gml$.  This condition is crucial in Lemma~\ref{mixedisorcat} which provides us with what will be called the level mixed $(\hml,\gml)$-model structure.  We know of no instances where the simpler (but more restrictive) demand \(\overline \hml^V \cap \gml^V=\hml^V\) could not be used.
\begin{example}
  For any $G$-typical family \(\hml\) we get a \(G\)-mixing pair \((\overline\hml,\All)\), where $\overline\hml$ is the closure of $\hml$ as in Definition~\ref{def:closureoffamily}.
\end{example}

\begin{example}\label{ex:qmix}
  For every \(G\)-typical family \(\hml\) of representations satisfying the condition~\ref{mixing} (iii), \((\hml,\hml)\) is a \(G\)-mixing pair. 
\end{example}

It is a consequence of (iii) that \(\gml\) is closed under direct sum, in
the sense that if \(P\) and \(Q\) are in \(\gml\), then so is \(P
\oplus Q\) (Definition~\ref{dfn:sumofrepns}).  Concretely, the condition is perhaps best understood by seeing it in explicit examples as in the proof of the following lemma.

\begin{lemma}
  All the families $\AI$, $\AI_+$, $\Qll$ and $\Qll_+$ of Example~\ref{ex:familiesofGreps} satisfy condition~\ref{mixing} (iii) and so give rise to $G$-mixing pairs of the form $(\hml,\hml)$.
\end{lemma}
\begin{proof}
  We start with the family $\AI$.
  Let $P\in\AI^V$ and $Q\in\AI^W$.  Every element of the \(G \times \Or{V\oplus W}\)-space  
\[\Or {V \oplus W} \times _{\Or V \times \Or W} ((G \times \Or V)/P \times (G \times \Or W)/ Q)\]
is represented by an element of the form \(x =
(\alpha,[g_V,I_V],[g_W,I_W])\) for some \(g_V,g_W \in G\) and \(\alpha
\in \Or {V \oplus W}\). If \((g,\beta)\) is in the isotropy group \((G
\times \Or {V \oplus W})_x\) of \(x\) for some \(g \in G\) and \(\beta
\in \Or {V \oplus W}\), then \(\beta\) is of the form \(\beta =
\alpha(\beta_V \oplus \beta_W)\alpha^{-1}\), and \((g_V^{-1}gg_V, \beta_V) \in P\) and \((g_W^{-1}gg_W,
\beta_W) \in Q\). 

If $(g,\beta)$ lies in the kernel of the projection $\proj_1 \colon    (G \times \Or {V \oplus W})_x \to G$, this means that $g=1$, and so $(1,\beta_V)\in P$ and $(1,\beta_W)\in Q$ which, since $P,Q\in\AI$, means that $\beta_V=I_V$ and $\beta_W=I_W$ and so $\beta=I_{V\oplus W}$.  Consequently  \(\proj_1 \colon    (G \times \Or {V \oplus W})_x \to G\)
is injective and $(G \times \Or {V \oplus W})_x\in\AI^{V\oplus W}$.

For the family $\Qll$ we must do some extra work.
If also $\proj_1\colon P\to G$ and $\proj_1\colon Q\to G$ are isomorphisms, we get that given $g\in G$ there are unique $\beta_V$ and $\beta_W$ such that $(g,\alpha(\beta_V\oplus\beta_W)\alpha^{-1})\in(G\times\Or{V\oplus W})_x$, and so \(\proj_1 \colon    (G \times \Or {V \oplus W})_x \to G\) is an isomorphism.  If $P'\subseteq P$ and $Q'\subseteq Q$ then the isotropy group of an element $\Or {V \oplus W} \times _{\Or V \times \Or W} ((G \times \Or V)/P' \times (G \times \Or W)/ Q')$ of the $G\times\Or{V\oplus W}$-action injects into the isotropy subgroups of the corresponding element of $\Or {V \oplus W} \times _{\Or V \times \Or W} ((G \times \Or V)/P \times (G \times \Or W)/ Q)$.  Hence, if $P'\in\Qll^V$ and $Q'\in\Qll^W$, then the isotropy subgroups of the   $G\times\Or{V\oplus W}$-action on $\Or {V \oplus W} \times _{\Or V \times \Or W} ((G \times \Or V)/P' \times (G \times \Or W)/ Q')$ are in $\Qll^{V\oplus W}$.

The proofs for $\AI_+$ and $\Qll_+$ are similar.
\end{proof}

Given a \(G\)-typical family of representations \(\hml\) and a sequence \(\gml = (\gml^V)_{V \in \catLr}\) of families of subgroups of \(G \times \Or V\) with \(\hml^V\) contained in \(\gml^V\) for all \(V \in \catLr\), we will consider the \((\hml^V,\gml^V)\)-model structure on \((G \times \Or V)\catT\) from Theorem~\ref{GFmixedmodelstr}.  
Explicitly, Theorem~\ref{GFmixedmodelstr} gives that for a \(G\)-mixing pair
\((\hml,\gml)\), the \((\hml^V,\gml^V)\)-model structure on \(G \times\Or V\)-spaces is cofibrantly generated, with set of
generating cofibrations
\begin{displaymath}
  I_{\gml^V} = \{ (i \times (G \times \Or V)/P)_+ \, \mid \,
  i \in I, \, P \in \gml^V \}\index{IGV@\(I_{\gml^V}\)}
\end{displaymath}
 and set of generating acyclic cofibrations 
\begin{displaymath}
  J_{\hml^V,\gml^V} = \{
  (j \times (G \times \Or V)/P)_+ \, \mid \,
  j \in J, \, P \in \gml^V \} \cup K\index{JHVGV@\(J_{\hml^V,\gml^V}\)}
\end{displaymath}
with
\begin{displaymath}
  K = \{ i \Box k_P \, \mid \, i \in I, \,
  k_P = s_P(G/e) \text{ for \(P \in \gml^V, P\notin \hml^V\)}\}
\end{displaymath}
where $s_P$ is is the $(\hml^V,\gml^V)$-localizer of Definition~\ref{definethesetS}.
\newcommand{\IV}{I_{\gml^V}}
\newcommand{\IW}{I_{\gml^W}}
\newcommand{\JV}{J_{\hml^V,\gml^V}}
\newcommand{\JW}{I_{\hml^W,\gml^W}}
Recall that
a map \(f\colon X \to Y\) of orthogonal spectra is a level
\(\hml\)-equivalence if and only if for each Euclidean space \(V\),
the map \(f_V\) is an 
\(\hml^V\)-equivalence. We can formulate this in terms of the genuine
model structure on \(\catGT\) of Definition~\ref{genuineGTstructure}. by noting that
\(f_V\) is an \(\hml^V\)-equivalence if and only if for every \(P \in
\hml^V\), the map \(f_V\) is a genuine equivalence of 
\(P\)-spaces. Writing \(H = \proj_1(P)\), the map \(f_V\) is a genuine
equivalence of \(P\)-spaces if and only if the map \(f_{V(P)}\) is an
genuine equivalence of \(H\)-spaces.  

  \begin{dfn}
  If $(\hml,\gml)$ is a $G$-mixing pair, the \emph{level $(\hml,\gml)$-model structure}\index{level $(\hml,\gml)$-model structure} on orthogonal $G$-spectra is the one given by the Assembling Theorem~\ref{puzzling} from the 
  \((\hml^{\R^n},\gml^{\R^n})\)-model structures 
  from Theorem~\ref{GFmixedmodelstr}.  
\end{dfn}
The level $(\hml,\gml)$-model structure is cofibrantly generated by the sets
$$I_{\gml}=\bigcup_n\Gr{}{\R^n}I_{\gml^{\R^n}},\qquad J_{\hml,\gml}=\bigcup_n\Gr{}{\R^n}J_{\hml^{\R^n},\gml^{\R^n}}.$$
\label{ex:levelqmodel}\index{level $q$-model structure}\index{model structure!level $q$-}

\begin{example}
  The mixing pair $(\Qll,\Qll)$  will give rise to the {\em level $q$-model structure} and $(\Qll_+,\Qll_+)$ to the {\em positive level $q$-model structure}).\index{positive level $q$-model structure}\index{model structure!positive level $q$-} 

Note that, in view of Lemma~\ref{lem:twistandshout}, the set 
$I_\Qll
$ (see Definition~\ref{mixedchangeofgroups}) of generating $q$-cofibrations could without changing the model structure be replaced with the set consisting of the inclusions 
$$\Fr{}V(G/H_+\smsh (S^{n-1}_+\subseteq D^n_+))$$ of {\em free} spectra, with $V$ an orthogonal $G$-representation and $H\subseteq G$ a closed subgroup.  Likewise for the positive level $q$-structure, leaving out the $G$-representations $V$ with only trivial fixed points.
\end{example}
\begin{remark}
  The (positive) $\Qll$-model structure is widely used,
  but has the disadvantage that the natural model structure it induces on commutative ring $G$-spectra has the annoying property that the forgetful functor to orthogonal $G$-spectra does not preserve cofibrations.

  This is one of the reasons for considering the following mixing pair,
  ultimately giving rise to a good stable structure on commutative orthogonal ring $G$-spectra.
\end{remark}

\begin{example}[Positive mixing pair]\label{positivemixingpair}
  The {\em positive mixing pair}\index{positive!mixing pair}\index{AIplus@$\AIplus$}\index{Allplus@$\All_+$} $(\AIplus,\All_+)$ for \(G\) is defined as follows: 
  For
  \(V \ne 0\), the family \(\All_+^V\) consists of all closed subgroups of \(G
  \times \Or {V}\), and \(\AIplusV V\) consists of the closed subgroups \(P\) of
  \(G \times \Or{V}\) with the property that \(\proj_1 \colon P \to G\) is
  injective. Finally, \(\All_+^0=\AIplusV 0=\emptyset\).
\end{example}
There is a maximal or ``nonpositive'' variant $(\AI,\All)$\index{AI@$\AI$}\index{All@$\All$} agreeing with $(\AIplusV V,\All_+^V)$ for all $V\neq0$, but with $\AI^0=\All^0$ consisting of all closed subgroups of $G\cong G\times\Or 0$.  This example is occasionally useful for comparison, but will play no significant role in what follows.

Given a subgroup \(P\) of a group \(A\) and an element \(a\) of \(A\), the conjugation homomorphism \(c_a \colon P \to aPa^{-1}\) is defined by \(c_a(p) = apa^{-1}\).
\begin{lemma}\label{doublecosettwist}
  Let \(P\) and \(B\) be subgroups of a compact Lie group \(A\), let \(a \in A\) and let \(X\) be a \(B\)-space. If \(a^{-1}Pa\) is contained in \(B\), then the map \(f \colon c_{a^{-1}}^* X \to PaB_+\wedge_B X\) with \(f(x) = [a \wedge x]\) is an isomorphism of \(P\)-spaces. If \(a^{-1}Pa\) is not contained in \(B\), then the space of \(P\)-fixed points of \(PaB_+\wedge_B X\) is the one-point space. In particular, the \(P\)-fixed subspace of \(A_+ \wedge _B X\) is the one point space unless \(P\) is subconjugate to \(B\).
\end{lemma}
\begin{proof}
  If \((PaB/B)^P\) is non-empty, then \(paB = aB\) for all \(p \in P\), so \(a^{-1}Pa \subseteq B\). Thus, if \(a^{-1}Pa\) is not contained in \(B\), then the space of \(P\)-fixed points of \(PaB_+\wedge_B X\) is the one-point space. Now suppose \(a^{-1}Pa\) is contained in \(B\). Then 
\[
  f(a^{-1}pax) = [a \wedge a^{-1}pax] = [aa^{-1}pa \wedge x] = [pa \wedge x] = p[a \wedge x] = pf(x),
\]
so \(f\) is a \(P\)-map. Multiplication with \(a^{-1}\) gives an inverse 
\[
  PaB_+ \wedge_B X \to c_{a^{-1}}^*(a^{-1}PaB_+\wedge_B X) = c_{a^{-1}}^*(B_+ \wedge_B X) \cong c_{a^{-1}}^* X
\]
to \(f\). The statement about the \(P\)-fixed subspace of \(A_+ \wedge _B X\) now follows from the fact that \(A\) is the disjoint union of subspaces of the form \(PaB\).
\end{proof}

\begin{prop}\label{mixedisorcat}
  Let \((\hml,\gml)\) be a \(G\)-mixing pair. For each
  Euclidean space \(V\) consider the \(\catM_V=(\hml^V,\gml^V)\)-model structures on \((G \times \Or
  V)\catT\) of Theorem~\ref{defineivandjv}. The collection \(\catM\) of model structures \(\{\catM_V\}_{V}\), where $V$ varies over Euclidean spaces, forms an \(\hml\)-model
  structure in the sense of Definition~\ref{hmlmonoidal} with the property 
  that if \(i \in\IV\) and \(j \in\IW\), then \({\Or{V\oplus
      W}}_+\smash_{\Or{V}\times\Or{W}} i \Box j\) is a cofibration in
  \(\catM_{V \oplus W}\). 
  
If \(\hml\) is {closed under retracts}, then \(\catM\)
  satisfies the \(\hml\)-pushout product axiom in the sense of Definition~\ref{generalizedpushoutproductaxiom}, and \(\Sp\) is cofibrant
  in the level \(\catM\)-model structure on \(\catGOS\). Letting
  \(I_{\catM_V} = \IV{}\) and \(J_{\catM_V} = \JV\)
  as in Theorem~\ref{defineivandjv}, 
  the sets \(\Gr{}{}I_{\catM}\) and
  \(\Gr{}{}J_{\catM}\) are sets of generating
  cofibrations and generating acyclic cofibrations respectively for
  the level \(\catM\)-model structure.
\end{prop}
\begin{proof}
  By design, the weak equivalences in \(\catM_V\) are the \(\hml^V\)-equivalences. The generating cofibrations for the  \((\hml^V,\gml^V)\)-model structure are Hurewicz cofibrations, and thus so are all cofibrations, and (i) of Definition~\ref{hmlmonoidal} holds.

  We need to work harder to verify part (ii) of Definition~\ref{hmlmonoidal}, that is, that the maps in $\Gr{}{}J$ are actually \(\hml\)-level equivalences. So let 
$j\colon X \rightarrow Y$ be a generating acyclic \((\hml^V,\gml^V)\)-cofibration. Then \((\Gr{}{V}j)_W\) is the map
\[\catOr(V,W)\smash_{\Or{V}} j \colon \catOr(V,W)\smash_{\Or{V}}X \rightarrow \catOr(V,W)\smash_{\Or{V}}Y.\] 
We have to show that this is an \(\hml^W\)-equivalence. Suppose there exists an isometric embedding \(\varphi \colon V \to W\). Otherwise \(\catOr(V,W)\) is the one-point space and \((\Gr{}{V}j)_W\) is obviously a weak equivalence. Let \(A = G \times \Or W\) and \(B = G \times \Or{\varphi^\perp} \times \Or {\varphi(V)}\).
Now, $\varphi$ provides us with a natural (in $X$) isomorphism $\catOr(V,W)\wedge_{\Or V}X\cong A_+\smsh_B(S^{\varphi^\perp}\smsh X)$, and so we consider the $P$ fixed points of 
\[
  A_+ \wedge_{B} (S^{\varphi^{\perp}} \wedge X) 
  \xrightarrow{A_+ \wedge_B (\id \wedge j)}
  A_+ \wedge_{B} (S^{\varphi^{\perp}} \wedge Y)
\]
for $P\in\hml^W$.
Lemma~\ref{doublecosettwist} says that the \(P\)-fixed points of the above spaces consist of just the base point unless \(P\) is subconjugate to \(B\).

Consider \(A\) as a \(P \times B^{\op}\)-space with \((p,b) \in P \times B^{\op}\) acting on \(a \in A\) by the rule \((p,b)a = pab\). By Illman's Theorem~\ref{illtriangmfd}, the space \(A\) is an \(P \times B^{\op}\) CW-complex. By Lemma~\ref{isotropytypeofcof} the cells of this \(P \times B^{\op}\) CW-complex are isomorphic to products \(PaB \times D^n\) of an orbit in \(A\) and a disc. Lemma~\ref{doublecosettwist} would imply that \(((PaB \times D^n) {\wedge_B} (\id \wedge j))^P\) and \(((PaB \times S^{n-1}) {\wedge_B} (\id \wedge j))^P\) are weak equivalences if we knew that \((\id \wedge j)^{a^{-1}Pa}\) were a weak equivalence, whenever \(a^{-1}Pa \subseteq B\). An induction on the cells using the Cube Lemma \cite[Lemma 5.2.6]{H} would then give that \((A_+ \wedge_B (\id \wedge j))^P\) is a weak equivalence. 

In order to finish the verification of part (ii) of Definition~\ref{hmlmonoidal} it thus suffices to show that if \(P \in \hml^W\) is a subgroup of \(B\), then \((\id \wedge j)^{P}\) is a weak equivalence. Let \(P_1\) be the image of \(P\) under the projection 
\[
  B = G \times \Or{\varphi^\perp} \times \Or {\varphi(V)} \to G \times \Or 
  {\varphi(V)} \cong G \times \Or {V},
\]
 where the rightmost isomorphism is by conjugation by $\varphi$.
Then by Lemma~\ref{fixedprop} we have \((\id \wedge j)^{P} = \id \wedge j^{P_1}\), and since \(X\) and \(Y\) are cofibrant in the mixed \((\hml^V,\gml^V)\)-model structure, it suffices to show that \(j^{P_1}\) is a weak equivalence.

Now, the composite $P\subseteq G \times \Or{\varphi^\perp}\to G$ (where the last map is the first projection) is injective and factors over the projection $P\to P_1$.
Hence, the projection $P\to P_1$ is actually an isomorphism, and has the same image, say $H\subseteq G$.
If $\alpha\colon H\to\Or{\varphi^\perp}$ and $\beta\colon H\to\Or{\varphi(V)}$ are representations so that 
$P=\{(h,\alpha(h),\beta(h)\mid h\in H\}$ is the graph, 
then necessarily $P_1=\{(h,\beta'(h)\,|\, h\in H\}$, where $\beta'\colon H\to\Or V$ is the conjugate of $\beta$ by $\varphi$.
Hence, $P_1$ is a retract of $P$ and by condition~(ii) of Definition~\ref{mixing} we have that either $P_1\in\hml^V$ or $P_1\notin\gml^V$.
If \(P_1 \in \hml^V\),  then \(j^{P_1}\) is a weak equivalence in \(\catT\). If \(P_1\) is not in \(\gml^V\), then, since both source and target of \(j\) are cofibrant in the mixed \((\hml^V,\gml^V)\)-model structure, this implies that both source and target of \(j^{P_1}\) are the one-point space. In particular \(j^{P_1}\) is a weak equivalence.

Next we show that if \(i \in \IV{}\) and \(j \in \IW\), then \({\Or{V\oplus W}}_+\smash_{\Or{V}\times\Or{W}} i \Box j\) is a cofibration.
Let 
\[i = (\left[S^{n-1}\rightarrow D^n\right] \times {\faktor{(G \times \Or V)}{P}})_+ \quad \text{ and } \quad j = (\left[S^{m-1}\rightarrow D^m\right] \times {\faktor{(G \times \Or W)}{Q}})_+\]
be maps in $\IV{}$ and $\IW$ respectively. Then \({\Or{V\oplus W}}_+\smash_{\Or{V}\times\Or{W}} i \Box j\) is isomorphic to:
\[(\left[S^{n+m-1}\rightarrow D^{n+m}\right] \times {\Or{V\oplus W}} \times_{\Or{V}\times\Or{W}}{\faktor{(G \times \Or {V})}{P}} \times {\faktor{(G \times \Or W)}{Q}})_+\,.\]
Since as a left adjoint, taking the product with a space preserves colimits, it suffices by Illman's Theorem~\ref{illtriangmfd} 
to note that in part (iii) of Definition~\ref{mixing} we require that 
the isotropy groups of the smooth \(G \times \Or {V \oplus W}\) manifold
\[
  {\Or {V \oplus W}} \times_{\Or V \times \Or W} ((G \times \Or V)/P \times (G \times \Or W)/Q) 
\]
are in \(\gml^{V \oplus W}\). 

Suppose that \(\hml\) is {closed under retracts}. We will show that if \(X_V\) is a cofibrant object in \(\catM_V\) and \(f\) is an \(\hml^W\)-equivalence of cofibrant objects of \(\catM_W\), then the map \({\Or {V \oplus W}}_+ \wedge_{\Or V \times \Or W} X_V \wedge f\) is a weak equivalence in \(\catM_{V \oplus W}\). Since the generating acyclic cofibrations for \(\catM_W\) are weak equivalences  of cofibrant objects and source and target of the generating cofibrations are cofibrant, this will imply the statement about acyclic cofibrations in the pushout product axiom.

Let \(A = G \times \Or {V \oplus W}\) and let \(B = G \times \Or V \times \Or W\). Given \(P \in \hml^{V \oplus W}\) it suffices by cell induction to show that \((PaB_+ \wedge_B (X_V \wedge f))^P\) is a weak equivalence. As above, we may without loss of generality assume that \(a^{-1}Pa\) is contained in \(B\). 
Let \(P_1\) be the image of \(a^{-1}Pa\) under the projection 
\[
  B = G \times \Or{V} \times \Or {W} \to G \times \Or {V}
\]
and let \(P_2\) be the image of \(a^{-1}Pa\) under the projection 
\[
  B = G \times \Or{V} \times \Or {W} \to G \times \Or {W}.
\]
By Lemma~\ref{doublecosettwist} \((PaB_+ \wedge_B (X_V \wedge f))^P\) can be identified with \(X_V^{P_1} \wedge f^{P_2}\). Since \(X\) is cofibrant in \(\catM\), the space \(X_V^{P_1}\) is cofibrant, and likewise the \(P_2\)-fixed points of the source and target of \(f\) are cofibrant. Thus it suffices to show that \(f^{P_2}\) is a weak equivalence. However, since \(V(P_2)\) is a subrepresentation of \(V(P)\) and \(\hml\) is {closed under retracts}, we have that \(P_2 \in \hml^W\), and thus \(f^{P_2}\) is a weak equivalence because \(f\) is an \(\hml^W\)-equivalence.

Finally, we verify that if \(\hml\) is {closed under retracts}, then \(\Sp = \Gr{} 0 S^0\) is cofibrant. For this it is enough to know that \(S^0\) is cofibrant in \(\catM_0\). 
By Lemma~\ref{lem:closedretracthmlisbig}
$\hml^0$ contains $G\times\{\id_0\}$.
 Since \(\hml^0 \subseteq \gml^0\) this implies that \(S^0\) is cofibrant in \(\catM_0\).
\end{proof}

\label{comebackwasworkinghereoct23}
\begin{dfn}\label{classdfnlevelsaysMartin}
Let \((\hml,\gml)\) be a \(G\)-mixing pair. The  
model structure on \(\catGOS\) 
in Proposition~\ref{mixedisorcat} is the {\em level mixed \((\hml,\gml)\)-model structure.}\index{model structure!level mixed \((\hml,\gml)\)-}\index{level!mixed \((\hml,\gml)\)-model structure} 
We say that  a morphism of orthogonal $G$-spectra $f\colon{} X \rightarrow Y$ is
\begin{indentpar}{0cm}
\begin{enumerate}
\item  a \emph{level \((\hml,\gml)\)-fibration}\index{level!mixed \((\hml,\gml)\)-fibration, cofibration} if for each \(V\) in \(\catLr\) the map $f_V : X_V \rightarrow Y_V$ is a fibration in the \((\hml^V,\gml^V)\)-model structure on \((G\times \Or V)\catT\).
\item  an \emph{\((\hml,\gml)\)-cofibration} if $f$ satisfies the left lifting property with respect to all maps that are both level \(\hml\)-equivalences and level \((\hml,\gml)\)-fibrations.
\end{enumerate}
\end{indentpar}
The classes of generating cofibrations and generating acyclic cofibrations
for the level mixed $\hml,\gml)$-model structure in Proposition~\ref{mixedisorcat} are referred to as   \(\Gr{}{}I_{\gml}\) and $\Gr{}{}J_{\hml,\gml}$.
\end{dfn}

  \begin{ex}
    \label{ex:level}\index{level $q$-structure}\index{level $\Sp$-structure}\index{positive level $q$-structure}\index{qstructurelevel@$q$-structure, level}\index{Sstructure@$\Sp$-structure, level}
    The mixed model structure with $(\hml,\gml)=(\AIplus,\All_+)$ of Example~\ref{positivemixingpair} is called the (positive) {\em level $\Sp$-structure}. The mixed model structure with $(\hml,\gml)=(\Qll,\Qll)$ of Example~\ref{ex:qmix} is called the {\em level $q$-structure} and the one with $(\hml,\gml)=(\Qll_+,\Qll_+)$ is called the positive {\em level $q$-structure}.  
  \end{ex}

  \begin{lemma}\label{lem:indepofuniv}
    The classes of level $(\hml,\gml)$-cofibrations and level $(\hml,\gml)$-acyclic fibrations do not depend on $\hml$.
  \end{lemma}

In hindsight it may be a bit disconcerting that we in Proposition~\ref{mixedisorcat} claimed that properties of $\hml$ can imply that $\Sp$ is cofibrant.  However, since $\hml\subseteq\gml$, knowledge of $\hml$ {\em can} be used to get knowledge about $\gml$; in this case what we needed was that $\gml^0$ was maximal.

\begin{lem}\label{understandcofprojsphere}
  Let \(P \in \hml^V\) with \(\proj_1(P) = H\) and let \(C\) be a cofibrant replacement of the sphere spectrum \(\Sp\) in the \((\hml,\gml)\)-model structure. The \(H\)-space \(C_{V(P)}\) is \(H\)-homotopy equivalent to the representation sphere \(S^{V(P)}\).
\end{lem}
\begin{proof}
  The map \(C_V \to S^V\) is an \(\hml^V\)-equivalence. Since \(P \in \hml^V\), it is a weak equivalence of cofibrant \(P\)-spaces. Thus it is a \(P\)-homotopy equivalence. Finally, \(H\) acts through the isomorphism \(\proj_1 \colon P \to H\).
\end{proof}

Note that since the mixed \((\hml^V,\gml^V)\)-model structures are left proper and cellular in the sense of Hirschhorn \cite[Definition 12.1.1]{Hir}, the level \((\hml,\gml)\)-model structure is also left proper and cellular. The following theorem is a recollection of results in this section.
\begin{thm}
  Let \((\hml,\gml)\) be a \(G\)-mixing pair. The level 
\((\hml,\gml)\)-model structure on \(\catGOS\) is a left proper and cellular 
\(G\)-topological model structure. The weak equivalences in this model structure are the level \(\hml\)-equivalences. The fibrations are the level \((\hml,\gml)\)-fibrations and the cofibrations are the \((\hml,\gml)\)-cofibrations. If \(\hml\) is {closed under retracts}, then the level \((\hml,\gml)\)-model structure on \(\catGOS\) is monoidal.
\end{thm}
We have seen that \(f\) is a level \(\hml\)-equivalence if and only if for all \(V\) in \(\catLr\) and all \(P \in \hml^V\), the map \(f_{V(P)}\) is an \(\fml^P\)-equivalence of \(H = \proj_1(P)\)-spaces for the family \(\fml^P = \{\proj_1(Q)\cap \proj_1(P) \, | \, Q \in \hml^V \}\) of subgroups of \(H\).

  Note that if \((\hml,\gml)\) and \((\hml,\gml')\) are \(G\)-mixing pairs with \(\gml \subseteq\gml'\), then the identity functor is a left Quillen functor from the level \((\hml,\gml')\)-model structure to the level \((\hml,\gml)\)-model structure on \(\catGOS\). In fact this is the left Quillen functor in a Quillen equivalence.

\newcommand{\iH}{i}
  Given a \(G\)-mixing pair \((\hml,\gml)\) and a subgroup \(H\) of \(G\) with inclusion homomorphism \(\iH \colon H  \to G\) we let \((\iH^*\hml,\iH^*\gml)\) be the \(H\)-mixing pair with \(\iH^* \hml^V\) consisting of the subgroups of \(H \times \Or V\) obtained by intersecting members of \(\hml^V\) with \(H \times \Or V\). Similarly \(\iH^* \gml^V\) consists of the subgroups of \(H \times \Or V\) obtained by intersecting members of \(\gml^V\) with \(H \times \Or V\).

\begin{lemma}\label{levelchangeofgroups}
  Given a \(G\)-mixing pair \((\hml,\gml)\) and a subgroup \(H\) of \(G\) with inclusion homomorphism \(\iH \colon H  \to G\) the 
  functor \(\iH^* \colon G\catOS \to H\catOS\) is 
  a right Quillen functor
  of level mixed model structures with respect to \((\hml,\gml)\) and \((\iH^*\hml, \iH^*\gml)\) respectively.  Hence, the adjoint functor $G_+\smsh_H-$ is a left Quillen functor, and in particular, $G_+\smsh_H-$ takes $\Gr{}{}I_{i^*\gml}$-cofibrations  to $\Gr{}{}I_{\gml}$-cofibrations.
\end{lemma}
\begin{proof}
  This is a direct consequence of Lemma~\ref{mixedchangeofgroups}.
\end{proof}

\chapter{Stable Structures}
\label{ch:ss}
In Chapter~\ref{ch:eos} we constructed mixed level model structures for orthogonal spectra, and in this chapter we localize to obtain the promised $\Sp$-model structure.  The $\Sp$-equivalences are the stable equivalences detected by stable homotopy groups and the $\Sp$-cofibrations we already know from the positive level $\Sp$-structure of Example~\ref{ex:level}.  This should be compared with the Quillen equivalent, but more classical $q$-model structure developed by Mandell and May.  

When we eventually move on to looking at the equivariant structure of smash powers, the $\Sp$-model structure will be much more convenient. Still, it is useful to compare the two model structures since the $q$-cofibrant objects are built out of {\em free} $G$-spectra only and hence are in some respects much better behaved.  Also, since all $\Sp$-fibrant spectra are $q$-fibrant, all results about the latter hold for the former.

We end the chapter with some brief remarks about algebraic structure, but we postpone the most interesting case, namely case of commutative orthogonal $G$-ring spectra, since this relies on some results about smash products fitting later in our narrative.

\section{Stable Equivalences  
}\label{stable_equivalences_ny}
In this section we work with a fixed \(\hml\)--model structure \(\catM\) as in Theorem~\ref{catMlevel}, satisfying the pushout product axiom.

\begin{dfn}\label{def:lambda}
  Given a Euclidean space \(V\) and a subgroup \(P\) of \(G \times \Or
V\), we define the \(G \times \Or V\)-space
\begin{displaymath}
  \widetilde S^P \defas (G \times \Or V)_+ \wedge_P S^V.\index{SP@$\widetilde S^P$}
\end{displaymath}
Here \(P\) acts on \(S^V\) through the projection to \(\Or V\).
Let 
\[\lambda_P \colon \Gr{} V \widetilde S^P \to \Sp\]\index{lambdap@$\lambda_P$}
be the adjoint to action map
\begin{displaymath}
  \widetilde S^P = (G \times \Or V)_+ \wedge_P S^V \to  S^V =  \Sp_V.
\end{displaymath}
Given an orthogonal \(G\)-spectrum \(X\), we let
\begin{displaymath}
  \lambda_P^X \colon \Gr{} V  \widetilde S^P \wedge X  \to X \index{lambdapX@$\lambda_P^X$}
\end{displaymath}
be the composition of \(\lambda_P  \wedge X\) and the structure
isomorphism \(  \Sp  \wedge X \cong X\). 
\end{dfn}
Note that for \(X\) of the
form \(X = \Gr {} W C\) for \(C\) a \(G \times \Or W\)-space, the map
\(\lambda_P^X\) is the composition of the isomorphism 
\[\Gr{} V
\widetilde S^P \wedge \Gr {} W C \cong \Gr{}{V \oplus W}({\Or{V \oplus W}}_+\wedge_{\Or V \times \Or W}
  \widetilde S^P \wedge C)\] 
of Proposition~\ref{semi-free-smash}
and a map of the form
\begin{displaymath}
  \Gr{}{V \oplus W}({\Or{V \oplus W}}_+\wedge_{\Or V \times \Or W}
  \widetilde S^P \wedge C) \to \Gr{}{W} (C).
\end{displaymath}

Using a shear map, we obtain an isomorphism of \((G \times \Or V)\)-spaces
\begin{displaymath}
  \widetilde S^P = (G \times \Or V)_+ \wedge _P S^V \cong 
  (G \times \Or V)/P_+ \wedge  S^V, \quad (g,A) \wedge x \mapsto (g,A)
  \wedge Ax
\end{displaymath}
where \(G \times \Or V\) acts on \(S^V\) 
through the projection to $\Or V$.

\begin{remark}
  In the situation where the projection \(\proj_1 \colon P \to G\) is an
isomorphism, we can interpret the \(G \times \Or V\)-space
\(\widetilde S^P\) as follows:
let \(\varphi \colon G \to \Or V\) be the composition of the inverse
of \(\proj_1 \colon P \to G\), the inclusion of \(P\) in \(G \times \Or
V\) and the projection \(\proj_2 \colon G \times \Or V \to \Or V\).
Let \(\Or {V(P)} \) be the space \(\Or {V} \) with \(G \times \Or
V\) acting via multiplication by elements of \(\Or V\) from the left
and action from the right by the inverse of the 
elements of \(G\).
There is an isomorphism \((G \times
\Or V)_+ \wedge_P S^V \cong {\Or {V(P)}}_+ \wedge S^{V(P)}\),
of \(G \times \Or V\)-spaces taking an element of \((G \times
\Or V)_+ \wedge_P S^V\) represented by \((g,\alpha)\wedge x \in (G \times
\Or V)_+ \wedge S^V\) to \(\alpha \varphi(g^{-1}) \wedge \varphi(g)x
\in \Or V \wedge S^V\). In Example~\ref{freeGorthsp} the
orthogonal \(G\)-spectrum \(\Gr {} V \widetilde S^P\) is denoted
\(\Fr{} V S^{V(P)}\) 
and in \cite{MM}
it is denoted \(F_{V(P)}
S^{V(P)}\), or \(F_W S^W\) for \(W\) any representation of \(G\).
\end{remark}

\begin{dfn}\label{pshiftedvloop}
  Let \(P\) be a subgroup of \(G \times \Or V\) and let \(X\) be an orthogonal \(G\)-spectrum.
  \begin{enumerate}
  \item The {\em negative \(P\)-shifted \(V\)-loop spectrum}\index{negative \(P\)-shifted \(V\)-loop spectrum} of \(X\) is the orthogonal \(G\)-spectrum
    \begin{displaymath}
      R_{P} X \defas \internhom(\Gr {} V \widetilde S^P,X)\index{RPX@$R_{P} X,\, R_{+P}$}
    \end{displaymath}
  \item The {\em positive \(P\)-shifted \(V\)-loop spectrum}\index{positive!\(P\)-shifted \(V\)-loop spectrum} of \(X\) is the orthogonal \(G\)-spectrum
    \begin{displaymath}
      R_{+P} X \defas \Gr {} V \widetilde S^P \wedge X
    \end{displaymath}
  \end{enumerate}
\end{dfn}

\begin{dfn}\label{momegaspectrum}
  An orthogonal \(G\)-spectrum \(X\) is an {\em
    \(\catM\)-\(\Omega\)-spectrum}\index{MOmega spectrum@\(\catM\)-\(\Omega\)-spectrum} if it is fibrant in the 
  level \(\catM\)-model structure and for every subgroup \(P\) of \(G
  \times \Or V\) in \(\hml^V\) with \(\proj_1(P) = G\), the map
  \(\internhom(\lambda_P,X) \colon 
  X \to R_P X\) induced by \(\lambda_P \colon
  \Gr{} V \widetilde S^P \to \Sp\) is a level \(\catM\)-equivalence. 
\end{dfn}

\begin{dfn}\label{mstableequivalence}
  A morphism \(f \colon X \to Y\) of cofibrant \(G\)-orthogonal
  spectra in the level \(\catM\)-model structure is an {\em
    \(\catM\)-stable equivalence}\index{M-stable equivalence@\(\catM\)-stable equivalence}\index{equivalence!\(\catM\)-stable} if for every 
  \(\catM\)-\(\Omega\) spectrum \(E\), the map \(\catGOS(f,E) \colon
  \catGOS(Y,E) \to 
  \catGOS(X,E)\) is an weak equivalence in \(\catT\). More generally,
  a morphism of arbitrary orthogonal \(G\)-spectra is an
  \(\catM\)-stable equivalence if the induced map of cofibrant
  replacements is an \(\catM\)-stable equivalence.
\end{dfn}
The proof of \cite[8.11]{MMSS} gives the following result. 
\begin{lem}\label{stablebetweenomegaislevel}
  Every \(\catM\)-stable equivalence between \(\catM\)-\(\Omega\)-spectra is a
  level \(\catM\)-equivalence.
\end{lem}
The proof of {\cite[III.3.4.]{MM}} gives the following result. 
\begin{lem}\label{levelFislevelH}
  Let \(\fml\) be the \(G\)-typical family of representations
  consisting of all trivial $G$-representations in $\hml$.
  A morphism between \(\catM\)-\(\Omega\)-spectra is a level
  \(\hml\)-equivalence if and only if it is a level \(\fml\)-equivalence.
\end{lem}

\begin{prop}\label{manystableone}
  If \(P\) is in \(\hml\) with \(\proj_1 \colon P \to G\) an
  isomorphism, then \(\lambda_P^X\) 
  is an \(\catM\)-stable equivalence for every cofibrant spectrum \(X\)
  in the level \(\catM\)-model structure.
\end{prop}
\begin{proof}
  We have to show that for all \(\catM\)-\(\Omega\)-spectra \(E\) the
  map
  \[\catGOS(\lambda^X_P,E) = \catGOS(X,\internhom(\lambda_P,E))\]
  is a weak equivalence. However \(\internhom(\lambda_P,E)\) is a level
  equivalence of  level fibrant objects. Since \(X\) is
  cofibrant, the map \(\catGOS(X,\internhom(\lambda_P,E))\) is a weak equivalence.
\end{proof}

\begin{dfn}\label{ploopandsuspension}
  Let \(V\) be a Euclidean space, let \(P\) be a subgroup of \(G \times \Or V\) such that \(\proj_1 \colon
  P \to G\) is an isomorphism and let \(X\) be an orthogonal \(G\)-spectrum.
  \begin{enumerate}
  \item The {\em \(P\)-loop spectrum}\index{Ploop spectrum@$P$-loop spectrum} of \(X\) is the orthogonal \(G\)-spectrum
    \begin{displaymath}
      \Omega^P X \defas \internhom(\Gr {} 0 S^{V(P)},X).\index{OmegaP@$\Omega^P$}
    \end{displaymath}
  \item The {\em \(P\)-suspension spectrum} \index{Psuspension spectrum@$P$-suspension spectrum}of \(X\) is the orthogonal \(G\)-spectrum
    \begin{displaymath}
      \Sigma^P X \defas \Gr {} 0 S^{V(P)} \wedge X.\index{SigmaP@$\Sigma^P$}
    \end{displaymath}
  \end{enumerate}
\end{dfn}
Note that the category of orthogonal \(G\)-spectra is enriched over
the category of \(G\)-spaces, and that \(\Sigma^P X = S^{V(P)}\otimes
X\) is the tensor of \(S^{V(P)}\) and \(X\), and that \(\Omega^P X\)
is the cotensor of \(S^{V(P)}\) and \(X\). Since 
the
\(\catM\)-level 
model structures on the category
of orthogonal \(G\)-spectra 
is a 
\(G\catT\)-model structure
we can
conclude that 
\((\Sigma^P,\Omega^P)\) is a
Quillen adjoint pair of endofunctors of the category of orthogonal
\(G\)-spectra. 
\begin{lem}\label{charstabeq}
  A map \(f \colon X \to Y\) of orthogonal
\(G\)-spectra is an \(\catM\)-stable equivalence if and only
if for every \(\catM\)-\(\Omega\) spectrum \(E\), the induced map \(f^*
\colon [Y,E] \to [X,E]\) of morphism sets in the level
\(\catM\)-homotopy category is a bijection.
\end{lem}
\begin{proof}
  One direction is easy since for \(\catM\)-level cofibrant \(Y\) and
  fibrant \(E\), we have that \([Y,E]\) is the set of components of
  \(\catGOS(Y,E)\). 
  Conversely, suppose that for every \(\catM\)-\(\Omega\) spectrum \(E\), the
  induced map \(f^* \colon [Y,E] \to [X,E]\) of morphisms sets in the
  level
  \(\catM\)-homotopy category is a bijection. Since \(\catM\)-level
  equivalences are \(\catM\)-stable equivalences we may without loss
  of generality assume that \(X\) and \(Y\) are cofibrant in the level
  \(\catM\)-model structure. Given an
  \(\catM\)-\(\Omega\) spectrum \(E'\) we show that the map
  \begin{displaymath}
    \internhom(f,E') \colon \internhom(Y,E') \to \internhom(X,E')
  \end{displaymath}
  is a level equivalence. Since \(X\) and \(Y\) and cofibrant, this
  implies by evaluating at level \(0\) and taking
  \(G\)-fixed points that 
  \begin{displaymath}
    \catGOS(f,E') \colon \catGOS(Y,E') \to \catGOS(X,E')
  \end{displaymath}
  is a weak equivalence. In order to show that \(\internhom(f,E')\) is a level
  equivalence, it suffices to show that for every source or target
  \(C\) of
  a generating cofibration for the level \(\catM\)-model structure,
  the morphism \([C,\internhom(f,E')]\) in the level \(\catM\)-homotopy
  category is an isomorphism. However, this map is isomorphic to the
  adjoint \([f,\internhom(C,E')]\), and since \(C\) is cofibrant, the spectrum
  \(\internhom(C,E')\) is an \(\catM\)-\(\Omega\)-spectrum.
\end{proof}

\begin{lem}\label{comhmlhml}
  Let \(V\) be a Euclidean space. For every member \(P\) of \(\hml^V\) the \(G \times \Or V\)-space \((G \times \Or V/P)_+\) is cofibrant in the \(\catM_V\)-model structure.
\end{lem}
\begin{proof}
  Since the isotropy groups of the \(G \times \Or V\)-CW-complex \((G
  \times \Or V/P)_+\) are subconjugate to \(P\) it is cofibrant in the
  mixed \((\hml^V,\hml^V)\)-model structure. By part (iv) of
  Definition~\ref{hmlmonoidal} this implies that it is cofibrant in the
  \(\catM_V\)-model structure.
\end{proof}
\begin{cor}
  If \(P \in \hml^V\), then \(\widetilde S^P\) is cofibrant in
  \(\catM_V\). 
\end{cor}

\begin{dfn}\label{shiftedspectra}
  Let \(P\) be a subgroup of \(G \times \Or V\) such that \(\proj_1 \colon
  P \to G\) is an isomorphism and let \(X\) be an orthogonal \(G\)-spectrum.
  \begin{enumerate}
  \item The {\em negative \(P\)-shift}\index{negative \(P\)-shift} of \(X\) is the orthogonal \(G\)-spectrum
    \begin{displaymath}
      s_{-P} X \defas \internhom(\Gr {} V(G \times \Or V/P)_+,X)\index{sp@$s_{-P},\, s_{+P}$}
    \end{displaymath}
  \item The {\em positive \(P\)-shift}\index{negative \(P\)-shift} of \(X\) is the orthogonal \(G\)-spectrum
    \begin{displaymath}
      s_{+P} X\defas \Gr {} V(G \times \Or V/P)_+ \wedge X
    \end{displaymath}
  \end{enumerate}
\end{dfn}
\begin{lem}\label{twistedorthogonalgroups}
  Let \(P\) be a subgroup of \(G \times \Or V\) such that \(\proj_1
  \colon P \to G\) is an isomorphism and let \(X\) be an orthogonal
  \(G\)-spectrum.
  \begin{enumerate}
  \item   There are isomorphisms
    \[
    s_{-P} X \cong \internhom(\Gr {} V {{\Or {V(P)}}_+},X) 
    \qquad \text{and} \qquad
    s_{+P} X \cong \Gr {} V {{\Or {V(P)}}_+} \wedge X.
    \]
    In particular \((s_{-P}X)_W \cong X_{W \oplus V(P)}\).
  \item  The \(P\)-shifted \(V\)-loop spectrum \(R_P X\)
    (cf. Definition~\ref{pshiftedvloop}) 
    is naturally \(G\) 
    isomorphic to both \(\Omega^P s_{-P} X\) and \(s_{-P} \Omega^P X\).
  \end{enumerate}
\end{lem}
\begin{proof}
  Part (i) is a consequence of the isomorphism \(G \times \Or V/P
  \cong \Or {V(P)}\) from Equation~\ref{ovpgrep}. Part (ii) is a direct
  consequence of the
  isomorphisms 
  \begin{displaymath}
    \widetilde S^P \cong 
    S^{V(P)} \wedge {\Or {V(P)}}_+
    \cong
    {\Or {V(P)}}_+ \wedge S^{V(P)}.
  \end{displaymath}
\end{proof}
\begin{lem}\label{RPperserveslevelhml}
  For every \(G\)-typical family of representations \(\hml\) and every
  subgroup \(P\) of \(G \times \Or V\) in \(\hml\) with \(\proj_1(P) = G\), the functor \(R_{P}\) preserves level \(\hml\)-equivalences.
\end{lem}
\begin{proof}
  This is a consequence of the fact that in the level \((\hml,\hml)\)-model structure \(\Gr {} V \widetilde S^P\) is cofibrant and every object is fibrant.
\end{proof}

\begin{prop}\label{lambdastarstableeq}
  Let \(X\) be a  level \(\catM\)-fibrant orthogonal
  \(G\)-spectrum. If \(P \in \hml\) with \(\proj_1(P) = G\), then the
  map \(i = \internhom(\lambda_P,X) \colon  X \to R_P X\) is an
  \(\catM\)-stable equivalence. 
\end{prop}
\begin{proof}
  If \(E\) is an \(\catM\)-\(\Omega\)-spectrum, that is, \(i \colon E
  \to R_PE\) is a level \(\catM\)-equivalence, then so is and
  \(R_PE\). Let \(X\) be an \(\catM\)-level fibrant orthogonal
  \(G\)-spectrum. 
  There is a commutative diagram of morphism sets in the
  level \(\catM\)-homotopy category:
  \begin{displaymath} 
    \xymatrix{ 
      [X,E]
      \ar[dr]^-{R_P} \ar[r]^-{i_*} & 
      [X,R_PE] \\ 
      [R_PX,E] \ar[u]_{i^*} \ar[r]^-{i_*} & \ar[u]_{i^*}
      [R_PX,R_PE], 
    }
  \end{displaymath} 
  where both of the functions labeled \(i_*\) are
  bijections. It follows that 
  \[R_P \colon [X,E] \to
  [R_PX, R_PE]\] 
  is a
  bijection.  Hence 
  \[i^* \colon [R_PX,E] \to [X,E]\] 
  is a
  bijection. Now we apply Lemma~\ref{charstabeq}.
\end{proof}
\begin{cor}\label{lambdastarstableeqcor}
 If \(P\) is in \(\hml^V\) and \(\proj_1 \colon P \to G\) is an
 isomorphism, then for every orthogonal \(G\)-spectrum 
 \(X\), the map \(i = \internhom(\lambda_P,X) \colon X \to R_P X\) is an
 \(\catM\)-stable equivalence. 
\end{cor}
\begin{proof}
  Let \(X \to Y\) be a level \(\catM\)-equivalence with \(Y\) level
\(\catM\)-fibrant. Thus \(X \to Y\) is a level \(\hml\)-equivalence,
and by Lemma~\ref{RPperserveslevelhml} the morphism \(R_PX \to R_P Y\) is
also a level \(\hml\)-equivalence. Since level \(\hml\)-equivalences
are stable \(\catM\)-equivalences the result is a consequence of \(Y
\to R_P Y\) being an \(\catM\)-stable equivalence by
Proposition~\ref{lambdastarstableeq} and the commutative square
  \begin{displaymath}
    \xymatrix{
      X \ar[r] \ar[d]_{\simeq} & R_PX \ar[d]_{\simeq} \\
      Y \ar[r]^-{\sim} & R_PY.
    }
  \end{displaymath}
\end{proof}

\section{Homotopy Groups 
}\label{homotopygroupsofspectra}

We fix a \(G\)-typical family of representations
\(\hml\) and an \(\hml\)-model structure \(\catM\). 
The following paragraphs paves the way for Definition~\ref{def:tildelambda} of the  \(\catT\)-functor
\begin{displaymath}
  \widetilde \lambda_G \colon (G\catLr)^{\op} \to G \catOS \cong \catOr G \catT\index{lambdaG@$\widetilde \lambda_G$}
\end{displaymath}
whose value at an object \(V\) of \(G\catLr\) is the functor
\(\widetilde \lambda_G(V) \colon \catOr \to G\catT\) taking an
Euclidean space \(W\)  to the \(G\)-space \(\widetilde \lambda_G(V)_W = \catOr(V,W) \wedge S^V\) with \(G\) acting diagonally on \(\catOr(V,W)\) and \(S^V\).

Recall from Definition~\ref{untwistingmap} the untwisting \(G\)-isomorphism
\begin{displaymath}
  \tau_{V,W} \colon \catOr(V,W) \wedge S^V \to \catLr(V,W) \wedge S^W
\end{displaymath}
taking an element in \(\catOr(V,W) \wedge S^V\) of the form \((f \colon V \to W, w \in f(V)^{\perp}, v \in V)\) to the element \((f, w + f(v))\) in \(\catLr(V,W) \wedge S^W\). The functoriality of \(\widetilde \lambda_G\) is described via a \(\catT\)-functor \(l \colon (G\catLr)^{\op} \wedge \catOr \to G\catT\) defined on objects by \(l(V,W) = \widetilde \lambda_G(V)_W = \catOr(V,W) \wedge S^V\). The map
\begin{displaymath}
  l \colon (G\catLr)^{\op}(V,V') \wedge \catOr(W,W') \to G\catT(\catOr(V,W) \wedge S^V,\catOr(V',W') \wedge S^{V'})\index{l@$l$}
\end{displaymath}
is adjoint to the \(G\)-map
\begin{displaymath}
  G\catLr(V',V) \wedge \catOr(W,W') \wedge \catOr(V,W) \wedge S^V \to \catOr(V',W') \wedge S^{V'}
\end{displaymath}
given as the composition
\begin{align*}
  G\catLr(V',V) \wedge \catOr(W,W') \wedge \catOr(V,W) \wedge S^V &\to
  G\catLr(V',V) \wedge \catOr(V,W') \wedge S^V \\
  &\to
  G\catLr(V',V) \wedge \catLr(V,W') \wedge S^{W'} \\
  &\to
  \catLr(V',W')  \wedge S^{W'} \\
  &\to
  \catOr(V',W') \wedge S^{V'},
\end{align*}
where the first map is induced by composition in \(\catOr\), the
second map is induced by the untwisting map \(\tau_{V,W'}\), the third
map is induced by composition in \(\catLr\) and the last map is the
inverse of \(\tau_{V',W'}\). Given
\begin{displaymath}
  f \wedge (g,x)\wedge (h,y) \wedge z \in G\catLr(V',V) \wedge
  \catOr(W,W') \wedge \catOr(V,W) \wedge S^V 
\end{displaymath}
with \(f \colon V' \to V\), \(g \colon W \to W'\) and \(h \colon V \to
W\)  embeddings and with \(x \in W'\) in the orthogonal complement of
\(g(W)\), \(y \in 
W\) in the orthogonal complement of \(h(V)\) and \(z \in V\) we have
\begin{displaymath}
  l(f \wedge (g,x)\wedge (h,y) \wedge z) = (ghf,w) \wedge v,
\end{displaymath}
where \(v \in V'\) and \(w \in W'\) are uniquely determined by
requiring \(ghf(v) + w = x + gy + ghz\).
It is a consequence of Lemma~\ref{magicaboutuntwisting} that this defines a functor \(l \colon (G\catLr)^{\op} \wedge \catOr \to G\catT\).
\begin{dfn}\label{def:tildelambda}
  We write 
$$\widetilde \lambda =\widetilde \lambda_G\colon (G\catLr)^{\op} \to G\catOS\index{lambda@$\widetilde\lambda$}$$
for the functor adjoint to the \(l \colon (G\catLr)^{\op} \wedge \catOr \to G\catT\) defined in the preceding paragraph.
\end{dfn}

This functor is closely related to the functor $\lambda_P \colon \Gr{} V \widetilde S^P \to \Sp$ of Definition~\ref{def:lambda}.
\begin{lem}
  If \(P \in \hml^V\) has \(\proj_1P = G\), then \(\lambda_P = \widetilde \lambda_G(0 \to V(P))\).
\end{lem}
We say that an element \(P\) in \(\hml^V\) 
is {\em \(\hml\)-irreducible}\index{irreducible!$\hml$} if \(\proj_1 \colon P \to G\) is an
isomorphism and \(P\) is not of the form \(Q \oplus R\)
for \(Q \in \hml^W\) and \(R \in \hml^{W'}\) with both \(W\) and
\(W'\) non-zero Euclidean spaces.  Observe that if \(V\) is a non-zero
Euclidean space, then every member of \(\hml^V\) is a direct sum of
\(\hml\)-irreducible elements of \(\hml\). 

Let \(B\) be a set containing one representative for each conjugacy 
class of \(\hml\)-irreducible elements of \(\hml\).  
We write \(B^*\) for the free monoid on the set
\(B\). Given two words \(\PP = P_1\dots P_m\) and \(\PP' = P'_1 \dots P'_n\) in
\(B^*\), we write \(\PP \le \PP'\) if there is an increasing sequence \(i_1
< \dots < i_m\) such that \(P_j = P'_{i_j}\) for \(j = 1,\dots , m\),
that is, if \(\PP\) can be obtained from \(\PP'\) by removing some letters. This is a partial order, and we consider \(B^* = (B^*,\le)\) as a category.

The morphisms \(\lambda_{\PP}\) for \(\PP \in B\)
give us a functor
\(\lambda_B^* \colon (B^*)^{\op} \to \catGOS\) taking \(\PP = P_1\dots
P_m\) to \(\lambda_B(\PP) = \bigwedge_{i = 1}^m \Gr {} {V(P_i)}
\widetilde S^{P_i}\). Consider the free commutative monoid \(\N\{B\}\)
as a filtered partially ordered set with \(\sum n_{\PP} \PP \le \sum m_{\PP} \PP\)
if and only if \(n_{\PP} \le m_{\PP}\) for all \(\PP \in B\). Choosing a total
order on \(B\), there is an order-preserving function \(\N\{B\} \to
B^*\) taking \(n_1 P_1 + \dots n_r P_r\) with \(P_1 < \dots < P_r\) to
\(P_1^{n_1} \dots P_r^{n_r}\).  
Given \(\PP = n_1 P_1 + \dots n_r P_r \in \N\{B\}\) with \(P_1 < \dots <
P_r\) we define \(V(\PP) = V(P_1)^{\oplus n_1} \oplus \dots\oplus
V(P_r)^{\oplus n_r}\), and we define \(\Omega^{\PP} =
(\Omega^{P_r})^{\circ n_r} \circ \dots \circ (\Omega^{P_1})^{\circ n_1}\)
\begin{dfn}\label{structuremapsforhomotopygroups}
  The functor \(\lambda_B \colon \N\{B\}^{\op} \to \catGOS\)\index{lambdaB@$\lambda_B$} is the composition of the functors \(\N\{B\}^{\op} \to (B^*)^{\op}\) and \(\lambda_B^* \colon (B^*)^{\op} \to \catGOS\).
\end{dfn}
\begin{dfn}
  Let $X$ be an othogonal $G$-spectrum.  If $V$ is a Euclidean space we consider the homotopy colimit of \(G\)-spaces
 $$(Q_{\hml}X)_V=\hocolim _{\PP \in \N\{B\}} \internhom(\lambda_B(\PP), X)_V\cong
    \hocolim_{\PP \in \N\{B\}} \Omega^{\PP} X_{V(\PP) \oplus V} 
$$
and let \(Q_{\hml}X\)\index{QHX@$Q_{\hml}X$} denote the resulting  orthogonal \(G\)-spectrum.
  
  The spectrum \(QX\)\index{QX@$QX$} is the fibrant replacement of \(Q_{\hml}X\) in
  the level \(\catM\)-model structure.
\end{dfn}

\begin{lem}\label{Qstablesame}
Inclusion of \(X \cong \internhom(\lambda_B(0),X)\) in the homotopy colimit defining \(QX\) gives a natural map \(X \to QX\) which is an \(\catM\)-stable equivalence.
\end{lem}
\begin{proof}
  Let \(\widetilde X\) be a cofibrant
  replacement of the functor \(\PP \mapsto \internhom(\lambda_B(\PP), X)\) in
  the  model structure on the category of \(\N\{B\}\)-diagrams
  in the level \(\catM\)-model structure on \(\catGOS\). In particular,
  given a morphism \(\beta \colon \PP \to \PP'\) in \(\N\{B\}\), the map
  \(\widetilde X^{\beta} \colon \widetilde X^{\PP} \to \widetilde X^{\PP'}\)
  is a map between
  cofibrant objects and \(\hocolim_{\PP \in \N\{B\}}
  \widetilde X^{\PP}\) is a cofibrant replacement of \(QX\).

  Let \(E\) be an
  \(\catM\)-\(\Omega\)-spectrum (Definition~\ref{momegaspectrum}) and let \(\beta \colon \PP \to \PP'\) be a
  morphism in \(\N\{B\}\). By Corollary~\ref{lambdastarstableeqcor} the map \(\widetilde X^{\beta}\) is an 
  \(\catM\)-stable equivalence. Given a Euclidean space \(V\) and
  \(P\) in \(\hml^V\), the \(G\)-orthogonal spectrum \(\Gr{} V (G
  \times \Or V/P_+)\) is cofibrant in the level \(\catM\)-model
  structure, and thus
  \begin{displaymath}
    \internhom(\widetilde X^{\beta},E)_V^P \cong \catGOS(\Gr{} V (G
    \times \Or V/P_+) \wedge \widetilde X^{\beta},E)
  \end{displaymath}
  is a weak equivalence. That is, \(\internhom(\widetilde
  X^{\beta},E)\) is a level \(\catM\)-equivalence. Now, since
  \(\hocolim_{\beta \in \N\{B\}} 
  \widetilde X^{\beta}\) is a cofibrant replacement of \(QX\), the isomorphism
  \begin{displaymath}
    \internhom(\hocolim_{\beta \in \N\{B\}} \widetilde X^{\beta},E) \cong
    \holim_{\beta \in \N\{B\}^{\op}} \internhom(\widetilde X^{\beta},E) 
  \end{displaymath}
  shows that \(X \to \hocolim_{\beta \in \N\{B\}} \widetilde X^{\beta}\) is an
  \(\catM\)-stable equivalence.
\end{proof}

\begin{prop}\label{QisOmega}
  For every orthogonal \(G\)-spectrum \(X\),
  the orthogonal \(G\)-spectrum
  \(QX\) is an \(\catM\)-\(\Omega\)-spectrum. 
\end{prop}
\begin{proof}
  It suffices to show that for every \(P \in \hml\) with
  \(\proj_1\colon P \to G\) an isomorphism the map 
  \[\internhom( 
  \lambda_{P},QX) \colon QX \to R_{P} QX\] 
  is a level
  \(\catM\)-equivalence. Since spheres are compact, the canonical map
  \(R_PQX \to 
  QR_PX\) is a level \(\catM\)-equivalence. Thus it suffices to show
  that the map 
  \[Q\internhom( \lambda_{P},X) \colon QX \to Q
  R_{P} X\] 
  is a level \(\catM\)-equivalence. 
  Given \(b\) in \(\N\{B\}\), we write \(X^b = \internhom(\lambda_B(b),X)\),
  and given a morphism    
  \(\beta\) in \(\N\{B\}\), we let 
  \(X^\beta = \internhom(\lambda_B(\beta),X)\).
  We have to show for every Euclidean space \(W\) the map
  \(i_X \colon X \to R_PX\) induces a \(\hml^W\)-equivalence
  \begin{displaymath} 
    \hocolim_{b \in \N\{B\}}  X^b_W \to 
    \hocolim_{b \in \N\{B\}} (R_PX)^b_W.
  \end{displaymath} 
  Choose an isomorphism \(V(P) \cong \bigoplus_{i = 1}^k V_{p_i}\) for some not
  necessarily distinct \(p_1,\dots,p_k
  \in B\) and let \(p = \sum_{i = 1}^k p_i \in \N\{B\}\). 
  We let \(\N\{P\} \subseteq \N\{B\}\) be the partially ordered
  subset
  consisting of elements of the form \(n p\) for \(n \in \N\), and we
  let \(B_P \subseteq B\) be the complement of \(\{p_1,\dots,p_k\}\) in
  \(B\). The sum in \(\N\{B\}\) induces a cofinal inclusion \(\N\{P\} \times
  \N\{B_P\} \to \N\{B\}\) of partially ordered sets. Therefore it
  suffices to show that the map
  \(i_X \colon X \to R_PX\) induces an \(\hml^W\)-equivalence
  \begin{displaymath} 
    \hocolim_{n \in \N}  X^{np}_W \to 
    \hocolim_{n \in \N} (R_PX)^{np}_W.
  \end{displaymath} 
  We thus have to show that for every \(k
  \ge 0\) and every
  \(Q\) in \(\hml^W\) (and for all choices of base point in
  \((\hocolim_{n \in \N}  X^{np}_{W})^{Q}\)), the homomorphism  
  \begin{displaymath} 
    \pi_k (\hocolim_{n \in \N}  X^{np}_{W})^{Q} \to 
    \pi_k (\hocolim_{n \in \N} (R_PX)^{np}_{W})^{Q}
  \end{displaymath} 
  is an isomorphism. By compactness of \(S^k \wedge (G \times
  \Or {W}/Q)_+\), this is an isomorphism if and only if the homomorphism 
  \begin{displaymath} 
    \colim_{n \in \N} \pi_k ( X^{np}_{W})^{Q} \to 
    \colim_{n \in \N} \pi_k ((R_PX)^{np}_{W})^{Q}
  \end{displaymath}  
  is an isomorphism.
  We consider diagrams of the form
  \begin{displaymath} 
    \xymatrix{  X^{np} \ar[rr]^{i_{ X^{np}}}
      \ar[d]_{(i_X)^{np}} &&  R_P  X^{np} \ar[d]^{R_P( i_X)^{np}}
      \\
      (R_P X)^{np}  \ar[rr]_-{i_{(R_P X)^{np}}}
      \ar[urr]
      && R_P(R_P X)^{np},
    }
  \end{displaymath}
  where the diagonal arrow is induced by the morphism shifting a
  factor of \(\Gr{} V  S^P\) from the front to the back. 
  The
  outer square in the above diagram commutes, as does the upper
  triangle, but the lower triangle does not. However, by Schur's Lemma,
  the space \(G\catLr(V(P),V(P)\oplus V(P))\) is path connected, so
  the lower triangle commutes up
  to homotopy. From the commutative diagram
  \begin{displaymath} 
    \xymatrix{ \pi_k( X^{np}_W)^{Q} \ar[rr]^{i_{ X^{np}}}
      \ar[d]_{(i_X)^{np}} &&  \pi_k(R_P  X^{np}_W)^{Q} \ar[d]^{R_P( i_X)^{np}}
      \\
      \pi_k((R_P X)^{np}_W)^{Q}  \ar[rr]_-{i_{(R_P X)^{np}}}
      \ar[urr]
      && \pi_k(R_P(R_P X)^{np}_W)^{Q},
    }
  \end{displaymath}
  we conclude that the homomorphisms
  \begin{displaymath} 
    \colim_{n \in \N} \pi_k ( X^{np}_{W})^{Q} \to 
    \colim_{n \in \N} \pi_k ((R_PX)^{np}_{W})^{Q}
  \end{displaymath}  
  are isomorphisms. (For all choices of base point in one of the spaces
  \((X^{np}_{W})^{Q}\).)
\end{proof}
\begin{cor}\label{Qcreatesanddetectsequivalences}
  A map \(X \to Y\) of orthogonal \(G\)-spectra is an \(\catM\)-stable equivalence
  if and only if \(QX\to QY\) is a level \(\catM\)-equivalence. 
\end{cor}
\begin{proof}
  If \(X \to Y\) is an \(\catM\)-stable equivalence,
  then by Lemma~\ref{Qstablesame} and Proposition~\ref{QisOmega} the induced map \(QX \to QY\) is a
  stable equivalence of \(\catM\)-\(\Omega\)-spectra,
  and by Lemma~\ref{stablebetweenomegaislevel} it is a level \(\catM\)-equivalence.
  Conversely, if \(QX \to QY\) is a level \(\catM\)-equivalence, 
  then it is also a stable equivalence, and \(X \to Y\) is an \(\catM\)-stable
  equivalence by Lemma~\ref{Qstablesame}.
\end{proof}
\begin{cor}\label{cor:stabdeponclosure}
  All \(\hml\)-model categories \(\catM\) have the same class of
  stable \(\catM\)-equiva\-lences as the \((\overline
  \hml,\overline \hml)\)-model structure,
where $\overline \hml$ is the closure of $\hml$ (c.f.~Definition~\ref{def:closureoffamily}). In particular, the class of
  \(\catM\)-stable equivalences only depends on 
  \(\overline \hml\).
\end{cor}
\begin{proof}
  Let \(\overline B\) be a set containing one representative for each conjugacy 
class of \(\overline \hml\)-irreducible elements of \(\overline
\hml\). Then \(\N\{B\}\) is cofinal in \(\N\{\overline B\}\), so for
every Euclidean space \(V\), the canonical map \(Q_{\hml}X \to Q_{\overline
  \hml}X\) is a level \(\overline \hml\)-equivalence, and these are
both level \(\catM\)-equivalent to \(QX\).
\end{proof}

In view of the above result, we say that a morphism of orthogonal
\(G\)-spectra is an {\em \(\hml\)-stable equivalence}\index{equivalence!\(\hml\)-stable}\index{Hstable equivalence@ \(\hml\)-stable equivalence} if it is a
\(\catM\)-stable equivalence for some (and hence every) \(\hml\)-model
structure \(\catM\).
\begin{cor}\label{Qtakesfibtofib}
  If \(X \to Y \to Z\) is a fibration sequence in the level
  \((\hml,\hml)\)-model structure, then so is \(QX \to QY \to QZ\).
\end{cor}
\begin{proof}
  Let \(b \in \N\{B\}\). If \(X \to Y \to Z\) is a fibration
  sequence in the level 
  \((\hml,\hml)\)-model structure, then so is the sequence
  \(\internhom(\lambda_B(b),X) \to
  \internhom(\lambda_B(b),Y) \to
  \internhom(\lambda_B(b),Z).
  \)
  From 
  \cite[Theorem 14.19]{HirschhornHocolim} we conclude that \(QX \to QY
  \to QZ\) is a fibration
  sequence in the level \((\hml,\hml)\)-model structure.
\end{proof}

\begin{dfn}\label{definitionofhomotopygroups}
  Let \(k\) be an integer, let \(X\) be an orthogonal \(G\)-spectrum
  and let \(H\) be a subgroup of \(G\). The {\em homotopy group}\index{homotopy group!\(\pi_k^H(X,\hml)\)}\index{pikHXH@\(\pi_k^H(X,\hml)\)}
  \(\pi_k^H(X,\hml)\) is defined as the following abelian group  
  \begin{displaymath}
    \pi^H_k(X,\hml) \defas
    \begin{cases}
      \pi_k (QX)_0^H & k \ge 0\\
      \pi_0 (QX)_{\R^k}^H & k < 0,
    \end{cases}
  \end{displaymath}
where the fibrant replacement $Q$ is with respect to the $(\hml,\overline\hml)$-structure.
\end{dfn}
By compactness, there are isomorphisms
\begin{displaymath}
  \pi^H_k(X,\hml) \cong
  \begin{cases}
    \colim_{b \in \N\{B\}} 
    \pi_k(\Omega^{V_b} X_{V_b})^H
    & k \ge 0 \\
    \colim_{b \in \N\{B\}} 
    \pi_0(\Omega^{V_b}X_{V_b \oplus \R^k})^H 
    & k < 0.
  \end{cases}
\end{displaymath}

\begin{prop}\label{prop:stableispistar}
  A map \(f \colon X \to Y\) of orthogonal \(G\)-spectra is an
  \(\hml\)-stable equivalence if and only if the induced homomorphism
  \(\pi^H_k(f,\hml)\) is an isomorphism for every subgroup \(H\) of
  \(G\) and for every integer \(k\).
\end{prop}
\begin{proof}
  Let \(\fml\) be the \(G\)-typical family of representations
  consisting of the trivial \(G\)-representations in $\hml$. 
  If the induced homomorphism
  \(\pi^H_k(f,\hml)\) is an isomorphism for every subgroup \(H\) of
  \(G\) and for every integer \(k\), then the map \(Qf \colon
  QX \to QY\) a level \(\fml\)-equivalence. Since it is a map of
  \(\catM\)-\(\Omega\)-spectra, Lemma~\ref{levelFislevelH} implies
  that it is a level
  \(\catM\)-equivalence. Thus \(f\) is an \(\catM\)-stable equivalence.

  Conversely, if \(f\) is an \(\catM\)-stable equivalence, then \(Qf\)
  is an level \(\catM\)-equivalence, and also a level
  \(\fml\)-equivalence. Thus the induced homomorphism
  \(\pi^H_k(f,\hml)\) is an isomorphism for every subgroup \(H\) of
  \(G\) and for every integer \(k\).
\end{proof}

\begin{remark}\label{rem:compMM}
  The infinite dimensional \(G\)-representation \(\universeU = \colim_{n \in
  \N} \bigoplus_{b \in B} V^n_b\) is a universe of \(G\)-representations in the sense of Mandell and May who define equivariant stable homotopy groups as follows:
\begin{displaymath}
  \pi_k^H(X,\universeU) =
  \begin{cases}
    \colim_{U \subseteq \universeU} 
    \pi_k(\Omega^{U} X_{U})^H
    & k \ge 0 \\
    \colim_{U \subseteq \universeU} 
    \pi_0(\Omega^{U}X_{U \oplus \R^k})^H 
    & k < 0.
  \end{cases}
\end{displaymath}
Here the colimits are taken over the partially ordered set of finite
dimensional representations in \(\universeU\). Since \(\N\{B\}\) is cofinal
in this partially ordered set, the groups \(\pi^H_k(X,\hml)\) and
\(\pi_k^H(X,\universeU)\) are isomorphic. 

If \(i \colon H \to G\) is the
inclusion of a subgroup, then by part (iv) of Definition~\ref{gtypicalrepresentations}, 
the groups
\(\pi^H_k(X,\hml)\) and \(\pi^H_k(i^*X,i^*\hml)\) are isomorphic. We
could have used the universe \(\universeU\) in the construction of
\(QX\). Working with \(\N\{B\}\) we are keeping our presentation close to the
one of \cite{MMSS}.
\end{remark}

Most of the following crucial result is taken directly from \cite[Theorem 7.4]{MMSS}.
\begin{thm}\label{MMSSthm}
  Let \(\hml\) be a \(G\)-typical family of representations.
  \begin{enumerate}
  \item For every \(G\)-CW complex $A$, the functor $- \wedge A
$ preserves \(\hml\)-stable equivalences of orthogonal \(G\)-spectra.
  \item A morphism $f$ of orthogonal \(G\)-spectra is an
    \(\hml\)-stable equivalence if and only if its suspension $\Sigma
    f$ is an \(\hml\)-stable equivalence. Moreover, the natural map
    $\eta\colon X \rightarrow \Omega^V\Sigma^V X$ is an
    \(\hml\)-stable equivalence for all orthogonal \(G\)-spectra $X$
    for all \(V\) in \(\hml\). 
  \item The homotopy groups of a wedge of orthogonal $G$-spectra are the direct sums of the homotopy groups of the wedge summands, hence a wedge of \(\hml\)-stable equivalences is an \(\hml\)-stable equivalences.
  \item Cobase changes of maps that are \(\hml\)-stable equivalences and levelwise Hurewicz cofibrations are \(\hml\)-stable equivalences.
  \item The generalized cobase change and cube lemmas (\ref{gen.cobch},\ref{gen.cube}) hold for all orthogonal \(G\)-spectra, levelwise Hurewicz cofibrations and \(\hml\)-stable equivalences.
  \item If $X$ is the colimit of a sequence of Hurewicz cofibrations $X_n \rightarrow X_{n+1}$, each of which is an \(\hml\)-stable equivalence, then the map from the initial term $X_0$ into $X$ is an \(\hml\)-stable equivalence.
  \item For every morphism $f\colon X \rightarrow Y$ of orthogonal \(G\)-spectra, there are natural long exact sequences
\[\cdots \rightarrow \pi^{H}_{k}(Ff,\hml) \rightarrow \pi^{H}_{k}(X,\hml)\rightarrow \pi^{H}_{k}(Y,\hml)\rightarrow \pi^{H}_{k-1}(Ff,\hml)
\rightarrow \cdots \]
and
\[\cdots \rightarrow \pi^{H}_{k}(X,\hml) \rightarrow \pi^{H}_{k}(Y,\hml)\rightarrow \pi^{H}_{k}(Cf,\hml)\rightarrow \pi^{H}_{k-1}(X,\hml)\rightarrow \cdots,\]
where $Ff$ and $Cf$ denote the levelwise homotopy fiber and cofiber of
$f$, that is, in each level it is given be the homotopy fiber and
cofiber. The natural map $Ff \rightarrow \Omega Cf$ is 
an \(\hml\)-stable equivalence. 
  \end{enumerate}
\end{thm}
\begin{proof}
  Part (i) is \cite[Theorem III.3.11]{MM}.
  For part (ii), note that a map \(f\) induced an isomorphism on
  homotopy groups if and only \(\Omega f\) is so. Thus \(f\) is an
  \(\hml\)-stable equivalence if and only if \(\Omega f\) is so. Now,
  the arguments used in the proof of Proposition~\ref{QisOmega} give the
  statements of (ii).
  For the rest of the statements, the arguments in the proof of
  \cite[Theorem 7.4]{MMSS} carry over to our situation.  
\end{proof}
\begin{cor}
  \label{stabeqstabundersusp}
  Let \(V\) be a Euclidean space and let \(P \in \hml^V\). A map \(f \colon X \to
  Y\) is an \(\hml\)-stable equivalence if and only if its
  \(V(P)\)-suspension \(\Sigma^{V(P)} f \colon \Sigma^{V(P)} X \to
  \Sigma^{V(P)} Y\) is 
  an \(\hml\)-stable equivalence.
\end{cor}
\begin{proof}
  Note that \(S^{V(P)}\) is a \(G\)-CW-complex. Thus (i) of Theorem~\ref{MMSSthm} implies
  that is \(f\) is
  an \(\hml\)-stable equivalence,
  then so is \(\Sigma^{V(P)} f\).
  Conversely, if \(\Sigma^{V(P)} f\) is
  an \(\hml\)-stable equivalence,
  then so is \(\Omega^{V(P)} \Sigma^{V(P)} f\). 
  Now use (ii) of Theorem~\ref{MMSSthm}.
\end{proof}
\begin{cor}
    Let \(\hml\) be a \(G\)-typical family of representations
    satisfying the pushout product axiom and let
    \(\catM\) be an \(\hml\)-model structure  (see Definition~\ref{classdfnlevel}).
    Consider a pull-back diagram of the form
    \begin{displaymath}
    \begin{CD}
      A @>f>> B \\
      @ViVV @VVjV \\
      X @>>g> Y.
    \end{CD}
    \end{displaymath}
    If \(j\) is a level \(\catM\)-fibration and \(g\) is an \(\hml\)-stable equivalence,
    then \(f\) is an \(\hml\)-stable equivalence.
  \end{cor}
  \begin{proof}
    The map \(Fi \to Fj\) induced by \(f\) from the fiber of \(i\) to the fiber of
    \(j\) is an isomorphism. From Theorem~\ref{MMSSthm} (vii) we obtain a
    morphism of long
    exact fibration sequences of homotopy groups \(\pi_k^H(-,\hml)\)
    for all subgroups \(H\) of \(G\). Applying the five lemma we see
    that the homomorphism \(\pi_k^H(f,\hml) \colon \pi_k^H(A,\hml) \to
    \pi_k^H(B,\hml)\) is
    an isomorphism for all \(H\) and \(k\). Now Proposition~\ref{prop:stableispistar}
    implies that \(f\) is an \(\hml\)-stable equivalence.
  \end{proof}

\section{The Stable Model Structure}
\label{sec:stable-model-struct}

In this section we give the various flavors of equivariant stable structures.  As before, there is considerable freedom in the choice of structure, but for much of what we are going to do, the by far most important instance is the \(\Sp\)-model structure of Definition~\ref{Spmodelstructure}, which builds on the positive mixing pair $(\AIplus,\All_+)$ of Example~\ref{positivemixingpair}.

Let \(\catM\) be an \(\hml\)-model structure, see Definition~\ref{hmlmonoidal} and Theorem~\ref{catMlevel}. 

\begin{rem}
  In the case where $\catM$ is the $(\hml,\gml)$-structure as in Definition~\ref{classdfnlevelsaysMartin} with $\hml$ closed under conjugation, we have by Corollary~\ref{cor:stabdeponclosure} that the stable equivalences of the model structure we give in Theorem~\ref{stablemodelstrfromMMSS} only depend on $\hml$ and by Lemma~\ref{lem:indepofuniv} the cofibrations only depend on $\gml$.  

Classically, the stable model structures were built upon a concept of a {\em universe} of representations.  A choice of universe was needed in order to define the stable equivalences and is mirrored in our choice of $\hml$, but was baked into the very category of spectra under considerations.  However, in our setting, the underlying category of orthogonal $G$-spectra is independent of such choices, and the fact that the cofibrations only depend on $\gml$ can be stated as the cofibrations not being dependent on the ``choice of universe'' and is a novel and very convenient feature of our approach.
\end{rem}

\begin{dfn}
  Let \(f \colon X \to Y\) be a map of orthogonal \(G\)-spectra. We
  say that \(f\) is:
  \begin{enumerate}
  \item a {\em stable \(\catM\)-acyclic cofibration}\index{stable!M-acyclic cofibration@\(\catM\)-acyclic cofibration, fibration} if it is an
    \(\hml\)-stable equivalence and a level
    \(\catM\)-cofibration;
  \item an \(\catM\)-stable fibration if it satisfies the right
    lifting property with respect to stably \(\catM\)-acyclic cofibrations;
  \item a {\em stable \(\catM\)-acyclic fibration} if it is an \(\hml\)-stable
    equivalence and a level \(\catM\)-fibration.
  \end{enumerate}
\end{dfn}

Recall from Definition~\ref{generatorsforMorthogonalspectra} that the
level \(\catM\)-model structure has the set \(\Gr{}{}I_\catM\) of generating
cofibrations and the set \(\Gr{}{}J_\catM\) of generating acyclic cofibrations.
 Also, recall the map $\lambda_P^X$ of Definition~\ref{def:lambda}.
\begin{dfn}\label{stableacyccof}
  We let \(S(\catM)\)\index{SM@\(S(\catM)\)} be the set consisting of the
  morphisms 
  \[\G/H_+ \wedge \lambda_P^X \colon G/H_+ \wedge \Gr{} V \widetilde
  S^P \wedge X \to G/H_+ \wedge X,\] 
  where \(H\) is a subgroup of
  \(G\), where \(P \in \hml^V\) for some \(V \in
  \catOr\) and where \(X = \Gr{} W C\) for \(C\) a cofibrant replacement
  of either a source or a target of a generating cofibration in one of
  the categories \(\catM_W\).
\end{dfn}

\begin{dfn}
  Given \(\lambda\) in \(S(\catM)\), let \(M\lambda\) be the
  mapping cylinder of \(\lambda\). Then \(\lambda\) factors as the
  composite of a level \(\catM\)-cofibration 
  \[k_\lambda \colon \Gr{}{V \oplus W}({\Or{V \oplus W}}_+\wedge_{\Or V \times \Or W}
  \widetilde S^P \wedge C) \to M\lambda\]\index{klambda@$k_\lambda$}
  and a deformation retraction
  \[r_\lambda \colon M\lambda \to \Gr{}{W} (C).\]\index{rlambda@$r_\lambda$}
  Let \(K_\lambda\) be the set of maps of the form \(k_\lambda \Box
  i\), where \(i\) is a generating cofibration
  for \(\catT\). Let \(K_\catM\) be the union of \(\Gr{}{}J_\catM\) and the
  sets \(K_\lambda\) for \(\lambda \in S(\catM)\).
\end{dfn}

\begin{dfn}\label{defineSMfibration}
  An {\em \(S(\catM)\)-fibration}\index{SM-fibration@\(S(\catM)\)-fibration} is a level \(\catM\)-fibration, say
  \(g \colon Z \to W\), of orthogonal \(G\)-spectra with the property
  that for every morphism \(\lambda \colon A \to B\) of \(S(\catM)\),
  the square 
  \begin{displaymath}
    \xymatrix{
      \catGOS(B,Z) \ar[r]^{g_*} \ar[d]_{\lambda^*} & \catGOS(B,W) \ar[d]^{\lambda^*} \\
      \catGOS(A,Z) \ar[r]^{g_*} & \catGOS(A,W)
    }
  \end{displaymath}
  is a homotopy pullback square in \(\catT\).
\end{dfn}

\begin{example}\label{notationofgencofandstablyacycliccof}
  Let us take a look at the \(S(\catM)\)-fibrations in the situation  of Proposition~\ref{mixedisorcat},
  where \(\catM\) is the 
  level \((\hml,\gml)\)-model structure for a \(G\)-mixing pair
  \((\hml,\gml)\). The set \(\Gr {}{} I_\catM=\Gr{}{}I_\gml\) of
  generating cofibrations consists 
  of maps of the form 
  \begin{displaymath}
    \Gr {} V (i \wedge (G \times \Or V)/P_+)
  \end{displaymath}
  for \(i \colon S^{n-1}_+ \to D^n_+\) a generating cofibration for
  \(\catT\) and \(P\) a member of \(\gml^V\). Since we work with model
  categories enriched over \(\catT\), the \(S(\catM)\)-fibrations do 
  not change if replace \(S(\catM)\) by the set of maps of the form
  \(\lambda^{\Gr {} V  C}_{W}\),\footnote{MORTEN 66.  This seems like nonsense/tyoo.  What do you want?} where \(C\) is a transitive \((G
  \times \Or V)\)-spaces 
  of the form \((G \times \Or V)/P_+\) for \(P \in \gml^V\). In this
  situation we write \(K_{\gml,\hml}\) for the set \(K_\catM\) of generating
  stably acyclic cofibrations. 
\end{example}
The arguments proving \cite[Proposition 9.5]{MMSS} give:
\begin{prop}
  A map \(p \colon E \to B\) of orthogonal \(G\)-spectra satisfies the
  right lifting property with respect to \(K_{\catM}\) if and only if it is an \(S(\catM)\)-fibration.
\end{prop}
\begin{cor}\label{RLPKisOmega}
  The map \(F \to *\) satisfies the right lifting property with
  respect to \(K_{\catM}\) if and only if \(F\) is an
  \(\catM\)-\(\Omega\)-spectrum. 
\end{cor}

\begin{cor}\label{cor:acyfiblevelvsstable}
  If \(p \colon E \to B\) is an \(\hml\)-stable equivalence that
  satisfies the right lifting property with respect to \(K_{\catM}\), then
  \(p\) is a both a level \(\hml\)-equivalence and a level
  \(\catM\)-fibration.  
\end{cor}
\begin{proof}
  Since \(\Gr{}{} J\) is contained in \(K\), we only need to prove that
  \(p\) is a level \(\hml\)-equivalence. Let \(F = p^{-1}(*)\) be the
  fiber over the base point. Since \(p\) satisfies the right lifting
  property with respect to \(K\), so does the map \(F \to *\), and
  thus by Corollary~\ref{RLPKisOmega} the orthogonal \(G\)-spectrum \(F\) is an
  \(\catM\)-\(\Omega\)-spectrum. Since \(p\) is an \(\hml\)-stable
  equivalence, the corollaries~\ref{Qtakesfibtofib} and~\ref{Qcreatesanddetectsequivalences} imply that \(F \to *\) is also an
  \(\hml\)-stable equivalence. By the level long exact sequences of
  homotopy groups, for each
  \(p_V \colon E_V \to B_V\) and each \(P\) in \(\hml^V\) and each \(k
  \ge 1\), the group homomorphism \(\pi_k^P(p_V)\) is an isomorphism. To
  see that \(\pi_0^P(p_V)\) is a bijection, note that \(\pi_0^P(p_V) =
  \pi_0^{H}(p_{V(P)})\) for a subgroup \(H = \proj_1(P)\) of \(G\).
  
  By part (iv) of Definition~\ref{gtypicalrepresentations} we may assume that
  \(\proj_1(P) = G\). 
  By part (iii) of
  Definition~\ref{gtypicalrepresentations} we may choose a Euclidean space \(W\)
  and 
  \(Q \in \hml^W\) so that \(V(Q)\) is a trivial representation of
  \(G\). The homotopy pullback diagram of Definition~\ref{defineSMfibration} associated to the map \(\lambda =
  \lambda_{Q}^{\Gr{} {\widetilde S^P}}\) is of the form
  \begin{displaymath}
    \xymatrix{
      E_{V(P)} \ar[rr]^{p_{V(P)}}  \ar[d]_{\lambda^*} &&
      \ar[d]^{\lambda^*} B_{V(P)} \\
      \Omega^{V(Q)} E_{V(P) \oplus {V(Q)}} \ar[rr]^{\Omega^{V(Q)} p_{V(P)\oplus {V(Q)}}} &&
      \Omega^{V(Q)} B_{V(P) \oplus {V(Q)}}.
      }
  \end{displaymath}
  The map \(\Omega^{V(Q)} p_{V(P)\oplus {V(Q)}}\) depends only on base point
  components and is a weak equivalence of \(G\)-spaces. Therefore
  \(p_{V(P)}\) is a weak equivalence of \(G\)-spaces. In particular it
  is a weak equivalence of \(H\)-spaces and
  \(\hml^V\)-equivalence as required. 
\end{proof}

Most of the following crucial result is taken directly from \cite[Theorem 7.4]{MMSS}.

\begin{thm}\label{stablemodelstrfromMMSS}\label{Gsmodel}
  The category \(\catGOS\) of orthogonal \(G\)-spectra is a cofibrantly
  generated proper \(G\)-topological model category with respect to
  the \(\hml\)-stable equivalences, \(\catM\)-stable fibrations and
  level \(\catM\)-cofibrations.
\end{thm}
\begin{rem}
  Actually it is a compactly generated model category, and thus also
  cellular. 
\end{rem}

In the main part of this paper we work with the following
\(\Sp\)-model structure taken from Shipley's paper \cite{Shconv}.
{Recall the positive mixing
  pair \((\hml,\gml)=(\AIplus,\All_+)\) of Example~\ref{positivemixingpair}. For
  \(V \ne 0\), the family \(\All_+^V\) consists of all subgroups of \(G
  \times \Or {V}\), and \(\AIplusV V\) consists of the subgroups \(P\) of
  \(G \times \Or{V}\) with the property that \(\proj_1 \colon P \to G\) is
  injective. When \(V = 0\), the families \(\All_+^V\) and \(\AIplusV V\)
  are empty. }
\begin{dfn}[The \(\Sp\)-model structure]\label{Spmodelstructure}
  The {\em \(\Sp\)-model structure}\index{Smodel structure@\(\Sp\)-model structure}\index{model structure!\(\Sp\)-} on \(\catGOS\) is the model
  category obtained in Theorem~\ref{Gsmodel} from the positive mixing
  pair \((\AIplus,\All_+)\) of Example~\ref{positivemixingpair}.

  We call the cofibrations and fibrations in this model 
  structure \(\Sp\)-cofibrations and \(\Sp\)-fibrations respectively.
  Moreover we use the notation \(\Sp I = \Gr{}{}I_\gml\) and \(\Sp J =
  \Gr{}{}K_{\hml,\gml}\)\index{SI, SJ@$\Sp I,\,\Sp J$} for the sets of generating cofibrations and
  generating acyclic cofibrations respectively for this model structure.

 In the $\Sp$-module structure, we  write {\em stable equivalence}\index{stable!equivalence}\index{equivalence!stable} or equivalently (in view of Proposition~\ref{prop:stableispistar})
  {\em \(\pi_*\)-isomorphism}\index{piisomorphism@\(\pi_*\)-isomorphism}
  instead of \(\hml\)-stable equivalence. 
\end{dfn}

 \begin{remark}\label{rem:newSstucture}
   The $\Sp$-model structure has more cofibrations than what was called
   the ``positive $\Sp$-model structure'' in the last author's thesis.
   There positive dimensional representations with trivial fixed point space
   where excluded when picking out the generating cofibrations.
   In the current indexing scheme (using semi-free spectra indexed by inner product spaces)
   these play no special r\^ole,
   and the ``positive'' aspect (critical for dealing with symmetric powers as, for instance,
   in Lemma~\ref{eqfreeup}) is handled simply by excluding the zero dimensional inner product space.  
  \end{remark}

\begin{lem}
    If \(P\) is a $G$-representation in \(\hml\), then the endofunctors \(s_{+P}\) and
    \(s_{-P}\) of \(\catGOS\) from Lemma~\ref{twistedorthogonalgroups} form
    a Quillen adjoint pair with respect to the stable \(\catM\)-model structure. 
\end{lem}
\begin{proof}
  The functors \(s_{+P}\) and \(s_{-P}\) form a Quillen adjoint pair
  with respect to the level \(\catM\)-model structure. Therefore it
  suffices to show that \(s_{+P}(\lambda_Q^X)\) is a stable
  equivalence for all elements \(\lambda_Q^X\) of
  \(S(\catM)\). However \(Y = \Gr{} V(G \times \Or V/P)_+ \wedge X\)
  is cofibrant and \(s_{+P}(\lambda_Q^X) = \lambda_Q^Y\), so by
  Proposition~\ref{manystableone} the morphism \(s_{+P}(\lambda_Q^X)\)
  is a stable equivalence. 
\end{proof}
\begin{lem}
  For every cofibrant \(G\)-space \(A\), the endofunctor \(X \mapsto A \wedge X\) is a left Quillen functor on \(\catGOS\) with respect to both the  level \(\catM\)-model structure and the \(\catM\)-stable model structure. 
\end{lem}
\begin{proof}
  Use the isomorphisms \(A \wedge \Gr{}V(f) \cong \Gr{}V(A \wedge f)\) and \(A \wedge \lambda^C_{P} \cong \lambda^{A \wedge C}_{P}\) together with Proposition~\ref{manystableone} and the assumption that \(\catM_V\) is a \(\catGT\)-model structure.
\end{proof}
\begin{thm}
  Let \(V\) be a Euclidean space and let \(P \in \hml^V\) with
  \(\proj_1 \colon P \to G\) an isomorphism. The pairs \((s_{+P},s_{-P})\) and \((\Sigma^V,\Omega^V)\) are Quillen equivalences of \(\catGOS\) in the \(\catM\)-stable model structure. 
\end{thm}
\begin{proof}
  By definition, for every \(\catM\)-\(\Omega\)-spectrum \(E\), the map 
\[\internhom(\lambda_P,E) \colon E \to R_P E\]
is a level equivalence. In particular it is a stable equivalence, and
thus the right derived functor of \(R_P\) is naturally isomorphic to
the identity. Since by Lemma~\ref{twistedorthogonalgroups} \(R_P \cong s_{-P} \Omega^V \cong
\Omega^V s_{-P}\), the right derived functors of \(s_{-P}\) and
\(\Omega^V\) are inverse equivalences of categories. Thus both
\(\Omega^V\) and \(s_{-P}\) are right adjoint functors in Quillen
equivalences. 
\end{proof}
\begin{cor}
  The stable \(\catM\)-model structure is a stable model structure in
  the sense that the homotopy category is a triangulated category
  (cf. \cite[Chapter 7]{H}.)
\end{cor}

\section{Smash Products in the Stable Category of Spectra}
\label{sec:bcdcomp}
In this section we provide a chain of natural stable equivalences between our smash powers and the handicrafted smash powers constructed in the spirit of B\"okstedt.  The meaning of ``natural'' is
dependent on the context.
Smash powers of orthogonal spectra are only functorial \wrt isomorphisms.
Smash powers of orthogonal ring spectra are functorial \wrt cyclic orderings and compares with B\"okstedt's original construction, see \eg~\cite{Shthh}, \cite{PS16} and \cite{DMPSW}).
Smash powers of commutative orthogonal ring spectra are functorial \wrt all maps of finite sets and the B\"okstedt type smash powers are given in \cite{BCD} as homotopy colimits over rather involved indexing categories.  Unfortunately \cite{BCD} was phrased in the language of $\Gamma$-spaces (which has the deficiency that not all connective orthogonal commutative ring spectra are modelled by commutative rings in $\Gamma$-spaces,~\cite{Tyler}), and we here compare with the parallel construction performed in orthogonal spectra.  The reader should beware that even if the constructions translates directly into orthogonal spectra, the connectivity offered by $\Gamma$-spaces was used in some of the proofs of the statements in \cite{BCD}, and so the correspondence below only allows us to transport results for the connective case.

We focus on the commutative case and provide natural transformations that have full functoriality.  On the other hand, the proofs that these natural transformations consist of stable equivalences is undertaken at the underlying spectral level via the following lemmas.

\begin{lemma}\label{lem:assemblyiseq}
  Let $L$ be an orthogonal spectrum.  Then for every $V$ the 
  assembly map
  $$\hocolim_{b \in \N\{b\}}\Omega^b(L_{V(b)}\smsh\Sp)\to\{W\mapsto\hocolim_{b \in \N\{b\}}\Omega^bL_{V(b)\oplus W}\}$$ is a level equivalence. 
\end{lemma}
\begin{proof}
  Both sides of the assembly map are homotopy functors (in $L$) and so it is enough to prove the statement for cellular $L$.  Since both functors commute with filtered colimits, it is by induction enough to handle the case where one attaches one cell.  Also, both sides send levelwise cofiber sequences to levelwise fiber sequences, so it is enough to prove the lemma for $L$ a single cell $L=\mathcal F_UK$ in which case the map is an isomorphism.  
\end{proof}
\begin{lemma}\label{lem:smashstab}
  If $A$ and $B$ are cofibrant spectra in the \(\Sp\)-model structure,
  then the map
  \begin{displaymath}
    \xymatrix{
      \hocolim_{(b,b') \in \N\{B\} \times \N\{B\}}
      \Omega^{V(b)\oplus V(b')}(A_{V(b)}\smsh B_{V(b')}\smsh\Sp) \ar[d] \\
      \hocolim_{(b,b') \in \N\{B\} \times \N\{B\}}\Omega^{V(b)\oplus V(b')}((A\smsh B)_{V(b)\oplus V(b')}\smsh\Sp)
    }
  \end{displaymath}
induced by the universal property of the smash product is a level equivalence.  Likewise for multiple smash factors.
\end{lemma}
\begin{proof}
  Since both functors are homotopy functors transforming levelwise cofiber sequences to levelwise fiber sequences it is enough to do the case where $A$ is a cell $A=\mathcal G_{U'}K$.  Likewise we reduce to the case $B=\mathcal G_{V'}L$.  Then
  $$A_U\smsh B_V\smsh\Sp_W=\catOr(U',U)\smsh \catOr(V',V)\smsh_{\Or{U'}\times \Or{V'}}K\smsh L\smsh S^W$$ whereas
  $$(A\smsh B)_{U\oplus V}\smsh\Sp_W=\catOr(U'\oplus V',U\oplus V)\smsh_{\Or{U'}\times \Or{V'}}K\smsh L\smsh S^W$$ and we show that the map from the former to the latter has connectivity $c_{U,V,W}$ such that  $c_{U,V,W}-\dim (U\oplus V\oplus W)$ goes to infinity with $\dim U$, $\dim V$ and $\dim W$.  This can be seeing by using the untwisting map~\ref{untwistingmap}.  We can assume $W=U'\oplus V'\oplus W'$ and under the untwisting map, the map in question is the map of suspensions
  $$\xymatrix{[(\catL(U',U)\times\catL(V',V))_+\smsh_{\Or{U'}\times\Or{V'}}(K\smsh L)]\smsh S^U\smsh S^V\smsh S^{W'}\ar[d]\\
    [\catL(U'\oplus V',U\oplus V)_+\smsh_{\Or{U'}\times\Or{V'}}(K\smsh L)]\smsh S^U\smsh S^V\smsh S^{W'}}
      $$
      having the desired connectivity since the spaces of linear isometries become highly connected with $U$ and $V$ (remember that $U'$ and $V'$ are fixed).
    \end{proof}

    We now come to the construction of the natural chain comparing our construction with that of \cite{BCD}.  We assume that $A$ is an $\Sp$-cofibrant commutative orthogonal ring spectrum so that we will end up with having functoriality with respect to all functions of finite sets, leaving the modifications necessary for the cases of spectra (natural wrt.~isomorphisms) and associative ring spectra (cyclic structure) aside.
    Comparing with \cite{BCD} necessitates adopting some of the machinery developed there, and we allow ourselves to use the notation of the B\"okstedt-type construction of the smash powers from \cite{BCD}.

Let $T$ be a finite set and $A$ an $\Sp$-cofibrant orthogonal ring spectrum.  Let $\Ib$ be the skeletal category of finite sets of the form $\{1,\dots,n\}$ ($n\geq 0$) and injections.  The functor $\mathcal I$ from finite sets to the category of pointed categories assigns to the set $T$ the category $\mathcal I(T)=(\Ib^T)_+$ (the $T$-fold product of $\Ib$ to which a base object is added). 
Let $$\mathcal G^A_T,\mathcal H^A_T,\mathcal M^A_T,\mathcal N^A_T\colon\mathcal I(T)\to \catOS$$ be the functors which to the object $i\in\mathcal I^T$ assigns the orthogonal spectra
\begin{align*}
  \mathcal G^A_T(i)&=\{V\mapsto\catT(\bigwedge_{t\in T}S^{i(t)},\bigwedge_{t\in T}\left[A(\R^{i(t)})\right]\smsh S^V)\}\\
\mathcal H^A_T(i)&=\{V\mapsto\catT(\bigwedge_{t\in T}S^{i(t)},\left[\bigwedge_{t\in T}A\right](\bigoplus_{t\in T}\R^{i(t)})\smsh S^V)\},\\
\mathcal M^A_T(i)&=\{V\mapsto\catT(\bigwedge_{t\in T}S^{i(t)},\left[\bigwedge_{t\in T}A\right](\bigoplus_{t\in T}\R^{i(t)}\oplus V))\},\\
\mathcal N^A_T(i)&=\bigwedge_{t\in T}A 
\end{align*}
(the last functor is constant).

The universal property of $\bigwedge_TA$ induces a natural transformation $\eta^A_T\colon\mathcal G^A_T\to \mathcal H^A_T$, the assembly map gives us a natural transformation $\epsilon^A_T\colon\mathcal H^A_T\to \mathcal M^A_T$ and the adjoint to the identity yields the natural transformation $\iota^A_T\colon\mathcal N^A_T\to \mathcal M^A_T$.  By Lemma~\ref{lem:assemblyiseq} $\iota^A_T$ is an equivalence.

Following the recipe of \cite[4.3]{BCD}, when varying the finite set $T$ the functors $\mathcal G^A_T$, $\mathcal H^A_T$, $\mathcal M^A_T$ and $\mathcal N^A_T$ become left lax transformations $\mathcal G^A$, $\mathcal H^A$, $\mathcal M^A$ and $\mathcal N^A$ from the functor $T\mapsto\mathcal I(T)$ to the constant functor with value the category of orthogonal spectra (c.f.~\cite[4.3]{BCD}) and $\eta^A_T$, $\epsilon^A_T$ and $\iota^A_T$ become modifications $\eta^A$, $\epsilon^A$ and $\iota^A$.

Let us spell this out for the case of $\epsilon^A$, since this is the source of much trouble in competing theories where the ``last face map'' in topological Hochschild homology tends to get disturbed by the suspension coordinate (see \eg \cite{PS16} and \cite{DMPSW}).  For $\phi\colon S\to T$ we must show that the two two-cells
$$\xymatrix{\mathcal I(S)\ar[r]^{\mathcal H^A_S}\ar[d]_{\mathcal I(\phi)}\ar@{}[dr]|{\mathcal H_\phi^A{\Downarrow}}&\catOS\ar@{=}[d]\\
  \mathcal I(T)\ar[r]_{\mathcal H^A_T}\ar@{=}[d]\ar@{}[dr]|{\epsilon_T^A{\Downarrow}}&\catOS\ar@{=}[d]\\
  \mathcal I(T)\ar[r]_{\mathcal M^A_T}&\catOS}
\qquad
\xymatrix{\mathcal I(S)\ar[r]^{\mathcal H^A_S}\ar@{=}[d]\ar@{}[dr]|{\epsilon_S^A{\Downarrow}}&\catOS\ar@{=}[d]\\
  \mathcal I(S)\ar[r]_{\mathcal M^A_S}\ar[d]_{\mathcal I(\phi)}\ar@{}[dr]|{\mathcal M_\phi^A{\Downarrow}}&\catOS\ar@{=}[d]\\
\mathcal I(T)\ar[r]_{\mathcal M^A_T}&\catOS}
$$
are equal.  To check this we need to spell out the ingredients.  Let $i\in\mathcal I(S)$, $W\in\catOr$ and $f\colon\bigwedge_{s\in S}S^{i(s)}\to\left[\bigwedge_{s\in S} A\right](\oplus_{s\in S}\R^{i(s)})\smsh S^W$ be an element in $\mathcal H^A_S(i)_W$.  Then $\mathcal H^A_\phi(i)_W(f)$ is {\em defined} to be the composite
\newcommand{\CDsmsh}[1]{\underset{#1}\bigwedge}
\newcommand{\CDoplus}[1]{\underset{#1}\bigoplus}
$$
\begin{CD}
  \CDsmsh{t\in T}S^{\coprod_{s\in\phi^{-1}(t)}i(s)}@.\left[\CDsmsh{t\in T}A\right](\CDoplus{t\in T}\R^{\coprod_{s\in\phi^{-1}(t)}i(s)})\smsh S^W\\
  @A{\cong}AA @AAA\\
  \CDsmsh{t\in T}\CDsmsh{s\in\phi^{-1}(t)}S^{i(s)}@.\left[\CDsmsh{t\in T}\CDsmsh{s\in\phi^{-1}(t)}A\right](\CDoplus{t\in T}\CDoplus{s\in\phi^{-1}(t)}\R^{i(s)})\smsh S^W\\
@A{\cong}AA @A{\cong}AA\\
\CDsmsh{s\in S}S^{i(s)}@>f>>\left[\CDsmsh{s\in S}A\right](\oplus_{s\in S}\R^{i(s)})\smsh S^W
\end{CD},
$$
where the unlabeled isomorphisms are induced by $\phi$.  The modification $\epsilon^A$ has the effect of post composing with the assembly, so the problem at hand reduces to whether the diagram of assemblies
$$
\begin{CD}
  \left[\CDsmsh{t\in T}A\right](\CDoplus{t\in T}\R^{\coprod_{s\in\phi^{-1}(t)}i(s)})\smsh S^W@>>>\left[\CDsmsh{t\in T}A\right](\CDoplus{t\in T}\R^{\coprod_{s\in\phi^{-1}(t)}i(s)}\oplus W)\\
  @AAA@AAA\\
  \left[\CDsmsh{t\in T}\CDsmsh{s\in\phi^{-1}(t)}A\right](\CDoplus{t\in T}\CDoplus{s\in\phi^{-1}(t)}\R^{i(s)})\smsh S^W@>>>
  \left[\CDsmsh{t\in T}\CDsmsh{s\in\phi^{-1}(t)}A\right](\CDoplus{t\in T}\CDoplus{s\in\phi^{-1}(t)}\R^{i(s)}\oplus W)\\
@A{\cong}AA@A{\cong}AA\\
\left[\CDsmsh{s\in S}A\right](\oplus_{s\in S}\R^{i(s)})\smsh S^W@>>>\left[\CDsmsh{s\in S}A\right](\oplus_{s\in S}\R^{i(s)}\oplus W)
\end{CD}
$$
commutes, which is true since $A$ is a commutative $\Sp$-algebra.

Now, as in \cite{BCD} there is a price to be paid by working with lax transformations and we perform exactly the same operation on the entire string
$$
\begin{CD}
  \mathcal G^A@>{\eta^A}>>\mathcal H^A@>{\epsilon^A}>>\mathcal M^A@<{\iota^A}<<\mathcal N^A,
\end{CD}
$$
\ie in the language of \cite[4.4.2]{BCD} we consider $\hocolim_{{\tilde{\mathcal I}(S)}}\mathcal G^A_S\circ r_S$,
where $r_S$ is the canonical map comparing the Street rectification $\tilde{\mathcal I}(S)$
and $\mathcal I(S)$ and likewise for the other terms using the full functoriality of the homotopy colimit as explained in \cite[4.4.2]{BCD}. We would like to thank Valentin Krasontovitch for pointing out that $r=e^{op} \wedge e$ as defined in \cite[4.4.1]{BCD} does not by itself form a left lax transformation. However, when we compose it with the projection $p_2$, we obtain a left lax transformation.
\newcommand{\we}{\overset{\sim}{\to}}
\newcommand{\ew}{\overset{\sim}{\gets}}
\begin{lemma}
  Let $A$ be a cofibrant orthogonal spectrum.  Then $\hocolim_{\mathcal I(T)} \eta^A_T$ and $\hocolim_{\mathcal I(T)} \epsilon^A_T$ are stable equivalences.
\end{lemma}
\begin{proof}
  Combine Lemma~\ref{lem:assemblyiseq} with Lemma~\ref{lem:smashstab}.
\end{proof}
This means that all the maps in the natural (in $T$) chain
$$\hocolim_{\tilde{\mathcal I}(T)}\mathcal G^A_T\circ r_T\we\hocolim_{\tilde{\mathcal I}(T)}\mathcal H^A_T\circ r_T\we\hocolim_{\tilde{\mathcal I}(T)}\mathcal M^A_T\circ r_T\ew\hocolim_{\tilde{\mathcal I}(T)}\mathcal N^A_T\circ r_T
$$
are weak equivalences, where the left hand term is the model for higher Hochschild homology of \cite{BCD} in the setting of orthogonal spectra.

Ultimately, this means that we can rewrite \cite{BCD} using $\mathcal N^A_T$ everywhere instead of $\mathcal G^A_T$.  However, since $\Ib$ has an initial object (the empty set),  the homotopy colimit -- now applied to the constant functor $\mathcal N^A_T=\bigwedge_{t\in T}A$ -- are redundant.  Consequently we get

\begin{thm}\label{thm:vsbcd}
  If $A$ is a cofibrant commutative orthogonal ring spectrum and $T$ is a finite set there is a chain of natural (in both $A$ and $T$) equivalences between $\bigwedge_TA$ and the construction of \cite{BCD}.
\end{thm}

\section{Ring- and Module Spectra}
\label{sec:eqringspec}
Let us fix a \(G\)-mixing pair \((\hml,\gml)\).
We use Theorem \cite[4.1]{SS} to lift the
\((\hml,\gml)\)-model structure to categories of modules and
algebras. First we need to verify the monoid axiom.

\begin{prop} \label{Gprespistar}
Cofibrant spectra in the \((\hml,\gml)\)-model structure are flat, in the sense that for any \((\hml,\gml)\)-cofibrant spectrum $X$, the functor $X \smash \,-$ preserves stable equivalences.
\end{prop}
\begin{proof}
  Since smashing with any spectrum preserves level Hurewicz cofiber
  sequences, and by the long exact sequence for homotopy groups
  (Theorem~\ref{MMSSthm} (vii)),
  it suffices to show that if $Z$ is an orthogonal spectrum with
  $\pi_*(Z)=0$, then also $\pi_*(X\smash Z)= 0$. Since smashing with $Z$
  preserves the cell complex construction, we can further reduce to the
  case where $X$ is either the source or the target of one of the
  generating \((\hml,\gml)\)-cofibrations, \ie $X$ is of the form
  $\Gr{}{V}{\left[{{\scriptsize\faktor{(G \times \Or{V})}{P}}}_+\smash
      S^k_+\right]}$  or $\Gr{}{V}{\left[{{\scriptsize\faktor{(G \times
            \Or{V})}{P}}}_+\smash D^k_+\right]}$ for \(P\) in
  \(\gml^V\). Since, for every space \(K\), the spectrum
  $\Gr{}{V}{\left[{{\scriptsize\faktor{(G \times 
            \Or{V})}{P}}}_+\smash K_+\right]}$ is equal to
  $\Gr{}{V}{\left[{{\scriptsize\faktor{(G \times \Or{V})}{P}}}_+\right]}
  \smash K_+$ and since we know from Theorem~\ref{MMSSthm} (i) that smashing
  with a cofibrant \(G\)-space preserves \(\hml\)-stable equivalences 
  it
  suffices to show that if \(Z\) has trivial homotopy groups, then also
  the spectrum $\Gr{}{V}{\left[{{\scriptsize\faktor{(G \times
            \Or{V})}{P}}}_+\right]} \smash Z$ has trivial homotopy
  groups. 

  In order to simplify notation, we let \(\Gor{V} = G \times \Or{V}\).
  Recall that 
  \begin{eqnarray*}\left(\Gr{}{V}{\left[{{\scriptsize\faktor{\Gor{V}}{P}}}_+\right]} \smash Z\right)_{V \oplus W} &=& {\Or{V \oplus W}}_+\smash_{\Or{V}\times\Or{W}}\left({{\scriptsize\faktor{\Gor{V}}{P}}}_+ \smash Z_W\right)\\
    &\cong&{{\scriptsize\faktor{\Gor{V \oplus W}}{P}}}_+\smash_{\Or{W}} Z_W,
  \end{eqnarray*}
  where the structure maps are the composites: 
  \[\xymatrix{{{{\scriptsize\faktor{\Gor{V \oplus
              W}}{P}}}_+\smash_{\Or{W}} Z_W \smash S^U}\ar^-{\id
      \smash \sigma}[r]&{{{\scriptsize\faktor{\Gor{V \oplus
              W}}{P}}}_+\smash_{\Or{W}} Z_{W \oplus
        U}}\ar^-{p\circ\inc}[d] \\
    &{{{\scriptsize\faktor{\Gor{V \oplus W
              \oplus U}}{P}}}_+\smash_{\Or{W \oplus U}} Z_{W \oplus
        U},}}\] 
  where \(\sigma\) is the structure map of \(Z\), the map \(\inc\) is
  induced by  the standard inclusion \( \Or {V 
    \oplus W} \to \Or {V \oplus W \oplus U}\) and \(p\) is the
  projection from \(\Or{W}\)-orbits to \(\Or {W \oplus U}\)-orbits.

  Recall that \(\{V_b\}_{b \in B}\) is a totally ordered set of
  representatives of 
  \(\hml\)-irreducible representations of \(G\). Let us represent
  elements of the free commutative monoid \(\N\{B\}\) as functions \(f
  \colon B \to \N\) with finite 
  support. We will
  take the liberty to use the symbol
  \(W\) to denote both an element \(W = f \in
  \N\{B\}\) and the \(G\)-representation \(V(f) = \bigoplus_{P \in B}
    V(P)^{\oplus f(P)}\) associated to \(f\).
  The homotopy groups $\pi_k^H(\Gr{}{V}{\left[{{\scriptsize\faktor{\Gor{V}}{P}}}_+\right]} \smash Z)$ are therefore isomorphic to the colimit:
  \begin{displaymath}
    \colim_{W \in \N\{B\}} \pi_k^H( \Omega^{W} (\Gor{V \oplus W}/P_+
    \wedge_{\Or{W}} Z_{W}))^H. 
  \end{displaymath}

  The projection
  \[\xymatrix{ {{{\scriptsize\faktor{\Gor{V \oplus
              W}}{P}}}_+\smash_{\Or{W}} Z_{W \oplus
        U}}\ar^-{p\circ\inc}[r]&{{{\scriptsize\faktor{\Gor{V \oplus W
              \oplus U}}{P}}}_+\smash_{\Or{W \oplus U}} Z_{W \oplus
        U}}}\]
  is the identity when \(U\) is zero, so it
  induces a surjective homomorphism from
  \begin{displaymath}
    \colim_{W \in \N\{B\}} \colim_{U \in \N\{B\}} \pi_k^H( \Omega^{W\oplus U} (\Gor{V \oplus W}/P_+ \wedge_{\Or{W}} Z_{W\oplus U}))^H
  \end{displaymath}
  to
  \begin{displaymath}
    \colim_{W \in \N\{B\}} \pi_k^H( \Omega^W (\Gor{V \oplus W}/P_+ \wedge_{\Or{W}} Z_W))^H.
  \end{displaymath}
  Let us
  write \(s_{-W} Z = s_{-P_W}Z\) for \(P_W \in \hml\) with \(W =
  V(P_W)\).   The group \(G \times \Or W\) acts
  on the orthogonal spectrum \(s_{-W} Z\) since \((s_{-W} Z)_U = Z_{W
    \oplus U}\) and with structure maps inherited from \(Z\). 
  The colimit
  \begin{displaymath}
    \colim_{U \in \N\{B\}} \pi_k^H( \Omega^{W \oplus U} (\Gor{V \oplus W}/P_+ \wedge_{\Or{W}} Z_{W\oplus U}))^H
  \end{displaymath}
  is the homotopy groups of the orthogonal spectrum
  ${{\scriptsize\faktor{\Gor{V \oplus W}}{P}}}_+\smash_{\Or{W}}
  s_{-W}Z$.
  Thus if we show that spectra of this type have trivial homotopy groups, we are done.\\
  Since \(G\) acts on \(W = V(P_W)\) and \(P\) is contained in \(\Or W\), we
  can consider the space
  ${\scriptsize\faktor{\Gor{V \oplus W}}{P}}$ as a \(G \times
  \Or W\)-space. Since \(\Or {V(P_W)} \cong (G \times \Or W)/P_W\),
  the isotropy groups of this \(G \times \Or W\)-space are all
  subconjugate to \(P_W\).  
  Therefore ${\scriptsize\faktor{\Gor{V \oplus W}}{P}}$ has \(G \times
  \Or W\)-cells of the form \(D^p_+ \wedge (G \times \Or W)/Q_+\)
  for \(Q \subseteq P_W\). Writing \(H = \proj_1(Q)\) and
  \(i \colon H \to G\)
  for the inclusion,
  there
  are isomorphisms 
  \[(G \times \Or W)/Q_+ \cong G_+ \wedge_H \Or
  {V(Q)} \cong G/H_+ \wedge \Or {V(P_W)}\]
  of \(G \times \Or W\)-spaces, and there is an isomorphism 
  \[(G \times \Or W)/Q_+
  \wedge_{\Or W} s_{-W} Z \cong G/H_+ \wedge (s_{-W} Z).\]

  Suppose that $*=X_0 \rightarrow X_1 \rightarrow \ldots\rightarrow
  X_l$ is the cellular filtration of ${\scriptsize\faktor{\Gor{V
        \oplus W}}{P}}_+$. 
  Both levelwise smash product with spaces and taking levelwise
  $\Or{W}$-orbits commute with (levelwise) colimits hence with the
  cell complex construction. 
  From the cell structure we get the gluing diagrams of
  spectra of the form: 
  \begin{displaymath}
    \xymatrix{
      {S^{p-1}_+\smash (G \times \Or W)/Q_+ \smash_{\Or{W}} s_{-W}
        Z}\ar[r]\ar_-{}[d]&{X_i\smash_{\Or{W}} s_{-W} Z}\ar[d] \\
      {D^{p}_+\smash (G \times \Or W)/Q_+ \smash_{\Or{W}} s_{-W}
        Z}\ar[r]&{X_{i+1}\smash_{\Or{W}} s_{-W} Z,}\pushout
    }
  \end{displaymath}
  that is, of the form
  \begin{displaymath}
    \xymatrix{
      {S^{p-1}_+\smash {G/H_+} \smash
       (s_{-W} Z)}\ar[r]\ar_-{}[d]&{X_i\smash_{\Or{W}} s_{-W} Z}\ar[d] \\
      {D^{p}_+\smash {G/H_+} \smash (s_{-W} Z)}\ar[r]&{X_{i+1}\smash_{\Or{W}}
        s_{-W} Z.}\pushout
    }
  \end{displaymath}
  Then as $Z$ was \(\hml\)-acyclic, and by Corollary~\ref{lambdastarstableeqcor} so is \(R_{P_W}
  Z \cong \Omega^{P_W} 
  s_{-W} Z\). Since \(\Sigma^{P_W}\) and \(s_{-W} = s_{-P_W}\) commute (i)
  and (ii) of Theorem~\ref{MMSSthm} imply that
  \(s_{-W}Z\) is \(\hml\)-acyclic. Thus the two spectra in the
  left column of the square 
  are \(\hml\)-acyclic by part (i) of Theorem~\ref{MMSSthm}. Assume
  that \({X_i\smash_{\Or{W}} s_{-W}Z}\) is 
    stably \(\hml\)-acyclic, so the top horizontal maps are
  \(\hml\)-stable equivalences, hence so are the bottom maps since the left
  vertical maps are Hurewicz cofibrations. Thus $X_{i+1} \smash s_{-W}Z$ is
  \(\hml\)-acyclic, and it follows by induction that $X_{l} \smash s_{-W}Z$ is
    \(\hml\)-acyclic. 
\end{proof}
Following \cite[Proposition 12.5]{MMSS} we obtain the following crucial result:
\begin{prop}\label{Gprespistar2}
     If \(i \colon A \to X\) is an acyclic cofibration in the
     stable \((\hml,\gml)\)-model structure and \(Y\) is
     any orthogonal \(G\)-spectrum, then the map \(i \wedge Y \colon A
     \wedge Y \to X \wedge Y\) is a stable equivalence. 
\end{prop}
\begin{proof}
  Let \(Z = X/A\). There is a Hurewicz cofiber sequence \(A \wedge Y \to X \wedge Y \to Z \wedge Y\). By part (vii) of Theorem~\ref{MMSSthm} it suffices to show that \(Z \wedge Y\) is acyclic. Let \(j \colon Y' \to Y\) be a cofibrant approximation of \(Y\). By Proposition~\ref{Gprespistar}, the map \(Z \wedge j\) is a stable equivalence. Thus, it is enough to show that \(Z \wedge Y\) is acyclic when \(Z\) is acyclic and both \(Z\) and \(Y\) are cofibrant. Now Proposition~\ref{Gprespistar} implies that the map \(* = * \wedge Y \to Z \wedge Y\) is a stable equivalence because \(Y\) is cofibrant and \(* \to Z\) is a stable equivalence.  
\end{proof}
\begin{prop}
  The stable \((\hml,\gml)\)-model structure is monoidal.
\end{prop}
\begin{proof}
  It follows from Proposition~\ref{mixedisorcat} that if \(i \colon A
  \to B\) and \(j \colon X \to Y\) are cofibrations, then also \(i
  \Box j\) is a cofibration. Now, suppose that \(j\) is an acyclic
  cofibration. Then the cofiber \(Y/X\) of \(j\) is acyclic, and there
  is a cofibration sequence 
  \begin{displaymath}
    A \wedge Y \cup_{A \wedge X} B \wedge X \xto{i \Box j} B \wedge Y \to B/A \wedge Y/X.
  \end{displaymath}
  By stability it suffices to show that the spectrum \(B/A \wedge Y/X\) is acyclic. However, since \(B/A\) is cofibrant and \(Y/X\) is acyclic, this is a consequence of Proposition~\ref{Gprespistar}. We have now shown that the push-out-product axiom holds in the stable \((\hml,\gml)\)-model structure. The criterion for being a monoidal model category concerning cofibrant replacement of the sphere spectrum is a direct consequence of Proposition~\ref{Gprespistar2}.
\end{proof}

Now \cite[Prop. 4.1]{SS} gives the following:

\begin{thm}\label{eqRmodel}Let $R$ be an orthogonal ring $G$-spectrum.
\begin{enumerate}
\item The category of left $R$-modules is a compactly generated proper
  model category with respect to the \(\hml\)-stable equivalences and
  the underlying \((\hml,\gml)\)-stable fibrations.
\item If $R$ is cofibrant in \(\catGOS\), then the forgetful functor
  from $R$-modules 
  to orthogonal spectra preserves cofibrations. Hence every cofibrant
  $R$-module is cofibrant as an orthogonal \(G\)-spectrum. 
\item Let $R$ be commutative. The model structure of $(i)$ is monoidal
  and satisfies the monoid axiom. 
\item Let $R$ be commutative. The category of $R$-algebras is a
  compactly generated right proper model category with respect to the
  stable \(\hml\)-equivalences and the underlying
  \((\hml,\gml)\)-fibrations.
\item Let $R$ be commutative. Every cofibration of $R$-algebras whose source is cofibrant as an $R$-module is also a cofibration of $R$-modules. In particular, every cofibrant $R$-algebra is cofibrant as an $R$-module.
\end{enumerate}
All the above model structures are Quillen equivalent to the one
obtained from the \((\overline \hml,\overline \hml)\)-model structure,
that is the 
classical one from \cite[III.7.6]{MM} with respect to the
\(G\)-universe
\(\universeU = \colim_{n \in
  \N} \bigoplus_{b \in B} V^n_b\), via the identity functor.
\end{thm}
We will deal with the lift to commutative algebras in a separate section (\ref{eqexttocomring}).

\chapter{Fixed Points}
\label{ch:filtering}

The reason for the choices made in the previous chapters is that we want a structure so that the smash product in spectra has good equivariant properties.  In order to verify this we need to study the fixed points.  These come in two main guises, the categorical fixed points and the geometric fixed points and the interplay between these two is one way to express the properties we want to uncover.
In this chapter we give the basic definitions and explore some of the fundamental properties, laying the foundations for our investigations of smash powers.

In Section~\ref{secfixedpointspectra} we fix some notation regarding fixed point constructions.  
Section~\ref{sec:orfi} explores an important class of inclusions of normal subgroups which we call \orfi. This class was studied by Kro in his thesis and includes all examples of finite index and inclusions of kernels of isogenies of the torus and suffices for our purposes.  
As we will show in Section~\ref{sec:fixandchange}, the \orfi case is important because we then have very good control over what happens when we change the underlying group.  This will turn out to be crucial in our approach when we approximate the torus by means of finite subgroups. 


\section{Fixed Point Spectra  
}
\label{secfixedpointspectra}\label{secfixedpointsofcells}
In this section we consider a fixed
short exact sequence of compact Lie-groups  
\[E: \:\:\:\:\:1 \rightarrow N \rightarrow \G
\stackrel{\epsilon}{\rightarrow} \J \rightarrow 1.\] 

  \begin{dfn}\label{def:catfix}
    The {\em (categorical) $N$-fixed point} functor 
      $$G\catOS \rightarrow J\catOS,\qquad A\mapsto A^N$$
is given by post-composition with the $N$-fixed point functor in $\catT$.
  \end{dfn}
Explicitly, if $A$ is an orthogonal $G$-spectrum and $V$ is a Euclidean space, then $(A^N)_V=(A_V)^N$.

\begin{rem}Note that for subgroups $H$ that are not necessarily
  normal, we can first restrict the $G$-actions to the normalizer
  $N_{\!H}$ of $H$ in $G$ before taking fixed points, to get a functor 
\[G\catOS \rightarrow N_{\!H}\catOS \rightarrow W_{\!H}\catOS,\]
to spectra with actions of the Weyl group $W_{\!H} \defas{} \faktor{N_H}{H}$.
\end{rem}
We now explain an alternative ``geometric'' way of forming \(N\)-fixed points. 
For this it is convenient to change the perspective slightly.  
We have firmly insisted that our spectra are indexed over Euclidean spaces only.  This has huge advantages over the classical setup in that the categorical framework is fixed throughout and we keep this orthodoxy except for a concession for geometric fixed points where temporarily indexing by representations simplify things.  
  \begin{dfn}\index{OGOJOT@$\catO_G$, $\catO_J$, $\catT_G$}
    \label{def:OG}For a compact Lie group $G$ and topological category $\catC$, let $\catC_G$ be the $G\catT$-category whose objects are $G$-objects in $\catC$ and where $G$ acts on the morphism space $\catC_G(V,W)=\catC(V,W)$ by conjugation.  
The $G\catT$-category of  (not necessarily $G$-equivariant) continuous functors $\catC_G\to\catT_G$ is called $\catC_G\catT$.
  \end{dfn}
In particular, we have the examples $\catO_G$ and $\catT_G$.
\begin{lemma}
  Restriction to trivial representations $\catO\subseteq\catO_G$ gives an equivalence between the category $G\catOT$ of orthogonal $G$-spectra and the $G$-fixed category $(\catO_G\catT)^G$ of $\catO_G\catT$; \ie the category of $G$-functors $\catO_G\to\catT_G$.
\end{lemma}

Consequently, under the fixed equivalence given by restriction, we will not distinguish between the category $G\catOT$ of orthogonal $G$-spectra and the category $(\catO_G\catT)^G$ of $G$-functors $\catO_G\to\catT_G$.

\begin{dfn}\label{dfnOE}\index{OE@$\catJE$}
Let $\catJE$ be the $N$-fixed category $(\catO_G)^N$ of $\catO_G$; explicitly, $\catJE$ is the $\J$-category whose objects
are the orthogonal
\(G\)-representations 
and whose morphism spaces are given by the $N$-fixed points \[\catJE(V,W) \defas{} \catO(V,W)^N,\]
 with \(G\) acting by conjugation.
An {\em $\catO_E$-space}\index{OEspace@$\catO_E$-space}\index{space!$\catO_E$-} is a  continuous $J$-functor $\catO_E\to\catT_J$.  
The category of $\catO_E$-spaces \index{EOT@$E\catOT$, $\catO_E$-space} is denoted $E\catO\catT$, as opposed to $\catO_E\catT$ which is the $J\catT$-category consisting of all continuous functors $\catO_E\to\catT_J$ (so that $E\catO\catT=(\catO_E\catT)^J$).
\end{dfn}

  \begin{dfn}\label{def:fixthroughE}\label{def:FixN}
    Post-composing with the $N$-fixed point functor defines a functor
    $$\Fix^N\colon G\catOT\to E\catOT,$$
    or in the enriched setting, a $J\catT$-functor
$$\Fix^N\colon(\catO_G\catT)^N\to\catO_E\catT$$ 
(from the $J$-category of $N$-functors $A\colon\catO_G\to\catT_G$ to to the $J$-category of functors $\catO_E\to\catT_J$).  Explicitly, $\Fix^N(A)_V=[A_V]^N$ with structure map the composite 
\begin{displaymath}
  \catJE(V,W) = \catO_G(V,W)^N \xto{A} \catT_G(A_V,A_W)^N \to \catT_J([A_V]^N,[A_W]^N),
\end{displaymath}
for $V,W$ orthogonal $G$-representations and 
where the last map takes an \(N\)-equivariant map of $G$-spaces to its induced
map of \(N\)-fixed points (which are $J$-spaces).
  \end{dfn}

\begin{dfn}\label{dfnphiphi}
The
\(J\catT\)-functor $\phi\colon \catOE \rightarrow \catOj$ is given on
objects by 
\begin{eqnarray*}
\phi\,\,\,\colon\catJE & \rightarrow & \catJJ\\
V&\mapsto&V^N
\end{eqnarray*}
and on morphisms by
\begin{eqnarray*}
\catOr(V,W)^N&\rightarrow&\catOJ(V^N,W^N).\\
g \wedge x&\mapsto& g^N \wedge x
\end{eqnarray*}
Let $\epsilon^*\colon\catO_J\to\catO_E$ be induced by restriction along $\epsilon\colon G\to J$; it is the identity $\catO_J(V,W)=\catO_E(\epsilon^*V,\epsilon^*W)$ on morphisms.
\end{dfn}

  \begin{remark}\label{notationFP}
    Let $t\colon\catO\to\catOJ$ be the inclusion of Euclidean spaces as trivial $J$-represent\-ations (inducing the equivalence between $\catOJ$-spaces and orthogonal $J$-spectra $t^*\colon\catOJ\catT\simeq J\catOT)$.  
The categorical $N$-fixed point functor of Definition~\ref{def:catfix} factors as 
$$\xymatrix{
(-)^N\colon G\catOT\simeq(\catO_G\catT)^G\ar[r]^-{\Fix^N}& 
(\catO_E\catT)^J
\ar[r]^-{\epsilon^*}&(\catOJ\catT)^J\simeq J\catOT,}$$ 
where $\Fix^N$ is the functor of Definition~\ref{def:fixthroughE} and the map marked $\epsilon^*$ restricts along $\epsilon^*\colon\catO_J\to\catO_E$.  

The geometric fixed point functor which we will define in a moment is the variant where one  uses the left Kan-extension $\FP_\phi\colon\catJE\catT\to\catOJ\catT$ (written out as a coend in Definition~\ref{dfngeomfixed} below) of the map induced by $\phi\colon\catJE\to\catOJ$ instead of $\epsilon^*$ to bridge the gap between $\catJE\catT$ and $\catOJ\catT$.  We'll come back to this point of view in Section~\ref{fixedpointsandlodaysubsect} (where functoriality becomes more ticklish), but for the moment it is convenient to express the geometric fixed points through an explicit formula.
  \end{remark}

\begin{dfn}\label{dfngeomfixed}
The {\em geometric \(N\)-fixed point functor}\index{geometric!fixed point spectrum} 
$$\PhiNH{}\colon G\catOT\to J\catOT$$
is given by  
$$\PhiNH{} A_V =\int^{W \in \catJE} \catOJ(W^N,tV) \wedge (A_W)^N
$$
for $A$ an orthogonal $G$-spectrum and $V$ a Euclidean space with $tV$ the associated trivial $J$-representation.
For $U\in\catJE$, the maps \(\phi \colon \catJE(W,U) \to \catO_J(W^N,U^N)\) induce a map
\begin{displaymath}
  (A_U)^N \cong \int^{W \in \catJE} \catJE(W,U) \wedge (A_W)^N \to
  \int^{W \in \catJE} \catOJ(W^N,U^N) \wedge (A_W)^N.
\end{displaymath}
Restricting to the situation where \(U=tV\) is a trivial
\(G\)-representation, we obtain a natural transformation of functors $G\catOT\to J\catOT$ called the {\em isotropy separation map}\index{isotropy separation map}  
$$\gamma \colon A^N \to
\PhiNH{} A$$ 
from the categorical fixed point functor to the geometric fixed point functor.
\end{dfn}

  \begin{remark}
It is sometimes useful to extend the definition of the geometric fixed points to $\catO_J\catT$, letting \(\PhiNH{} A_U=\int^{W \in \catJE} \catOJ(W^N,U) \wedge A^N_W\) for \(U\) in \(\catOJ\).
  \end{remark}

Before going into more detail, we will list properties of the fixed point functors and study their interaction with free spectra, which form the basis for computations on fixed points of smash powers.

For an orthogonal $G$-representation $V$ and $G$-space $K$, recall the free orthogonal $G$-spectrum $\Fr GV=\catO(V,-)\smsh K$ from Example~\ref{freeGorthsp}.  These free spectra formed the basis for the treatment of Mandell and May and they are particularly nice since they commute with geometric fixed points in the following sense (c.f. \cite[IV.4.5]{MM}).  Since $\catOE(V,W)=\catO(V,W)^N$, the dual Yoneda isomorphism takes the following form 
  \begin{align*}
    \Phi^N\Fr{\G}{V} K&=\int^{W\in\catOE}\catO(W^N,-)\smsh[\catO(V,W)\smsh K]^N \\
    &=\int^{W\in\catOE}\catO(W^N,-)\smsh\catO(V,W)^N\smsh K^N
    &=\int^{W\in\catOE}\catO(W^N,-)\smsh\catOE(V,W)\smsh K^N\\
    &
      \cong\catO(V^N,-)\smsh K^N=\Fr{\J}{V^N}K^N.
  \end{align*}
For reference:
\begin{prop}
\label{geomfree}
For any finite dimensional $\G$-representation $V$ and any $\G$-space $K$, the dual Yoneda map is a natural isomorphism of $J$-spectra
\[\Phi^N\Fr{\G}{V} K\cong \Fr{\J}{V^N}K^N .\]
\end{prop}

\begin{prop}
\label{geomcolim}
 The geometric fixed point functor $\PhiNH{}$ preserves coproducts, push\-outs along Hurewicz cofibrations, sequential colimits along Hurewicz cofibrations and tensors with spaces.
\end{prop}
\begin{proof}
 The functor from \(\catGOS\) to the category of
 \(J\)-functors from \(\catJE\) to \(J\catT\) taking \(A\) to \(W \to
 A^N_W\) preserves these by Lemma~\ref{fixedprop}. So does the continuous prolongation
 functor \(Y \mapsto \int^{W \in \catJE} \catOJ(W^N,(-)^N) \wedge
 Y_W\) from \(\catJE\)-spaces to \(\catOJ\)-spaces.
\end{proof}


Recall the $q$-cofibrations from Example~\ref{ex:levelqmodel}.
 \begin{lemma}
    \label{lem:phipreservesstableeq}
    The geometric $N$-fixed point functor preserves $q$-cofibrations and acyclic $q$-cofibrations.
  \end{lemma}
   \begin{proof}
     Follows by induction from Proposition~\ref{geomfree} and Proposition~\ref{geomcolim}.
   \end{proof}
By its very construction as postcomposition by fixed points and the left Kan extension, the geometric fixed point functor is lax symmetric monoidal.  Explicitly, for orthogonal $G$-spectra $X$ and $Y$ the required  natural transformation of orthogonal $J$-spectra 
\[\alpha_{X,Y}\colon \PhiNH{} X \smash \PhiNH{} Y \rightarrow
\PhiNH{}(X\smash  Y),\]
 is, by the universal property of the smash product, given by the map
$$\smallint^{U_1}\catOJ(U_1^N,V_1)\smsh(X_{U_1})^N\smsh
\smallint^{U_2}\catOJ(U_2^N,V_2)\smsh(X_{U_2})^N\to
\smallint^{U}\catOJ(U^N,V_1\oplus V_2)\smsh((X\smsh Y)_{U})^N
$$
induced by $\oplus$, the canonical isomorphisms $(U_1\oplus U_2)^N\cong U_1^N\oplus U_2^N$ and $(X_{U_1})^N\smsh(Y_{U_2})^N\cong (X_{U_1}\smsh Y_{U_2})^N$ and the structure map of the smash product $X_{U_1}\smsh Y_{U_2}\to (X\smsh Y)_{U_1\oplus U_2}$.  
\begin{prop}
\label{philaxmonoidal}
The geometric fixed point functor $\PhiNH{}$ is lax symmetric monoidal.

For $q$-cofibrant orthogonal $G$-spectra $X$ and $Y$ the lax monoidal structure map
\[\alpha\colon \PhiNH{} X \smash \PhiNH{} Y \rightarrow
\PhiNH{}(X\smash  Y),\]
is an isomorphism of $J$-spectra.   
\end{prop}
\begin{proof}
That $\alpha$ is an isomorphism if $X$ and $Y$ are $q$-cofibrant follows exactly as in~\cite[V.4.7]{MM} from Proposition~\ref{geomfree} and Lemma~\ref{geomcolim}. 
\end{proof}

\section{Orthogonally final subgroups}
\label{sec:orfi}
Changing the ambient compact Lie group $G$ in the definition of geometric $N$-fixed points is a well-known sticky point, classically involving a significant amount of ``changes of universe''.  The main case we are interested in satisfies a property which simplifies these obstacles enormously and we have dubbed ``\orfi''.
This property was originally considered without a name in Kro's thesis.
  \begin{dfn}\label{cond1}
    If $G$ is a compact Lie group and $N$ a normal subgroup, we say that the inclusion $i\colon N\subseteq G$ is {\em\orfi}\index{\orfi} if for any orthogonal $N$-representation $W$ there is an orthogonal $G$-representation $V$  and an \(N\)-linear isometric embedding $W \rightarrow i^*V$
  inducing an isomorphism  $W^N \cong (i^*V)^N$ on $N$-fixed spaces. 
  \end{dfn}

Important examples of \orfi $N\subseteq G$ include the case where $G/N$ is finite and the case of $G$ being the torus, see Lemma~\ref{cond1lem} and Lemma~\ref{toruscond}.  We know of no situation that is not \orfi.

Note that if $N\subseteq G$ is \orfi and $H$ is a subgroup of $G$ containing $N$, then $N\subseteq H$ is also \orfi.  Note also that the condition is purely one of existence, there is no implicit functoriality or continuity attached to the claim that  whenever given $W$ the isometry $W\to i^*V$ exists.

Another noteworthy aspect of the concept is that it is model dependent: it is custom built for the approach where we index our spectra by isometric embeddings.  Other flavors of spectra could conceivably benefit from different perspectives and cover applications that are not \orfi.

The reason for the name \orfi is that it can be rephrased in terms of a certain functor being final.  Namely, if $\catC_N$ (resp.~$\catC_E$) is the category of orthogonal $N$-(resp.~$G$)-representations and $N$-equivariant linear isometries that induce isomorphisms on $N$-fixed points, then $i$ is \orfi if and only if the restriction $$i^*\colon\catC_E\to\catC_N$$ is final.  More is true: 
\begin{lemma}\label{pushoutinO}
  Let $i\colon N\subseteq G$ be \orfi and let $[N]\catO\subseteq N\catO$ be the subcategory containing all objects and whose morphisms are those that become isomorphisms upon taking $N$-fixed points.  Suppose $W_0,W_1,W_2$ are orthogonal $N$-representations, 
$\omega\in N\catO(W_0,W_1)=\catO(W_0,W_1)^N$
and 
$\alpha\in[N]\catO(W_0,W_2)$.  Then there exists an orthogonal $G$-representation $V$ and maps $\beta\in[N]\catO(W_1,i^*V)$ and $\nu\in N\catO(W_2,i^*V)$ such that the diagram
$$\xymatrix{W_0\ar[r]^\omega\ar[d]^{\alpha}&W_1\ar[d]^{\beta}\\
W_2\ar[r]^{\nu}&i^*V}$$
in $N\catO$ commutes.
\end{lemma}
\begin{proof}
Write $\omega=(f_\omega,x_\omega)$ where $f_\omega\colon W_0\to W_1$ is an $N$-linear isometry and $x_\omega$ is a point in the one-point compactification of the orthogonal complement of the image of $f_\omega$.  Likewise for the maps $\alpha$, $\beta$ and $\nu$. Consider the pushout
$$\xymatrix{W_0\ar[r]^{f_\omega}\ar[d]^{f_\alpha}&W_1\ar[d]^{f_{\beta_1}}\\
W_2\ar[r]^{f_{\nu_1}}&P}
$$
of orthogonal $N$-representations.  Since $N\subseteq G$ is \orfi there exists a $G$-representation $V$ and an $N$-linear isometry $f_{\beta_2}\colon P\to i^*V$ such that the map of fixed points $f_{\beta_2}^N\colon P^N\to i^*V^N$ is an isomorphism.  Now, setting $\beta=(f_{\beta_2}f_{\beta_1},f_{\beta_2}f_{\beta_1}x_\omega)$ and $\nu=(f_{\beta_2}f_{\nu_1},f_{\beta_2}f_{\nu_1} x_\alpha)$ we are done.
\end{proof}

\begin{lem}\label{cond1lem}
  Inclusions of subgroups of finite index are \orfi.
\end{lem}

\begin{proof}
  Let $N\subseteq G$ be of finite index.  As observed by Kro \cite[3.8.11]{Kro}, it suffices to look at the
irreducible representations. If \(V\) is the
trivial representation of \(N\), there is nothing to do. If $V$ is
an irreducible and non-trivial representation of \(N\), then the $N$-fixed points of both $V$ and the induced $G$-representation (which is still
finite dimensional) are both of dimension zero, so the induced representation extends $V$ in the desired way.
\end{proof}

\section{Change of Group and Geometric Fixed Points}
\label{sec:fixandchange}
  In the next chapter Theorem~\ref{geomsmashsemifree} will form the basis for
  the identification in Theorem~\ref{geomScofib} of the geometric fixed points of smash powers of
  general $\Sp$-cofibrant spectra. We need one
  more crucial input, namely Kro's observation on the
  interaction of the geometric fixed point functor with induced
  spectra 
  from \cite[3.8.10]{Kro}. Since we will in particular need analogous
  results for the functors restricting to $H$-spectra for $H$ a
  subgroup of $G$ 
  to lift our results to the case of $G$ a compact Lie group, we
  will go into more details in the next section. 

For the rest of the section, let $G$ be a compact Lie-group and $H$ a closed
subgroup of $G$ with inclusion map $i\colon H\rightarrow G$. 
We consider the restriction functor 
\[i^*\colon \catGOS \rightarrow H\catOS,\] 
and its left adjoint induction functor 
\[G_+ \smash_H (-) \colon H\catOS \rightarrow \catGOS.\] 
The following lemma will be helpful later, it is a direct consequence
of Lemma~\ref{inducingsmashcomp}.
\begin{lem}\label{inducingsmashspeccomp}
  For an orthogonal $H$-spectrum $X$ and an orthogonal $G$-spectrum $Y$,
  the identity on \(G_+ \smash X \smash Y\) induces a natural
  isomorphism of \(G\)-spectra
  \[(G_+ \smash_H X) \smash Y \cong G_+ \smash_H ( X \smash i^*Y).\] 
\end{lem}
Reintroducing the normal subgroup $N$, let $j\colon E_0\rightarrow E$ be the injection of short exact sequences of compact Lie groups:
\labeleq{grpsequences}{\xymatrix{{\E_0\colon}\ar_{j}[d]&{1}\ar[r]&{N}\ar@{=}[d]\ar[r]&{\h}\ar[r]\ar_-{i}[d]&{\J_0}\ar_-{i_J}[d]\ar[r]&{1}\\{E\colon}&{1}\ar[r]&{N}\ar[r]&{\G}\ar[r]&{\J}\ar[r]&{1}}} 
($N$ is normal in $G$ and hence in $H$, but $i\colon H\subseteq G$ need not be normal). 
Our notation $\Phi^N$ for the geometric fixed points doesn't distinguish between the case of $H$- and $G$-spectra, which is innocent by type-checking the input and, furthermore, in the \orfi case we will see that it matters even less.  

\begin{dfn}
  \label{def:jOE} The full and faithful functor 
$j^* \colon \catO_{E} \to \catO_{E_0}$
 is given by taking a \(G\)-representation \(V\) to its
restricted \(H\)-representation \(i^*V\) and on morphism spaces it is the
identity \(\catO(V,W)^N = \catO(i^*V,i^*W)^N\) of subspaces of
\(\catO(V,W)\). 
If we let $i_J^*\catO_E$ be the category $\catO_E$ considered as a $J_0$-category through restriction along $i_J\colon J_0\to J$, we may consider $j^*$ as a \(J_0\)-functor 
$$j^* \colon i_J^* \catO_{E} \to \catO_{E_0}.$$
\end{dfn}

\newcommand{\jota}{i}
\begin{dfn}\label{def:inducedspectrum}
Let $Y$ be an $\catO_{E_0}$-space (\ie a $J_0$-functor $\catO_{E_0}\to\catT_{J_0}$), then the \emph{induced $\catO_E$-space $J_+ \smash_{J_0} Y$}\index{induced $\catO_E$-space} is given by sending the orthogonal $G$-representation $V$ to
\[(J_+ \smash_{J_0} Y)_{V} \defas{} J_+ \smash_{J_0} Y_{\jota^* V},\]
and on morphisms given as the composite of the isomorphism from Lemma~\ref{inducingsmashcomp} and the structure map of $Y$: 
\begin{align*}
\catO(V,W)^N\smash (J_+ \smash_{J_0} Y_{\jota^*V})&
\cong J_+\smash_{J_0} (\jota_J^*\catO(V,W)^N\smash Y_{\jota^* V})\\
&=J_+ \smash_{J_0} (\catO(\jota^* V,\jota^* W)^N \smash Y_{\jota^* V})
\rightarrow J_+ \smash_{J_0} Y_{\jota^* W}.
\end{align*}
\end{dfn}

The following Lemma is where demanding that $N\subseteq G$ is \orfi plays a crucial r\^ole: it reinvents the geometric fixed points of an $H$-spectrum in terms of $\catO_E$.  The result depends on some technical observations regarding cofinality which are collected in Section~\ref{sec:coendcof}. 
\begin{lem}\label{condition1lemma}
  Suppose $N\subseteq G$ is \orfi. If $X$ is an orthogonal $H$-spectrum, then the map  of $J_0$-spectra
  $$
    j\colon\int^{V \in\catO_{E}}
    \catO(i^*V^N,-) \wedge X_{i^*V}^N
    \to
    \int^{W \in \catO_{E_0}}
    \catO(W^N,-) \wedge X_{W}^N=\Phi^N X
$$
 induced by $j^*\colon\catO_E\to\catO_{E_0}$ is an isomorphism onto the geometric fixed points.
\end{lem}
\begin{proof}
  If \(F \colon 
  \catO_{E_0}^{\op} \wedge \catO_{E_0} \to
  \catOT\)  is the functor given by \(F(W,W') = \catO(W^N,-)
  \wedge (X_{W'})^N\), then Lemma~\ref{pushoutinO} and Lemma~\ref{fromcond1tocofinal} proves that  \(j^*\colon \catO_{E} \to \catO_{E_0}\) is \(F\)-cofinal (see Definition~\ref{def:Fcofinal}). Thus Proposition~\ref{Fcofinaliso} implies that the map \(j\) is an isomorphism.
\end{proof}
The application of the (unnatural) Lemma~\ref{pushoutinO} in the proof of Lemma~\ref{condition1lemma} totally disregards the $H$-action of the objects in $\catO_{E_0}$ and only focuses on the $N$-action.

\subsubsection{Induced and Restricted Spectra and (Geometric) Fixed Points}
Thanks to Lemma~\ref{condition1lemma}, we are now in a position to show that the geometric fixed points commute both with induction and restriction along similar ideas as in {\cite[3.8]{Kro}} when $N\subseteq G$ is \orfi.
\begin{prop}\label{3810}
Assume that $N\subseteq G$ is \orfi.  If $X$ is an orthogonal $H$-spectrum $X$, then the natural map
\[J_+ \smash_{J_0}(\Phi^NX) \xto{}  \PhiNH{}(G_+
  \smash_H X)\]
is an isomorphism of $J$-spectra.
\end{prop}
\begin{proof}
  We prove that the precomposite with the isomorphism $j$ of Lemma~\ref{condition1lemma} is an isomorphism:
  the left vertical isomorphism is the natural redistribution and the lower horizontal isomorphism is given by  Lemma~\ref{lem:inducedfixpoints}.
  $$\xymatrix{J_+ \wedge _{J_0} \int^{V \in \catO_{E}} i_J^* \catO(V^N,-) \wedge X_{i^*V}^N
    \ar[r]^-{\id\smsh j}_-\cong\ar[d]^\cong&
    J_+ \wedge _{J_0} \Phi^N X\ar[d]\\
    \int^{V \in \catO_{E}}\catO(V^N,-) \wedge (J_+ \wedge _{J_0} X_{i^*V}^N)\ar[r]_-\cong&
    \int^{V \in \catO_{E}}\catO(V^N,-) \wedge (G_+\wedge_H X_{i^*V})^N. 
    }
    $$
    The diagram commutes when the right vertical map is the desired map (with target $\PhiNH{}(G_+
  \smash_H X)$) which consequently is an isomorphism.
\end{proof}
\newcommand{\ione}{i_J}

For orthogonal $G$-spectra, taking categorical $N$-fixed points
commutes with the restriction to $H$-spectra, \ie 
\begin{lemma}
  If $Y$ is an orthogonal $G$-spectrum, then the natural
map \[\ione^*(Y^N) \to (i^*Y)^N\]
is an
isomorphism of $J_0$-spectra.
\end{lemma}

The following result says that when $N\subseteq G$ is \orfi, the same is true for geometric fixed points, a result we will eventually need when passing from finite groups to compact
Lie groups in Theorem~\ref{generalliegroup}.
\begin{prop}\label{geomfixedrestrict}
 Assume $N\subseteq G$ is \orfi. 
Then taking
  geometric $N$-fixed 
 points commutes with the restriction to $H$-spectra. More precisely, if $Y$ is an orthogonal $G$-spectrum, then the functor
\(j^*\colon\catO_E \to \catO_{E_0}\) induces
 a
 natural (in $Y$) isomorphism
 \[\ione^*\PhiNH{} Y \xto \cong
 \Phi^N(i^* Y)\]
 of $J_0$-spectra.
\end{prop}
\begin{proof}
The left hand side is 
$$i_J^* \PhiNH{} Y = 
  i_J^* \int^{V \in \catO_E} \catO(V^N,-) \wedge Y_V^N 
  =\int^{V \in \catO_E}  \catO(i^*V^N,-) \wedge (i^* Y_{V})^N, 
$$
which by means of $j^*$ maps as 
$$\int^{V \in \catO_E}  \catO(i^*V^N,-) \wedge (i^* Y_{V})^N\to
\int^{W \in \catO_{E_0}} \catO(W^N,-) \wedge i^*Y_W^N
  = \Phi^N i^*Y
$$
which is an isomorphism by Lemma~\ref{condition1lemma} with $X=i^*Y$. 
\end{proof}

\chapter{Finite Smash Powers and Commutativity}
\label{ch:finsmashcom}

In this chapter we demonstrate that, under mild cofibrancy conditions, we have tight equivariant control over finite smash powers.  This is in turn fed into the constructions in Chapter~\ref{ch:smashpow} to give invariants of commutative orthogonal ring spectra, generalizing topological cyclic homology and covering homology.

The crowning result of this chapter is Theorem~\ref{geomScofib}, which identifies the so-called geometric fixed points (see Definition~\ref{dfngeomfixed}) of smash powers over a finite free $G$-set with the smash power over the {\em orbits}.  When applied to smash powers of commutative ring spectra this can be rephrased as saying that the geometric fixed points give a concrete model for the {\em homotopy orbits in commutative orthogonal ring spectra}.  The technical underpinnings of this statement are improved on in Chapter~\ref{ch:smashpow} and the result itself reappears as an appropriately natural statement in Theorem~\ref{thm:diagisoisnaturalcom}.  This naturality will in turn be crucial when extending to compact Lie groups.

The correspondence between the the smash powers over orbits and the geometric fixed points over the full smash power is provided by the \emph{geometric diagonal} which we define in Section~\ref{specialcellssubsect}.  We prove that the $\Sp$-model structure of Definition~\ref{Spmodelstructure} provides us with good control over smash powers while cooperating tolerably well with the geometric fixed points.

We apply these constructions to smash powers in Section~\ref{subsectcellularfiltrations} and attain the mentioned results via a careful cellular induction started in Section~\ref{sec:cellfilt}. 

As corollaries of our understanding of smash powers  we also collect some basic but crucial facts pertaining to commutative orthogonal ring spectra in Section~\ref{sec:eqcomconv}.  We end the chapter with some categorical considerations relating to the tensor structure on commutative orthogonal ring spectra in Section~\ref{subsecttensors}.

Throughout this chapter we work with the positive model structure on
orthogonal \(G\)-spectra from Definition~\ref{Spmodelstructure},
and we let 
$(\hml,\gml)=(\AIplus,\All_+)$ 
be the positive mixing pair from Example~\ref{positivemixingpair}.

\section{The Geometric Diagonal and Induced Regular Cells  
}\label{specialcellssubsect}

Unlike the category of spaces, the category of spectra doesn't have a ``diagonal''.  However, there is a very nice substitute, namely the ``geometric diagonal'' which we define now.  This construction will be crucial to our further investigations. 

Smash powers are more ``regular'' than most equivariant spectra in a way that makes it possible to start an induction process towards proving that the geometric diagonal is an isomorphism.  We end this section by studying the ``induced regular cells'' and how the geometric diagonal handles these.  The induction process will be followed up in the next sections.

In this section \(G\) is a finite group, $N\subseteq G$ a normal subgroup and \(X\) is a finite \(G\)-set.
\begin{dfn}\label{definitionofgeomfixedpoints}
  Let \(E\) be an orthogonal spectrum and let \(X\) be a finite
  \(G\)-set with \(N\)-orbits \(X_N\). 
The {\em geometric diagonal  map}
  \begin{displaymath}
   \Delta_XE= \Delta \colon 
    E^{\wedge X_N} \to \PhiNH{} (E^{\wedge X})
  \end{displaymath}
  is defined as follows:  Given vector spaces
  \((V_{[x]})_{[x] \in X_N}\), let 
  \(W = \bigoplus_{x \in X} V_{[x]}\) and \(U = \bigoplus_{[x] \in
    X_N} V_{[x]}\).  The diagonal embedding \(U \to W\) gives an isomorphism \(W^N \xto{\cong} U\). Considering this isomorphism as  an element \(\delta \in \catOJ(W^N,U)\),
  the structure map for the smash product gives a
  morphism
  \begin{displaymath}
    \delta\colon(\bigwedge_{x \in X} E_{V_{[x]}})^N \to \catOJ(W^N,U) \wedge
    \left(E^{\wedge X}_{W}\right)^N.
  \end{displaymath}
  The geometric diagonal $\Delta$ is defined by requiring that the following diagram commutes:
\labeleq{geomdiagdiagr}
  {
    \xymatrix{
      \bigwedge_{[x] \in X_N} E_{V_{[x]}} \ar[r] \ar[d]^-{\cong}&
      E^{\wedge X_N}_{U}  \ar[r]^-{\Delta} &
      (\Phi^N E^{\wedge X})_{U} \\
      (\bigwedge_{x \in X} E_{V_{[x]}})^N \ar[rr]^-{\delta} &&
      \catOJ(W^N,U) \wedge
      \left(E^{\wedge X}_{W}\right)^N,
      \ar[u]
      }
    }
  where the unlabeled arrows are given by the universal property of
  smash-products and by inclusion in the coend defining geometric
  fixed points by considering \(W\) as an object of \(\catOE\).  
\end{dfn}

Direct inspection shows that the geometric diagonal is monoidal in the following sense.
  \begin{lemma}\label{lem:geomdiagismonoidal}
    If $E$ and $F$ are orthogonal spectra and $X$ a finite $G$-set, then the diagram
$$\xymatrix{
E^{\smsh X_N}\smsh F^{\smsh X_N}\ar[rrr]^-{\text{shuffle}}_\cong
\ar[d]^{\Delta_XE\smsh\Delta_XF}&&&
(E\smsh F)^{\smsh X_N}\ar[d]^{\Delta_X(E\smsh F)}\\
\Phi^N(E^{\smsh X})\smsh \Phi^N(F^{\smsh X})\ar[r]^-\alpha&\Phi^N(E^{\smsh X}\smsh F^{\smsh X})\ar[rr]^-{\text{shuffle}}_\cong&&\,\Phi^N((E\smsh F)^{\smsh X}),}
$$
commutes, where the horizontal isomorphisms are the natural shuffle permutations of smash factors.  With repeating spectra entries indexed by a finite set $Y$ this takes the form of the commutative diagram
$$\xymatrix{
(E^{\smsh X_N})^{\smsh Y}\ar[rrr]^-{\text{shuffle}}_\cong
\ar[d]^{(\Delta_XE)^{\smsh Y}}&&&
(E^{\smsh Y})^{\smsh X_N}\ar[d]^{\Delta_X(E^{\smsh Y})}\\
(\Phi^N(E^{\smsh X}))^{\smsh Y}\ar[r]^-\alpha&\Phi^N((E^{\smsh X})^{\smsh Y})\ar[rr]^-{\text{shuffle}}_\cong&&\,\Phi^N((E^{\smsh Y})^{\smsh X}).}
$$
Similarly, the diagram
$$\xymatrix{
(E^{\smsh Y})^{\smsh X_N}\ar[rr]^-{\text{shuffle}}_\cong
\ar[d]^{\Delta_X(E^{\smsh Y})}&&
E^{\smsh (Y\times X)_N}\ar[d]^{\Delta_{Y\times X}E}\\
\Phi^N((E^{\smsh Y})^{\smsh X})\ar[rr]^-{\text{shuffle}}_\cong&&\,\Phi^N(E^{\smsh (Y\times X)})}
$$
commutes.
  \end{lemma}

The class of semi-free $G$-spectra is too big to fully control the geometric fixed point functor. Our studies of the smash powers of semi-free non equivariant spectra has given us a specific example of a class where such control is possible. Now we will define classes of \emph{regular spectra}, and \emph{induced regular spectra} which the smash powers are examples of.
\begin{dfn}\label{def:regularspectra}
  Let \(\varphi \colon G \to \Or{V}\) be a \(G\)-representation and
  let \(P\) be a
  \(G\)-invariant subgroup of \(\Or V\). Given a finite free \(G\)-set
  \(X\), we consider the group \(G \ltimes \prod_X P\), where \(G\)
  acts on the product by the action corresponding to conjugation of
  functions from \(X\) to \(P\).
  Let \(\psi \colon G \to \Or {V^{\oplus X}}\) be the direct sum representation.

Given a \(G\)-representation \(\rho \colon G \to \Or{W}\), we
  say that a \(G \ltimes_{\rho} \Or  W \)-space $K$ (in the notation of Remark~\ref{shearingsemi}) is {\em \(W\)-regular}\index{regular!\(W\)-}\index{Wregular@\(W\)-regular} if such data $(\phi, V, P, X)$ exists and $\rho$  contains the direct sum representation and if $K$ 
  is isomorphic to   
$$(\Or W)_+ \wedge_{\prod_X P} L$$ for a genuinely cofibrant \(G \ltimes \prod_X P\)-space
  \(L\).
  
A semi-free $G$-spectrum is called \emph{regular}\index{regular spectrum} if it is isomorphic to one of
the form  \(\Gr{\rho}{W} K\) (cf. Definition~\ref{twistedsemifree}) for \(\rho \colon G \to \Or{W}\)
a
\(G\)-representation and a \(W\)-regular \(G \ltimes_{\rho} \Or W\)-space $K$. 
\end{dfn}
\begin{rem}\label{freeareregular}
Free spectra are regular, since we can choose \(V = 0\).
\end{rem}
Note that since, in Definition~\ref{def:regularspectra}, the subspace $\prod_X P \subseteq \Or{W}$ is $G$-invariant, we have an inclusion of $N$-fixed subgroups $(\prod_X P)^N \subseteq (\Or{W})^N$ for any subgroup $N$ of $G$.
\begin{rem}\label{cofibrantregular}
  Since all involved functors preserve colimits, and since inducing up preserves genuine cofibrations (the right adjoint is clearly a right Quillen functor), the regular semi-free spectra are $\Sp$-cofibrant and genuine cofibrations induce $\Sp$-cofibrations.
\end{rem}
\begin{ex}
The most important example of a regular semi-free $G$-spectrum is the smash power \[(\Gr{\varphi}{V}K)^{\smash X} \cong \Gr{\psi}{\dsi{X} V}[ {\Or{{\dsi{X}{V}}}}_+ \smash_{\pri{X}\Or{V}} K^{\smash X}],\]
for $X$ a finite free $G$-set (c.f.~Proposition~\ref{semifreesmashpower}). 
\end{ex}

Let
$\rho \colon G \to \Or{W}$ be a $G$-representation.
By Lemma~\ref{lemma:semifreeadjoint} and Definition~\ref{twistedsemifree},
for \(K\) an \(\Or{W} \times G\)-space the
regular semi-free $G$-spectrum
$\Gr{\rho}{W} K = \Gr{}{W} \catOr(W,W) \wedge_{\Or{W}} K$
is determined by its value
\begin{displaymath}
  (\Gr{\rho}{W} K)_V = \catOr(W,V) \wedge_{\Or{W}} K
\end{displaymath}
at \(V \in \catOr\). An element \((g,\alpha)\) of the group \(G \times
\Or{V}\) acts on 
\(\gamma \wedge k \in \catOr(W,V) \wedge_{\Or{W}} K\) by the formula
\begin{displaymath}
  (g,\alpha) \cdot (\gamma \wedge k) = \alpha \gamma \rho(g^{-1})
  \wedge gk.
\end{displaymath}
We write $\bar\rho\colon J\to\Or{W^N}$ for the $J$-representation on the fixed points $W^N$ given by the restriction of $\rho$ to $W^N$.

\begin{prop}\label{geomfixofregfree}
Let $N$ be a normal subgroup of $G$ and let $X$ be a finite free $G$-set. 
Given $G$-representations $\phi\colon G\to \Or V$ and $\rho\colon G\to \Or W$ with $V^{\oplus X}\subseteq W\in\catJE$, if $P\subseteq \Or V$ is a $G$-invariant subgroup and $L$ is a cofibrant $G\ltimes\prod_XP$-space, consider the $W$-regular space $K=(\Or W)_+\wedge_{\prod_XP}L$.  Then the dual Yoneda lemma and Proposition~\ref{regularfix} induce a natural isomorphism 
\[\Gr{\overline \rho}{W^N}((\Or{W^N})_+\wedge_{\Or{W}^N}K^N)\to\PhiNH{}(\Gr{\rho}{W} K).\]
More precisely, the following chain of natural transformations consists of isomorphisms
 \begin{align*}
  \Gr{\overline \rho}{W^N}((\Or{W^N})_+\wedge_{\Or{W}^N}K^N) &=
  \Or{}(W^N,-)\wedge_{\Or{W^N}}((\Or{W^N})_+\wedge_{\Or{W}^N}[(\Or{W})_+\wedge_{\prod_XP}L]^N)\\
&\cong\Or{}(W^N,-)\wedge_{\Or{W}^N}[(\Or{W})_+\wedge_{\prod_XP}L]^N\\
&\gets
\Or{}(W^N,-)\wedge_{\Or{W}^N}[(\Or{W}^N)_+\wedge_{(\prod_XP)^N}L^N]\\
&\cong\Or{}(W^N,-)\wedge_{(\prod_XP)^N}L^N\\
&\underset{\text{Yoneda}}{\overset{\cong}\gets}\smallint^{U\in\catJE}\Or{}(U^N,-)\smsh[\Or{}(W,U)^N\wedge_{(\prod_XP)^N}L^N]\\
&\to\smallint^{U\in\catJE}\Or{}(U^N,-)\smsh[\Or{}(W,U)\wedge_{\prod_XP}L]^N\\
&=\PhiNH{}(\Gr{\rho}{W} K),
 \end{align*}
 where the unlabelled isomorphisms are instances of (associativity and) the  isomorphism $\id\cong \id\smsh_HH_+$ (for $H=\Or{W^N}$ and $H=\Or W^N$), the first and the last arrows are instances of the universality of limits and the middle arrow is the isomorphism given by the dual Yoneda lemma.
\end{prop}
\begin{proof}
The arrows given by universality of limits are isomorphisms by Proposition~\ref{regularfix} as applied to $Z=\Or{}(W,T)\wedge L$ (with $T=W$ and $T=U$).
\end{proof}

When the action $\phi$ is trivial and $W=V^{\oplus X}$ this becomes particularly attractive:
\begin{thm}\label{geomsmashsemifree}
  Let \(V\) be an inner product space, let \(X\) be a discrete set with free action of a finite discrete
  group \(G\) and let $M$ be a genuinely cofibrant $G\times\Or V$-space.
Let \(N\) be a normal subgroup of \(G\) with factor group \(J\) and
  let $E = \Gr{}V(M)$.
  Then the geometric diagonal map (Definition~\ref{definitionofgeomfixedpoints})
  \begin{displaymath}
    \Delta \colon E^{\wedge X_N} \to  \PhiNH{} (E^{\smsh X})
  \end{displaymath}
  is an  isomorphism of  $J$-spectra.
\end{thm}
\begin{proof}
In the setup of Proposition~\ref{geomfixofregfree}, let $W=V^{\oplus X}$, 
$K=(\Or{W})_+\smsh_{\prod_X\Or V}M^{\smsh X}$ and $P=\Or V$.  Using the definitions of the smash product, semi-free spectra, the natural diagonal isomorphisms $V^{\oplus X_N}\cong [V^{\oplus X}]^N=W^N$, $\prod_{X_N}\Or{V}\cong[\prod_X\Or{V}]^N$ and Proposition~\ref{regularfix} as in Proposition~\ref{geomfixofregfree}, we get the vertical isomorphisms in the diagram
$$
\xymatrix{\Gr{}{W^N}((\Or{W^N})_+\wedge_{\Or{W}^N}K^N)\ar[d]_\cong
\ar[rr]^-{\text{Proposition~\ref{geomfixofregfree}}}_\cong&&
\PhiNH{}(\Gr{}{W} K)\ar[d]_\cong\\
\Or{}(V^{\oplus X_N},-)\smsh_{\prod_{X_N}\Or V}M^{\smsh X_N}\ar[rr]&&
\smallint^U\Or{}(U^N,-)\smsh[\Or{}(V^{\oplus X},U)\smsh_{\prod_X\Or V}M^{\smsh X}]^N
 \\
 (\Gr{}VM)^{\smsh X_N}\ar[u]^\cong
\ar[rr]^-{\Delta}&&
 \PhiNH{}(\Gr{}VM)^{\smsh X}.\ar[u]^\cong
}
$$
The middle horizontal map is given by choosing $U=W=V^{\oplus X}$ and the inverse of the diagonal isomorphism $V^{\oplus X_N}\cong U^N$.  The topmost square commutes by design
(for instance, the leftmost top vertical isomorphism is the composite
\
\begin{align*}
  \Gr{}{W^N}((\Or{W^N})_+\wedge_{\Or{W}^N}K^N)
&=\Or{}(W^N,-)\smsh_{\Or{W^N}}(\Or{W^N})_+\smsh_{\Or W^N}[(\Or W)_+\smsh_{\prod_X\Or V}M^{\smsh X}]^N\\
& \cong 
\Or{}(W^N,-)\smsh_{\Or W^N}[(\Or W)_+\smsh_{\prod_X\Or V}M^{\smsh X}]^N \\
& \cong 
\Or{}(W^N,-)\smsh_{\Or W^N}(\Or W)^N_+\smsh_{[\prod_X\Or V]^N}[M^{\smsh X}]^N \\
& \cong \Or{}(V^{\oplus X_N},-)\smsh_{\prod_{X_N}\Or V}M^{\smsh X_N}
\end{align*}
used in Proposition~\ref{geomfixofregfree}, and similarly but simpler for the rightmost map) and the lower square commutes by the definition of the geometric diagonal,~\ref{definitionofgeomfixedpoints}, (evaluate at any $(V_{[x]})_{[x]}$ and use the defining diagram~\ref{geomdiagdiagr}).
\end{proof}
{Let $i\colon H\subseteq G$ be a subgroup containing $N$.
\begin{dfn}\label{indregsemif}
A semi-free $G$-spectrum $E$ is called \emph{induced 
regular}\index{regular!induced \(\hml\)-}\index{induced regular} if there is an isomorphism of $G$-spectra
\[E \cong G_+ \smash_H F,\]
for $H$ a subgroup of $G$ and $F$ a regular semi-free $H$-spectrum.
\end{dfn}
Lemma~\ref{cond1lem} implies that an inclusions $N\subseteq G$ of a normal subgroup in a finite group $G$ is \orfi,
so we get the following characterizations of geometric fixed points for induced regular spectra which is analogous to Kro's Lemma 3.10.8:
\begin{thm}\label{geomfixedofinduced}
  Let $G$ be a finite group and $H$ and $N$ subgroups with $N$ normal.
  Let $\rho\colon G\to\Or W$ be a representation and $K$ a $W$-regular space.
  
  The geometric fixed points of the induced regular semi-free
 $G$-spectrum $G_+ \smash_H \Gr{\rho}{W}K$  
are given by the natural isomorphism:
\[\PhiNH{}(G_+ \smash_H \Gr{\rho}{W}K)
\cong \begin{cases}
{\scriptsize{\faktor{G}{N}}_+\smash_{\scriptsize{\faktor{H}{N}}}
[\Gr{\bar\rho}{W^N}((\Or{W^N})_+\smsh_{\Or W^N}K^N)}]&\text{
    if  } N \subset H\\ \,\,\,\,\,\,\,\,\,* &\text{
    otherwise.}\end{cases}\]
\end{thm}
\begin{proof}
  The case \(N \subseteq H\) is a direct consequence of Theorem~\ref{geomfixofregfree} and
  Proposition~\ref{3810}. For the second part, note that already for $H$-spaces
  $A$ we have $(G_+ \smash_H A)^N \cong *$ if $N$ is not contained
  in $H$, hence \((G_+ \smash_H \Gr{\rho}{W}K)^N_V \cong *\) for all \(V\) (cf.~\cite[3.10.8]{Kro}).
\end{proof}

Since we must consider $G$- and $H$-representations simultaneously,
we allow ourselves to decorate our semi-free spectra when this can aid readability,
so that $\Gr{H}{V}K$\index{GHVK@$\Gr{H}VK$} will signify the semi-free $H$-spectrum where $V$ is an $H$-representation
(the underlying homomorphism $H\to \Or V$ will be suppressed from the notation in these cases).

Induction $G_+\smsh_H-\colon H\catOT\to G\catOT$ is easily analyzable on semi-free spectra: given a subgroup $P\subseteq H\times\Or V$, then
 $$G_+\smsh_H[\catOr(V,-)\smsh_{\Or V}(H\times\Or V)/P_+]
\cong\catOr(V,-)\smsh_{\Or V}(G\times\Or V)/P_+,
 $$
and so induction preserves $\Sp$-cofibrations and we  get the following:
\begin{lem}\label{inducesmashpresscof}
  Let $V$ be an $H$-representation,
  $X$ a finite free $H$-set, $P$ an $H$-invariant subgroup of
  $\Or{V}$, for $Q = \Pi_X P$ and $i$ a genuine cofibration of $H
  \ltimes Q$-spaces, the map 
  \[G_+ \smash_H \Gr{H}{W}[{\Or{W}}_+ \smash_Q i]\]
of $G$-spectra is an $\Sp$-cofibration for every
\(G\)-representation \(W\) containing \(V^{\oplus X}\).
\end{lem}

\begin{dfn}
  Let $\SpGreg$.\index{SGreg@$\SpGreg$} be the class of \Sp-cofibrations
  of the form 
  \[G_+ \smash_H \Gr{H}{W}[{\Or{W}}_+ \smash_Q i]\]
  in Lemma~\ref{inducesmashpresscof}.
\end{dfn}
\begin{rem}\label{stupidremark}
For $\SpGreg$-cell complexes, Theorem~\ref{geomfixedofinduced} allows us to compute the geometric fixed points via a cell induction. This could be used to define a class of cofibrations very much in the spirit of Kro's \emph{induced cells} (\cite[3.4.4]{Kro}) and \emph{orbit cells} (\cite[3.4.6]{Kro}), but more general than both. Since we are not going to construct model structures or even replacement functors for any of these classes, we will not go into this generality. Instead we will focus on the type of cells that will appear in the cell structure for the smash powers.
\end{rem}
Finally we give a name for the cells that we will use in the equivariant filtration theorem:
\begin{dfn}
 An \emph{induced regular cell}\index{induced!regular cell} is an \Sp-cofibration of the form
\[f= G_+ \smash_H i^{\square H}.\]
for $H\subseteq G$ a normal subgroup  and  $i$ a generating $\Sp$-cofibration of orthogonal spectra.
Here $\square$ is the pushout
product construction from Definition~\ref{pushoutproduct}.

 We let
$\Indreg$\index{IndregG@$\Indreg$, the induced regular cells} denote the
\emph{class of all induced regular cells}.  
\end{dfn}
\begin{rem}\label{rem:sourceandtarget}
Both source and target of an induced regular cell are induced regular in the sense of Definition~\ref{indregsemif}, since if $i$ is the map
$$\Gr{}{V}\left({\faktor{\Or{V}}{P}}_+\smash [S^{n-1} \rightarrow D^n]\right),$$
then the dual Yoneda isomorphism as in Proposition~\ref{semi-free-smash} together with the fact that inducing up preserves colimits and the identification of $(S^{n-1}\subseteq D^n)^{\square H}$ with $S^{n|H|-1}\subseteq D^{n|H|}$ we have that $i^{\square H}$ corresponds to the map
\[\Gr{}{\dsi{H}{V}}\left({\faktor{\Or{\dsi{H}{V}}}{\Pi_H P}}_+ \smash[S^{|H|n-1} \rightarrow D^{|H|n}]_+\right), \]
where $|H|$ is the order of $H$. In other words, $i^{\square H}$ is represented by
the inclusion of the boundary sphere of the $H \ltimes \Pi_H P$ space
$D^{|H|n}$, where $\Pi_H P$ acts trivially, and $H$ acts by permuting
coordinates blockwise. 
\end{rem}
}

The geometric fixed points are well behaved for $q$-cofibrant $G$-spectra, but for our results about the geometric diagonal to be homotopically well behaved we'll need to extend this good behavior to the $\Indreg$-cellular case.  
Mandell and May's proof {\cite[V.4.17]{MM}} applies verbatim and we'll indicate only the necessary emendations.

Recall the definition of the
universal $\aml$-space $E\aml$ for a closed family of subgroups of $G$
from Definition~\ref{univFspace}.
\begin{dfn}
  For $N$ a normal subgroup, let $\nml$\index{N@$\nml$, the family of subgroups that do not contain $N$}
  be the family of subgroups of $G$ that do not contain $N$.
  Let $E\nml$ be the universal $\nml$-space, and let $\tilde{E}\nml$\index{EN@$\tilde{E}\nml$} be the homotopy cofiber of the quotient map $E\nml_+ \rightarrow S^0$ that collapses $E\nml$ to the non base point. For orthogonal $G$-spectra $X$, the map $S^0 \rightarrow \tilde{E}\nml$ induces a natural map
  $$\lambda\colon X \rightarrow X\smash\tilde E\nml.$$ \end{dfn}

Let $r\colon\id\we R$ be a $q$-fibrant replacement.
\begin{prop}\label{fundcofindreg}
For $\Indreg$-cellular orthogonal $G$-spectra $Y$, the natural maps
\labeleq{htpytypecor}{\xymatrix{
R(Y \smash \tilde{E}\nml)^N \ar[r]^-\gamma&\Phi^NR(Y \smash \tilde{E}\nml) &
\Phi^N(Y \smash \tilde{E}\nml)\ar[l]_-{\Phi^Nr}&\Phi^N Y\ar[l]^-\cong_-{\Phi^N\lambda}
}
}
where $\gamma$ is the isotropy separation map of Definition~\ref{dfngeomfixed}, are $\pi_*$-isomorphisms of orthogonal $J$-spectra.
\end{prop}
\begin{proof}
  Since $(\tilde E \nml)^H = S^0$ if $H$ contains $N$, and $(\tilde{E}\nml)^H$ is contractible otherwise, the map $\Phi^N\lambda$ of geometric diagonal is an isomorphism.  Since $r$ is a acyclic $q$-cofibration Lemma~\ref{lem:phipreservesstableeq} gives that $\Phi^Nr$ is an acyclic $q$-cofibration.  Hence we are left with analyzing the isotropy separation map.   When $Y$ is $q$-cofibrant this is handled in \cite[V.4.17]{MM}, and the same proof works for $\Indreg$-cellular orthogonal $G$-spectra $Y$ with the following two modifications.

  The cell induction is over the induced regular cells
    \[G_+ \smash_H (\Gr{}{V}(L)^{\smash H}),\]
    (not the free cells $\Fr{}ZA$ as in \cite[V.4.12]{MM}) with $L$ a genuine $\Or V$-cell complex, and, for $J_1 = \faktor{H}{N}$, we must show that the map
    \[p \colon \hocolim\limits_{U^N = W} \catOE(V^{\oplus{H}},U) \rightarrow \catOJ(V^{\oplus J_1},W)\]
    induced by the restriction to the $N$-fixed space $V^{\oplus J_1} \subset V^{\oplus H}$
is a $J_1\ltimes\Pi_{J_1} \Or V$-homotopy equivalence (cf.~Definition~\ref{dfnphiphi}).

From the definition of $\catOE$ (\ref{dfnOE}), recall that
$\catOE(\dsi{H}{V},U) = \catOr(\dsi{H}{V},U)^N$. Note that any
$N$-equivariant isometry has to preserve fixed spaces and
isotypical factors.
Hence, orthogonal decompositions $W \oplus U' \cong U$ and $\dsi{J_1}{V} \oplus
V' \cong \dsi{H}{V}$ where \(N\) acts
trivially on \(W\) and \(U'\) contains
no summands with trivial \(N\)-action, induce an
isomorphism
\[\catOE(\dsi{H}{V},U) \cong \catOJ(\dsi{J_1}{V},W) \times \catOr(V',U')^N.\]
Here \(\catOr(V',U')\) is a sphere bundle over $\catL(V',U')$ whose
dimension depends linearly on the dimension of \(U'\) and the map $p$
is induced from the projection to the first factor. Thus it suffices to prove that $\catL(V',\colim U')^N \rightarrow *$ is a $J_1\ltimes\Pi_{J_1}\Or{V}$-homotopy equivalence. Since $V'$ is the orthogonal complement of $\dsi{J_1}{V}$, the $\Pi_{J_1} \Or{V}$ action is trivial, so \cite[II.1.5]{LMS} gives the desired result. 
    
Lastly, in \cite[V.4.16]{MM}, the reference to \cite[9.3]{LMS} must be extended beyond $G$-CW-complexes to genuine $G$-cell complexes.  In particular there are bijections \[[A, B \smash \tilde{E}\nml]_G \cong [A^N, B \smash \tilde{E}\nml]_G \cong [A^N, B]_G\]
between sets of $G$-homotopy classes for $A$ any representation sphere and $B$ a level of an induced regular spectrum, which is a genuine $G$-cell complex by Lemma~\ref{inducesmashpresscof} and Corollary~\ref{illsubgroup}.
\end{proof}

\begin{rem}\label{correcttype}
  There are alternative definitions for the geometric fixed points of a $G$-spectrum. Classically, one would take the leftmost $J$-spectrum $[R(Y \smash \tilde{E}\nml)]^N$ in the zigzag~\eqref{htpytypecor} as the definition. 
The reason is that we want to use the isotropy separation in the form of the homotopy-cofiber sequence of (non equivariant) orthogonal spectra:
\labeleq{fundamentalcofiber}{\xymatrix{{[R(Y \smash E\nml_+)]^N}\ar[r]&{[R(Y)]^N}\ar[r]&{[R(Y \smash \tilde{E}\nml)]^N}},}
which arises from the defining homotopy-cofiber sequence
$$\xymatrix{E\nml_+\ar[r]& S^0\ar[r]^-\lambda&\tilde{E}\nml}.$$
We saw above that the homotopy groups of $[R(Y \smash \tilde{E}\nml)]^N$ are closely related to the homotopy groups of the geometric fixed points of $Y$, partially justifying us calling $\gamma\colon Y^N\to\Phi^NY$ the isotropy separation map. 
\end{rem}
\begin{dfn}\label{def:Mortenamliso}
  Let $\aml$ be a closed family of normal subgroups of $G$. A morphism $f$ of $G$-spectra is a {\em $\pi_*^\aml$-isomorphism}\index{piAisomorphism@\(\pi_*^\aml\)-isomorphism} if $\pi_*^Hf$ is an isomorphism for all $H\in\aml$. 
\end{dfn}
Note that a morphism $f$ of $G$-$\Omega$-spectra is a $\pi_*^\aml$-isomorphism if and only if $\Fix^H f$ is a non equivariant $\pi_*$-isomorphism for all $H \in \aml$.

\begin{prop}\label{fundcofiber}
Let $\aml$ be a family of subgroups of $G$ and let $X$ and $Y$ be $\Indreg \cup \Fr{}{}I_G$-cellular. Then for a morphism $f \colon X \rightarrow Y$, the following are equivalent:
\begin{enumerate}
\item 
  The map $f$ is a $\pi_*^\aml$-isomorphism.
\item For all $H \in \aml$ the map $\Phi^H f$ is a (non equivariant) $\pi_*$-isomorphism.
\end{enumerate}
\end{prop}
\begin{proof}
  We use induction on the size of the family $\aml$, which is possible
since $G$ is compact (cf.~\cite[1.25.15]{tD}). For the trivial family,
the result is true.

Let $r\colon\id\we R$ be a $q$-fibrant replacement with respect to $G$ and hence with respect to any subgroup of $G$.

We compare the maps induced on the homotopy cofiber sequences for $N \in \aml$:
\labeleq{cofibercomp}{\xymatrix{{[R(X \smash E\nml_+)]^N}\ar[d]\ar[r]&{[R(X)]^N}\ar[d]\ar[r]&{[R(X \smash \tilde{E}\nml)]^N}\ar[d]\\
{[R(Y \smash E\nml_+)]^N}\ar[r]&{[R(Y)]^N}\ar[r]&{[R(Y \smash
  \tilde{E}\nml)]^N}}}
where, as before, $\nml$ is the family of subgroups not containing $N$.

Note that since $\Indreg$ consists of $\Sp$-cofibrations, $X$ and $Y$ are levelwise genuine $G$- hence $N$-complexes. 
Property $(i)$ is equivalent to the second vertical map in \eqref{cofibercomp} being a $\pi_*$-isomorphism for all $N\in\aml$, and Proposition~\ref{fundcofindreg} establishes that $(ii)$ is equivalent to the third vertical map in \eqref{cofibercomp} being a $\pi_*$-isomorphism.  Hence, if we can show that both $(i)$ and $(ii)$ imply that the left vertical map is a $\pi_*$-isomorphism we will be done.

Fix an $N\in\aml$.
For $(i)$ the left map is a $\pi_*$-isomorphism by Lemma~\ref{stablebetweenomegaislevel}. 
For $(ii)$ we use the induction hypothesis.  Since $N\in\aml$, but $N\notin\nml$ the family $\aml\cap\nml$ is strictly smaller than $\aml$ which implies that $f$ is an $\aml \cap \nml$-equivalence.  Furthermore,  $\aml\cap \nml$ contains all proper subgroups of $N$, but not $N$ itself and so, as an $N$-space  $E\nml$ is a universal $\nml_N$-space where $\nml_N$ is the family of proper subgroups of $N$.
 Thus $R(X \smash E\nml_+) \rightarrow R(Y \smash E\nml_+)$ is a levelwise $\nml_N$-equivalence between genuine $\nml_N$-complexes, thus an $N$-homotopy equivalence and the claim follows.
\end{proof}

\section{Cellular Filtrations}\label{sec:cellfilt}
We need some additional control of how smash powers of orthogonal spectra can be assembled.
In Section~\ref{subsectcellularfiltrations} we construct the cellular structures that form the technical heart of our constructions. We generalize Kro's approach from \cite[2.2]{Kro}, correcting some minor mistakes along the way. In particular we drop the assumption that all $\lambda$-sequences are $\N$-sequences in order to be able to attach cells one at a time, and work in general categories. This allows us to apply the theory in a lot of different contexts, cf.~Corollary~\ref{eqfreeup2}, but also for a potential extension of our results to multiplicative norm constructions (cf.~Remark~\ref{multiplicativenormscomment}).

For reference, in this subsection, we list some easily checked facts
about ``cellular 
filtrations'' which hold for any closed symmetric category $\catC$
with all small colimits (most of the facts do not need all this
structure). 
First recall the following definition (\eg \cite[2.1.9]{H}):
\begin{dfn}\label{cellularMaps}
 Let 
$I$ be 
a class of morphisms of $\catC$. Then a morphism $f\colon A
\rightarrow X$ in $\catC$ is \emph{$I$-cellular}, if it is a
transfinite composition of pushouts of coproducts of elements of
$I$. In this case we also say that \(f \colon A \to X\) is an
\(I\)-cell complex.
\end{dfn}
Given morphisms \(f\) and \(g\) in \(\catC\), we write \(f \square g\) for
their pushout product of Definition~\ref{pushoutproduct}.
If \(I\) and \(J\) are sets of morphisms in \(\catC\), we write \(I \square J\)
for the set of morphisms of the form \(f \square g\) for \(f \in I\)
and \(g \in J\).
\begin{thm}\label{cellulartheorem}
Let $(\catC,\smash)$ be closed symmetric monoidal with all small
colimits. Let $I$ and $J$ be sets of morphisms in $\catC$ and let
$f\colon A\rightarrow X$ and $g\colon B\rightarrow Y$ be $I$- and
$J$-cellular, respectively. Then their pushout product $f \square g$
is relative $(I\square J)$-cellular.\\ 
In particular, if $\lambda$ and $\mu$ are partially ordered indexing
sets for cells of $f$ and $g$, respectively, then $\lambda\times \mu$
is a partially ordered indexing set for cells of $f \square g$. 
\end{thm}

It seems hard to 
actually find a proof of the above theorem explicitly spelled out in
the literature. 
Since we are going to have to work with such filtrations in more
detail later, we give them below.

\subsubsection{Pushouts and Pushout Products}
\begin{lemma}\label{composepushout}
Consider the following commutative diagram in $\catC$:
$${\xymatrix{{A}\ar[r]\ar[d]&{B}\ar[r]\ar[d]&{C}\ar[d]\\{D}\ar[r]&{P}\ar[r]&{\,Q.}}}
$$
\begin{indentpar}{1cm}
\begin{enumerate}
\item If both the left and the right subsquare of the diagram are pushout diagrams, then so is the outer rectangle.
\item If both the left subsquare and the outer rectangle are pushout diagrams, then so is the right subsquare.
\end{enumerate}
\end{indentpar}
\end{lemma}

\begin{lemma}\label{pushoutcube}
Let $\catC$ have all pushouts. Consider a commutative cube in $\catC$, where either the top and bottom faces or the left and right faces are pushouts:
\labeleq{pushoutcubecube}{\xymatrix@!0{
&{A_0}\ar[rr]\ar'[d][dd]\ar[dl]&&{X_0}\ar[dd]\ar[dl]\\
{Y_0}\ar[rr]\ar[dd]&{}\ar@{}[r]&{P_0}\ar[dd]\\
&{A_1}\ar@{}[r]\ar'[r][rr]\ar[dl]&&{X_1}\ar[dl]\\
{Y_1}\ar[rr]&&{P_1}}.}
Then the induced square
$$\xymatrix{X_0{\coprod_{A_0}}A_1\ar[r]\ar[d]&X_1\ar[d]\\P_0{\coprod_{Y_0}}Y_1\ar[r]&P_1}
$$
is again pushout.
 \end{lemma}
\begin{lemma}\label{levelwisesquare}
Consider two pushout squares
\[\xymatrix{{A_b}\ar^-{g_b}[r]\ar[d]&{X_b}\ar[d]&{}&{}&{B_b}\ar^-{h_b}[r]\ar[d]&{Y_b}\ar[d]\\{A_f}\ar_-{g_f}[r]&{X_f}\pushout&{}&{}&{B_f}\ar_-{h_f}[r]&{\,Y_f.}\pushout}\]
Their row-wise pushout product is also a pushout square
$$\xymatrix{
{(A_b\smsh Y_b)\coprod_{A_b\smsh B_b}(X_b\smsh B_b)}\ar[d]\ar^-{g_b\square h_b}[r]&{X_b\smash Y_b}\ar[d]\\
{(A_f\smsh Y_f)\coprod_{A_f\smsh B_f}(X_f\smsh B_f)}\ar_-{g_f\square h_f}[r]&{\,X_f \smash Y_f.}\pushout}
$$
\end{lemma}
\subsubsection{Cellular Maps}
We will use Lemma~\ref{levelwisesquare} to recognize a \relative cellular structure on the $\square$-product of \relative cellular maps. Recall the following definition (\eg \cite[2.1.9]{H}):
\begin{dfn}
 Let $I$ be 
a class of morphisms of $\catC$. Then a morphism $f\colon A \rightarrow X$ in $\catC$ is \emph{\relative $I$-cellular},\index{relative I-cellular@relative $I$-cellular}
 if it is a transfinite composition of pushouts of coproducts of elements of $I$.
\end{dfn}
\begin{rem}
 Let $f\colon A \rightarrow X$ be a \relative $I$-cellular map, and let $A = X_0 \rightarrow X_1 \rightarrow \ldots$ be a $\lambda$-sequence that exhibits this structure, \ie $\lambda$ an ordinal and for any $\alpha \leq \lambda$ we have pushout diagrams
\[\unskew\xymatrix{{S_\alpha}\ar^-{\sigma_\alpha}[r]\ar_{i_\alpha}[d]&{\colim\limits_{\beta<\alpha} X_\alpha}\ar^-{f_\alpha}[d]\\{D_\alpha}\ar[r]&{X_\alpha,}\pushout}\]
where $i_\alpha$ is a coproduct $\coprod\limits_{c \in C_\alpha}i_c$, with all the maps $i_c$ in $I$, and $C_\alpha$ empty whenever $\alpha$ is a limit ordinal. Then the union of the $C_\alpha$ is partially ordered, with $i_c \in C_\alpha$ smaller than $i_d \in C_\beta$ if and only if $\alpha < \beta$. In this situation, we say that $\bigcup_{\alpha \leq \lambda} C_\alpha$ indexes the attached cells of $f$ in the $\lambda$-sequence.
\end{rem}
The following lemma helps with keeping the ``length'' of the transfinite composition in check when the domains of the morphisms in $I$ are sufficiently small:
\begin{lemma}\label{kappaseq}
 Assume that the domains of the maps in $I$ are $\kappa$-small. For
 every $I$-cell complex $f\colon A \rightarrow X$ there is a
 $\kappa$-sequence of maps exhibiting $f$ as \relative $I$-cellular. 
\end{lemma}
\begin{proof}
Assume that $f$ is the transfinite composition of a $\lambda$-sequence $\{X_\alpha\}_{\alpha \leq \lambda}$ that exhibits the a cellular structure, \ie for $\alpha < \lambda$ there are pushout diagrams
\[\unskew\xymatrix{{S_\alpha}\ar^-{\sigma_\alpha}[r]\ar_{i_\alpha}[d]&{\colim\limits_{\beta<\alpha} X_\alpha}\ar^-{f_\alpha}[d]\\{D_\alpha}\ar[r]&{X_\alpha,}\pushout}\]
such that $i_\alpha$ is the identity of the initial object for $\alpha$ a limit ordinal, and a coproduct of maps in $I$ otherwise. For $\gamma \leq \kappa$, define sets $C^<_\gamma$ and $C_\gamma$ as well as commutative diagrams \labeleq{filterxy}{\unskew\xymatrix{{\colim\limits_{\delta < \gamma}X_\delta}\ar[r]\ar[d]&{\colim\limits_{\delta < \gamma}Y_\delta}\ar[d]\ar[dr]\\{X_\gamma}\ar@(dr,dl)[rr]\ar[r]&{Y_\gamma,}\ar[r]&{X}}}
{\color{white}by}\\
\noindent by transfinite induction: Let $C_0 \defas{} {0}$ and $Y_0 \defas{} X_0 = A$. Continuing, for $\mu$ a limit ordinal let $C_\mu$ be empty.

Otherwise define the set \[ C^<_{\gamma}\defas{} \{\alpha \leq \lambda, \sigma_\alpha\text{   factors through  } Y_{\gamma-1}\}.\]
Furthermore, let $C_{\gamma} \defas{} C^<_{\gamma} \setminus \bigcup_{\delta <\gamma} C^<_\delta$. Finally, define $Y_{\gamma}$ as the pushout
\[\unskew\xymatrix{{\coprod_{\alpha \in C_\gamma}S_\alpha}\ar[d]\ar[r]&{\colim\limits_{\delta < \gamma}Y_\delta}\ar[d]\\{\coprod_{\alpha \in C_\gamma} D_\alpha}\ar[r]&{Y_\gamma}\pushout}\]
Define a map $Y_\gamma \rightarrow X$ on the attached cells $S_\alpha \rightarrow D_\alpha$ by going through the $X_\alpha$. Note that $\sigma_\gamma\colon S_\gamma \rightarrow X$ factors through $\colim_{\delta < \gamma}X_\delta$, hence we get a map $X_\gamma \rightarrow Y_\gamma$ which fits into the diagram~\ref{filterxy}.  Finally, note that since $S_\alpha$ is $\kappa$-small, all attaching maps $\sigma_\alpha$ for $\alpha \leq \lambda$ factor through some $X_\gamma$, hence through $Y_\gamma$. the union $\bigcup_{\gamma \leq \kappa} C_\gamma$ contains all $\alpha \leq \lambda$. Therefore there are canonical maps in both directions between the colimits \[\colim\limits_{\alpha\leq \lambda}X_\alpha \cong \colim\limits_{\gamma \leq \kappa} Y_\gamma,\]
which are isomorphisms by cofinality.
\end{proof}
\begin{rem}\label{partialthuslinearindexing}
Note that attaching cells via coproducts, gives a partial order on the set of cells. Every such partially ordered set can be linearly ordered as in \cite[2.1.11]{H}, which corresponds to giving a $\lambda$-sequence in which the cells are attached one at a time. Lemma~\ref{kappaseq} gives us a much more convenient way to revert this process, than simply forgetting the extra information.
Returning to a closed symmetric monoidal category $(\catC, \smash)$,
observe that the coproduct distributes with the smash product, hence
also with the $\square$-product. We therefore allow ourselves to
switch freely between attaching cells one at a time or in bigger
groups via the coproduct. 
\end{rem}
\begin{proof}[Proof of Theorem~\ref{cellulartheorem}]
We assume without loss of generality (cf.~\ref{partialthuslinearindexing}) that $\lambda$ and $\mu$ are ordinals linearly indexing the cells of $f$ and $g$, respectively. That is for each $\alpha \leq \lambda$ we have a pushout diagram
\[\unskew\xymatrix{{S_\alpha}\ar[d]\ar^-{i_\alpha \in I}[r]&{D_\alpha}\ar[d]\\{\colim\limits_{\gamma < \alpha}X_{\gamma}}\ar_-{f_\alpha}[r]&{X_{\alpha},}\pushout}\]
such that the $\lambda$-sequence $A = X_0 \rightarrow X_\lambda = X$
is the map $f$, and and analogous for $g$. 
Let \(<\) be the partial order on $\lambda \times \mu$ with
$(\gamma,\delta) < (\alpha,\beta)$ if and only if $\gamma < \alpha$
and $\delta < \beta$. Let $E\colon \lambda\times\mu \rightarrow \catC$
be the sequence defined by the pushout diagrams
\[\xymatrix{{A\smash B}\ar[r]\ar[d]&{X_\alpha\smash B}\ar[d]\\{A\smash Y_\beta}\ar[r]&{E_{\alpha,\beta},}\pushout}\]
and note that $E_{\lambda,\mu}$ is the source of $f \square g$. We
claim that the desired filtration is then given by the
$\lambda\times\mu$-sequence $\{F_{\alpha,\beta}\}$ where
$F_{\alpha,\beta}$ is the pushout of the diagram:
\labeleq{defF}{\xymatrix{{E_{\alpha,\beta}}\ar[r]\ar[d]&{X_\alpha
      \smash
      Y_\beta}\ar[d]\\{E_{\lambda,\mu}}\ar[r]&{F_{\alpha,\beta},}\pushout}} 
where $E_{\lambda,\mu}$ is the constant sequence. 
To prove the claim, note that the transformation $E \rightarrow
X_{(-)}\smash Y_{(-)}$ of sequences factors through the sequence $\{P_{\alpha,\beta}\}$,
where $P_{\alpha,\beta}$ is the pushout 
\labeleq{defP}{\unskew\xymatrix{{\colim\limits_{\gamma<\alpha, \delta
        < \beta}X_\gamma \smash Y_\delta}\ar[r]\ar[d]&{\colim 
\limits_{\delta < \beta}X_\alpha \smash
Y_\delta}\ar[d]\\{\colim\limits_{\gamma<\alpha}X_\gamma \smash
Y_\beta}\ar[r]&{P_{\alpha,\beta}}\pushout\ar^-{f_\alpha\square
g_\beta}[r]&{X_\alpha \smash Y_\beta.}}} 
That is, $P_{\alpha,\beta}$ is the source of the map $f_\alpha \square
g_\beta$. We apply the cobase change as in Diagram~\ref{defF} to this
factorization to get the diagram 
\labeleq{cobchcool}{\xymatrix{{E}\ar[r]\ar[d]&{P}\ar^-{f_{(-)}\square
      g_{(-)}}[r]\ar[d]&{X_{(-)}\smash
      Y_{(-)}}\ar[d]\\{E_{\lambda,\mu}}\ar[r]&{P\amalg_E
      E_{\lambda,\mu}}\pushout\ar[r]&{F_{}.}\pushout}} 
Now comparing the colimits pointwise, cofinality lets us identify
$P_{\alpha, \beta}\amalg_{E_{\alpha,\beta}} E_{\lambda,\mu}$ as the
following pushout: 
\[{\unskew\xymatrix{{\colim\limits_{\gamma<\alpha, \delta < \beta}F_{\gamma,\delta}}\ar[r]\ar[d]&{\colim
\limits_{\delta <
  \beta}F_{\alpha,\delta}}\ar[d]\\{\colim\limits_{\gamma<\alpha}F_{\gamma,\beta}}\ar[r]&{P_{\alpha,\beta}\amalg_{E_{\alpha,\beta}}
E_{\lambda,\mu}}\pushout\ar[r]&{F_{\alpha,\beta}.}}}\] 
In particular we have
\[P\amalg_E E_{\lambda,\mu} \cong \colim\limits_{(\gamma,\delta)<(\alpha,\beta)}F_{\alpha,\beta},\] and the map \[\colim\limits_{(\gamma,\delta)<(\alpha,\beta)}F_{\gamma,\delta} \rightarrow F_{\alpha, \beta}\] is a cobase change of $f_\alpha \square g_\beta$. Therefore to show that $F$ is indeed a filtration by $I \square J$-cells, it suffices by Diagram~\ref{cobchcool} to show that $f_\alpha \square g_\beta$ is the attaching of a $I\square J$-cell.
This is a consequence of Lemma~\ref{levelwisesquare}, which implies that there are pushout diagrams
\[\xymatrix{{S_{(\alpha,\beta)}}\ar[d]\ar^-{i_\alpha \square j_\beta \in I\square J}[rrr]&&&{D_{(\alpha,\beta)}}\ar[d]\\{P_{(\alpha,\beta)}}\ar^-{f_\alpha\square g_\beta}[rrr]&&&{X_{\alpha}\smash Y_\beta.}\pushout}\]
Note that as in Remark~\ref{partialthuslinearindexing} we can extend the partial order on $\lambda \times \mu$ to a linear one, finishing the proof.
\end{proof}
\begin{rem}
  In many cases of interest, for example $\catC = \catT$, with $I$ and $J$ the sets of generating (acyclic) cofibrations, we will actually have that $I \square J \subset J(\cell)$, such that the above proposition also gives $f\square g$ the structure of a \relative $J$-cellular map.
\end{rem}
\begin{rem}\label{coliminclu}
 In categories where the maps in $I$, $J$ and $I\square J$ are
 inclusions, the
 intuition behind the filtration in the theorem simplifies
 significantly. In particular the cellular maps then give a
 $\lambda$-sequence of inclusions of subobjects \[A \hookrightarrow
 X_1 \hookrightarrow \ldots \hookrightarrow X,\] and similar for $B
 \hookrightarrow Y$. Note that $f\square g$ is the inclusion \[f
 \square g \colon X\smash B\cup_{A\smash B} A \smash Y = X\smash B
 \cup A\smash Y\hookrightarrow X\smash Y,\] and the filtration given
 by the theorem is through objects\[F_{\alpha,\beta} = X\smash B \cup
 X_\alpha \smash Y_\beta \cup A \smash Y.\] 
\end{rem}
\begin{cor}
 The monoidal product $X \smash Y$ of an $I$-cellular $X$ object with a $J$-cellular object $Y$ is $I \square J$-cellular.
\end{cor}
\begin{cor}\label{relativecells}
 In the situation of Theorem~\ref{cellulartheorem} let $K$ is the set
 of maps $I\smash B \cup A \smash J$. Then the map $f \smash g$ is
 \relative $(I\square J)\cup K$-cellular.  In particular, if $A$ and
 $B$ are themselves respectively $I$-cellular and  $J$-cellular,
 {then} $f \smash g$ is $I \square J$-cellular.   
\end{cor}
\begin{proof}
 Apply Theorem~\ref{cellulartheorem} on the maps $\star \rightarrow A
 \rightarrow X$ and $\star \rightarrow B \rightarrow Y$ which are
 respectively $I \cup \{\star\rightarrow A\}$-cellular and $J \cup
 \{\star \rightarrow B\}$-cellular. Note that the indexing of the
 filtrations is shifted, and the new filtration factors through
 $F_{1,1} = A \smash B$. All the later cells are then of type $(I
 \square J)\cup K$. 
\end{proof}

Note that since the $\square$-product is associative,
Theorem~\ref{cellulartheorem} immediately gives specific filtrations
for iterated $\square$-products of maps. The indexing set for the
cells of the iterated $\square$ is always given by the product of the
indexing sets with some (linear) order that is compatible with the
product partial order.

\section{Equivariant Cellular Filtrations of Smash
  Powers}\label{subsectcellularfiltrations}
\newcommand{\lie}{{Q}}

We finally give the equivariant cellular structure for smash powers of orthogonal $G$-spectra.
In this section $G$ is a fixed finite discrete group. 

Analyzing this structure we get that the geometric diagonal is an isomorphism for sufficiently cofibrant spectra.
Specializing and stripping the central result Theorem~\ref{cellular} of all bells and whistles, the important upshot is that if $L$ is $\Sp$-cellular, $G$ finite and $X$ a free $G$-set, then the $X$-fold smash product $L^{\smsh X}$ has a filtration by ``induced regular cells'', \ie as pushouts along maps of the form
$$G_+\smsh_H[\Gr{}{V}(S^{n-1}\times\Or V/P)_+]^{\Box H}\to G_+\smsh_H[\Gr{}{V}(S^{n-1}\times\Or V/P)_+]^{\smsh H}$$
for subgroups $H\subseteq G$ and $P\subseteq \Or V$.  The importance of such a filtration is driven home by Theorem~\ref{geomScofib} which shows that the geometric diagonal in that case becomes an isomorphism
$$L^{\smsh X_N}\cong \Phi^N(L^{\smsh X})$$
for all normal subgroups $N\subseteq G$.

Our intended applications are to smash powers of orthogonal ring spectra with a priori no group $Q$ acting on the incoming spectra. 
However, with our setup where orthogonal $Q$-spectra are just $Q$-object in orthogonal spectra, the difference between the two is taken care of by (enriched) functoriality in the incoming spectra, plus that the generating cofibrations -- of course -- must know about $Q$.  Since there are interesting applications of the statements we prove in this section (for instance in \cite{HHR}), we give the fuller generality as an illustration of the power of the methods.

Also, we are only intending to use the case when the incoming spectrum is $\Sp$-cofibrant, but in case one is interested in other applications we keep track of how smash powers modify the cell structure.
For this purpose, consider a Lie group $\lie$ and a $\lie$-mixing pair $(\hml,\gml)$ (with the base case being when $\lie$ is trivial and $(\gml,\hml)=(\AI_+,\All_+)$).

If $X$ is a finite $G$-set, then
the Lie group $\lie^{X} = \Map(X,\lie)$ has a $G$-action by
precomposition: \(gf(x) = f(g^{-1}x)\) for \(f \colon X \to \lie\),
\(g \in G\) and \(x \in X\).

 Let $L$ be a \(\gml\)-cellular orthogonal
  $\lie$-spectrum. 
For $G$ a finite group and $X$ a finite $G$-set, we
will specify a \(G \ltimes Q^X\)-mixing pair
$$(\hml(G,X),\gml(G,X))$$
and give a  filtration of the map  
\[(\star \rightarrow L)^{\square X} \cong (\star \rightarrow L^{\smash X})
\] 
by \(\gml(G,X)\)-cofibrations using Theorem~\ref{cellulartheorem}. 


Given \(V = \bigoplus_{x \in X} V_x\) and a subgroup \(P_x \subseteq Q
\times \Or{V_x}\) for each \(x \in 
X\), we consider \(P = \prod_{x \in X} P_x\) as a subgroup of \(Q^X
\times \Or V\). We let
\(\hml(G,X)\)\index{HGX@$\hml(G,X)$} be the smallest \(G \ltimes Q^X\)-typical family of
representations so that if every \(P_x\) is
in \(\hml\) and \(H\) is any subgroup of \(G\), then \(H \ltimes P
\subseteq G \ltimes (Q^X \times \Or V) = (G \ltimes Q^X) \times \Or
V\) is a member of \(\hml(G,X)\). (Here we use the trivial action
of \(G\) on \(V\). Later, when some of the factors \(V_x\) are
identical, we will also make use of the non-trivial actions of
subgroups of \(G\) on \(V\) given by permutation of identical
factors.)

Finally, we let \(\gml(G,X)\)\index{GGX@$\gml(G,X)$} be the smallest collection of families of subgroups
of \((G \ltimes Q^X)\times \Or V\) for \(V\) in \(\catOr\) so that
\((\hml(G,X),\gml(G,X))\) is a \(G \ltimes Q^X\)-mixing pair, and so that
if every \(P_x\) is
in \(\gml\) and \(H\) is any subgroup of \(G\), then \(H \ltimes P
\subseteq G \ltimes (Q^X \times \Or V) = (G \ltimes Q^X) \times \Or
V\) is a member of \(\gml(G,X)\). 

When $X$ is $G$-free it will follow from the construction that all the
attaching maps are  \(\gml(G,X)\)-cofibrations between induced
regular \(\gml(G,X)\)-cofibrant $G \ltimes \lie^X$-spectra. 
Our methods are inspired by \cite[3.10.1]{Kro}, where a similar filtration is given for the case that $X = G = C_q$ a finite cyclic group.

\begin{dfn}\label{dfn:smashcells}
  If $\lambda$ is an ordinal, and $X$ a $G$-set, we define a partial
order on the product
$\lambda^{X}$: for $\alpha =\{\alpha_{x}\}_{x\in X},
\beta=\{\beta_{x}\}_{x\in X}\in\lambda^X$ we say that $\beta\le \alpha $ if
for all $x\in X$ we have that 
$\beta_x \le \alpha_x$.
The group $G$ acts on $\lambda^X$ by letting
$(g\alpha)_x=\alpha_{g^{-1}x}$.  This induces a partial ordering on the set
$(\lambda^{X})_G$ of $G$-orbits by declaring for 
\(u,v \in (\lambda^{X})_G\) that $u\le v$ if there exist $\alpha\in u$
and $\beta\in v$ such that $\alpha\leq\beta$.

As in Remark~\ref{coliminclu}, since the maps in 
$\Gr{}{}I_{\gml}$ are levelwise inclusions,
we lose no generality by assuming that a given \relative
$\Gr{}{}I_{\gml}$-cell complex is an inclusion. 

Consider a 
$\lambda$-sequence  $K=L_0 \subseteq L_1 \subseteq \ldots \subseteq L$
of $\lie$-spectra, exhibiting $K\subseteq L$ as \relative
$\Gr{}{}I_{\gml}$-cell complex,
\ie if $a \in \lambda$ is a limit ordinal, then 
\(L_a = \bigcup_{b < a} L_b\), and if
$a \in \lambda$ is not a limit ordinal,
there is an object \(V_a\) of
\(\catOr\), a member \(P_a\) of  \(\gml^{V_a}\) 
and a
pushout diagram $\D_{a}$ of \(\lie\)-spectra 
\[
\xymatrix{
  {\Gr{}{V_a}
    (S^{n_a-1} \times 
    \faktor{(\lie \times \Or{V_a})}{P_a})_+}
  \ar[d]\ar^-{i_a}[r]&
  {\Gr{}{V_a}(D^{n_a}
    \times 
    \faktor{(\lie \times \Or{V_a})}{P_a})_+}
  \ar[d]\\
  {\bigcup\limits_{b <
      a}L_b}\ar[r]&{\,L_{a},}
  \pushout
}
\] 
with the inclusion $i_a$ an element in the set of generating cofibrations 
\(\Gr{}{}I_{\gml}\).
\end{dfn}
\begin{thm} \label{cellular}\label{relativeversion} 
Let $\lie$ be a compact Lie group and let 
$K\subseteq L$ be a \relative \(\gml\)-cellular inclusion of
orthogonal \(\lie\)-spectra with cells indexed by the ordinal $\lambda$. Let $G$ be a finite group and let
$X$ be a finite $G$-set.  Then the smash power
$$K^{\smash X}\subseteq L^{\smash X}$$ is a \relative
\(\Gr{}{}I_{\gml(G,X)}\)-cell complex with
$(\lambda^{X})_G$ indexing the 
\(\Gr{}{}I_{\gml(G,X)}\)-cells in the following sense:

If $\alpha \in \lambda^{X}$ with $G$-orbit \( u = [\alpha]\) and
stabilizer group $G_\alpha\subseteq G$ we let \(L^{\smash X}_{u}\) be the
${G \ltimes \lie^{X}}$-subspectrum of $L^{\smash \X}$ 
defined as the union
$$L^{\smash X}_{u} = \bigcup_{\beta \in u} L^{\smash
  X}_{\beta}$$
with $L^{\smash X}_\beta = \bigwedge_{x \in X}
L_{\beta_x}$
(that is, $L^{\smash X}_{u} = \bigcup_{g\in G}\bigwedge_{x\in X} L_{\alpha_{g^{-1}x}}$).
If \(\alpha_x\) is a limit
ordinal for some \(x \in X\), then 
\({L^{\smash \X}_{u}} = {\bigcup\limits_{v < u}L^{\smash \X}_{v}}\). Otherwise
there is a Euclidean
vector space \(V\), a closed
subgroup \(P\) of ${(G \ltimes \lie^{X}) \times \Or V}$
in \(\gml(G,X)^V\)and a pushout diagram 
of ${G \ltimes \lie^{X}}$-spectra of the form
\labeleq{tagtest}{
\xymatrix{ {\Gr{}{V}
    (S^{n-1} \times 
      \faktor{{((G \ltimes \lie^{X}) \times \Or V)}}{P})_+}
    \ar[r]^-{k_\alpha}_--\subseteq \ar[d]& 
    \Gr{}{V}(D^{n} \times
      \faktor{{((G \ltimes \lie^{X}) \times \Or V)}}{P})_+ \ar[d]\\
  {\bigcup\limits_{v < u}L^{\smash \X}_{v}}\ar[r]^\subseteq& {L^{\smash
      \X}_{u}}.\pushout}
}
Varying $u$, the inclusions induce an isomorphism
\[
  \colim_{u\in (\lambda^{X})_G} L^{\smash X}_u \cong \bigcup_{u\in (\lambda^{X})_G} L^{\smash X}_u = L^{\wedge X}.
\]
In particular, extending the partial order \((\lambda^X)_G\) to a
total order \(\lambda'\), the smash power \(K^{\wedge X} \subseteq
L^{\wedge X}\) is 
a \(\Gr{}{}I_{\gml(G,X)}\)-cell complex indexed by \(\lambda'\). 

 Furthermore:
\begin{enumerate}
\item
The $\square$-product $(K\subseteq L)^{\square X}$ is in this notation
given by $\bigcup_{u}L^{\smsh X}_u\subseteq L^{\smsh X}$, where $u$
varies over the orbits in $(\lambda^X)_G$ whose representatives take
the value $0$ at least once, and both the inclusions 
$K^{\smsh X}\subseteq\bigcup_{u}L^{\smsh X}_u\subseteq L^{\smsh X}$
are \relative \(\Gr{}{}I_{\gml(G,X)}\)-cell complexes.

\item If $X$ is a finite free $G$-set, then the top inclusion $k_\alpha$ of the
pushout diagram~\ref{tagtest} of $G \ltimes \lie^X$-spectra corresponds under the standard identification of iterated pushout products to one on the form
$G_+\smsh_{G_\alpha}(k^{\square G_\alpha})$ with $k\in \Gr{}{}I_{\gml}$.
In this form, the pushout diagram~\ref{tagtest} is referred to as $\D_u$.
\end{enumerate}
\end{thm}
\begin{proof}
Let $K=L_0 \subseteq L_1 \subseteq \ldots \subseteq L$ be the
$\lambda$-sequence exhibiting $L$ as \(\Gr{}{}I_{\gml}\)-cellular, using the
names $V_a$, $P_a$, $n_a$, $i_a$ and $\D_a$ for each $a\in \lambda$ as
in Definition~\ref{dfn:smashcells}.

Let \(\alpha \in \lambda^X\). 
If $\alpha_x$ is a limit ordinal for any $x \in X$, then 
$${L^{\smash \X}_{u}} = {\bigcup\limits_{v < u}L^{\smash \X}_{v}}.$$
Suppose now that none of the \(\alpha_x\) are limit ordinals.
 Taking the row-wise
$\square$-product over $x\in X$ of all the diagrams $\D_{\alpha_x}$ as
displayed just before the start of the theorem, we
get by Lemma~\ref{levelwisesquare} a pushout diagram of
\({G_\alpha \ltimes \lie^{X}}\)-spectra whose top map is
$\square_{x\in X}i_{\alpha_x}$, which we by  Proposition~\ref{semi-free-smash} may
identify with the following pushout diagram $\D_{\alpha}$ 
\[
\xymatrix{
{\Gr{}{V_{\alpha}}(S^{n_\alpha-1} \times \faktor{(\lie^X \times \Or{V_\alpha})}{P_\alpha})_+}\ar[d]\ar[r]^-{i_\alpha}_-{\subseteq}&
{\Gr{}{V_\alpha}(D^{n_\alpha} \times \faktor{(\lie^X \times \Or{V_\alpha})}{P_\alpha})_+}\ar[d]^{e_\alpha}\\
{\bigcup\limits_{\gamma < \alpha}L^{\wedge X}_{\gamma}}\ar[r]^-{\subseteq}&
{L^{\wedge X}_{\alpha},}\pushout}\]
where $V_\alpha = \bigoplus\limits_{x\in\X}\!V_{\alpha_x}$, $n_\alpha
= \sum\limits_{x\in\X}n_{\alpha_x}$ and $P_\alpha =
\pr{x\in\X}{P_{\alpha_x}}$ considered as a subgroup of \({\lie^{X}}
\times \Or{V_\alpha}\) via the inclusion 
\[
({\lie^{X}}
\times  \pr{x\in \X} \Or{V_{\alpha_x}}) \subseteq {\lie^{X}}
\times \Or{V_\alpha}  ,
\]
and where $G_\alpha$ acts by
permuting the $X$-coordinates on each of these ingredients of $\D_\alpha$.
Via the shear isomorphism \(G_\alpha \ltimes \Or {V_\alpha} \cong
G_\alpha \times \Or {V_\alpha}\) we change the \(G_\alpha\)-action so
that \(G_\alpha\) acts trivially on \(\Or {V_\alpha}\), and so that
the above diagram is a pushout of \((G_\alpha \ltimes Q^X) \times \Or
{V_\alpha}\)-spaces.

If $\alpha\in\lambda^X$ and $g\in G$,
then conjugation yields an isomorphism $c_g\colon G_\alpha\cong G_{g\alpha}$,
$c_g(h)= ghg^{-1}$.  
Furthermore, $\alpha\mapsto\D_\alpha$ is natural in the sense that acting by $g$ (\ie permuting the coordinates in each of
the vertices of the square) defines
an isomorphism between $\D_\alpha$ and $\D_{g\alpha}$, and even a
$G_\alpha$-isomorphism $\D_g\colon\D_\alpha\cong c^*_g\D_{g\alpha}$.

Fixing the orbit $u$ of \(\alpha\), we write $\D^u_\alpha$ for what we just called
$\D_\alpha$, and for $\gamma \in \lambda^{X}$ with $[\gamma]< [\alpha]
= u$ we let $\D^u_{\gamma}$ be the pushout diagram
\[
\xymatrix{
pt \ar[d]\ar[r]& pt \ar[d]\\
L^{\wedge X}_\gamma \ar[r]& L^{\wedge X}_\gamma.
\pushout
}\]
The assignment $\gamma \mapsto \D_\gamma^u$ is functorial on the
partially ordered set consisting of those $\gamma \in \lambda^{X}$
with $[\gamma] \le u$. Taking the colimit of the pushout diagrams
$\D_\gamma^u$ we obtain a pushout diagram 
 \[
 \xymatrix{
 {\bigvee_{\beta \in u} 
  \left[\Gr{}{V_\beta}(S^{n_\beta-1}\times 
\faktor{(\lie^X \times \Or{V_\beta} )}{P_\beta})_+\right]}
\ar[d]\ar[r]^{\vee i_\beta}&
 {\bigvee_{\beta \in u} \left[\Gr{}{V_\beta}(D^{n_\beta} \times      
\faktor{(\lie^X \times \Or{V_\beta})}{P_\beta})_+\right]}\ar[d]^{\sum
   e_\beta}\\
 {\bigcup\limits_{v<u}L^{\smash \X}_{v}}\ar[r]&{L^{\smash
       \X}_{u}}
 \pushout
}\]
of $G \ltimes \lie^X$-spectra where $g \in G$ acts by taking factors of the
form \(S^{n_\beta - 1}\) and \(D^{n_\beta}\) to 
\(S^{n_{g \beta} - 1}\) and \(D^{n_{g \beta}}\) respectively and by
the permutation action on \(Q^X\).

Fixing one $\alpha$ in the orbit $u$ and using the isomorphisms
translating between the different $i_\beta$s and our $i_\alpha$ we can
rewrite this diagram on the form 
\[\xymatrix{
{G}_+ \smash_{{ G_\alpha}}{\left[\Gr{{G_\alpha}}{V_\alpha}(S^{n_\alpha-1}
    \times \faktor{(\lie^X \times \Or{V_\alpha})}{P_\alpha})_+\right]}\ar[r]^{k_\alpha}\ar[d]&
{G_+ \smash_{G_\alpha}\left[\Gr{G_\alpha}{V_\alpha}(D^{n_\alpha} \times \faktor{(\lie^X \times \Or{V_\alpha})}{P_\alpha})_+\right]}\ar[d]\\
{\bigcup\limits_{v < u}L^{\smash \X}_{v}}\ar[r]&
{L^{\smash \X}_{u},}\pushout}\]
where the top map is $k_\alpha=G_+\smsh_{G_\alpha}i_\alpha$.

We claim that the inclusion $k_\alpha$ is a \(\gml(G,X)\)-cofibration of
${G \ltimes \lie^X}$-spectra.  To see this consider the inclusion
$S^{n_\alpha-1}\subseteq D^{n_\alpha}$ of $G_\alpha$-spaces, where
\(G_\alpha\) acts by permutation of coordinates.  This is
a genuine $G_\alpha$-cofibration (as can be seen by viewing
$D^{n_\alpha}$ as a cone on the barycentric subdivision of an
$(n_\alpha-1)$-simplex), and so \(i_\alpha\) is a \relative cell
complex with cells of the form
$${\Gr{}{V_{\alpha}}((S^{m-1} \times G_\alpha/H_\alpha)\times
  \faktor{(\lie^X \times \Or{V_\alpha})}{P_\alpha})_+}\subseteq
{\Gr{}{V_\alpha}((D^{m}\times G_\alpha/H_\alpha)
 \times \faktor{(\lie^X \times \Or{V_\alpha})}{P_\alpha})_+}$$
for \(H_\alpha\) a subgroup of \(G_\alpha\), and thus \(k_\alpha\) is a
\relative cell complex with cells of the form 
$${\Gr{}{V_{\alpha}}((S^{m-1} \times G/H)\times
  \faktor{(\lie^X \times \Or{V_\alpha})}{P_\alpha})_+}\subseteq
{\Gr{}{V_\alpha}((D^{m}\times G/H)
 \times \faktor{(\lie^X \times \Or{V_\alpha})}{P_\alpha})_+}$$
for \(H\) a subgroup of \(G_\alpha\). Thus it suffices to show that this 
is a \({\Gr{}{} I_{\gml(G,X)}}\)-cell.  This is seen by first noting
that since \(P_\alpha\) is a \(G_\alpha\)-invariant subgroup of \(Q^X
\times \Or{V_\alpha}\),
there is an isomorphism 
$$G_\alpha/H\times\faktor{(\lie^X \times \Or{V_\alpha})}{P_\alpha}
\to 
\faktor{(G_\alpha \ltimes ({\lie^{X}} \times \Or{V_\alpha}))}{(H
  \ltimes P_\alpha)}
$$
of \((G_\alpha \ltimes Q^X) \times \Or {V_{\alpha}}\)-spaces
sending $(gH,(A,q)P_\alpha) $ to $(g,q,A)(H \ltimes P_\alpha)$. Here
we use 
that since \(G_\alpha\) acts trivially on \(\Or {V_\alpha}\), the groups
\((G_\alpha \ltimes Q^X) \times \Or {V_{\alpha}}\) and  \(G_\alpha
\ltimes (Q^X \times \Or {V_{\alpha}})\) are identical.
The projection 
\((G \ltimes Q^X) \times \Or {V_{\alpha}} \to G\)
induces an isomorphism between
\begin{displaymath}
  ((G \ltimes Q^X) \times \Or {V_{\alpha}}) \times_{(G_\alpha \ltimes Q^X)
    \times \Or {V_{\alpha}}} 
  (G_\alpha/H\times\faktor{(\lie^X \times \Or{V_\alpha})}{P_\alpha})
\end{displaymath}
and
\begin{displaymath}
  G \times_{G_\alpha} 
(G_\alpha/H\times\faktor{(\lie^X \times \Or{V_\alpha})}{P_\alpha})
 \cong 
G/H\times\faktor{(\lie^X \times \Or{V_\alpha})}{P_\alpha}
\end{displaymath}
of \((G \ltimes Q^X) \times \Or {V_{\alpha}}\)-spaces.
Thus inducing up the action of \((G_\alpha \ltimes Q^X)
    \times \Or {V_{\alpha}}\) to an action of \((G \ltimes Q^X)
    \times \Or {V_{\alpha}}\)
we obtain an isomorphism
$$G/H\times\faktor{(\lie^X \times \Or{V_\alpha})}{P_\alpha}
\to 
\faktor{(G \ltimes ({\lie^{X}} \times \Or{V_\alpha}))}{(H
  \ltimes P_\alpha)}
$$
of \((G \ltimes Q^X) \times \Or {V_{\alpha}}\)-spaces.
By design, the subgroup \(H \ltimes P_\alpha\) is in
\(\gml(G,X)\). We conclude that \(k_\alpha\) is a
\relative \(\Gr {}{} I_{\gml(G,X)}\)-cell complex.

The statement about the $\square$-product is immediate from the construction.

Lastly, we treat the special case when $X$ is free as a $G$-set.
Recall that the inclusion $i_\alpha$ in the diagram $\D_{\alpha}$ 
was given by the iterated $\square$-product 
$\square_{x\in X}i_{\alpha_x}$, with the $G_\alpha$-action permuting
$\square$-factors that are identical. Since $G$ acts freely on $X$, so
does $G_\alpha$. Therefore, choosing any system of representatives $R$
for the $G_\alpha$-orbits of $X$ and letting $k\in \Gr{}{} I_\gml$ be
the inclusion 
$$\Gr{\lie}{V_{R}}(S^{n_R-1} \times 
\faktor{(\lie \times \Or{V_R})}{P_R})_+\subseteq
{\Gr{\lie}{V_R}(D^{n_R} \times 
\faktor{(\lie \times \Or{V_R})}{P_R})_+},$$
where $V_R= \bigoplus\limits_{r\in R}\!V_r$, 
$n_R= \sum\limits_{r\in  R} n_r$ and $P_R= \pr{r\in R}{P_r}$, we 
 get a $G_\alpha$-equivariant
 isomorphism $i_\alpha\cong k^{\square G_\alpha}$.
\end{proof}

\begin{example}\label{dfn.smash-inducedcell}
Consider the case where $G=\Sigma_n$ acts on $X=\{1,\dots,n\}$ and
$K\subseteq L$ is obtained by attaching a single $i\colon s\to t\in
\Gr{}{} I_\gml$. 
Then we can think of the map $\star \rightarrow K$ as another single
cell as in Corollary~\ref{relativecells}, \ie think of
$\lambda=\{0<1<2\}$ with $L_0\subseteq L_1\subseteq L_2$
being $\ast\subseteq K\subseteq L$. 
Then all $\square$-product summands containing $*$ are trivial, so we
consider the $\alpha$s with only $1$ and $2$ as values. The
stabilizer group $G_\alpha$ is isomorphic to
$\Sigma_m\times\Sigma_{n-m}$ where $m=|\alpha^{-1}(2)|$, and we get a
pushout diagram of the form
\[
\xymatrix{
A\smsh K^{\smsh\alpha^{-1}(1)}\ar[rrr]^-{i^{\square\alpha^{-1}(2)}\smsh K^{\smsh\alpha^{-1}(1)}}\ar[d]&&&L^{\smsh\alpha^{-1}(2)}\smsh K^{\smsh\alpha^{-1}(1)}\ar[d]\\
{\bigcup\limits_{\gamma < \alpha}L^{\wedge X}_{\gamma}}\ar[rrr]^-{\subseteq}&&&
{L^{\wedge X}_{\alpha},}\pushout}\]
where $A$ is the source of $i^{\square\alpha^{-1}(2)}$.
Hence, the 
the cell
attached to reach $L^{\smash X}_{[\alpha]}$ is of the form
\[(\Sigma_n)_+ \smash_{\Sigma_m \times
  \Sigma_{n-m}}i^{\square m} \smash K^{\smash n-m}.\]
Varying $m$ from $1$ to $n$ we run through all the conjugacy
classes $[\alpha]$ of such $\alpha$s and get the $n$ cells needed to build $L^{\smsh X}$ from $K^{\smsh X}$.
\end{example}

Together with Theorem~\ref{geomfixedofinduced}, we can now calculate the geometric fixed points of smash powers of $\Sp$-cofibrant spectra. 
\begin{theorem}\label{geomScofib}
  Let
\[1 \rightarrow N \rightarrow \G \stackrel{\epsilon}{\rightarrow} \J \rightarrow 1,\]
be a short exact sequence of finite discrete groups.
 If $L$ is an $\Sp$-cofibrant orthogonal $Q$-spectrum, with $Q$ a compact Lie group
 and $X$ a finite
  free $G$-set, then the geometric diagonal map
  from Definition~\ref{definitionofgeomfixedpoints} is an 
  isomorphism of $(G\ltimes Q^X)/N$-spectra
  \[\Delta \colon L^{\wedge X_N} \xto \cong \Phi^{N}(L^{\wedge X}).\]
\end{theorem}
\begin{proof}
This is again quite similar in spirit to 
\cite[3.10.7]{Kro}. It suffices to look at $\Sp I$-cellular $L$, since
retracts are preserved by any functor. We keep the notation from the
proof of Theorem~\ref{cellular} as far as possible. We let \(\sigma\)
be the set of \(G\)-orbits of \(\lambda^{\times X}\) and we let
$\sigma'$ be the $J$-orbits of $\lambda^{\times X_N}$. The projection
  $X\to  X_N$  induces a ``diagonal'' \(G\)-map  
\begin{eqnarray*}\epsilon^* \colon \,\,\,\, \lambda^{\times X_N}
  &\rightarrow& \lambda^{\times X},\\ 
\{\kappa_{[x]}\}_{[x]\in X_N}&\mapsto& \{\kappa_{[x]}\}_{x\in X}
\end{eqnarray*}
which descends to orbits to give a map $\epsilon^*\colon \sigma'
\rightarrow \sigma$. We show that, in the cell structure for $L^{\smash X}$, the cells that are not indexed by $\epsilon^*[\sigma']$ do not
contribute to the geometric fixed points. By induction over the
cellular filtration, assume that for all $[\beta] < [\alpha]$ in
$\sigma$ we have  
\[\Phi^N({\bigcup_{{{[\delta] \leq [\beta]}}} L^{\smash X}_{[\delta]}}) \cong \Phi^N({\bigcup_{[\epsilon^* \kappa]\leq \beta} L^{\smash X}_{[\epsilon^* \kappa]}}).\]
Then we claim that the same is true for $\beta$ replaced by $\alpha$:
In the case $\alpha \in \epsilon^*\sigma'$ there is nothing to do,
otherwise note that for $[\alpha] \not\in \epsilon^*\sigma'$, the
group $N$ is not contained in the stabilizer group of $\alpha$, hence in the attaching
diagram $\D_{[\alpha]}$ (\ie diagram~\ref{tagtest} in the form of specified for free $G$-set $X$),
the top row has trivial geometric fixed points
by Theorem~\ref{geomfixedofinduced}.

Since taking geometric fixed points commutes with the cell-complex
construction, we therefore get colimit diagrams for $\Phi^N{L^{\smash
    X}}$ and $L^{\smash X_N}$ of exactly the same shape connected by the geometric diagonal.
By induction on \([\kappa] \in \lambda^{X_N}\) we show that the
geometric diagonal $\Delta\colon L^{\smash X_N}_{[\kappa]}\to\Phi^N(L^{\smash
  X}_{\epsilon^*{[\kappa]}})$ is an isomorphism of $(G\ltimes Q^X)/N$-spectra.  Letting, for typographical reasons, the inclusion 
$$Y=(S^{n-1}\times(Q\times\Or V)/P)_+\subseteq Z=(D^n\times(Q\times\Or V)/P)_+$$ induce the last cell, we get that the geometric diagonal $L^{\smash X_N}_{[\kappa]}\to\Phi^N(L^{\smash
  X}_{\epsilon^*{[\kappa]}})$ is is the pushout of the map of spans
$$\xymatrix{{\bigcup\limits_{[\gamma] < [\kappa]} (L^{\smash \X_N})_{[\gamma]}}\ar[d]&
{\J_+\smash_{\J_0} [\Gr{Q}{V}Y]^{\square{\J_0}}}\ar[l]\ar[r]\ar[d]&
{\J_+\smash_{\J_0} [\Gr{Q}{V}Z]^{\smsh{\J_0}}}\ar[d]\\
{\bigcup\limits_{[\gamma] < [\kappa]} \Phi^N(L^{\smash \X})_{[\epsilon^*\gamma]}}&
{\Phi^N\left(\G_+\smash_{\h}[\Gr{Q}{V}Y]^{\square\h} \right)}\ar[l]\ar[r]&
{\Phi^N\left(\G_+\smash_{\h}[\Gr{Q}{V}Z]^{\smsh\h} \right)}}
$$
also given by the geometric diagonal.  Since the leftmost vertical map is assumed to be an isomorphism by induction and the two others are isomorphism by Theorem~\ref{geomfixedofinduced} we are done.
\end{proof}

\begin{cor}\label{cor:phistrongmon}
 Let $L$ and $L'$ be $\Sp$-cofibrant orthogonal spectra, then the lax monoidal structure map
\[\alpha\colon \Phi^N(\bigwedge_{X}L) \smash \Phi^N(\bigwedge_{X} L') \to \Phi^N (\bigwedge_X L \smash \bigwedge_X L')\]
from Proposition~\ref{philaxmonoidal} is an isomorphism.
\end{cor}
\begin{proof}
Combine Lemma~\ref{lem:geomdiagismonoidal} and Theorem~\ref{geomScofib}.
\end{proof}
\begin{remark}
  Naturality of the isomorphism $\Delta$ for all
morphisms between $\Sp$-cofibrant orthogonal spectra and for all
isomorphisms of finite free $G$-sets will follow Proposition~\ref{diagzigzag}. 
\end{remark}

\section{Commutative Ring Spectra  
}\label{sec:eqcomconv}\label{eqexttocomring}

In this section, we define a convenient model structure on the category of commutative
orthogonal ring \(G\)-spectra. We  work in the \(\Sp\)-model
structure on \(\catGOS\) of Definition~\ref{Spmodelstructure} with $\Sp I$ ($\Sp J$) as the set of generating (acyclic) cofibrations. 

\begin{dfn}
  A map of commutative orthogonal ring $G$-spectra is a {\em stable equivalence}\index{stable!equivalence!of commutative orthogonal ring $G$-spectra} if the underlying map of orthogonal $G$-spectra is a stable equivalence.
\end{dfn}

\label{eqexttocommring}

Given \(i \ge 1\) we define
\(E_G \Sigma_i\)\index{EGSigma@$E_G \Sigma_i$} to be the equivariant universal space 
 for the family of
subgroups of \(G \times \Sigma_i\) intersecting \(\Sigma_i\) trivially (\ie given a subgroup $H\subseteq G\times\Sigma_i$, the fixed point space $(E_G\Sigma_i)^H$ is contractible if $H$ intersects  \(1\times\Sigma_i\) trivially, and otherwise the fixed point space is empty).
\begin{lemma}\label{eqfreeup}
 Let $Y$ be an orthogonal $G$-spectrum and consider the semi-free orthogonal \(G\)-spectrum
 $X=\Gr{}{V}
   \left(({G \times \Or V}/{P})_+ \smash K
   \right)$ from Definition~\ref{semifreespectra}, 
 for $K$ a based $CW$-complex, $V \ne 0$ in \(\catOr\) and \(P
 \subseteq G\times \Or V\).
 Then the map \[q \colon \left(E_G{\Sigma_i}_+
   \smash_{\Sigma_i} X^{\smash i} \right)\smash Y \rightarrow
 X^{\smash i}_{\Sigma_i}\smash Y\] 
 induced by \(E_G{\Sigma_i}_+ \to S^0\) is a $\pi_*$-isomorphism. Indeed,
 if \(W\) in \(\catOr\) is of dimension at least
 \(i\) times the dimension of \(V\), then the map \(q_W\) is a
 \(G\)-equivalence. 
\end{lemma}
\begin{proof}
 We prove that $q$ is a $\G$-equivalence in all levels $W =
 \ds{i}{V}\oplus V'$. 
 Let
 $$\Gamma \defas{} \G \times
{(\Sigma_i \ltimes P^{\times i}) \times
  \Or {V'}}$$ and let \(p \colon \Gamma \to G \times \Sigma_i\) be the
projection onto the first and second factor. The space 
\(p^* E_G\Sigma_{i+}\)
is $\Gamma$-homotopy
equivalent to $E\fml$, where $\fml$ is the family of subgroups 
 of \(\Gamma\) whose image under \(p \colon \Gamma \to G \times \Sigma_i\) intersects
\(\Sigma_i\) trivially.
Since \(V \ne 0\), the homomorphism 
\({(\Sigma_i \ltimes \Or {V}^{\times i}) \to \Or {V^{\oplus i}}}\)
sending \((\sigma,(\alpha_1,\dots,\alpha_i))\) to \((\alpha_1
\oplus \dots \oplus \alpha_i)\sigma\) is
injective. We let
\(\Sigma_i \ltimes (G \times \Or {V})^{\times  i}\) 
act on \(G^{\times i} \times \Or {V^{\oplus i}}\) via this
homomorphism on the second factor and via the standard action of
\(\Sigma_i \ltimes 
G^{\times i}\) on \(G^{\times i}\) on the second factor. Observe that this action is free.
In particular \(G^{\times i} \times \Or W\) is a free \({(\Sigma_i \ltimes P^{\times i}) \times \Or
  {V'}}\)-space. This means that its \(\Gamma\)-isotropy groups
intersect \({(\Sigma_i \ltimes P^{\times i}) \times \Or
  {V'}}\) trivially, and the image under \(p \colon \Gamma \to G
\times \Sigma_i\) of these isotropy groups intersect \(\Sigma_i\)
trivially. Thus by Theorem~\ref{illtriangmfd} \(G^{\times i} \times \Or W\) is
an \(\fml\)-complex.

Evaluating at $W$ and spelling out the smash power of the semi-free spectrum, the map
 $q_W$ may be identified as the
 \(\Sigma_i \times \Or {V'}\)-orbits of  the \(\Sigma_i \times \Or {V'} \times G \times \Or W\)-map 
 \newcommand{\Jother}{(E_G\Sigma_{i+}\to S^0)}
\[\Jother \wedge {\Or{W}}_+ \smash _{\catOr_{V}^{\times i}}
    \left({\faktor {(G\times \Or V)}{P} }\right)_+^{\wedge i}
    \smash K^{\smash i} \smash Y_{V'}.\]
  The redistribution
  $ (\G^{\times i} \times \Or {W})_+ \wedge_{P ^{\times i}}K^{\smash i}\smash Y_{V'}\cong{\Or{W}}_+ \smash _{\catOr_{V}^{\times i}}\left({\faktor {(G\times \Or V)}{P} }\right)_+^{\wedge i}
     \smash K^{\smash i} \smash Y_{V'}$
allows us to further identify $q_W$ with the \(\Sigma_i \times \Or {V'}\)-orbits of 
  \[
\Jother{} \wedge(\G^{\times i} \times \Or {W})_+ \wedge_{P ^{\times i}}K^{\smash i}\smash Y_{V'}.
\] 
Using this identification, we can rewrite \(q_W\) as 
\[
p^*\Jother{} \wedge
(\G^{\times i} 
\times 
\Or {W})_+ 
\wedge_{(\Sigma_i \ltimes P^{\times i}) \times \Or {V'}}
K^{\smash i} \smash Y_{V'}.\]
Since
\(G^{\times i} \times \Or W\) is an \(\fml\)-complex, the map
\(E\fml_+ \wedge G^{\times i} \times \Or W \to S^0 \wedge G^{\times i}
\times \Or W\) is a \(\Gamma\)-homotopy equivalence, and thus \(q_W\)
is a \(G\)-homotopy equivalence.
\end{proof}

\begin{dfn}\label{def:freecom}
  The functor \(\Comm \colon \catGOS \to \Sp_{\catcAlg}\)\index{E@$\Comm$, free commutative algebra functor} from
  orthogonal \(G\)-spectra to commutative orthogonal \(G\)-ring
  spectra is left adjoint to the forgetful functor. Explicitly
  \(\Comm Y = \bigvee_{i} Y^{\smash i}_{\Sigma_i}\).
\end{dfn}
\begin{prop}\label{eqconven}
  If $i$ is a generating cofibration in the $\Sp$-model structure and $k>0$, then both the source and target of the \(\Sigma_k\)-orbits of the \(k\)-fold smash power $i^{\smsh k}_{\Sigma_k}$ are $\Sp$-cofibrant.  

In more detail, let \(V \ne 0\), \(k > 0\) and $Y = E_+ \wedge \Gr{}{V}((G \times \Or V/P)_+)$ where $E$ is either a disk or a sphere. Then the
\(\Sigma_k\)-orbit 
spectrum $Y^{\wedge k}_{\Sigma_k}$ of the \(k\)-fold smash power is
\(\Sp\)-cofibrant.  Consequently, the inclusion $\Sp \rightarrow \Comm Y$ is an $\Sp$-cofibration. 
\end{prop}
\begin{proof}
Considered as a \(\Sigma_k \times G\)-spectrum, the $k$-fold smash power
of $Y$ is isomorphic to the spectrum
\begin{displaymath}
  \Gr{}{V^{\oplus k}}((E^{\times k} \times G^{\times k} \times \Or{V^{\oplus k}}/
  P^{\times k})_+),
\end{displaymath}
where \(\Sigma_k \times G \times \Or{V^{\oplus k}}\) acts on the space  
\(K = E^{\times k} \times G^{\times k} \times \Or{V^{\oplus k}}/ 
P^{\times k}\) as follows:
The symmetric group \(\Sigma_k\) acts by permuting coordinates and
summands of \(V\). The group \(G\) acts
through the diagonal inclusion \(G \to
  G^{\times k}\) and 
\(\Or{V^{\oplus k}}\) acts by multiplication from the left.
Note that \(K\) is the \(P^{\times k}\)-orbit space of the 
\((\Sigma_k \ltimes P^{\times k}) \times G \times \Or{V^{\oplus
    k}}\)-manifold
\(M = E^{\times k} \times G^{\times k}
\times \Or{V^{\oplus k}}\).

By Illman's Theorem~\ref{illtriangmfd} the manifold \(M\)
is a 
\((\Sigma_k \ltimes G^{\times k}) \times G \times \Or{V^{\oplus k}}\)-CW-complex, and thus
  the orbit space  
  \(N = (E^{\times k} \times G^{\times k} \times \Or{V^{\oplus k}}/P^{\times k})/\Sigma_k\) is a 
  \(G \times \Or{V^{\oplus
      k}}\)-CW-complex.
Since \(V^{\oplus k} \ne 0\), the spectrum
\(Y^{\wedge k}_{\Sigma_k} \cong \Gr{}{V^{\oplus k}} N\) is
\(\Sp\)-cofibrant.
\end{proof}
\begin{cor}\label{eqfreeup2}
For every $\Sp$-cofibrant orthogonal spectrum \(X\) the quotient map 
\[q: E_G{\Sigma_i}_+ \smash_{\Sigma_i} X^{\smash i}  \rightarrow X^{\smash i}_{\Sigma_i}\]
is a $\pi_*$-isomorphism.
\end{cor}
\begin{proof}
We proceed by induction on $i$ and a $
G$-equivariant
cellular filtration of $X
$.  For the induction start $i=1$,
the statement is  trivially true. Hence let us assume it holds for all
$j < i$ and for 
an $\Sp$-cofibrant spectrum \(A\) such
that $X$ is built from $A$ by attaching a single cell
$\Gr{}{W}\left[(S^{n-1} \subseteq D^n)_+ \smash ({\scriptsize\faktor{G \times \Or{W}}{P}})_+ \right]$ with $W\neq 0$. Then, as explained in Remark~\ref{dfn.smash-inducedcell}, Theorem~\ref{cellular} states that $X^{\smash i}$ is built from $A^{\smash i}$ by attaching induced cells of the form 
$${\Sigma_{i+} \smash_{\Sigma_j \times \Sigma_{i-j}}
  \left(\Gr{}{W}(S^{n-1} \subseteq D^n)_+ \smash ({\scriptsize\faktor{G
        \times \Or{W}}{P}})_+ \right)^{\square j} \smash A^{\smash {(i-j)}},}
$$
with cofiber $\Sigma_{i+} \smash_{\Sigma_j \times \Sigma_{i-j}}
\left(\Gr{}{W}S^n_+ \smash ({\scriptsize\faktor{G \times \Or{W}}{P}})_+ \right)^{\smash j} \smash A^{\smash {(i-j)}}$. 
For ease of notation, we let \(Y =  S^n_+ \smash \Gr{}{W}
({\scriptsize\faktor{G \times \Or{W}}{P}})_+\). Thus,
for the induction step, it suffices to show that the projection
\[E_G\Sigma_{i+}
\smash_{\Sigma_j \times \Sigma_{i-j}} Y^{\smash j} \smash A^{\smash
  {(i-j)}} \rightarrow \left[Y^{\smash j} \smash A^{\smash
    {(i-j)}}\right]_{\Sigma_j \times \Sigma_{i-j}}\] is a
$\pi_*$-isomorphism. We use that $E_G\Sigma_i$ is $\Sigma_j \times
\Sigma_{i-j}$-equivariantly homotopy equivalent to $E_G\Sigma_j \times
E_G\Sigma_{i-j}$ to
reduce the problem to showing that the composite 
\[\xymatrix{
  {{E_G\Sigma_{j+} \smash_{\Sigma_j} Y^{\smash j} \smash
      E_G\Sigma_{i-j+}\smash_{\Sigma_{i-j}} A^{\smash
        i-j}}}\ar^-{{q}\smash \id}[d]\\
{{Y^{\smash j}_{\Sigma_j} \smash E_G\Sigma_{i-j+}\smash_{\Sigma_{i-j}} A^{\smash i-j}}}\ar^-{\id \smash {q}\smash \id}[d]\\{{Y^{\smash j}_{\Sigma_j} \smash A_{\Sigma_{i-j}}^{\smash i-j}}}}\]
(where the maps marked $q$ are induced by $E_G\Sigma_n\to*$) is a $\pi_*$-isomorphism.
Here the first map is a level homotopy equivalence by
Lemma~\ref{eqfreeup}. For \(j>0\) the second map is a
$\pi_*$-isomorphism by the induction hypothesis 
using that  $Y^{\smash j}_{\Sigma_j}$ is $\Sp$-cofibrant.
For \(j=0\) the second map is a $\pi_*$-isomorphism by assumption on \(A\).
\end{proof}
\begin{cor}
Let $V$ be a non-trivial Euclidean space, $K$ a $G$-manifold,  \(P\) a subgroup of \(G\times \Or V\)  and let  $Y = \Gr{}{V}((G \times \Or V/P)_+\smsh K_+)$.
Then the functor $\Comm Y \smash (-)$ on orthogonal spectra preserves
stable equivalences.
\end{cor}
\begin{proof}
Given a stable equivalence of orthogonal spectra \(f:Z\to Z'\) it is
enough to show that each summand \(Y^{\smsh i}_{\Sigma_i}\smsh f\) is a stable
equivalence.  By Lemma~\ref{eqfreeup} it is enough to show that
\(E_G{\Sigma_i}_+\smsh_{\Sigma_i}Y^{\smsh i}\smsh f\) is a stable
equivalence.  Since \(E_G\Sigma_i\) is a free \(\Sigma_i\)-cell complex it is
enough to show that  \(Y^{\smsh i}\smsh f\) is a stable equivalence, which
is a consequence of Proposition~\ref{Gprespistar2}. 
\end{proof}

\begin{cor}
The functor $\Comm \colon \catGOS \to \Sp_{\catcAlg}$ preserves stable equivalences between
$\Sp$-cofibrant orthogonal spectra. In particular, each map in $\Comm\Sp J$ is a stable equivalence.
\end{cor}
\begin{proof}
Iterated use of the pushout product axiom (cf.{} Definition~\ref{generalizedpushoutproductaxiom}) for the $\Sp$-model
structure of Definition~\ref{Spmodelstructure} implies that the $i$-fold smash power of an
acyclic cofibration between $\Sp$-cofibrant spectra is an acyclic
cofibration. Both ${E_G\Sigma_i}_+ \smash_{\Sigma_i} (-)$ and taking
wedges preserve $\pi_*$-isomorphisms, that is, stable equivalences. Hence Ken Brown's Lemma gives the result.
\end{proof}
We will need to calculate realizations of simplicial objects and some other specific colimits in the category of commutative orthogonal $G$-ring spectra. 
\begin{lemma}[{\cite[15.11]{MMSS}}]
  Let $R_0\to R_1\to R_1\to\dots $ 
  be a sequence of maps of commutative
orthogonal ring \(G\)-spectra that are Hurewicz cofibrations of orthogonal
\(G\)-spectra. Then the underlying orthogonal \(G\)-spectrum of the colimit
of the sequence in commutative orthogonal ring \(G\)-spectra is the colimit
of the sequence computed in the category of orthogonal \(G\)-spectra.  
\end{lemma}
The following Proposition is inspired by Lemma 15.9 in \cite{MMSS}. It deals with the other part of the cofibration hypothesis for $\Comm \Sp I$ and brings us closer to the convenience property~\ref{convenient}.
\begin{prop}\label{barprop}
Let $f\colon X \rightarrow Y$ be a wedge of maps in $\Sp I$ and let
$\Comm X \rightarrow R$ be a map of commutative orthogonal ring
$G$-spectra. Then, considered as a map in \(\catGOS\), the cobase change $j\colon R \rightarrow R \smash_{\Comm X} \Comm Y$ is a Hurewicz cofibration. If in addition, smashing with $R$  preserves $\Sp$-cofibrations, $j$ is even an $\Sp$-cofibration.
\end{prop}
\begin{proof}
Recall the two-sided bar construction of Definition~\ref{barconstr}.  To distinguish the two cases, we use roman bold fonts $\B_*(-,-,-)$ for the bar with respect to the smash product of commutative orthogonal ring $G$-spectra and normal fonts $B_*(-,-,-)$ for the bar with respect to wedge sum of spaces.  We let $\B(-,-,-)=|\B_*(-,-,-)|$ be the realization in the category of commutative orthogonal ring $G$-spectra and likewise we set $B(-,-,-)=|B_*(-,-,-)|$.  Note that the symmetric monoidal structures are the coproduct in both cases, so all objects are symmetric monoids.

As in the proof of {\cite[15.9]{MMSS}}, we identify the inclusion
$S^n_+ \rightarrow D^{n+1}_+$ of spaces with the realization of the
inclusion of the \(S^n_+\) summand in the $0$-simplices of the
two-sided bar construction 
$B_*(S^n_+,S^n_+,S^0)$.
 Its $q$-simplices are
given by a wedge of $q + 1$ copies of $S^n_+$ with $S^0$, with
degeneracy maps the inclusions of wedge summands and face maps induced
from folding maps $S^n_+ \vee S^n_+ \rightarrow S^n_+$, respectively
the collapse map $S^n_+ \vee S^0  \rightarrow S^0$ for the last face
in each simplicial level. Note that both the smash product with an
$\Or{V}$-orbit and the semi-free functors $\Gr{}{V}$ preserve colimits
and tensors, hence the simplicial realization. Thus we can express $f$
analogously as the realization of the inclusion of a summand of the $0$-simplices of the simplicial
orthogonal spectrum $B_*(X,X,T)$, where $X= \bigvee_i
\Gr{}{V_i}{S^{n_i}_+ \smash
  \left({\scriptsize\faktor{G \times \Or{V_i}}{P_i}}\right)_+}$ and $T =
\bigvee_i \Gr{}{V_i}{S^{0} \smash
  \left({\scriptsize\faktor{G \times \Or{V_i}}{P_i}}\right)_+}$.

Applying $\Comm$ takes coproducts to smash products, fold maps to
multiplication maps, and inclusions of the base points to unit maps of
commutative orthogonal $G$-ring spectra. It also preserves tensors over
$\catU$, \ie sends $X \smash A_+$ to $X \otimes A$,
and so we get canonical isomorphisms
$$\Comm Y\cong \Comm B(X,X,T)\cong\B(\Comm X,\Comm X,\Comm T).
$$
Furthermore, since we can compute geometric realizations
in terms of the underlying spectra ({Proposition~\ref{realizeinspec}, whose proof clearly does not rely on the result we are now proving}), and since
smashing with $R$ commutes with this realization, we can identify  
\labeleq{barident}{R \smash_{\Comm X} \Comm Y \cong R \smash_{\Comm X} \B(\Comm X, \Comm X, \Comm T) \cong \B(R, \Comm X, \Comm T).}
We look at $\B_*(R, \Comm X, \Comm T)$ in more detail: $R$ includes
into the $0$-simplices $R \smash \Comm T$ as a wedge summand, \ie via
a Hurewicz cofibration. All the other wedge summands are of the form 
\[R
\smash \left(\Gr{}{V_i}{S^0 \smash
    ({\scriptsize\faktor{G \times \Or{V_i}}{P_i}})_+}\right)^{\smash k 
}_{\Sigma_k},\] 
hence they are $\Sp$-cofibrant if smashing with $R$
preserves $\Sp$-cofibrations by Proposition~\ref{eqconven}.
Then in particular $R\smash \Comm T$ is $\Sp$-cofibrant and the inclusion of $R$ is an
$\Sp$-cofibration.

The degeneracy maps are given by inclusions \[R
\smash (\Comm X)^{\smash q} \smash \Comm T = R \smash (\Comm
X)^{\smash r} \smash \Sp \smash (\Comm X)^{\smash q-r} \smash \Comm T
\longrightarrow R \smash (\Comm X)^{\smash q+1} \smash \Comm T.\] 
Therefore the inclusion of degenerate simplices (\ref{degsimplices}) is in each level $q$ given by the map
\[R \smash (\Sp \rightarrow \Comm X)^{\square q+1} \smash \Comm T,\]
which is a Hurewicz cofibration because $\Sp \rightarrow \Comm X$ is an inclusion of a wedge summand. Furthermore Proposition~\ref{eqconven} states that $\Sp \rightarrow \Comm X$ is an $\Sp$-cofibration. Hence by the pushout product axiom, the inclusion of degenerate simplices is an $\Sp$-cofibration if smashing with $R$ preserves $\Sp$-cofibrations. Hence the bar construction is Hurewicz proper 
(Definition~\ref{dfn:Cproper} of $C$-proper with $C$ being the class of Hurewicz cofibrations), 
and even $\Sp$-proper for the stronger assumption on $R$. The result then follows using Proposition~\ref{cofproperfiltr}.
\end{proof}
\begin{rem}
Note that the sphere spectrum $\Sp$ is an important specific example of a ring spectrum that preserves
$\Sp$-cofibrations under the smash product. 
\end{rem}
\begin{cor}
The set of maps of commutative orthogonal ring $G$-spectra $\Comm \Sp I$
satisfies the cofibration hypothesis~\ref{cofhyp}. Since it consists of $\Comm \Sp
I$-cell complexes, so does $\Comm\Sp J$. 
\end{cor}
\begin{lemma}
Let $i\colon R \rightarrow R'$ be an $\Sp$-cofibration of commutative
orthogonal ring $G$-spectra. Then the functor $(-)\smash_R R'$ on
commutative $R$-algebras preserves stable equivalences.
\end{lemma}
\begin{proof}
We proceed as in {\cite[15.12]{MMSS}}.
Assume inductively that $i$ is a cobase change of a wedge of maps in
$\Comm \Sp I$. Then as in \ref{barident} we can identify
$(-)\smash_R R'$ with an appropriate $\B(-,\Comm X, \Comm T)$. This
functor preserves stable equivalences by Proposition~\ref{cofproperfiltreq},
since the bar construction is Hurewicz proper.    
\end{proof}
Finally we get the analog of {\cite[15.4]{MMSS}}, using the same proof as in the classical case (cf.~\cite[p. 490]{MMSS}):
\begin{prop}
Every \relative $\Comm \Sp J$-cell complex is a stable equivalence.
\end{prop}
This once more allows us to use Lemma \cite[2.3]{SS}, and we obtain the $\Sp$-model structure for commutative orthogonal ring $G$-spectra:
\begin{theorem} \label{modelstruct}\label{eqmodelstruct}
The underlying $\Sp$-fibrations and stable equivalences give
a compactly generated proper topological model structure on the
category of commutative orthogonal ring \(G\)-spectra. The generating
(acyclic) cofibrations are given by the sets $\Comm \Sp I$ and
$\Comm\Sp J$, respectively.\\Again the identity functor gives a
Quillen equivalence to the classical model structure from
\cite[15.1]{MMSS}.  
\end{theorem}
We call cofibrant objects in this model structure simply
$\Sp$-cofibrant, inspired by the following Theorem, which is implied
by the second statement of Proposition~\ref{barprop}, and provides the
main motivation for the constructions in this section: 
\begin{theorem} \label{convenient}\label{eqconvenient}
The $\Sp$-model structure on commutative orthogonal ring $G$-spectra is
``convenient''\index{convenient model structure}\index{model structure!convenient} in the sense that
\begin{quote}
  the forgetful map from commutative orthogonal ring $G$-spectra to orthogonal $G$-spectra preserves $\Sp$-cofibrations:
\end{quote}
if $A$ is a commutative orthogonal
ring \(G\)-spectrum that is $\Sp$-cofibrant, it is also 
$\Sp$-cofibrant as an orthogonal \(G\)-spectrum.  
\end{theorem}

Even slightly more is true:
\begin{theorem}
Let $f\colon R \rightarrow R'$ be a map of commutative orthogonal ring
\(G\)-spectra, that is a cofibration in the model structure of
Theorem~\ref{modelstruct}. If the smash product with $R$ preserves
$\Sp$-cofibrations of orthogonal \(G\)-spectra, then $f$ is an
$\Sp$-cofibration also as a map of orthogonal $G$-spectra. 
\end{theorem} 
\begin{proof}
Reduce to the case of a $\Comm \Sp I$-cell complex. Induction on the cellular filtration and the stronger second statement of Proposition~\ref{barprop} give the result.
\end{proof}
\begin{theorem} \label{eqRalg}\label{Ralg}
For a commutative orthogonal $G$-ring spectrum $R$, the $\Sp$-model structure induces a compactly generated proper topological model structure on commutative $R$-algebras. This $R$-model structure is convenient with respect to the $R$-model structure on $R$-modules from Theorem~\ref{eqRmodel}$(i)$.\\ The identity functor on commutative $R$-algebras induces a Quillen equivalence to the classical model structure one would get by applying \cite[3.10]{DS} to the structure from \cite[III.8.1]{MM}.
\end{theorem}
\begin{proof}
We can use \cite[3.10]{DS}. An analog of Theorem~\ref{eqconvenient} is then immediate, since the free commutative $R$-algebra functor $\Comm_R$ satisfies $\Comm_R(-) \cong R\smash \Comm(-)$, and thus any cofibration of commutative $R$-algebras is an underlying $R$-cofibration.
\end{proof}


\section{Realization and Tensor in Commutative 
  Ring Spectra}\label{subsecttensors}
In this section we prove some standard facts pertaining to realizations and tensors.  We have already had the the occasion to use Proposition~\ref{realizeinspec}.

The functor category $\catGOS$ is tensored
  and cotensored over $G\catT$. Since the functor adding a disjoint
  base point is monoidal, it is also tensored and cotensored over the
  category \(G\catU\). 

From \cite[VII 2.10]{EKMM} and \cite[II.7.2]{EKMM} we get the following
\begin{prop}
  Let $R$ be an orthogonal ring spectrum and let \(G\) be a compact
  Lie group.
  \begin{enumerate}
  \item The category of $R$-modules with action of \(G\) is topologically bicomplete with limits, colimits, tensors and cotensors calculated in the category $\catGOS$.
  \item If $R$ is commutative, 
    then the category of
    (commutative) $R$-algebras with action of \(G\) is topologically
    bicomplete with limits and cotensors created in $R$-modules with
    action of \(G\).
  \end{enumerate}
\end{prop}

For discrete spaces, tensors are easily computable:
\begin{prop}\label{disctensor}
Let $\catD$ be a category which is enriched and tensored over 
$G\catU$. Let $A$ be an object of $\catD$ and $X$ a {\em discrete} space in $G\catU$. Then the universal property of the coproduct induces a natural (in both $X$ and $A$) isomorphism
\[\coprod_X A\cong A\otimes X.\]
\end{prop}
In the case we are most interested in, using Lemma~\ref{commcoprod}, we get
\begin{cor}
Let $A$ be a commutative orthogonal ring spectrum, $X$ a discrete
space in \(G\catU\). 
Then the universal properties induce a  natural isomorphism 
\[\bigwedge_X A\cong A \otimes X\] between the tensor of $A$ with $X$ and the $X$-fold smash power of $A$, \ie the $X$-fold smash product of $A$ with itself.
\end{cor}
This is the main point of motivation for the translation from Hochschild homology.
  
Next we treat geometric realization, c.f.~Definition~\ref{geometricrealiz} for general background.  In particular, $|-|_\catC$ denotes the realization in a given category $\catC$, but for the special case when $\catC$ is the category of orthogonal $G$-spectra we drop the decoration and simply write $|-|$.

The below results hold for $\catC$ being either one of the categories of \(G\)-objects in
$R$-modules, $R$-algebras or commutative $R$-algebras, but for concreteness we write out only the case we actually need.  The proofs in \cite[VII.3.2]{EKMM} and  {\cite[VII.3.3]{EKMM}} can be transported without any loss.  In the following statement we let $\catDelta$ and $\catDelta_\catC$ be the singular functors adjoint to the realizations as in Definition~\ref{ralizationadj}. 
\begin{prop}\label{EKMMVII32}
Let $\catC$ be the category of commutative orthogonal $G$-ring spectra, let $A,B\in\catC$ and let $X_*$ be a simplicial set.
Then the definitions of tensor and realization induce a string of isomorphisms 
\begin{eqnarray*}
\catC(A \otimes_\catC |X_*|, B)&\cong&\catU(|X_*|,\catC(A,B))\\
&\cong&s\catU(X_*, \catDelta \catC(A,B))\\
&\cong&s\catU(X_*, \catC(A,\catDelta_\catC B))\\
&\cong&s\catC((A \otimes_\catC X)_*, \catDelta_\catC B)\\
&\cong&\catC(|(A\otimes_\catC X)_*|_\catC,B),
\end{eqnarray*} 
and so a natural isomorphism
\[A \otimes_\catC \vert X_*\vert \cong \vert A\otimes_\catC X_*\vert_\catC.\]
\end{prop}

Let $U$ be the forgetful functor from the category $\catC$ of commutative orthogonal $G$-ring spectra to orthogonal $G$-spectra and let $\Comm{}$ be the left adjoint.  Explicitly, $\Comm{}X=\bigvee_{n\geq0}(X^{\smsh i})_{\Sigma_i}$.  %
Since, in addition to commuting with all colimits, realization in $\catT$ is strong monoidal, we have a natural (in the simplicial orthogonal $G$-spectrum $X_*$) isomorphism 
$$\epsilon_{X_*}\colon\Comm{}|X_*|\cong |\Comm{}X_*|_\catC.$$
If $A_*$ is a simplicial commutative orthogonal ring $G$-spectrum, we let $\mu_{A_*}\colon\Comm UA_*\to A_*$ be the structure map (adjoint to the identity on $UA_*$).  Composing we get a natural map
$|\mu_{A_*}|_\catC\,\epsilon_{UA_*}\colon\Comm{}|UA_*|\to |A_*|_\catC$,
whose adjoint
$$\eta_{A_*}\colon |UA_*|\to U|A_*|_\catC$$
we now show is a natural isomorphism.
\begin{prop}\label{realizeinspec}
Let  
$A_*$ be a simplicial object of $\catC$, the category of commutative orthogonal $G$-ring spectra. Then $\eta_{A_*}\colon |UA_*|\to U|A_*|_\catC$ is an isomorphism,
  endowing $|UA_*|$ with the structure of a commutative orthogonal ring $G$-spectrum and an isomorphism to $|A_*|_\catC$. 
\end{prop}\label{comeback181120}
\begin{proof}
  Consider the map of reflexive coequalizers
$$\xymatrix{
  |UA_*|\ar[d]^{\eta_{A_*}}&&
  |UEUA_*|\ar[ll]_{|U\mu_{A_*}|}\ar[d]^{\eta_{\Comm UA_*}}&&
  |UEUEUA_*|\ar[d]^{\eta_{\Comm U\Comm UA_*}}\ar@<-1ex>[ll]_{|U\mu_{\Comm UA_*}|} \ar@<1ex>[ll]^{|U\Comm U\mu_{A_*}|}\\
  U|A_*|_\catC&&U|EUA_*|_\catC\ar[ll]_{U|\mu_{A_*}|_\catC}&&U|EUEUA_*|_\catC\ar@<-1ex>[ll]_{U|\mu_{\Comm UA_*}|_\catC} \ar@<1ex>[ll]^{U|\Comm U\mu_{A_*}|_\catC}
}$$
(realization and the forgetful map preserve these reflexive coequalizers).  Each of the diagrams commute
and the two rightmost vertical maps are isomorphisms by construction and so $\eta_{A_*}$ is an isomorphism too.
\end{proof}

\begin{remark}\label{simplicialtensor}
  With the structure provided by Proposition~\ref{realizeinspec} on $|UA_*|$ we have a natural isomorphism to $\vert A_*\vert_\catC$ of commutative orthogonal $G$-ring spectra and we will drop the subscript $\catC$ from the notation, $|A_*|=|A_*|_\catC$, and write the isomorphism of Proposition~\ref{EKMMVII32} as
  \[A \otimes_\catC \vert X_*\vert \cong \vert A\otimes_\catC X_*\vert.\]
\end{remark}

\subsubsection{Hochschild Homology} In \cite{HH}, Hochschild defined a cohomology theory for bimodules of algebras over a field, mirroring the definition of group homology in terms of the bar complex.  Hochschild cohomology and its dual, Hochschild homology,  have proven valuable tools in a wide variety of mathematical disciplines, spanning from algebra and topology to mathematical physics and functional analysis. 

For future comparison, we will very briefly recall some basics, not touching on most of the rich theory that follows. For further details, the reader should consult Loday's book \cite{LCH} which provides a very readable and thorough introduction.

We shall focus on the commutative setting, and will always assume that the coefficients are the commutative algebra itself. Let for the whole section $R$ denote a commutative unital ground ring, and let $\otimes=\otimes_R$ denote the tensor product over $R$.
\begin{dfn}\label{algHH}
Let $A$ be a commutative $R$-algebra. \emph{Hochschild homology of
  $A$}\index{Hochschild homology} is the homology $HH(A)$ of the simplicial commutative
\(R\)-algebra \(Z(A)\)
\[\small\xymatrix@=.5cm@R=.4cm@C=.1cm{
{}&{}&{}&{\vdots}\ar[d]&{}&{}&{}\\
{}&{A}&{\otimes}&{A}\ar@<-2ex>[d] \ar[d] \ar@<2ex>[d]&{\otimes}&{A}&{}\\
{}&{}&{A}&{\otimes}\ar@<-1ex>[d] \ar@<1ex>[d]
\ar@<-1ex>[u] \ar@<1ex>[u]&{A}&{}&{}\\
{}&{}&{}&{A} \ar[u] &{}&{}&{}\\}\]
with \(Z(A)_n = A^{\otimes n+1}\) and with face maps
$b_i \colon Z(A)_n \to Z(A)_{n-1}$ given by \[b_i(a_0, \ldots, a_n)\defas{}(a_0,\ldots,
a_i.a_{i+1},\ldots, a_n),\] for \(i < n\) and
\[b_n (a_0 , \ldots, a_n) \defas{}(a_n.a_0, a_1\ldots, a_{n-1}).\]
The degeneracy operators \(s_i \colon Z(A)_n \to Z(A)_{n+1}\) are given
by
$$s_i(a_0,\ldots,a_n) = (a_0,\dots, a_i,1,a_{i+1},\dots,a_n).$$
\end{dfn}
Loday realized that this definition could be seen as a special case of a functorial construction on $R$-algebras (cf.~\cite[4.2]{Lpap}):
\begin{dfn}\label{algloday}
Let $\catFin$ be the category of finite sets and $\catRcalg$ the category of commutative $R$-algebras. Define the \emph{algebraic Loday functor}\index{algebraic Loday Functor}\index{Loday Functor!algebraic }
\[{\bigotimes}_{(-)}(-)\colon \catFin \times \catRcalg \rightarrow \catRcalg,\]
on objects as the iterated tensor product
\begin{eqnarray*}
  {\bigotimes}_{X}(A),
\end{eqnarray*}
that is, taking \(X,A)\) to the coproduct of \(X\) copies of \(A\) in the category \(\catRcalg\).
The Loday functor takes a morphism $(f,\phi): (X,A)\rightarrow (Y,B)$
to the \(A\)-linear map given by
\[(f,\phi)(a_x)_{x\in X} \defas{} \left(\prod_{f(x)=y}{\phi(a_x)}\right)_{y\in Y},\]
where the product over an empty indexing set is to be understood as the unit in $R$.
\end{dfn}
\begin{rem}\label{nonfinremark}
Note that we really use that $A$ is a \emph{commutative} $R$-algebra when defining the functor from $\catFin$. If we would restrict to the category of finite sets and isomorphisms, the same formulas would give a functor with input $R$-modules. Extending to injective maps requires a unit map of some sort, and non injective but at least monotonous maps between ordered sets would only require an associative multiplication.
\end{rem}
It is this formula that we emulate in the topological setting of commutative orthogonal ring spectra, where we construct a continuous analog to the Loday functor. One application of this topological Loday functor, is a convenient definition of topological Hochschild homology in the same spirit as the following definition:\\
The algebraic Loday functor gives an explicit example of functors for which we can define Hochschild homology:
\begin{dfn}
Let $F \colon \catFin \rightarrow \catRAlg$ be a functor. Define its {\em Hochschild homology}\index{Hochschild homology} $HH(F)$ as the homology of the simplicial algebra
\[\Delta^{\op} \stackrel{S^1}{\rightarrow} \catFin \stackrel{F}{\rightarrow}\catRAlg.\]
\end{dfn}
Then immediately, inspection of the defining chain complex yields that Hochschild homology of commutative algebras is the same as Hochschild homology of the Loday functor:
\labeleq{hochschildloday}{HH(A) = HH(\bigotimes_{(-)}A).}

\subsubsection{The Loday Functor}
\label{sectlodayfunct}
\newcommand{\lodayR}[2]{{\bigwedge_{#1}}^R{#2}}


The coproduct in the category of commutative monoids in a symmetric monoidal category $(\cC,\otimes,e)$ is nothing but $\otimes$ itself (c.f.~Lemma~\ref{commcoprod}).  In particular, with the considerations in Section~\ref{tensors} 
we get

\begin{cor}\label{4.3.6}
  Let $(\cC,\otimes,e)$ be a symmetric monoidal category.  Then the category of commutative monoids in $\cC$ is tensored over the category of finite sets,  sending a commutative monoid $M$ and a finite set $S$ to $\bigotimes_SM=M^{\otimes S}$.  
\end{cor}
As discussed in Section~\ref{tensors}, $M^{\otimes\{1,\dots,n\}}=(\dots(M\otimes M)\otimes\dots)\otimes M$, while
the functoriality in $S$ given by sending a function of finite sets $f\colon S\to T$ to the composite 
$$\bigotimes_SM\cong \bigotimes_{t\in T}\bigotimes_{f^{-1}(t)}M\to \bigotimes_TM,$$
where the isomorphism is a (uniquely given by $f$) structure isomorphism in $(\cC,\otimes,e)$ and the morphism is the tensor of the maps $\bigotimes_{f^{-1}(t)}M\to M$ induced by multiplication (the unit if $f^{-1}(t)$ is empty).

In particular, if $R$ is a commutative orthogonal ring spectrum (\ie a commutative $\Sp$-algebra in $\catOS$), $A$ a commutative $R$-algebra, and $X$ a finite set, considered as a discrete space. Then Corollary~\ref{4.3.6} gives a natural isomorphism
$$A \otimes_{\catRcalg} X\cong  A\smash_R \ldots \smash_R A$$
($|X|$-fold smash).  Recall from Section~\ref{tensors} that the right hand side is functorial in bijections even if $A$ is only an $R$-module, and in injections if $A$ is an $R$-module under $R$.

Since tensors commute with colimits we get that if $X$ is a set considered as a discrete space, then $A \otimes_{\catRcalg} X$ is the filtered colimit of $S\mapsto A \otimes_{\catRcalg} S$, where $S$ varies over the finite subsets of $X$. A priori, this colimit is in $\catRcalg$, but is created in $\catRmod$.  Note that, since the colimit is over inclusions, this makes no use of the multiplication, and applies to any $R$-module under $R$.  

\begin{dfn}\label{Lodayfunctor}\label{lodaysets}\label{colimloday}\label{discretelodayfunctor}
The {\em discrete Loday functor}\index{discrete Loday functor}\index{Loday functor!discrete} 
applied to an $R$-module $M$ under $R$ and set $X$ is given by
$$\lodayR XM= \colim_{U\subseteq X}\lodayR UM$$
(functorial in injections).  The {\em Loday functor} applied to a commutative $R$-algebra $A$ and space $X$ is given by the tensor
$$\lodayR XA=A\otimes_{\catRcalg} X.$$
In either case, $\bigwedge_X^RA$ is also referred to as the {\em $X$-fold smash power of $A$ (over $R$)}\index{smash power}
\end{dfn}
As commented above, if $A$ is a commutative $R$-algebra and $X$ a set, then the two potential interpretations of the smash powers 
as modules under $R$ are naturally isomorphic.

Finally note that Proposition~\ref{EKMMVII32} implies the following lemma, which makes it possible to realize $\lodayR{|X|}A$ concretely as a ``smash power indexed over $X$'' for simplicial sets $X$:
\begin{lem}\label{simplicialloday}
 Given a simplicial space $Y$ and $A$ a commutative $R$-algebra, there is a natural isomorphism (using the convention of Remark~\ref{simplicialtensor})
 \[\lodayR{{|Y|}}(A) \cong |\lodayR{Y} (A)|,\]
 of commutative $R$-algebras.
\end{lem}

We will mostly be interested in the case where $R$ is the sphere spectrum $\Sp$, and for simplicity of notation we restrict ourselves to that case in the following and omit $R$ from the notation.

A similar construction has been carried out in in \cite[Section 4]{BCD}. The construction we give here is much simpler than the one presented in op.~cit., so it will be crucial to study its properties in detail, so as to make sure, that we indeed capture all the desired information. In particular the equivariant homotopy type, when $X$ is equipped with an action of some (compact Lie-) group $\G$, will require some care.
We admit that the choice in \cite{BCD} of $\Gamma$-spaces as a model for connective spectra was not optimal, but with very minor rewriting the paper makes sense when based on symmetric spectra or simplicial functors.  This rewriting is necessary for applying \cite{BCD} to commutative ring spectra that do not have strictly commutative models in $\Gamma$-spaces \cite{Tyler}.

\begin{rem}
For a morphism $\varphi: R \rightarrow S$ of commutative orthogonal ring spectra, we get adjoint pairs of functors between the categories of $R$-modules and  $S$-modules, and also for algebras and for commutative algebras. All the left adjoints are given by \emph{induction}, \ie using $S \smash_R (-)$, which is strong monoidal and preserves tensors, hence commutes with smash powers.
The respective right adjoint functors do in general not exhibit similar properties.
\end{rem}
Note that since all group actions are through isomorphisms, a (continuous) action of a (topological) group $G$ on $X$ induces (continuous) actions on the smash powers
by precomposition as follows 
\[{G}\to{\catA(X,X)}\to{\catC(\lodayR XA, \lodayR XA)},\]
where ($\catA$, $\catC$) is any of the pairs of categories for which we defined the 
$A \in \catC$ and $(\catA$, $\catC)$ is any of the pairs of categories for which we defined the smash power. In this light, we can for each of the above definitions and any (topological) group $G$ consider equivariant analogs
\[\lodayR {(-)}(-) \colon [G,\catA] \times \catC \rightarrow [G,\catC].\]
As already discussed in the Introduction, it is crucial to investigate the equivariant properties of the smash power,
to make sure that they are usable for our applications. The next sections will be devoted to this topic.

\chapter{The Geometric Diagonal}
\label{ch:smashpow}
This chapter contains our main results: we give a vast generalization of the ``cyclotomic'' structure of topological Hochschild homology, see \eg Theorem~\ref{thm:diagisoisnaturalcom} for the case of finite groups and Theorem~\ref{gocompact} for the case when the Lie group in question is a torus. Under certain conditions -- which we prove for finite groups and for the torus when the base ring is the sphere spectrum --  this generalizes to other compact Lie groups $G$ in Theorem~\ref{generalliegroup} and to other base rings in Section~\ref{thm:geodiagoverR}.  

One -- derived -- way of phrasing the conclusion is that, when $X$ is a free $G$-space and $A$ is a commutative orthogonal ring spectrum, then
\begin{quote}
  the geometric diagonal of Definition~\ref{definitionofgeomfixedpoints} for a finite normal subgroup $N$ of $G$ induces a natural equivalence {\em  in the category of commutative orthogonal ring spectra}  between
  \begin{itemize}
  \item the $N$-orbits of $\bigwedge_XA$ and 
  \item the { geometric $N$-fixed points} $\bigwedge_XA$;
  \end{itemize} 
\end{quote}
since tensors commute with colimits  both constructions yield the smash power $\bigwedge_{X_N}A$ over the orbits.

{\em Natural} should be emphasized here.  The very definition of the equivalence in question
is dependent on getting this naturality sufficiently general for the special case of finite groups.  We spend Section~\ref{fixedofsmashsect} getting this exactly right with respect to the space $X$, culminating in Theorem~\ref{thm:diagisoisnaturalcom} and Proposition~\ref{diagsimplicial} which we couple with the crucial Proposition~\ref{diagrespres} which tells us how to move from one finite group to another.  With this in place we go from the finite to the compact Lie group case by simplicial approximation. 
 
This functoriality opens in Section~\ref{sec:complie} for an approach towards extending the geometric diagonal to compact Lie groups that are in an appropriate way possible to approximate by finite groups.  More precisely, we need a situation that is \orfi and also has compatible simplicial approximations.  We explain the conditions in more detail and prove that they are satisfied for tori.  We also explore to what extent it is possible to use other base rings than the sphere spectrum.

\section[Functoriality for Fixed Finite Group]{Functoriality for Fixed Finite Group%
              \sectionmark{Functoriality of the Geometric Diagonal}}
\sectionmark{The Geometric Diagonal for Finite Groups}
\label{fixedofsmashsect}\label{fixedpointsandlodaysubsect}
We now return to the geometric diagonal of Definition~\ref{definitionofgeomfixedpoints} for finite groups and show the relevant functoriality properties in the space variable.  This relies heavily on our work on filtrations of smash powers in Section~\ref{sec:cellfilt}.

In this section we consider a  fixed short exact sequence
$$E\colon\,1\to N\to G\to J\to 1$$
of finite groups.  We will consider what happens when we change $G$ in Section~\ref{sec:gpchange}.

For a finite free $G$-set $X$ and an $\Sp$-cofibrant orthogonal spectrum $L$, in Theorem~\ref{geomScofib} we defined an isomorphism between the geometric $N$-fixed points $\Phi^NL^{\smsh X}$ and the smash $L^{\smsh X_N}$ indexed over the orbits.  In this section we show that this construction is as functorial as one can expect (and perhaps a bit more). 

Proposition~\ref{diagzigzag} below initializes our study of the important naturality properties
which we need in order to generalize the correspondence between the geometric fixed points of smash powers and smash powers over orbits from finite sets to infinite and
non-discrete spaces. It is inspired by Kro's proposed ``diagonal map''
$L^{q} \rightarrow \Phi^{C_r}L^{rq}$ in the case of finite cyclic
groups. The write-up in \cite[3.10.4]{Kro}, however, appears to
mix up the left and right adjoints involved.  We
will instead define a natural zig-zag of maps, where the arrow in the
``wrong'' direction becomes an isomorphism for $\Sp$-cofibrant input.

Recall the categories 
\(\catO_J\) and $\catO_E$ from Section~\ref{secfixedpointspectra}. 

\begin{dfn}
  Let
$\catO^\reg_G \subset \catO_G$\index{Oregg@$\catO^\reg_G$} and $\catO^\reg_J\subset\catO_J$\index{OregJ@$\catO^\reg_J$} be the
full subcategories of objects of the form \(V^{\oplus X}\) and
$(V^{\oplus X})^N$ respectively for \(V\) an object of \(\catOr\). 
Via the diagonal isomorphism we will when convenient allow ourselves to write $V^{\oplus X_N}$ instead of $(V^{\oplus X})^N$ for the objects of $\catO^\reg_J$.
Let
$\catO_E^\reg$\index{Orege@$\catO^\reg_E$} be the $N$-fixed category of $\catO_G^\reg$,
and consider the commutative diagram
\[
\xymatrix{{\catO_E}\ar_-{\phi}[d]&
{\catO_E^\reg}\ar_-{i}[l]\ar^-{\phi^\reg}[d]\\
{\catO_J}&{\,\catO_J^\reg,}\ar^-{j}[l]}
\]  
where $i$ and $j$ are the inclusions and $\phi^\reg$\index{phi@$\phi,\,\phi^\reg$} is the restriction of $\phi$ (of Definition~\ref{dfnphiphi}) sending a regular
representation $V^{\oplus X}$ to $(V^{\oplus X})^N$ (think $V^{\oplus X_N}$). 
\end{dfn}

Note that $\phi^\reg$ is full and a bijection on objects.

\begin{dfn}
  Define the section of $\phi^\reg$
$$\Delta\colon\catO_J^\reg\to\catO_E^\reg$$ by $\Delta(V^{\oplus X_N})=V^{\oplus X}$ on objects and the diagonal injection 
$$\catO(V^{\oplus X_N},W^{\oplus X_N})\to\catO(V^{\oplus X},W^{\oplus X})^N$$ 
on morphisms (explicitly, an element $f\smsh y\in\catO(V^{\oplus X_N},W^{\oplus X_N})$ is sent to the pair $\Delta f\smsh\Delta y$, where the diagonal $\Delta f\colon V^{\oplus X}\to W^{\oplus X}$ has $(x,x')$-coordinate equal to the $([x],[x'])$-coordinate of $f$ and $\Delta y$ is the image of $y$ under the diagonal map $S^{W^{\oplus X_N}}\to S^{W^{\oplus X}}$).\end{dfn}

Using the convention that  $\FU$ signify precomposition and $\FP$ the left adjoint \index{P@$\FP_j,\, \FP_{\phi^\reg}, \, \FP_{\phi},\, \FP_i$}\index{UP@$\FU_j,\, \FU_{\phi^\reg}, \, \FU_{\phi},\, \FU_i$} we get diagrams
 \labeleq{iphiphii}{\xymatrix{{\catO_E\catT}\ar^-{\FU i}[r]&{\catO_E^\reg\catT}\\
 {\catO_J\catT}\ar^-{\FU \phi}[u]\ar_-{\FU j}[r]&
{\,\catO_J^\reg\catT}\ar_-{\FU\phi^\reg}[u]
}\qquad 
\xymatrix{{\catO_E\catT}\ar_-{\FP\phi}[d]&
{\catO_E^\reg\catT}\ar_-{\FP i}[l]\ar^-{\FP\phi^\reg}[d]\\
 {\catO_J\catT}&{\,\catO_J^\reg\catT,}\ar^-{\FP j}[l]
}}
commuting up to preferred natural isomorphisms.

Recall from Definition~\ref{def:FixN} the fixed point $J$-functor $\Fix^N\colon[\catO_G\catT]^N\to\catO_E\catT$ taking a \(G\)-spectrum \(E\) to \(V \mapsto (\Fix^NE)_V = (E_V)^N\), so that \(\FP_\phi \Fix^N\) is isomorphic to the geometric fixed point functor $\Phi^N$ (c.f.~Remark~\ref{notationFP}).
We construct a natural transformation
$$\xymatrix{
\catOT\ar[rr]^{{L\mapsto\bigwedge_{X_N}L}}\ar[d]_{{L\mapsto\bigwedge_{X}L}}&&
{\catO_J\catT}\ar[rr]^{\FU_j}\ar@{}_\Downarrow^{\tilde\Delta}[d]&&{\catO^\reg_J\catT}\\
{[\catO_G\catT]^N}\ar[rr]^{\Fix^N}&&{\catO_E\catT}\ar[rr]_{\FU_i}&&
{\catO_E^\reg\catT}\ar[u]^{\FU_\Delta}
}
$$ 
(all functors are $J$-functors, but since the source category is $J$-fixed we could just as well have restricted to the $J$-fixed categories everywhere) as follows.
For $V^{\oplus X_N}\in\catO^\reg_J$, let $\tilde\Delta$ be the natural map
$$\xymatrix{(\FU_j L^{\smash X_N})_{V^{\oplus X_N}}\ar@{=}[d]\ar[r]^-{\tilde\Delta}&\FU_\Delta \FU_i\Fix^N(L^{\smash X})_{V^{\oplus X_N}}\ar@{=}[d]\\
  {\int^{(V_{[x]})_{[x]}
    }\catO(\bigoplus_{[x]}V_{[x]},V^{\oplus X_N})\smsh\bigwedge_{[x]}L_{V_{[x]}}}\ar[r]&
  {\left[\,\, \int^{(W_x)_x
      }\catO(\bigoplus_xW_x,V^{\oplus X})\smsh\bigwedge_{x\in X} L_{W_x}\right]^N}}
$$
obtained by sending the $(V_{[x]})_{[x]}$-summand to the $(V_{[x]})_x$-summand (which is $N$-fixed) via the diagonal
$\catO(\bigoplus_{[x]}V_{[x]},V^{\oplus X_N})\smsh\bigwedge_{[x]}L_{V_{[x]}}\to
[\catO(\bigoplus_{x}V_{[x]},V^{\oplus X})\smsh\bigwedge_{x}L_{V_{[x]}}]^N$ of spaces.
Direct inspection gives the following naturality result.
\begin{lemma}\label{lem:firstnat}
  The map  of $J$-spectra 
$\tilde\Delta\colon \FU_j L^{\smash X_N}\to\FU_\Delta \FU_i\Fix^N(L^{\smash X})$ is natural \wrt maps of orthogonal spectra $L\to L'$ and \wrt isomorphism of finite free $G$-sets $X\cong X'$.
\end{lemma}

\begin{dfn}\label{dfn:tildedelta}
  We define the natural map $\tilde\Delta(X,L)$ of $J$-spectra by the diagram
$$\xymatrix{\FP_j\FU_j L^{\smash X_N}\ar[d]^{\FP_j\tilde\Delta}\ar[rr]^-{\tilde\Delta(X,L)} &&\FP_\phi\FP_i\FU_i\Fix^N(L^{\smash X})\\
  \FP_j\FU_\Delta \FU_i\Fix^N(L^{\smash X})\ar[r]^-\cong&
  \FP_j\FP_{\phi^\reg}\FP_\Delta\FU_\Delta \FU_i\Fix^N(L^{\smash X})\ar[r]^-\cong&
  \FP_\phi\FP_i\FP_\Delta\FU_\Delta \FU_i\Fix^N(L^{\smash X}),\ar[u]^{\epsilon_\Delta}
}$$
where the first isomorphism is induced by the identity $\phi^\reg\Delta=\id_{\catO^\reg_J}$, the second is that of~\ref{iphiphii} and $\epsilon_\Delta$ is the unit of adjunction \wrt $\Delta$.
\end{dfn}

\begin{prop}\label{diagzigzag}\label{functorialitydiag}
 For $G$ a finite group and $X$ a finite free $G$-space, and $L$ any
 orthogonal spectrum the solid arrows in 
 $$\xymatrix{L^{\smash   X_N}\ar@{.>}_-{\Delta(X,L)}[d]&\FP_j\FU_jL^{\smash X_N}\ar_-{\varepsilon_j}[l]\ar[d]_{\tilde\Delta(X,L)}\\
     \Phi^N(L^{\smash X})&\FP_\phi \FP_i \FU_i \Fix^N(L^{\smash X})\ar_-{\FP_\phi(\varepsilon_i)}[l]
   }$$
   (which we will refer to as the  \emph{diagonal zig-zag})
   are natural maps  of
 $J$-spectra (\wrt maps of orthogonal spectra $L\to L'$ and \wrt isomorphism of finite free $G$-sets $X\cong X'$),
   where $\varepsilon_j$ and $\varepsilon_i$ are the counits of the adjoint pair $(\FP_j,\FU_j)$, respectively $(\FP_i,\FU_i)$ and $\tilde\Delta(X,L)$ is given in Definition~\ref{dfn:tildedelta}.
   
For $\Sp$-cofibrant spectra $L$, $\epsilon_j$ is an isomorphism, and so the (indicated with dots) composite $\Delta(X,L)$ is uniquely defined.  In this case, $\Delta(X,L)$ coincides with the natural
isomorphism from Theorem~\ref{geomScofib}, which we therefore call \emph{the
  diagonal isomorphism}. \index{diagonal isomorphism, the} 
\end{prop}
\begin{proof}
  The naturality claim follows from Lemma~\ref{lem:firstnat}.
To see that the instance of $\varepsilon_j$ is an isomorphism for
$\Sp$-cofibrant $L$, note that smash-powers of semi-free orthogonal
spectra are semi-free $J$-spectra of the form $\Gr{}{V}{K}$ with $V$
regular and that \(\varepsilon_j\) is  an isomorphism for any
semi-free \(J\)-spectrum of this form. 
By Theorem~\ref{cellular}, a cell induction
then gives the result.

Similarly it suffices to show that for each
semi-free spectrum $L=\Gr{}{R}K=\catO(R,-)\smsh K$ the zig-zag yields the isomorphism from
Theorem~\ref{geomsmashsemifree}.
 We follow the zig-zag and observe that under the isomorphisms $L^{\smsh X}\cong\Gr{}{R^{\oplus X}}K^{\smsh X}=\catO(R^{\oplus X},-)\smsh_{\prod_{X} \Or{R}} K^{\smsh X}$  and  and $L^{\smsh X_N}\cong\Gr{}{R^{\oplus X_N}}K^{\smsh X_N}=\catO(R^{\oplus X_N},-)\smsh_{\prod_{X_N} \Or{R}} K^{\smsh X_N}$, analyzing the maps in the diagonal zig-zag
$$\xymatrix{\Gr{}{R^{\oplus X_N}}K^{\smsh X_N}\ar[d]^{\Delta(X,L)}&
  \int^{V^{\oplus X_N}\catO^\reg_J}\catO(V^{\oplus X_N},-)\smsh \Gr{}{R^{\oplus X_N}}K^{\smsh X_N}_{V^{\oplus X_N}}\ar[l]\ar[d]_{\tilde\Delta(X,L)}\\
   \int^{W\in\catO_E}\catO(W^N,-)\smsh[\Gr{}{R^{\oplus X}}K^{\smsh X}_W]^N&
   \int^{V^{\oplus X}\in\catO^\reg_E}\catO([V^{\oplus X}]^N,-)\smsh [\Gr{}{R^{\oplus X}}K^{\smsh X}_{V^{\oplus X}}]^N
   \ar[l]}$$
 we see that $\Delta(X,L)$ is given by the isomorphism
 $$\catO(R^{\oplus X_N},-)\smsh_{\prod_{X_N} \Or{R}} K^{\smsh X_N}\to
 \smallint^{W\in\catO_E}\catO(W^N,-)\smsh[\catO(R^{\oplus X},W)\underset{\prod_{X} \Or{R}}\smsh K^{\smsh X}]^N$$
 defined by mapping into the $W=R^{\oplus X}$ summand by the isomorphism $W^N\cong R^{\oplus X_N}$ and the diagonal.  Hence, $\Delta(X,L)$ coincides with the description of the 
isomorphism of Theorem~\eqref{geomsmashsemifree}. 
\end{proof}

\begin{cor}\label{orbitsplit}
The diagonal isomorphism respects decompositions of the finite free $G$-set $X$ into $G$-orbits, in the sense that for a choice of $G$-isomorphism $f\colon X\cong G\times X_G$ and $L$ an $\Sp$-cofibrant orthogonal spectrum, we get a commutative diagram of isomorphisms
$$\xymatrix{
(L^{\smsh G/N})^{\smsh X_G}\ar[rr]_\cong
\ar[d]^{\Delta(G,L)^{\smsh X_G}}&&
(L^{\smsh X_G})^{\smsh G/N}\ar[d]^{\Delta(G,L^{\smsh X_G})}\ar[r]_\cong&
L^{\smsh X_N}\ar[d]^{\Delta(X,L)}\\
(\Phi^N(L^{\smsh G}))^{\smsh X_G}\ar[r]^-\alpha&\Phi^N((L^{\smsh G})^{\smsh X_G})\ar[r]_\cong&
\Phi^N((L^{\smsh X_G})^{\smsh G})\ar[r]_\cong&\,\Phi^N(L^{\smsh X}),}
$$
where the horizontal isomorphisms are given by shuffle permutations of smash factors, the chosen isomorphism $f$ and the lax monoidal structure map $\alpha$ from Proposition~\ref{philaxmonoidal}.
\end{cor}
\begin{proof}
  This follows by Lemma~\ref{lem:geomdiagismonoidal}, strong monoidality of Corollary~\ref{cor:phistrongmon} and the the naturality of the geometric diagonal of Proposition~\ref{diagzigzag}.
\end{proof}

\subsubsection{Enhanced naturality of the geometric diagonal} \label{generalx}
In Section~\ref{sectlodayfunct}, we have defined several versions of smash powers, 
with the input categories varying between finite discrete sets together with general orthogonal spectra, to general spaces and commutative orthogonal ring spectra. In the above discussion we have restricted the first input further, to finite free $G$-sets, to study the equivariant structure induced on the output.
In view of the functoriality of Proposition~\ref{functorialitydiag}, we will from now view the diagonal isomorphism as a natural transformation of functors 
\[\Delta(\cdot,-)\colon\,\,\,\,\,\, \bigwedge_{(\cdot)_N}(-) \,\,\,\longrightarrow \,\,\,\Phi^N(\bigwedge_{(\cdot)}(-)),\]
and study how we can vary the input categories.
Until further notice, \(G\) is a fixed finite group.

We begin with checking naturality of the diagonal isomorphism with respect to injections of finite free $G$-sets. As usual, we have to adapt the cofibrancy condition.
\begin{dfn} If $\Sp\to L$ is an $\Sp$-cofibration of orthogonal spectra we say that $L$ is \emph{$\Sp$-cofibrant under $\Sp$}\index{Scofibrant under@$\Sp$-cofibrant under $\Sp$}.
\end{dfn}
\begin{ex}
Every $\Sp$-cofibrant commutative orthogonal ring spectrum is $\Sp$-cofibrant under $\Sp$ via its unit map by Proposition~\ref{barprop}, which in particular says that the unit maps for the commutative ring spectra appearing in the generating $\Sp$-cofibrations is an 
$\Sp$-cofibration. 
\end{ex}
\begin{lem}
The diagonal isomorphism is natural with respect to finite free $G$-sets and equivariant inclusions, and spectra $\Sp$-cofibrant under $\Sp$.
\end{lem}
\begin{proof}
The naturality in the spectrum direction is clear from the construction.  Hence it is enough to treat the case where $\Sp\to L$ is an $\Sp I$-cellular inclusion and $X \rightarrow Y$ is an
equivariant inclusion of finite free $G$-sets.

 Then $L^{\smash X}
\rightarrow L^{\smash Y}$ is an inclusion of an equivariant subcomplex
as in Theorem~\ref{cellular}. Similarly $L^{\smash X_N}$ is an
equivariant subcomplex of $L^{\smash Y_N}$ and it follows as from the
proof of Theorem~\ref{geomScofib} that the diagonal isomorphism respects the
inclusion of subcomplexes.\end{proof} 
Moving towards infinite free $G$-sets $X$, the smash power is defined in Definition~\ref{colimloday} as the colimit along inclusions of finite free $G$-subsets of $X$.\[\bigwedge_X(L) \defas{}\colim\limits_{F \subset X \text{ finite}} \bigwedge_F L.\] Note that by the proof of the previous lemma, this is a filtered colimit along Hurewicz cofibrations, so that it is preserved by $\Phi^N$. 
\begin{lem}\label{naturalinclusions}
The diagonal isomorphism exists and is natural with respect to free $G$-sets and equivariant inclusions, and spectra $\Sp$-cofibrant under $\Sp$.
\end{lem}
\begin{proof}We begin with the existence. Let $L$ be cofibrant under $\Sp$, and let $X$ be a free $G$-sets. The finite subsets of $X_N$ are orbits of finite subsets of $X$ hence there is a natural map $\bigwedge_{X_N} L \rightarrow \Phi^N (\bigwedge_{X}L)$ which is the colimit of isomorphisms hence an isomorphism itself. Naturality follows since equivariant inclusions induce inclusions of indexing categories for the colimits.
\end{proof}
As an alternative proof, we can use Corollary~\ref{orbitsplit}: Let $f \colon X \rightarrow Y = X\cup Z$ be an equivariant inclusion of free $G$-sets. Then $f$ respects the orbit decomposition, \ie \[X \cong_G \bigcup_{X_G} G, \,\,\,\,\,\,\,\,\,\,\,\,\,\,\,\,\,\,\,Y\cong_G \bigcup_{X_G} G \cup \bigcup_{Z_G} G,\]
and $f$ corresponds to the obvious inclusion.
Hence $\bigwedge_f L$ is isomorphic to the map \[\bigwedge_G (\bigwedge_{X_G} L) \cong \bigwedge_G (\bigwedge_{X_G} \smash \bigwedge_{Z_G} \Sp) \rightarrow \bigwedge_{Y_G} (\bigwedge_{Y_G} L),\]
\ie it is the smash power of a map of $\Sp$-cofibrant spectra, so Proposition~\ref{diagzigzag} gives the result. 
In particular, we even get that the map $\bigwedge_{X_G} L \rightarrow \bigwedge_{Y_G}L$ is a (non equivariant) $\Sp$-cofibration, thus the induced map of smash powers is an $\Sp$-cofibration of $G$-spectra by Lemma~\ref{inducesmashpresscof}. 
To see this, filter $Y_G$ through finite sets $Y_i$ and let $X_i = f^{-1} Y_i$, $Z_i = Y_i \setminus f(X_i)$. 
As in Theorem~\ref{cellular}, each of the maps $L^{\smash X_i} \cong L^{\smash X_i}\smash \Sp^{\smash Z_i} \rightarrow L^{\smash Y_i}$ is an $\Sp$-cofibration, hence so is their colimit.\\
Finally we move towards $\Sp$-cofibrant commutative orthogonal ring spectra, where we want to work with the definition of the smash power 
in terms of the categorical tensor with spaces, \ie for $X$ a space and $A$ an $\Sp$-cofibrant commutative orthogonal ring spectrum,  $\bigwedge_X A = A \otimes X,$
where the tensor is in the category of commutative orthogonal ring spectra (cf.~\ref{subsecttensors}). As discussed in Proposition~\ref{disctensor}, the tensor specializes to the smash power for discrete inputs $X$, so all the results from above still apply. Note that we can now extend the naturality results to not necessarily injective maps:
\begin{thm}\label{thm:diagisoisnaturalcom}
The diagonal isomorphism from Lemma~\ref{naturalinclusions} is natural with respect to free $G$-sets and equivariant maps, and $\Sp$-cofibrant commutative orthogonal ring spectra.
\end{thm}
\begin{proof}
  First observe that both source, $A^{\smsh X_N}$ and target $\Phi^N(A^{\smsh X})$ are fully functorial in $A$ and $X$, and furthermore that the diagonal isomorphism is natural \wrt $A$.
  
  Let $f \colon X \rightarrow Y$ be an equivariant map between free $G$-sets. Similar to the above discussion, we filter $X$ and $Y$ by finite free $G$-sets $X_i$ and $Y_i$ and consider $f$ as a colimit of maps $f_i\colon X_i \rightarrow Y_i$, where the transformations $\bigwedge_{f_i} A \rightarrow \bigwedge_{f_j}A$ for $i \leq j$ are along $\Sp$-cofibrations.

  Thus it suffices to check naturality for 
equivariant maps between finite free $G$-sets.
  Any such map can be factored into isomorphisms, inclusions and surjections of the form $\id_G\times p\colon G\times S\to G\times T$, where $p$ is a surjection of finite sets.  Isomorphisms and inclusions having already been dealt with, we only need to consider the latter type, which can be reinterpreted as functoriality in $A$ as follows.  If $B=A^{\smsh{S}}$ and $C=A^{\smsh{T}}$, then $p$ induces a map $\tilde p\colon B\to C$ of $\Sp$-cofibrant orthogonal ring spectra.  This way, the diagram
  $$\xymatrix{
    A^{\smsh (G\times S)_N}\ar[rr]^{A^{\smsh(\id_G\times p)_N}}\ar[d]^{\Delta(G\times S,A)}&&
      A^{\smsh (G\times T)_N}\ar[d]^{\Delta(G\times T,A)}\\
      \Phi^NA^{\smsh (G\times S)}\ar[rr]^{\Phi^NA^{\smsh(\id_G\times p)}}&&\Phi^NA^{\smsh (G\times T)}
    }$$ is transformed into the commutative diagram
     $$\xymatrix{B^{\smsh G/N}\ar[r]^{\tilde p^{\smsh G/N}}\ar[d]^{\Delta(G,B)}&C^{\smsh G/N}\ar[d]^{\Delta(G,C)}\\
       \Phi^NB^{\smsh G}\ar[r]^{\Phi\tilde p^{\smsh G}}&\Phi^NC^{\smsh G}
     }$$
    representing naturality in $\Sp$-cofibrant commutative orthogonal ring spectra.
\end{proof}
We can finally move on towards non discrete $G$-spaces (however, at this stage, the group $G$ is still finite and discrete).
We begin with $G$-spaces that are geometric realizations of simplicial $G$-sets, since there we have Proposition~\ref{realizeinspec},
which makes computing the smash powers much easier and in particular allows the following extension of the diagonal isomorphism:

\begin{prop}\label{diagsimplicial}
The smash power and geometric fixed points commute with realization and so the diagonal isomorphism from Theorem~\ref{thm:diagisoisnaturalcom} yields a transformation of commutative orthogonal $G/N$-ring spectra  
$$\xymatrix{
\Delta_{|X|}\colon\bigwedge_{|X|_N} A\cong |\bigwedge_{X_N}A|\ar[rr]^{|\Delta_{X_N}|} &&
|\Phi^N(\bigwedge_{X}A) |\cong \Phi^N(\bigwedge_{|X|}A)}$$
which is natural with respect to free simplicial $G$-sets $X$ and $\Sp$-cofibrant commutative orthogonal ring spectra $A$.
\end{prop}
\begin{proof}
By Proposition~\ref{realizeinspec}, for a free $G$-simplicial set $X$, and an $\Sp$-cofibrant commutative orthogonal ring spectrum $A$, the tensor $A \otimes |X|$ is naturally isomorphic to the realization of the simplicial orthogonal spectrum 
$\{q \mapsto (A \otimes X_q) = \bigwedge_{X_q} A\}$.
By Proposition~\ref{skeletonfiltration}, the geometric realization of this spectrum is the colimit along the skeleton filtration, which is along levelwise Hurewicz cofibrations since the simplicial spectrum is Hurewicz good (cf.~Definition~\ref{dfngood} and~\ref{goodisproper} and Proposition~\ref{cofproperfiltr}): Every simplicial degeneracy map $s_i$ is an injection of free $G$-sets, hence as in the comment before Lemma~\ref{naturalinclusions}, it induces an $\Sp$- hence Hurewicz cofibration of smash powers.
 In particular, taking the geometric fixed points commutes with the geometric realization, since $\Fix^N$ does. Therefore the diagonal maps $\Delta(X_q,A)$ for each simplicial level induce an isomorphism on realizations. It is natural since maps of free $G$-simplicial sets are levelwise maps of free $G$-sets.
\end{proof}



\section{Change of Group}
\label{sec:gpchange}
This concludes our study of the case of a {\em fixed} finite group $G$. So far we have not touched upon functoriality of the geometric diagonal with respect to the ambient group,
 so let now $\phi\colon H \rightarrow G$ be a homomorphism of topological groups. As usual we can look at the restriction functor $\phi^*\colon \catGT \rightarrow H\catT$, and it is immediate, that it commutes with smash powers
in the sense that $\bigwedge_{\phi^* X} A = \phi^*\bigwedge_{X} A$ for orthogonal ring spectra $A$.
If $\phi^*$ is not injective it does not necessarily send
free $G$-sets to free $H$-sets and we have in general less control over the geometric diagonal. 
\begin{prop}\label{diagrespres}
  Let $N\subseteq H\subseteq G$ be finite groups with $N$ normal in $G$, and $X$ a finite free $G$-set.  Let $i\colon H\to G$ and $i_1\colon H/N\to G/N$ be the inclusions.  
  The restriction functor preserves the geometric diagonal of Proposition~\ref{diagzigzag} in the sense that if $A$ is an orthogonal spectrum, then the diagram
$$\xymatrix {
    i_1^*\bigwedge_{X_N}A\ar@{=}[d]& \ar_-{\varepsilon_j}[l]{i^*_1\FP_j\FU_j\bigwedge_{X_N} A}\ar^-{}[r]\ar[d]&{i^*_1\PhiNH{}\bigwedge_XA}\ar[d]^\cong\\
  \bigwedge_{(i^*X)_N}A& 
\ar_-{\varepsilon_{j_1}}[l]{\FP_{j_1}\FU_{j_1} \bigwedge_{(i^*X)_N}A}\ar^-{}[r]&
{\Phi^N\bigwedge_{i^*X}A}}$$
commutes, where $j_1\colon\catO_{H/N}^\reg\subseteq\catO_{H/N}$ is the inclusion, the horizontal maps are those of Proposition~\ref{diagzigzag} and the vertical maps are induced by restriction.  The rightmost restriction is an isomorphism.

If $A$ is an $\Sp$-cofibrant commutative orthogonal ring spectrum and $X$ is the realization of a free simplicial $G$-set, then the geometric diagonal isomorphism $\Delta(i^*X,A)$ for $A$ \wrt $H$ and $i^*X$ is equal to the the restriction to $H/N$ of the geometric diagonal isomorphism $\Delta(X,A)$ for $A$ \wrt $G$ and $X$ composed with the $H/N$-isomorphism  
$i^*_1\PhiNH{}(\bigwedge_{X}A)\cong \Phi^N(\bigwedge_{i^*X}A)$.
\end{prop}
\begin{proof}
  For a finite free $G$-set $X$, the commutativity of the diagram follows directly from the definition of the maps in Proposition~\ref{diagzigzag}.  Since $G$ is finite we have by Lemma~\ref{cond1lem} that $N\subseteq G$ is \orfi and so Proposition~\ref{geomfixedrestrict} gives that the rightmost restriction is an isomorphism. 

  Letting $X$ be the realization of a free simplicial $G$-set and $A$ an $\Sp$-cofibrant commutative orthogonal ring spectrum, the restriction functor preserves the colimit along the inclusions of finite subsets as well as the geometric realization which we used to extend the geometric diagonal in Proposition~\ref{diagsimplicial}.  Hence the last statement follows.
\end{proof}
\begin{rem}\label{multiplicativenormscomment}
  The relation between smash powers and the induction of an $H$-set $X$ along the inclusion $i \colon H \rightarrow G$ is more subtle. 
Intuitively smash powers should be viewed as the $\smash$-induction of a spectrum with action of the trivial group to a spectrum with $G$-action.
  If one wants to start with an $H$-spectrum instead, this generalizes the study of multiplicative norm constructions,\index{multiplicative norm} where $X$ is assumed to be a discrete subgroup $H$ of $G$. These norm functors have been famously put to use in the proof of the Kervaire-invariant problem by Hill, Hopkins and Ravenel. An introduction can be found in \cite[A.3,4]{HHR}, or \cite[8,9]{Schworth}, both of which only became available very shortly before the third author's thesis (which very much is at the core of this paper) was finished. The third author learned about the interplay from Stefan Schwede, during a visit in Bonn in November 2010, where he presented the results of his thesis. As the study of multiplicative norms of course has to address some of the same questions we discussed here, we point out some similarities and differences to \cite{HHR}. Due to the fact that the works are independent, the notation and viewpoint are rather different. First of all one should note the difference in model structures.
  We worked with the $\Sp$-model structure instead of the classical $q$-model structure on commutative orthogonal ring spectra, in order to get around the $q$-cofibrant replacement. 
  The paper \cite{HHR} addresses this problem by proving that the symmetric powers appearing in the generating $q$-cofibrations are ``very flat'' (cf.~\cite[B.13,63]{HHR}), which allows them to construct a natural weak equivalence calculating geometric fixed points. Their method has the advantage that it is more easily applicable to the general multiplicative norm case they aim to study. Our method on the other hand allows us to recognize the diagonal map as a natural \emph{isomorphism}, strengthening their statement. Note that the ``Slice Cells'' discussed in \cite[4.1]{HHR} are special cases of generating $\Sp I_G$-cofibrations in our language, and from this viewpoint the filtration given in \cite[A.4.3]{HHR} and our Theorem~\ref{cellular} achieve similar goals -- an equivariant filtration of the smash power -- with different methods. Finally note that the change of the indexing of the smash power away from a non discrete set we worked for in this section, is only addressed in the side note \cite[A.35]{HHR}. Since all groups discussed there are finite, this is not a major point in \cite{HHR}, but as we are going to move towards tori and more general compact Lie groups now, the details become important.
\end{rem}

\section{The Geometric Diagonal for Compact Lie Groups}
\label{sec:complie}
We now leave the realm of finite groups and move back to the case of
compact Lie groups. 
In particular, to deal with (higher) topological Hochschild homology and cyclic homology, we are interested in the case where $G$ is a
torus. The first thing we should address is that this case is \orfi in the sense of Definition~\ref{cond1}.

\begin{lemma}\label{toruscond}
  Inclusions of kernels of isogenies of the $n$-torus are \orfi.
\end{lemma}
\begin{proof}
  An isogeny of the $n$-torus $\R^n/\Z^n$ is given by multiplication by an integral $n\times n$-matrix $A$ with nonzero determinant -- making $A$ invertible over the rationals.  The kernel $H$ of the isogeny is then $A^{-1}\Z^n/\Z^n\subseteq\R^n/\Z^n$.  The columns $c_1,\dots,c_n$  (mod $\Z^n$) of $A^{-1}$ generate $H$ and an action $\phi$ of $H$ on a Euclidean space $W$ is uniquely given by specifying commuting elements $\phi(c_1),\dots,\phi(c_n)\in\Or{W}$, so that $\phi(\sum_{i=1}^nn_kc_k)=\prod_{i=1}^n\phi(c_k)^{n_k}$ for $n_k\in\Z$.

  We must find a representation $V$ of the torus and an $H$-linear isometric embedding $W\to V$ inducing an isomorphism of fixed points $W^H\cong V^H$.  We may restrict to irreducible $W$ so that $W=\R$ or $W=\C$.  If $W=\C$ we see that none of the $\phi(c_k)$s can be reflections: if one were it'd split $\C$ in two eigen spaces which \emph{all} the $\phi(c_k)$s would have to respect, contradicting the irreducibility of $W$.  So,  in either case $\phi(c_k)$ is given by multiplication by a root $\zeta_k$ of unity: if $W=\R$ then $\zeta_k=\pm1$ and in general $\zeta_k=e^{2\pi i\frac{n_k}{\rho_k}}$ where $n_k$ is an integer and $\rho_k$ is the order of $c_k$.

  If $W=\R$ and $\phi$ is trivial, we let $V=W$ be the trivial representation of the torus.

  If $\phi$ is nontrivial we get that $W^H=0$. We then let $V=\C$ provided with the torus action obtained by letting $(x_1,\dots,x_n)\in\R^n$ act by multiplication by $\prod_{k=1}^n\zeta_k^{x_k}$. Note that there are no non-zero fixed point of this action either.   If $W=\R$ then the standard embedding $\R\subseteq\C$ is our desired extension and if $W=\C$ we can simply use the identity $\C=\C$. 
\end{proof}

 We will need one more fact about the torus $\T^n$, namely that it is possible to realize our theorems about finite group actions on sets to actual $\T^n$-equivariant maps.  We give a name to the property we will need.

  \begin{dfn}
    \label{def:simpap}
    Let $G$ be a topological group and $\mathcal F$ a class of finite subgroups.  We say that $G$ has {\em simplicial approximations with respect to $\mathcal F$}\index{simplicial approximation wrt.\ a class of subgroups} if for each $H\in\mathcal F$, there is a free simplicial $H$-set $\sa HG$ together with an $H$-equivariant homeomorphism $D_H\colon|\sa HG|\to G$ such that for any $K\in\mathcal F$ and normal subgroup $N$ of $G$ contained in $H\cap K$ the diagram
$$\xymatrix{
\bigwedge_{|\sa HG|_N}A\ar[d]^{\bigwedge_{(D_K^{-1}D_H)_N}A}&
|\bigwedge_{\sa HG_N}A|\ar[l]_\cong\ar[rr]^-{|\Delta_{\sa HG}|}
&&
 |\Phi^N\bigwedge_{\sa HG}A|\ar[r]^\cong& 
\Phi^N\bigwedge_{|\sa HG|}A\ar[d]^{\Phi^N\bigwedge_{D_K^{-1}D_H}A}
\\
\bigwedge_{|\sa KG|_N}A&|\bigwedge_{\sa KG_N}A|\ar[l]_\cong\ar[rr]^{|\Delta_{\sa KG}|}&&
 |\Phi^N\bigwedge_{\sa KG}A|\ar[r]^\cong& \Phi^N\bigwedge_{|\sa KG|}A
}
$$
commutes.  
  \end{dfn}

\begin{rem}
As we shall see below, the torus $\T^n$ \siap
with respect to all kernels of isogenies. To the knowledge of the
authors,  the general case when $G$ is a compact Lie group is unknown. Note that Illman's triangulation theorem~\ref{illtriangmfd} constructs an $H$-equivariant triangulation of the smooth $H$-manifold $G$, but since the usual methods that produce simplicial sets from simplicial complexes fail for Illman's equivariant simplices (\cite[\S 3]{Ill}), this is not enough.
\end{rem}

  \begin{lemma}\label{toruscond2}
    The torus $\T^n$ has simplicial approximations with respect to all kernels of isogenies.
  \end{lemma}
  \begin{proof}
\newcommand{\sd}{\mathrm{sd}}
    We do the case $n=1$ first.  Recall the edgewise subdivision $\sd_r$ and the homeomorphisms $D_r\colon|\sd_{rs}X|\to|\sd_sX|$ for cyclic spaces $X$ as in \cite[Section 1]{BHM}.  

If $S^1=\Delta[1]/\partial\Delta[1]$ is the standard simplicial circle, we claim that the edgewise subdivisions provide simplicial approximations for $\T^1$ with respect to the kernels of isogenies (\ie finite cyclic groups).

The $s$th edgewise subdivision $\sd_sX$ has a natural simplicial action by the cyclic group $C_s$ of order $s$, and with respect to this action the map $D_s\colon|\sd_sX|\cong|X|$ is $C_s$-equivariant (where $X$ has the $\T^1$-action coming from the cyclic structure). 
Also, $\sd_{rs}=\sd_r\sd_s$.   Now, the subdivision appears as the precomposition of a functor (also called $\sd_r$ in \cite{BHM}) and as such commutes with any functor; in particular, $\bigwedge_{\sd_{rs}S^1}A=\sd_r\bigwedge_{\sd_sS^1}A$ and for $N\subseteq C_s$ we have that $\Phi^N\sd_{rs}=\sd_r\Phi^N\sd_s$.
Hence, if $N\subseteq H=C_s\subseteq K=C_{rs}$ our claim follows by the decomposition
$$\xymatrix{
\bigwedge_{|\sd_sS^1|_N}A\ar[d]^{\bigwedge_{(D_{rs}^{-1}D_s)_N}A}&
|\bigwedge_{\sd_sS^1_N}A|\ar[l]_\cong\ar[rr]^-{|\Delta_{\sd_sS^1}|}\ar[d]^{D_{rs}^{-1}D_s}
&&
 |\Phi^N\bigwedge_{\sd_sS^1}A|\ar[r]^\cong\ar[d]^{D_{rs}^{-1}D_s}& 
\Phi^N\bigwedge_{|\sd_sS^1|}A\ar[d]^{\Phi^N\bigwedge_{D_{rs}^{-1}D_s}A}
\\
\bigwedge_{|\sd_{rs}S^1|_N}A&|\bigwedge_{\sd_{rs}S^1_N}A|\ar[l]_\cong\ar[rr]^{|\Delta_{\sd_{rs}S^1}|}&&
 |\Phi^N\bigwedge_{\sd_{rs}S^1}A|\ar[r]^\cong&\, \Phi^N\bigwedge_{|\sd_{rs}S^1|}A.
}
$$

When $n$ is greater than $1$ we combine this insight with the investigations in the proof of Lemma~\ref{toruscond} whose notation we use freely below. Let $\alpha\colon\T^n\to\T^n$ be an isogeny with kernel $H$ containing $N$ (which then is the kernel of another isogeny).  For the intuition that underlies the following it is best to think of $\R^n$ as the $n$-fold product of the infinite simplicial complex $\R$, which has vertices lying on the points in $\Z$, and edges between them. 
We identify the action of $H$ on this complex, and then produce a finer complex where the action is simplicial. As in the proof of Lemma~\ref{toruscond}, associate to the isogeny $\alpha$ a matrix $A \in M_n(\Z) \cap \GL_n(\Q)$. The action of an element of $H$ on an element $[r] \in \faktor{\R^n}{\Z^n}$ corresponds to adding a linear combination of columns of $A^{-1}$ to $r$. By Cramer's rule all the columns of $A^{-1}$ are in $\det(A)\cdot\Z^n$. Hence the difference of $r$ to its image under the action of any $h \in H$ is in $\det(A)\cdot\Z^N$ as well. Note also that even if $A$ is not uniquely given by $H$, the absolute value $|\det A|$ of the determinant is uniquely determined by $H$.
Thus we can choose our simplicial approximation to be 
$$\sa HG = \sd_{|\det A|}(S^1\times\dots S^1)=\sd_{|\det A|}S^1\times\dots\sd_{|\det A|}S^1,$$ where $S^1$ is, as before, the standard simplicial circle. The action of $H$ on the realization corresponds to a simplicial action on the resulting simplicial complex, hence $G^H$ is the desired $H$-simplicial set.  The compatibility now follows as for the one-dimensional case.

  \end{proof}

Hence we are in a position to use our construction of the geometric diagonal in the simplicial context to prove that, for $N$ the kernel of an isogeny, the geometric diagonal provides us with an isomorphism between $\bigwedge_{\T^n/N}A$ and $\Phi^N\bigwedge_{\T^n}A$ that is at least equivariant with respect to finite subgroups of $\T^n/N$.  The compatibility we have proved for extensions of groups will give us  full $\T^n/N$-equivariance, but before we state the result we make an observation making it possible to extend the range somewhat further.
\begin{lemma}
  \label{lem:Deltaforfree}
  Let $G$ be a compact Lie group, $N$ a normal subgroup and $J=G/N$.  Let $F$ be a lax symmetric monoidal functor from $G$-spaces to commutative orthogonal $J$-ring spectra and let $\Delta_G\colon\bigwedge_JA\to F(G)$ be a $J$-ring homomorphism commuting with the right $G$-action on $J$ and $G$.  Then there is a unique monoidal extension of $\Delta_G$ to a natural transformation
$$\Delta_{X}\colon\bigwedge_{X_N}A\to F(X)$$
where $X$ ranges over the category of free cofibrant $G$-spaces.
\end{lemma}
\begin{proof}
  This is formal, but let us spell out some details.  Let $\catC$ (resp.~$J\catC$) be the category of commutative orthogonal (resp.~$J$-) ring spectra.  If $T\in J\catC$ and $X\in J\catT$, then $J\catC(\bigwedge_YA,T)\cong J\catT(Y,\catC(A,T))$, so a $J$-ring homomorphism $\bigwedge_YA\to T$ is uniquely determined by the composites with the maps $\bigwedge_JA\to\bigwedge_YA$ induced by the inclusion of $J$-orbits of $Y$. 

As we demanded that $\Delta$ should be monoidal, we must have that if $K$ is a space, then $\Delta_{G\times K}$ must be the composite 
$$\xymatrix{\bigwedge_{J\times K}A\cong\bigwedge_K\bigwedge_JA\ar[rr]^{\bigwedge_K \Delta_G}&&\bigwedge_KF(G)\to F(K\times G)}$$
of the monoidal structure maps with $\bigwedge_K \Delta_G$.  Uniqueness then gives that we have defined $\Delta$ for the full subcategory of $G$-spaces of the type $G\times K$.  There is a slight subtlety here: a $G$-equivariant map $f\colon G\times K\to G\times L$ is of the form $f(g,k)=(g\cdot f_1k,f_2k)$, and it is the presence of $f_1k$ that forced our demand that $\Delta_G$ must commute with the right action of $G$.

We do induction on cells.  Assume we have defined $\Delta$ naturally also over some $G$-map $G\times S^{n-1}\to X$ of free $G$-spaces.  Then we again define $\Delta_Y$ for $Y=X\coprod_{G\times S^{n-1}}(G\times D^n)$ by the strong monoidality of $\bigwedge_?A$.
If $G\times S^{k-1}\to Y$ is any $G$-map, then any orbit in $G\times S^{k-1}$ lifts (not necessarily uniquely) to $X$ or to $G\times D^n$, and so since $\Delta_Y$ is built out of $\Delta_X$ and $\Delta_{G\times D^n}$, uniqueness gives that the crucial diagram
$$\xymatrix{\bigwedge_{J\times S^{k-1}}A\ar[r]\ar[d]^{\Delta_{G\times S^{k-1}}}&
  \bigwedge_{Y_N}A\ar[d]^{\Delta_Y} \\
F({G\times S^{k-1}})\ar[r]&F(Y)
}$$
commutes.  Concretely, given a $G$-orbit $T$ in $G\times S^{k-1}$ that maps to an orbit in $Y$ which lifts to $Z$ (where $Z$ is either $X$ or $G\times D^n$), then all the faces but possibly the front one in the cube 
$$\xymatrix{
&{\bigwedge_{T_N}A}\ar[rr]\ar'[d][dd]_-{\Delta_T}\ar_-{\bigwedge_{\text{incl}}A}[dl]&&
{\bigwedge_{Z_N}A}\ar_>>>>>>>>{\Delta_Z}[dd]\ar[dl]\\
{\bigwedge_{J\times S^{k-1}}A}\ar[rr]\ar_>>>>>>>>{\Delta_{G\times S^{k-1}}}[dd]&
{}\ar@{}[r]&{\bigwedge_{Y_N}A}\ar_>>>>>>>>{\Delta_Y}[dd]\\
&{F(T)}\ar@{}[r]\ar'[r][rr]\ar_-{F(\text{incl})}[dl]&&{F(Z)}\ar[dl]\\
{F({G\times S^{k-1}})}\ar[rr]&&{F(Y)}}
$$
commute.  Hence, the composites around the front square with $\bigwedge_{\text{incl}}A$ are equal.

Induction then gives that we can define $\Delta_Z$ for any cell presentation of a free $G$-space $Z$ and uniqueness also gives that $\Delta_Z$ does not depend on the presentation.  We extend to all cofibrant $G$-spaces by retraction and again see that naturality follows by uniqueness.
\end{proof}

\begin{thm}\label{gocompact}
Let $N$ be the kernel of an isogeny of the $n$-torus $\T^n$.
For an $\Sp$-cofibrant commutative orthogonal ring spectrum $A$, the geometric diagonal of Proposition~\ref{diagsimplicial} applied to the simplicial $N$-set $\sa H\T^n$ of Lemma~\ref{toruscond2} extends uniquely to a natural (in both $A$ and $X$) isomorphism 
\[\xymatrix{\Delta(X,A)\colon\bigwedge_{X_N}A \ar[r]& \Phi^N \bigwedge_XA}\]
of commutative orthogonal $\T^n/N$-ring spectra, where $X$ is a free cofibrant $\T^n$-space.
\end{thm}
\begin{proof}
Firstly, by Lemma~\ref{toruscond}, $N\subseteq\T^n$ is \orfi.  For a finite subgroup $H\subseteq \T^n$ with $N\subseteq H$, Lemma~\ref{toruscond2} gives us a simplicial approximation $\sa H\T^n$ and compatible $H$-equivariant homeomorphisms $D_H\colon|\sa H\T^n|\cong\T^n$ and so Proposition~\ref{diagsimplicial} produces a diagonal isomorphism 
$$\xymatrix{\Delta(\T^n,A)\colon\bigwedge_{\T^n_N}A \ar[r]& \Phi^N \bigwedge_{\T^n}A}$$
 of commutative orthogonal $\faktor{H}{N}$-ring spectra.  Apart from not giving the full $\T^n/N$-action, this is all the input that is needed to evoke Lemma~\ref{lem:Deltaforfree}, so we must show that the action can actually be improved on.  Note that, thanks to the functoriality of Proposition~\ref{diagsimplicial} and the commutative diagram of the very definition of being a simplicial approximation (c.f~Definition~\ref{def:simpap}), the geometric diagonal coming from the simplicial approximations attached to different finite subgroups containing $N$ define the same underlying map and are compatible on the intersection in the following sense.

If $H_1$ and $H_2$ are finite subgroups of the torus, both containing $N$, then 
the two geometric diagonals they define restrict to the same map of $\faktor{H_1 \cap H_2}{N}$-spectra, since the diagonal maps are preserved under restriction (Proposition~\ref{diagrespres}). In particular all of these maps restrict to the same underlying map of spectra. Note that for normal subgroups $N \subset G$ an element $[g]$ in the quotient group $\faktor{G}{N}$ has finite order if and only if the subgroup generated by $g$ intersects {$N$} non trivially, and in particular if and only if the subgroup generated by $\{g\} \cup N$ contains $N$ as a subgroup of finite index. 
This implies that we constructed a map between spectra with $G/N$-action, which is equivariant with respect to the action of all points in $G/N$ that have finite order. Since {$\faktor{G}{N}$} is isomorphic to $G$, we know that the points of finite order are exactly the rational points in $\faktor{G}{N}\cong \T^n$. Since the rational points are dense in $\T^n$, and the actions on $G/N$-spectra are continuous, the map is indeed $G/N$-equivariant.
\end{proof}

Since every element of a compact Lie group lies in a maximal torus,
and in particular contains points of finite order in every one of its
neighborhoods, this argument can be used in more general \orfi setting with simplicial approximations.

\begin{thm}\label{generalliegroup}
Let $G$ be a compact Lie group, and $N$ 
an \orfi normal subgroup. Assume $G$ has simplicial approximations (c.f.~Definition~\ref{def:simpap}) with respect to all finite subgroups containing $N$.
For an $\Sp$-cofibrant commutative orthogonal ring spectrum $A$
and a \naively cofibrant $G$-space $X$, there is an isomorphism of
$\faktor{\G}{N}$-spectra 
$$\Delta(X,A)\colon\xymatrix{\bigwedge_{X_N}A \ar[r]& \Phi^N\bigwedge_XA.}$$ 
The isomorphism is natural and restricts to the diagonal map under the restriction of the action to any finite subgroup of $\faktor{G}{N}$.
\end{thm}
A cute side effect of the geometrical diagonal often being an isomorphism is that we get good equivariance properties with weak assumptions.

\begin{cor}\label{lodaypreservespistar}
Let $G = \T^n$ be the $n$-torus and let $\aml$ be the family of kernels of isogenies and let $X$ be a free cofibrant $\T^n$-space. Then the 
$X$-fold smash power $\bigwedge_X$ sends $\pi_*$-isomorphisms between $\Sp$-cofibrant commutative orthogonal ring spectra to $\pi_*^\aml$-isomorphisms of commutative orthogonal $\T^n$-spectra in the sense of Definition~\ref{def:Mortenamliso}.
\end{cor}
\begin{proof}
Since the $\Sp$-model structure on commutative orthogonal ring spectra is topological, we know that the tensor with a cofibrant space non equivariantly preserves $\pi_*$-isomorphisms between cofibrant commutative $\Sp$-cofibrations. By Theorem~\ref{gocompact}, it therefore also induces non equivariant $\pi_*$-isomorphisms in geometric fixed points with respect to all subgroups $H \in \aml$, so Proposition~\ref{fundcofiber} gives the result.
\end{proof}
Similarly for more general compact Lie groups, Theorem~\ref{generalliegroup} gives the following analog:
\begin{cor}\label{lodaypreservespistarG}
Let $G$ be a compact Lie group and let $\aml$ be a closed family of
normal subgroups that is closed under extensions of finite index, such
all inclusions $H \subseteq G$ with \(H \in \hml\)  are \orfi (c.f~Definition~\ref{cond1})  and 
$G$ \siap (c.f~Definition~\ref{def:simpap}) with respect
to all $H$. Let $X$ be a free cofibrant $G$-space.

Then the 
$X$-fold smash power $\bigwedge_X$ sends $\pi_*$-isomorphisms between
$\Sp$-cofibrant commutative orthogonal ring spectra to
$\pi_*^\aml$-equivalences of commutative orthogonal $G$-spectra. 
\end{cor}
\begin{rem}Analogous results hold for the cases of mere $\Sp$-cofibrant spectra and finite free $G$-sets $X$ as well as spectra $\Sp$-cofibrant under $\Sp$ and infinite free $G$-sets $X$ (cf.~Subsection~\ref{generalx}).
\end{rem}

\begin{remark}\label{comeback}
  An alternative inroad to these results is offered via Theorem~\ref{thm:vsbcd} by \cite{BCD},  which operates with a ``norm cofiber sequence'' \cite[Lemma 5.1.3]{BCD} which in our setup perhaps would be called an ``isotropy separation sequence'' and Corollary~\ref{lodaypreservespistarG} is parallel to \cite[Corollary 5.1.4]{BCD}.  
\end{remark}

\section[Other ground rings, actions on input and isovariance.]{Afterthoughts and generalizations: geometric diagonals for other ground rings, actions on input and the role of freeness.}\label{thm:geodiagoverR}
The existence of the geometric diagonal is very much dependent on special properties of the sphere spectrum.

We now explore what structure an $\Sp$-cofibrant commutative orthogonal ring spectrum $R$ needs to satisfy in order to support geometric diagonals for commutative $R$-algebras, and see that the only thing that is needed is that the $\Sp$-algebra $\Phi^NR$ of geometric fixed points of $R$ (with trivial $G$-action) inherits an $R$-algebra structure compatible with the usual geometric diagonal.
For the sphere spectrum, the $\Sp$-algebra structure comes from the lax monoidality of $\Phi^N$, and the unit map $\Sp\to\Phi^N\Sp$ is even an isomorphism.

For the remainder of the section we fix
\begin{itemize}
\item a compact Lie group $G$ with simplicial approximation, 
\item an orthogonally final normal subgroup $N$ and 
\item an $\Sp$-cofibrant commutative orthogonal ring spectrum $R$ equipped with the trivial action of $G$.
\end{itemize}

If $A$ is a commutative $R$-algebra, then $\bigwedge_X^RA=R\smsh_{\bigwedge_XR}\bigwedge_XA$ is the tensor $X\otimes A$ in the category of commutative $R$-algebras.  The category of commutative $R$-algebras has an ``$R$-model structure'' which is built exactly as the $\Sp$-model structure except that all the generating (acyclic) cofibrations are smashed with $R$.

\begin{dfn}
  A \emph{structure of $N$-geometric diagonals}\index{structure of geometric diagonals} for $R$ is a map of commutative orthogonal ring spectra
  $$\phi_N\colon R\to\Phi^NR$$ such that the diagram
  $$
  \xymatrix{
    \bigwedge_{G/N}R\ar[rr]^-{\Delta(G,R)}\ar[d]^{\text{mult.}}&&\Phi^N(\bigwedge_GR)\ar[d]^{\Phi^N\text{mult.}}\\
    R\ar[rr]^{\phi_N}&&\Phi^NR
  }
  $$
  commutes, where the vertical maps are induced by multiplication in $R$.
\end{dfn}
One could ask that the $\phi_N$ are compatible under restriction of subgroups, but since we will in the following only concern ourselves with the existence of a $R$-geometric diagonal for a given $N$ we do not bother imposing such extra structure.

  We now show how a structure of geometric diagonals $\phi_N\colon R\to\Phi^NR$ for $R$ gives rise to natural ``$R$-geometric diagonal'' maps
  $$\Delta^R(X,A)\colon\bigwedge^R_{X_N}A\to \Phi^N(\bigwedge^R_XA)$$
  of commutative $R$-algebras $A$.
  If $X$ is a free cofibrant $G$-space, then induction on a cell structure gives that
   $$
  \xymatrix{
    \bigwedge_{X_N}R\ar[rr]^-{\Delta(X,R)}\ar[d]^{\text{mult.}}&&\Phi^N(\bigwedge_XR)\ar[d]^{\Phi^N\text{mult.}}\\
    R\ar[rr]^{\phi_N}&&\Phi^NR
  }
  $$
  commutes, which ensures that the map $\phi_N\smsh \Delta(X,A)$ appearing in the following definition is well-defined.  
  \begin{dfn}\label{def:R-geodiag}
    If $\phi_N\colon R\to\Phi^NR$ is a structure of geometric diagonals for $R$, the associated \emph{$R$-geometric diagonal} $\Delta^R(X,A)$ is the natural (natural in the $R$-cofibrant commutative $R$-algebra $A$ and the free cofibrant $G$-space $X$) map defined as the composite
    $$\xymatrix{\bigwedge^R_{X_N}A\ar@{=}[d]\ar[rrr]^{\Delta^R(X,A)}&&&
      \Phi^N(\bigwedge^R_XA)\ar@{=}[d]
      \\
      R\smsh_{\bigwedge_{X_N}R}\bigwedge_{X_N}A\ar[rr]^-{\phi_N\smsh\Delta(X,A)}&&
      \Phi^NR\smsh_{\Phi^N(\bigwedge_XR)}\Phi^N(\bigwedge_XA)\ar[r]^-\alpha&\Phi^N(R\smsh_{\bigwedge_XR}\bigwedge_XA).
      }$$
    \end{dfn}

    \begin{remark}
      The $R$-geometric diagonal may or may not be an isomorphism or stable equivalence depending on properties of $R$ and $A$.
      The assumptions on $R$ and $A$ assure that the geometric diagonals $\Delta(X,R)$ and $\Delta(X,A)$ will be isomorphisms (and also that all the smash products involved are sensible homotopically), and so the question hinges on the properties of $\phi_N\colon R\to\Phi^NR$ and the associative structure $\alpha$ on the geometric fixed points.
      
      For $R=\Sp$ (which admittedly is not $\Sp$-cofibrant, a fact not affecting the argument) both $\phi_N$ and $\alpha$ are isomorphisms, reclaiming our result that the geometric diagonal for commutative $\Sp$-algebras is an isomorphism. 
    \end{remark}

\begin{remark}
  An avenue we feel guilty about not having explored (beyond that we set up the basic machinery for doing so in Section~\ref{subsectcellularfiltrations}) is the case when $A$ (and $R$?) already comes equipped with group actions.  This {\em is} important for some applications and it is mainly a result of fatigue that we have not rewritten everything to accommodate for this.  The pieces we have explored to full depth all turned out to generalize nicely.
\end{remark}

\begin{remark}
  Lastly, the above is about smash powers over free $G$-spaces. The action being free simplifies some matters considerably, and fits with the applications we have in mind.  As has been investigated by the second author, the statements we have made are for the most part false without this assumption, but can be extended to the isovariant case.
\end{remark}

\chapter{Category Theory}
\label{ch:cat}
We recall some of the basics of category theory. We assume that the
reader is familiar with the notions in this chapter, but the explicit
definitions allow for an easier transition to monoidal and enriched
categories. The canonical reference and source
for these definitions is Chapter I of \cite{McL}, though we have
allowed ourselves some reformulations for the sake of uniformity when
switching to the enriched setting. 

\label{catprod}
If $\catC$ and $\catD$ are categories with $\catC$ small, then the \emph{functor category} $\cat(\catC,\catD)$, has as objects the class of functors $\catC\to\catD$ and the morphisms are the natural transformations.

\section{Monoidal Categories}
We repeat the basic definitions as far as they will be used in the enriched setting in the next section. The definitions are only slight reformulations of the ones in \cite[VII]{McL}, adapted to our needs. 
\begin{dfn}
A \emph{monoidal category}\index{monoidal category} consists of the following data:
\begin{indentpar}{1 cm}
\begin{itemize}
\item An category $\catC$.
\item A bifunctor (\ie a functor out of the product category) $\otimes\colon\catC \times \catC \rightarrow \catC$, called the monoidal product. 
\item An object $\I$ of $\catC$, called the identity object.
\item Natural isomorphisms $\lambda\colon (\I\otimes \id) \rightarrow \id$ and $\rho\colon (\id\otimes\I) \rightarrow \id$ expressing that $\I$ is a left and right identity object for the monoidal product.
\item A natural isomorphism $a \colon [(\id \otimes \id) \otimes \id]
  \rightarrow [\id \otimes (\id \otimes \id)]$ expressing that the
  monoidal product is associative.  
\end{itemize}
The natural transformations have to satisfy the following two coherence conditions:\end{indentpar}
\begin{indentpar}{1 cm}
\begin{itemize}
\item For all objects $A$ and $B$ of $\catC$, the following diagram commutes:\[\xymatrix{{(A\otimes \I) \otimes B}\ar^-{a_{A,I,B}}[rr]\ar_-{\rho_A \otimes \id_B}[dr]&{}&{A \otimes (\I \otimes B)}\ar^-{\id_A \otimes \lambda_B}[dl]\\{}&{A \otimes B}}\]
\item For all objects $A$, $B$, $C$ and $D$ of $\catC$, the following pentagon commutes:
\[\tiny\xymatrix@1@=0pt@R=1cm{{}&{}&{(A\otimes (B\otimes C))\otimes D}\ar_-{a_{A,B\otimes C,D}}[drr]\\{((A\otimes B)\otimes C)\otimes D}\ar^-{a_{A,B,C}\otimes \id_D}[urr]\ar_-{a_{A\otimes B,C,D}}[dr]&{}&{}&{}&{A\otimes((B\otimes C)\otimes D)}\ar^{id_A\otimes a_{B,C,D}}[ld]\\{}&{(A\otimes B)\otimes (C\otimes D)}\ar_{a_{A,B,C\otimes D}}[rr]&{}&{A\otimes (B \otimes (C \otimes D))}}\]
\end{itemize}
\end{indentpar}
Instead of the tuple $(\catC,\otimes,\I,\lambda,\rho,a)$, we often just refer to the monoidal category as $(\catC,\otimes,\I)$ or even just to $\catC$, when it is clear which monoidal structure is meant.
\end{dfn}
\begin{dfn}
  \label{def:strict}\index{monoidal category!strict}
  A monoidal category $ (\catC,\otimes,\I,\lambda,\rho,a)$ is {\em strict} if $\lambda$, $\rho$ and $a$ are all identity morphisms.
\end{dfn}

\begin{dfn}
A \emph{lax monoidal functor}\index{monoidal!functor!lax}\index{functor!lax monoidal} $\FF \colon (\catC,\otimes,\I) \rightarrow (\catD,\times,\II)$ between monoidal categories consists of the following data:
\begin{indentpar}{1cm}
\begin{itemize}
\item A functor $\FF \colon \catC \rightarrow \catD$.
\item A natural transformation $\mu \colon [\FF \times \FF] \rightarrow \FF(- \otimes -)$.
\item A morphism $\iota \colon \II \rightarrow \FF(\I)$ in $\catD$.
\end{itemize}
\end{indentpar}
These have to satisfy the following coherence conditions:
\begin{indentpar}{1cm}
\begin{itemize}
\item For all objects $A$, $B$ and $C$ of $\catC$, the following diagram commutes in $\catD$:
\[\xymatrix{{(\FF(A)\times \FF(B))\times \FF(C)}\ar_-{\mu_{A,B}\times \id}[d]\ar^-{a_{\catD}}[r]&{\FF(A)\times (\FF(B)\times \FF(C))}\ar^-{\id \times \mu_{B,C}}[d]\\{\FF(A\otimes B)\times \FF(C)}\ar_-{\mu_{A\otimes B,C}}[d]&{\FF(A)\times (\FF(B\otimes C))}\ar^-{\mu_{A,B\otimes C}}[d]\\{\FF((A \otimes B)\otimes C)}\ar_-{\FF(a_\catC)}[r]&{\FF(A\otimes(B\otimes C))}}\]
\item For every object $A$ of $\catC$, the following diagrams commute in $\catD$:
\[\xymatrix{{\FF(A)\times \II}\ar^{\rho_\catD}[r]\ar_-{\id \times \iota_\catD}[d]&{\FF(A)}&{}&{\II \times \FF(A)}\ar^{\lambda_\catD}[r]\ar_-{\iota_\catD \times\id }[d]&{\FF(A)}\\{\FF(A)\times \FF(\I)}\ar_-{\mu_{A,\I}}[r]&{\FF(A \otimes \I)}\ar_-{\FF(\rho_\catC)}[u]&{}&{\FF(\I)\times \FF(A)}\ar_-{\mu_{\I,A}}[r]&{\FF(I \times A)}\ar_-{\FF(\lambda_\catC)}[u]}\]
\end{itemize}
\end{indentpar}
A lax monoidal functor $(\FF,\mu,\iota)$ is \emph{strong monoidal}\index{monoidal!functor!strong}\index{functor!strong monoidal} if $\mu$ and $\iota$ are (natural) isomorphisms, it is \emph{strict monoidal}\index{monoidal!functor!strict}\index{functor!strict monoidal} if they are the identity (transformation).\\
Again we often only refer to $\FF$ as the monoidal functor, suppressing $\mu$ and $\iota$ in the notation, where they are not critical to the discussion.
\end{dfn}
\begin{dfn}
A monoidal category $(\catC,\otimes,\I)$ is called {\em cartesian}\index{monoidal!category!cartesian}, if $\otimes$ is the categorical product and $\I$ is a terminal object.
\end{dfn}
Some monoidal categories have additional extra structure:
\begin{dfn}\label{homadj}
A monoidal category $(\catC,\otimes,\I)$ is called \emph{closed},\index{monoidal!category!closed} if for all objects $A$ of $\catC$, the functor $(- \otimes A)\colon \catC \rightarrow \catC$ has a right adjoint (\cite[IV.1]{McL}), denoted  by $\Hom(A,-)$. \\
Objects of $\catC$ the form $\Hom(A,B)$ are called {\em internal $\Hom$ objects},\index{internal $\Hom$ object} the counits of these adjunctions are usually called the evaluations $
\Hom(A,B)\otimes A \rightarrow B$.
\end{dfn}
\begin{lemma} \label{homhomadj}
If $(\catC,\otimes,\I)$ is closed monoidal, then there is a natural isomorphism:
\[\Hom(A\otimes B, C)\cong\Hom(A, \Hom(B,C)).\]
\end{lemma}
  \begin{dfn}\label{underlyset}
    If $(\catC,\otimes,\I)$ is a monoidal category, the {\em underlying set functor} is the set-valued functor $\catC(\I,-)\colon\catC \rightarrow \catset$\index{Set@$\catset$, sets}.  
With the cartesian monoidal structure on $\catset$, the underlying set functor is lax monoidal with the following structure: 
The unit morphism $\iota$ sends the terminal one-point set to the identity morphism of $\I$ and the natural transformation \[\mu \colon \catC(\I,A)\times \catC(\I,B)\stackrel{\otimes}{\longrightarrow} \catC(\I\otimes\I,A\otimes B) \cong \catC(\I,A\otimes B),\] uses the isomorphism $\lambda_\I\colon \I\otimes\I \rightarrow\I$. If $C$ is an  objects of $\catC$ we refer to $\catC(\I,C)$ as the \emph{underlying set} of $C$.
  \end{dfn}
If, in addition, $\catC$ is closed and $A$ and $B$ are objects of $\catC$, then the adjunction \[(-\otimes A)\colon\catC\rightleftarrows\catC\colon\Hom(A,-)\] gives the natural isomorphism
\[\catC(\I\otimes A,B)\cong\catC(\I,\Hom(A,B)).\]
Since $\I\otimes A$ is isomorphic to $A$ via $\lambda_A$, this implies that \[\catC(A,B) \cong \catC(\I,\Hom(A,B)),\] \ie the underlying set of the internal $\Hom$ object $\Hom(A,B)$ is naturally isomorphic to the morphism set $\catC(A,B)$.\\
Considerations in this spirit lead to the study of enriched categories. We will discuss these further in Section~\ref{enrichedcat}.

For any monoidal category, there are categories of \emph{monoids}\index{monoid} and (\emph{left} or \emph{right}) \emph{modules}\index{module} over such. Definitions can for example be found in \cite[VII.3,4]{McL}, and will be omitted here.\section{Symmetric Monoidal Categories}
\label{tensors}
In general, the functor $\twist\colon\catC\times\catC\to\catC\times\catC$ permutes the input.
\begin{dfn}
A \emph{symmetric monoidal category}\index{monoidal category!symmetric} is
\begin{itemize}
\item a monoidal category $(\catC, \otimes, \I, \lambda, \rho, a)$, together with 
\item a natural isomorphism $\tau$
\[\tau: \otimes \rightarrow \otimes \circ \twist, \qquad\tau_{A,B}\colon A\otimes B\cong B\otimes A\]
\end{itemize}
  satisfying the following conditions:
\begin{indentpar}{1cm}
\begin{itemize}
\item The composition of $\tau$ with itself is the identity, \ie \[\tau_{B,A} \circ \tau_{A,B} = \id_{A\otimes B},\] for all objects $A$ and $B$ of $\catC$.
\item Compatibility with the unit, \ie \[\rho_A = \lambda_A \circ \tau_{A,\I}.\]
\item For all objects $A$, $B$, and $C$ of $\catC$, the following hexagon commutes:
\[\xymatrix{{(A \otimes B) \otimes C}\ar^-{a}[r]\ar_-{\tau \otimes \id_C}[d]&{A \otimes (B \otimes C)}\ar^-{\tau}[r]&{(B \otimes C) \otimes A}\ar^-{a}[d]\\{(B \otimes A) \otimes C}\ar_-{a}[r]&{B \otimes (A \otimes C)}\ar_-{\id_B \otimes \tau}[r]&{B \otimes ( C\otimes A)}}\]
\end{itemize} 
\end{indentpar}
\end{dfn}
Note that all cartesian or cocartesian monoidal categories are
symmetric with monoidal product given by cartesian product and
coproduct respectively
  \begin{dfn}\label{def:permutative}\index{permutative category}
    A {\em permutative category} is a symmetric monoidal category whose underlying monoidal category is strict in the sense of Definition~\ref{def:strict}.
  \end{dfn}

\begin{dfn}
A \emph{commutative monoid}\index{monoid!commutative} in a symmetric monoidal category $(\catC, \otimes, \I, \tau)$ consists of the following data:
\begin{indentpar}{1cm}
\begin{itemize}
\item An object $M$ of $\catC$.
\item A morphism $\eta: \I \rightarrow M$ in $\catC$, called the unit of $M$.
\item A morphism $\mu: M \otimes M \rightarrow M$ in $\catC$, called the multiplication of $M$.
\end{itemize}
\end{indentpar}
such that the following diagrams are commutative:
\begin{indentpar}{1cm}
\begin{itemize}
\item \emph{unit:}\[\xymatrix{{\I \otimes M}\ar^-{\eta \otimes \id_M}[r]\ar_-{\lambda}[dr]&{M \otimes M}\ar_-{\mu}[d]&{M \otimes \I}\ar_-{\id_M \otimes \eta}[l]\ar^-{\rho}[dl]\\&{M}}\]
\item \emph{associativity:} \[\xymatrix{{(M \otimes M) \otimes M}\ar_-{\mu \otimes \id_M}[d]\ar^-{a}[r]&{M \otimes (M\otimes M)}\ar^-{\id_M \otimes \mu}[r]&{M \otimes M}\ar^-{\mu}[d]\\{M\otimes M}\ar_-{\mu}[rr]&{}&{M}}\]
\item \emph{commutativity:}
\[\xymatrix{{M \otimes M}\ar_-{\mu}[dr]\ar^-{\tau}[rr]&{}&{M\otimes M}\ar^-{\mu}[dl]\\{}&{M}}\]
\end{itemize}
\end{indentpar}
\end{dfn}
Like in the non-symmetric situation 
we can define categories of commutative monoids and module categories
of commutative categories. We work extensively with these in the case
of $\catC$ being the category $\catOS$ of orthogonal spectra. The
following lemmas are well known, but it seems hard to find explicit
references: 
\begin{lemma}
Let $M$ be a commutative monoid, then the categories of left $M$-modules and right $M$-modules are isomorphic.
\end{lemma}
If $M$ is a commutative monoid in a symmetric monoidal category $\catC$ with coequalizers and $V$ and $W$ are $M$-modules, define $V \otimes_M W$ of as the coequalizer
\labeleq{smashcoeq}{{V \otimes M \otimes W} \rightrightarrows {V \otimes W}\rightarrow{V \otimes_M W,} }
where one of the arrows uses the action map on $V$, and the other the action on $W$ precomposed with the twist $V \otimes \tau_{M,W}$.
Similarly,  if $\catC$ is closed and has equalizers, the internal $\Hom$ object $\Hom_M(V,W)$ is the equalizer
\[{\Hom_M(V,W)}\rightarrow{\Hom(V,W)}\rightrightarrows{\Hom(V \otimes M, W),} \]
where one of the arrows is induced by the action map of $V$, and the other one is induced by the adjoint of the action map of $W$, using the isomorphism $\Hom(V \otimes M, W) \cong \Hom(V, \Hom(M,W))$ (cf.~\ref{homhomadj}).
\begin{lemma}\label{monoidalmodules}
Let $M$ be a commutative monoid in the closed symmetric monoidal category $(\catC,\otimes,\I,\tau)$. Assume that $\catC$ has equalizers and coequalizers. Then the category of (right) $M$-modules with $\otimes_R$ and $\Hom_M$ is a closed symmetric monoidal category.  
\end{lemma}
\begin{lemma} \label{commcoprod}
Let $(\catC,\otimes,\I)$ be a symmetric monoidal category. Then $\otimes$ is the coproduct in the category of commutative monoids in $\catC$.
\end{lemma}

We will discuss monoids, (commutative) algebras and modules over such in various (symmetric) monoidal categories $(\catC, \otimes, \I)$. Often we use that the forgetful functors to $\catC$ have left adjoints, and hence we recall how these adjoints are formed in general:
\begin{lemma}\label{freealg}Let $R$ be a monoid in $(\catC,\otimes,\I)$.  Then the functors
\begin{itemize}
\item  $-\otimes R$ from $\catC$ to right $R$-modules
\item  $R\otimes -$ from $\catC$ to left $R$-modules
\end{itemize}
are left adjoint to the forgetful functors.
If $R$ is commutative and the category of $R$-modules has coproducts, then  \begin{itemize}
\item the functor $\Ass \defas{} \coprod_{i \in \N} (-)^{\otimes_R i}$ is left adjoint to the forgetful functor from $R$-algebras to $R$-modules,  where $(-)^{\otimes_R i}$ is the $i$-fold tensor power over $R$, with the convention that $M^{\otimes_R0}=R$.
\end{itemize} If the category of $R$ modules is cocomplete, then \begin{itemize}
\item the functor $\Comm \defas{} \coprod_{i \in \N} [(-)^{\otimes_R i}]_{\Sigma_i}$ is left adjoint to the forgetful functor from commutative $R$-algebras to $R$-modules, where $[-]_{\Sigma_i}$ is the orbits of the action of $\Sigma_i$ that permutes tensor factors, \ie the action induces a functor from $\Sigma_i$ viewed as a one-object category (cf.~\ref{Gset}) and $[-]_{\Sigma_i}$ denotes its colimit. 
\end{itemize} In both cases multiplication is by 
concatenating coproduct factors and the unit map is the inclusion of $R$ as the factor indexed by $0$.
\end{lemma}
We could of course have given each of these functor in terms of the monads that the unit of the adjunction induces on $\catC$.

\subsubsection{Categories of Finite Sets}\label{sec:fin}

\begin{dfn}\label{def:fin}
  Let $\Fin$\index{Fin@$\Fin$, finite sets} be the category of finite sets, and consider the skeleton
$\skFin\subseteq\Fin$\index{Fin@$\skFin$, skeleton of $\Fin$} consisting of the finite sets
$\bn=\{1,2,\dots,n\}$\index{n@$\bn=\{1,2,\dots,n\}$} for $n\in\bN$. Disjoint union is modelled
by 
$$\bm+\bn=\{1,\dots,m+n\}$$
with structure maps $\bm\to\bm+\bn$ sending $i$ to $i$, whereas $\bn\to\bm+\bn$ sends $j$ to $n+j$, giving a permutative structure. 
\end{dfn}
Choose, once and for all,
an inverse strongly symmetric monoidal equivalence $\Fin\to\skFin$
(\ie choose an ordering for each finite set; the associated natural
isomorphisms are forced).  For convenience we choose this equivalence
so that finite sets of integers retain their order. 

\begin{dfn}
  Let $\Sigma\subseteq\bo\subseteq\skFin$\index{Sigma@$\Sigma\subseteq\skFin$}\index{J@$\bo\subseteq\skFin$} be the permutative subcategories of respectively bijections and injections. 
\end{dfn}
 Note that any injection $\phi\colon\bn\to\bm$ can uniquely be factored as the inclusion $\bn\subseteq\bm$ composed with a bijection $\bm\cong\bm$.  

Let $\cC$ be a symmetric monoidal category with monoidal product $\otimes$ and neutral element $e$. Let $e/\cC$ be the category of objects under the neutral element $e$ and $\comC$\index{comC@$\comC$} the category of symmetric monoids in $\cC$.

Coherence for $\cC$ amounts to saying that the assignment 
$$(\bn,c)\mapsto c^{\otimes\bn}=(\dots(c\otimes c)\otimes\dots c)\otimes c$$ ($n$ copies of $c$ with parentheses moved as far left as possible) defines functors that are strong symmetric monoidal in each variable
$$\Sigma\times\cC\to\cC,\quad \bo\times e/\cC\to e/\cC,\quad \Fin\times\comC\to\comC.$$

Since this is central to our discussion we spell out some of the details, but for clarity restrict ourselves to a permutative $\cC$ (so that the units and associators are identities).
\begin{dfn}
  If $\sigma\in\Sigma_n$ is a bijection, the commutator associated with $\sigma$ is an isomorphism $c^{\otimes\sigma}\colon c^{\otimes\bn}\cong c^{\otimes\bn}$ (remember that $\otimes$ is strictly associative).  This defines the functor $\Sigma\times\cC\to\cC$, which is strong symmetric monoidal in each factor via $c^{\otimes (\bm+\bn)}= c^{\otimes\bm}\otimes c^{\otimes\bn}$ and the shuffle commutator $(c\otimes d)^{\otimes\bn}\cong c^{\otimes\bn}\otimes d^{\otimes\bn}$.  

For the functor $\bo\times e/\cC\to e/\cC$, the object $(\bn,f\colon e\to c)$ in $\bo\times e/\cC$ is sent to the composite $f^{\otimes \bn}\colon e= e^{\otimes\bn}\to c^{\otimes\bn}$.  On morphisms in $\bo$ we only lack the (first) inclusions $\bm\subseteq\bm+\bn$ which is sent to the morphism under $e$
$$\id^{\otimes\bm}\otimes f^{\otimes\bn}\colon c^{\otimes\bm}=c^{\otimes\bm}\otimes e^{\otimes\bn} \to c^{\otimes\bm}\otimes c^{\otimes\bn}=c^{\otimes(\bm+\bn)}.$$

For the functor $\Fin\times\comC\to\comC$, we first define a $\skFin\times\comC\to\comC$ compatible with the first two definitions and then use our chosen equivalence $\Fin\to\skFin$ to
extend these functors to a functor defined on $\Fin$.
If $c$ is a symmetric monoid in $\cC$ with structure maps $\mu\colon c\otimes c\to c$ and $e\to c$, we define a morphism
$c^{\otimes f}\colon c^{\otimes \bm}\to c^{\otimes \bn}$ for any function $f\colon\bm\to \bn\in\skFin$ as the composite
$$c^{\otimes\bm}\cong c^{\otimes f^{-1}(1)}\otimes\dots\otimes c^{\otimes f^{-1}(n)}\to c^{\otimes\bn},$$
where the isomorphism is obtained by multiple applications of the
functor \(\tau\)  and the second map is the tensor of the
multiplications $\mu\colon c^{\otimes f^{-1}(j)}\to c$  (unit if
$f^{-1}(j)$ is empty).  Lastly, we extend to $\Fin$.
\end{dfn}

\begin{remark}
  Since small colimits can be chosen functorially, any category with finite coproducts is tensored (c.f.~Definition~\ref{tenscotens}) over $\Fin$.  In particular, the coproduct in $\comC$ is $\otimes$, which can easily cause a conflict of notation.  Hence we retain the notation $c^{\otimes S}$ (or $\bigotimes_S c$) for this coproduct indexed over the finite set $S$, and let this be our {\em choice} of tensor.  

In our applications, $\cC$ is cocomplete and closed.  Then we can extend $\Fin\times\comC\to\comC$ to a functor (strong monoidal in each variable)
$$\Ens\times\comC\to\comC, \qquad 
S\mapsto \bigotimes_Sc=\colim_{T\subseteq S}\bigotimes_Tc,$$
where $T$ varies over the finite subsets of the set $S$, which agrees with the cotesor of $\comC$ over $\Ens$. 
\end{remark}

\section{Enriched Category Theory} \label{enrichedcat}
Let $(\catV,\otimes,\I)$ be a monoidal category.
\begin{dfn}
A \emph{category $\catC$ enriched over \catV},\index{enriched!category}\index{category!enriched} or a {\em $\catV$-category}\index{Vcategory@$\catV$-category}\index{category!V@$\catV$-}
\cite[1.2]{Kel} amounts to the following structure:
\begin{indentpar}{1cm}
\begin{itemize}
\item A class $\Ob(\catC)$ of \emph{objects of $\catC$}.
\item For every two objects $A$ and $B$ of $\catC$ an object $\catC(A,B)$ of $\catV$ called the Hom-object of $A$ and $B$.
\item For every object $A$ of $\catC$, a distinguished morphism $\id_A\colon \I \rightarrow \catC(A,A)$ in $\catV$, called \emph{the identity of $A$}.
\item For every three objects $A$, $B$ and $C$ of $\catC$ a morphism $\gamma\colon\catC(B,C)\otimes\catC(A,B)\rightarrow \catC(A,C)$, called \emph{the composition in $\catC$}.
\end{itemize}
\end{indentpar}
This data has to satisfy the following two conditions:
\begin{indentpar}{1cm}
\begin{itemize}
\item For all objects $A$,$B$,$C$ and $D$ of $\catC$, the following diagram commutes in $\catV$:
\[\tiny\xymatrix{{(\catC(C,D)\otimes\catC(B,C))\otimes \catC(A,B)}\ar^-{a}[rr]\ar^-{\gamma \otimes \id}[d]&{}&{\catC(C,D)\otimes(\catC(B,C)\otimes\catC(A,B))}\ar_-{\id \otimes \gamma}[d]\\{\catC(B,D)\otimes\catC(A,B)}\ar_-{\gamma}[dr]&{}&{\catC(C,D)\otimes\catC(A,C)}\ar^-{\gamma}[dl]\\{}&{\catC(A,D)}}\]
\item For all objects $A$ and $B$ in $\catC$, the following diagram commutes:
\[\xymatrix{{\I\otimes \catC(A,B)}\ar_-{\id_B\otimes \id}[d]\ar^-{\lambda}[rd]&{}&{\catC(A,B)\otimes \I}\ar^{\id \otimes \id_A}[d]\ar_-{\rho}[dl]\\{\catC(B,B)\otimes \catC(A,B)}\ar_-{\gamma}[r]&{\catC(A,B)}&{\catC(A,B)\otimes \catC(A,A)}\ar^-{\gamma}[l]}\]
\end{itemize}
\end{indentpar}
\end{dfn}
\begin{example}\label{trivialVcat}
The \emph{trivial $\catV$-category} $\star$ has one object $C$, and the morphism object $\star(C,C) = \I$. 
\end{example}
\begin{dfn}
A \emph{functor enriched over $\catV$}\index{functor!enriched}\index{enriched!functor}\index{functor!V@$\catV$} from $\catD$ to $\catC$ consists of the following data:
\begin{indentpar}{1cm}
\begin{itemize}
\item A function $\FF\colon \Ob(\catD)\rightarrow \Ob(\catC)$.
\item For each object $A$ and $B$ of $\catD$, a morphism $\FF_{A,B}\colon \catD(A,B)\rightarrow \catC(\FF(A),\FF(B))$ in $\catV$.
\end{itemize}
\end{indentpar}
These have to satisfy the following coherence conditions:
\begin{indentpar}{1cm}
\begin{itemize}
\item \emph{(identity)} For all objects $A$ of $\catD$ the following diagram commutes in $\catV$: \[\xymatrix{{}&{\catD(A,A)}\ar^-{\FF_{A,A}}[dd]\\{\I}\ar_-{\id_{\FF(A)}}[dr]\ar^-{\id_A}[ur]&{}\\{}&{\catC(\FF(A),\FF(A))}}\]
\item \emph{(composition)} For all objects $A$, $B$ and $C$ of $\catD$ the following diagram commutes in $\catV$:
\[\xymatrix{{\catD(B,C)\otimes\catD(A,B)}\ar^-{\gamma}[r]\ar_-{\FF \otimes \FF}[d]&{\catD(A,C)}\ar^-{\FF}[d]\\{\catC(\FF(B),\FF(C))\otimes\catC(\FF(A),\FF(B))}\ar_-{\gamma}[r]&{\catC(\FF(A),\FF(C))}}\]
\end{itemize}
\end{indentpar}
\end{dfn}
\begin{dfn}
For two functors $\FF,\GG \colon \catD \rightarrow \catC$ enriched over $\catV$, an \emph{enriched natural transformation $\alpha\colon \FF \rightarrow \GG$}\index{enriched!natural transformation} consists of morphisms $\I \rightarrow \catC(\FF(A),\GG(A))$ in $\catV$ for all objects $A$ of $\catD$, such that the following coherence diagram commutes in $\catV$ for all objects $A$ and $B$ of $\catD$:
\labeleq{enrnattransdiag}{\xymatrix{{}&{\I\otimes\catD(A,B)}\ar^-{\alpha_B\otimes\FF}[r]&{\catC(\FF(B),\GG(B))\otimes\catC(\FF(A),\FF(B))}\ar^-{\gamma}[dr]&{}\\{\catD(A,B)}\ar^-{\lambda^{-1}}[ur]\ar_-{\rho^{-1}}[dr]&{}&{}&{\catC(\FF(A),\GG(B))}\\{}&{\catD(A,B)\otimes \I}\ar_-{\GG\otimes\alpha_A}[r]&{\catC(\GG(A),\GG(B))\otimes \catC(\FF(A),\GG(A))}\ar_-{\gamma}[ur]&{}}}
\end{dfn}
\begin{rem} \label{2-cat}
This definition gives the class $[\catD,\catC]_0(\FF,\GG)$ of functors enriched over $\catV$ the structure of a category (one checks that there is an identity transformation and that composition of natural transformations is associative). This makes the category $\catV$-$\cat$\index{VCat@$\catV$-$\cat$} of $\catV$-enriched categories and $\catV$-enriched functors into a $2$-category, \ie into a category enriched over $\cat$.
\end{rem}
\begin{rem}\label{enrfunctorcat}
Under more assumptions, one can also define a $\catV$-enriched functor
category $[\catD,\catC]$. Let $\catV$ be closed and complete and
$\catD$ be
a small \(\catV\)-category. Then for two enriched
functors $\FF$ and $\GG$, the following end exists and forms the
morphism $\catV$-space $[\catD,\catC](\FF,\GG)$: 
\[\int_{d \in \catD} \catC(\FF(d), \GG(d)) \rightarrow \prod_{d \in \catD} \catC(\FF(d), \GG(d)) \rightrightarrows \prod_{d,d' \in \catD}\Hom(\catD(d, d'),\catC(\FF(d), \GG(d'))).\]
As indicated it can be expressed as the equalizer along two maps adjoint to the two ways around diagram \eqref{enrnattransdiag} above. Composition and identities are then inherited from $\catC$ (cf.~\cite[2.1]{Kel}).
\end{rem}

\begin{con} \label{enrtransport}
Note that if we have a lax monoidal functor $(\FM,\mu,\iota)\colon(\catV,\otimes,\I)\rightarrow (\catW,\boxtimes,\II)$, any category $\catC$ enriched over $\catV$ gives a category enriched over $\catW$, by just applying $\FM$ to all the Hom objects. The identity morphisms are defined as the composites \[\id'_{A}\colon \II \stackrel{\iota}{\rightarrow} \FM[\I]\stackrel{\FM[\id_A]}{\longrightarrow}\FM[\catC(A,A)].\] The composition is given by \[\gamma' \colon \FM[\catC(B,C)]\boxtimes\FM[\catC(A,B)] \stackrel{\mu}{\rightarrow}\FM[\catC(B,C)\otimes\catC(A,B)] \stackrel{\FM[\gamma]}{\longrightarrow} \FM[\catC(A,C)].\] One checks that the coherence diagrams still commute.\\
Also, in the same way $\catV$-enriched functors give $\catW$-enriched
functors and $\catV$-enriched natural transformations give
$\catW$-enriched natural transformations via the lax monoidal functor $\FM$. (One checks that $\FM(\FF)$ still takes identities to identities and respects composition, and that the appropriate diagram for $\FM\alpha$ still commutes using the structure maps of $\FM$).
\end{con}
\begin{rem}
In the spirit of the above Remark~\ref{2-cat}, one can check that $\FM$ induces a $\cat$-enriched, or $2$-functor $\FM\colon\catV$-$\cat\rightarrow \catW$-$\cat$, \ie that $\FM$ takes the identity $\catV$-enriched natural transformations to the identity $\catW$-enriched natural transformations, and that it respects composition of enriched natural transformations.
\end{rem}
\begin{ex} \label{underlyingcat}
In this way, if $\catV$ is a locally small monoidal category, every category enriched over $\catV$ has a canonical underlying ``normal'' category, \ie one enriched over $\catset$, with the same objects. The morphism sets are obtained by using the monoidal functor $\catV(\I,-)$ from Definition~\ref{underlyset} in the way described above.
\end{ex}
\begin{rem}\label{underlfunctcat}
For $\catD$ and $\catC$ categories enriched over $\catV$ as in Remark~\ref{enrfunctorcat}, the underlying underlying $\catset$-category of the functor $\catV$-category $[\catD,\catC]$ is $[\catD,\catC]_0$, if the former exists.
\end{rem}
\begin{ex} \label{selfenriched}
Let $(\catV_0,\otimes,\I)$ be a closed monoidal category. Then there is a $\catV_0$-category $\catV$ that restricts to $\catV_0$ along the monoidal functor $\catV(\I,-)$.\\
Define $\catV$ as having the same objects as $\catV_0$ and for morphism objects set $\catV(A,B) = \Hom(A,B)$. Then composition is adjoint to iterated evaluation, and the axioms for an enriched category trivially hold. When discussing a specific category $\catV_0$, we will often identify $\catV$ and $\catV_0$ and therefore say that $\catV$ is enriched over itself, but there are also important cases where we explicitly keep the notation separate (\eg~\ref{gspaces}).
\end{ex}
\begin{rem}
When viewing $\catV$ as enriched over itself in this sense, Lemma~\ref{homhomadj} can be reformulated to state that the adjunctions between $-\otimes A$ and $\Hom(A,-)$ are actually enriched, \ie imply natural isomorphisms even on morphism objects.
\end{rem}
\begin{ex} \label{catT}
Let $\catV = \catTop$\index{Top@$\catTop$} the cartesian monoidal
category of topological spaces. Then a category $\catC$ enriched over
$\catTop$ is a usual category, with a choice of topology on each
morphism set, such that the composition law gives continuous
maps. More important for us is the closed monoidal variation
$\catU$,\index{U@$\catU$, compactly generated weak Hausdorff spaces}
containing only the \emph{compactly generated weak Hausdorff
  spaces.}\\ 
For another example let $\catV$ be the category \emph{$\catT$ of based
  compactly generated weak Hausdorff spaces},\index{T@$\catT=*/\catU$}
\ie objects of $\catU$ with a distinguished base point. We will usually
drop the extra adjectives and just call these \emph{spaces}.\\
Since
$\catT$ has products and coproducts, it is monoidal in several ways:
with the cartesian product $\times$ and unit a one point space
$\{*\}$, or, more importantly for us, with respect to the smash
product $\smash$ and unit $S^0$, the $0$-sphere. The latter choice
makes $\catT$ closed monoidal, and we will denote the internal $\Hom$
spaces merely as $\catT(-,-)$ in agreement with
Example~\ref{selfenriched}. The identity functor $(\catT,\smash,S^0)
\rightarrow (\catT,\times,\{*\})$ is lax monoidal, just as the functor
$(\catT,\smash,S^0) \rightarrow (\catU, \times, \{*\})$ that forgets
the base points. The monoidal 
structure maps are in both cases given by the projections $X\times Y \rightarrow X
\smash Y$ and the inclusion of $\{*\}$ as the non-base point of
$S^0$. These functors give us a canonical way to view a category
enriched over $(\catT,\smash,S^0)$ as one enriched over
$(\catT,\times,\{*\})$, or $\catU$. 
The forgetful functor from $\catU$ to $\catset$ preserves products and
is therefore strict monoidal, indeed it is isomorphic to the functor
described in Definition~\ref{underlyset}. Hence a category enriched over either
monoidal structure on $\catT$ (or $\catU$) is a category. In the other
directions, including sets as discrete topological spaces and adding
disjoint base points to spaces in $\catU$ give left adjoints to the
forgetful functors and are also (strong) monoidal. Hence together with
Example~\ref{selfenriched} we can view $\catU$ and $\catT$ as enriched over
either themselves or each other. Generally, categories enriched over
any of the above are called \emph{topological
  categories}.\index{topological!category}\index{category!topological}
Enriched functors between  $\catTop$-, $\catU$- or
$\catT$-categories are usually called {\em continuous
  functors}.\index{continuous!functor}\index{functor!continuous}\\ 
\end{ex}
\begin{dfn} \label{Gset}
  For $\G$ a group, the
  \emph{category associated to
  $\G$} has one object $\star$, and the set of morphisms
is given as the group $\G$. The neutral element
of the group is the identity morphism and the group multiplication
gives composition of morphisms. Often we use the group $\G$ and its
associated category synonymously.\end{dfn} 
If $\G$ is a topological group, its associated category is canonically
a topological category. If $\G$ is in $\catU$, its associated category
is canonically enriched over $\catU$, and, adding a disjoint
base point, enriched over $\catT$.  
\begin{dfn}We denote the category of functors $\G \rightarrow \catset$
  and natural transformations between them by $\catGset$ instead of
  $[G,\catset]_0$, its objects are called $\G$-sets. 
\end{dfn} 
Note that a
  $\G$-set is the same as a set with a (left) action of $\G$, and a
  morphism of $\G$-sets is a $\G$-equivariant map.

Just like $\catset$, the category $\catGset$\index{GSet@$\catGset$, $G$-sets} is a closed cartesian monoidal category with respect to the usual
cartesian product of sets, which is given the diagonal
$\G$-action. The unit object is the trivial $\G$-set consisting of
only one point, and the internal morphism object of maps from $X$ to $Y$ consists of all functions of sets $X\to Y$ equipped with the conjugation action.
Note that there are {\bf{two}} obvious monoidal
functors $\catGset \rightarrow \catset$. One is the forgetful functor,
which is obviously product preserving, but this is \emph{not} the
functor described in Definition~\ref{underlyset}. In fact, $\catGset(\star,X)$
assigns to a $\G$-set $X$ its set of $\G$-fixed points $X^G$, and this
gives the second monoidal functor. We distinguish this in language by
saying $X$ \emph{is} a set, but \emph{has} $X^G$ as its underlying set
(of $\G$-fixed points). 
\begin{dfn}
Let $\G$ be a group, a \emph{$\G$-category}\index{Gcategory@$\G$-category}\index{category!G@$G$-} is a category enriched over $\catGset$. We call the elements of the morphism $G$-sets \emph{morphisms},\index{morphism of $G$-sets} whereas the elements of the underlying $G$-fixed point sets are called \emph{$\G$-maps}.\index{Gmap@$\G$-map} As above, every $\G$-category is also a category, and has an underlying $\G$-fixed category.\\ A \emph{$\G$-functor}\index{functor!G@$\G$-} $F \colon \catD \rightarrow \catC$ between $\G$-categories is an enriched functor of enriched categories, \ie the induced maps on morphism $\G$-sets\[F\colon \catD(X,Y)\rightarrow\catC(F X,F Y)\] have to be $\G$-equivariant.\\ 
Two types of natural transformations are important for us: 
A \emph{natural $\G$-transformation}\index{natural Gtransformation@natural $\G$-transformation} $\alpha\colon F \rightarrow F'$
between two $\G$-functors from \(\catD\) to \(\catC\), is an enriched natural transformation of
enriched functors, \ie it consists of a $\G$-map 
$\alpha_X \in \catC(F
X,F' X)^G$ for every object $X$ of $\catD$ such that the
diagrams \[\xymatrix{{F X}\ar_-{\alpha_X}[d]\ar^-{F f}[r]&{F
    Y}\ar^-{\alpha_Y}[d]\\{F' X}\ar_-{F' f}[r]&{F' Y}},\] commute in
$\catC$ for all $f \in \catD(X,Y)$. These are the morphisms in the
functor category $[\catD,\catC]_0$.\\ 
On the other hand, there are the natural transformations, given as collections of maps $\alpha_X \in \catC(F X,F' X)$. On the set of these transformations $\G$ again acts by conjugation. Then, as indicated in Remark~\ref{underlfunctcat}, the $\G$-natural transformations are exactly the $\G$-fixed natural transformations, so that the functor $\G$-category $[\catD,\catC]$ has the functor category $[\catD,\catC]_0$ as its underlying ($\G$-fixed) category.
\end{dfn}

The following combination of the above definitions will be important in our studies of equivariant orthogonal spectra. Let $\G$ be a (compactly generated weak Hausdorff) topological group, respectively the associated one object $\catT$-category with morphism space $\G_+$
\begin{dfn} \label{gspaces}
The category $\catGT$\index{GT@$\catGT$, category of $G$-spaces} of $\G$-spaces consists of (continuous) functors $\G \rightarrow \catT$ and natural transformations between them. In particular, objects of $\catGT$ are spaces with a (left) action of $\G$ and morphisms are $\G$-equivariant continuous maps.\\
Giving smash products the diagonal $\G$-action, $\catGT$ inherits a closed symmetric monoidal structure from $\catT$. Again this allows us to view $\catGT$ as enriched over itself, and we shall use the notation $\catTG$\index{TG@$\catTG$} for the ensuing enriched category (\ref{selfenriched}), as well as $\catTG(-,-)$ for the internal $\Hom$-functor of $\catGT$. Then $\catTG$ has $\G$-spaces as objects, and morphisms are (not necessarily $\G$-equivariant) continuous maps.
\end{dfn}
\begin{dfn}
A category $\catCG$ is called a \emph{topological $\G$-category}\index{topological!G category@$\G$-category}\index{category!topological $\G$-} if it is enriched over $\catGT$. Such a $\catCG$ has a \emph{$\G$-fixed category}\index{category!Gfixed@$\G$-fixed} $\catGC$ that is obtained by applying the fixed point functor to the morphism $\G$-spaces.\\
The appropriate functors enriched over $\catGT$ are called \emph{continuous $\G$-functors}.\index{functor!continuous $G$-}\index{continuous!Gfunctor@$G$-functor} 
The appropriate enriched natural transformations are called
\emph{continuous natural $\G$-transformations} \index{continuous!natural Gtransformation@natural $G$-transformation} (\cite[p. 27]{MM} calls these natural $\G$-maps between functors). We will often drop the extra adjective ``continuous'' in the future.
\end{dfn}
\begin{rem}
Note that the fixed point functor $(-)^\G \colon \catGT \rightarrow
\catT$ has a left adjoint giving a space the trivial $\G$-action. This
left adjoint preserves (smash-) products and is therefore strict monoidal.
\end{rem}
Monoidal functors starting in $\catGT$ allow us to transport enrichments as in Construction~\ref{enrtransport}. Transportation along functors in the commutative diagram
\labeleq{enrdiag}{\xymatrix{{}&{\catT}\ar[r]&{\catU}\ar[r]&{\catset}\\{}\ar^-{\rm{forget.}}[u]\ar_-{(-)^G}[d]&{\catGT}\ar[u]\ar[r]\ar[d]&{\G\catU}\ar^-{}[u]\ar[r]\ar[d]&{\catGset}\ar^-{}[u]\ar[d]\\{}&{\catT}\ar[r]&{\catU}\ar[r]&{\catset,}}}
as well as their left adjoints, and even variations only using
subgroups of $\G$ (Lemma~\ref{H-fixed}) appears at various points when doing equivariant homotopy theory. 
\begin{ex}\label{gtmonoidal}
As it is defined, the $\catGT$-category $\catTG$ has the underlying $\G$-fixed $\catT$-category $\catGT$, which is closed symmetric monoidal. Also, $\catTG$ is closed symmetric monoidal itself, when viewing it as a mere category, using the same smash product and internal hom functor as in $\catT$. One choice of internal $\Hom$-functor for $\catTG$ is $\catTG(-,-)$, and we agree to use this choice.
\end{ex}

\section{Tensors and Cotensors}
Detailed treatment of the concepts of \emph{(indexed) limits and
  colimits} in $\catV$-enriched categories can be found in Chapter 3
of \cite{Kel}. We will mainly be concerned with the special case of
tensors and cotensors, and for convenience we repeat the definition.
\begin{dfn}\label{tenscotens}
Let $\catC$ be enriched over the closed symmetric monoidal category
$\catV$. Let $V$ be an object of $\catV$ and $A$ an object of
$\catC$. Then their \emph{tensor product}\index{tensor} $V \otimes A$
is an object of $\catC$, together with a
$\catV$-natural isomorphism
\[\catC(V \otimes A, B) \cong \Hom(V,
  \catC(A,B))\]
for $B$ in $\catC$.
Here $\Hom$ denotes the internal $\Hom$-object in $\catV$.\\
The \emph{cotensor product}\index{cotensor} $\catC(V, A)$ is an object of $\catC$,
together with a $\catV$-natural isomorphism
\[\catC(B, \catC(V,A)) \cong \Hom(V, \catC(B,A)).\]
If all such (co-)tensor products exist we call $\catC$ \emph{(co-)tensored}.\index{tensored (over $\catV$)}\index{cotensored (over $\catV$)} If we consider $\catC$ as enriched over different monoidal categories, we clarify the one used for (co-)tensors by saying that $\catC$ is \emph{(co-)tensored over \catV}.
\end{dfn}
\begin{rem}
Note that for $\catV = \catset$, being tensored and cotensored over $\catset$ is equivalent to having all small copowers $\coprod\limits_X{A}$. Dually, being cotensored over $\catset$ is equivalent to having all small powers $\prod\limits_X A$.
\end{rem}
\begin{ex}
Considering the closed symmetric monoidal category $\catV$ as enriched over itself (cf.~Example~\ref{selfenriched}), it is both tensored and cotensored over itself, by the defining adjunction of the internal $\Hom$-space, cf.~Definition~\ref{homadj}.
\end{ex}
\begin{ex}\label{GTtensor}
As mentioned in Remark~\ref{gtmonoidal}, the category $\catTG$ is enriched over $\catGT$, but also over itself, \ie $\Hom(X,Y) = \catTG(X,Y)$. This immediately implies that $\catTG$ is both tensored and cotensored over both itself and $\catGT$, where both are displayed by the same natural isomorphisms, considered either in $\catGT$ or $\catTG$:
\[\catTG(D,\catTG(A,B)) \cong \catTG(D \smash A, B) \cong \catTG(A, \catTG(D,B)).\]
Since $\catTG$ has $\catGT$ as its underlying $\G$-fixed category, this implies natural isomorphisms in $\catT$:
\[\catGT(S,\catTG(A,B)) \cong \catGT(S \smash A, B) \cong \catGT(A, \catTG(S,B)).\]
For $S$ any object of $\catT$, \ie with trivial $\G$-action, this reduces to:
\[\catT(S,\catGT(A,B)) \cong \catGT(S \smash A, B) \cong \catGT(A, \catTG(S,B)),\] which shows that $\catGT$ is tensored and cotensored over $\catT$.
\end{ex}
The following construction is important for the compatibility of an enrichment and the model structures on the involved categories, and also appears prominently in a lot of our constructions of cellular filtrations:
\begin{dfn}\label{pushoutproduct}
Let $(\catV, \smash, \I)$ be a closed symmetric monoidal category. Let $\catC$ be enriched and tensored over $\catV$, and have pushouts.
For $i: A\rightarrow B$ a morphism in $\catV$, $j:X\rightarrow  Y$ a
morphism in the underlying category $\catC_0$ of $\catC$, define the
\emph{pushout product $i\square j$}\index{pushout product!$i\square j$}\index{ij@$i\square j$, pushout product}  to be the dotted map in
\(\catC_0\) from the pushout in the diagram:
\[\xymatrix{{A\otimes X}\ar[r]^-{\id\otimes j}\ar[d]_-{i\otimes\id}&{A\otimes Y}\ar[d]\ar@(r,u)^-{i\otimes\id}[ddr]\\{B\otimes X}\ar[r]\ar@(d,l)_-{\id\otimes j}[drr]&{P}\pushout\ar@{.>}^{i\square j}[dr]\\&&{B\otimes Y}}\] 
\end{dfn}

The dual construction is the following:
\begin{dfn}\label{cotensorbox}
Let $(\catV, \smash, \I)$ be a closed symmetric monoidal category. Let $\catC$ be enriched and cotensored over $\catV$, and have pullbacks.
For $i\colon A\rightarrow B$ a morphism in $\catV$, $p\colon
E\rightarrow  F$ a morphism in the underlying category $\catC_0$ of
$\catC$, define the map \emph{\(\catC(i^*,p_*)\)}\index{Cip@\(\catC(i^*,p_*)\in \catC_0\)} in \(\catC_0\) to be the dotted map to the pullback in the diagram:
\[\xymatrix{{\catC(B,E)}\ar@(d,l)_-{p_*}[ddr]\ar@(r,u)^-{i^*}[drr]\ar@{.>}^{\catC(i^*,p_*)}[dr]&{}\\
{}&{Q}\pullback\ar[r]^-{}\ar[d]_-{}&{\catC(A,E)}\ar^-{p_*}[d]\\
&{\catC(B,F)}\ar_-{i^*}[r]&{\catC(A,F)}}\]
\end{dfn}
This again has an analog living in the category $\catV$:
\begin{dfn}\label{defliftchar}
Let $(\catV, \smash, \I)$ be a closed symmetric monoidal category
having pullbacks. Let $\catC$ be enriched over $\catV$. For $j:
X\rightarrow Y$ and $p:E\rightarrow  F$ be morphisms in the underlying
category $\catC_0$ of $\catC$, define the map $\catC(j^*,p_*)$\index{Cjp@\(\catC(j^*,p_*)\in \catV\)} in
\(\catV\) to be the dotted map to the pullback in the diagram in $\catV$:
\[\xymatrix{{\catC(Y,E)}\ar@(d,l)_-{p_*}[ddr]\ar@(r,u)^-{j^*}[drr]\ar@{.>}^{\catC(j^*,p_*)}[dr]&{}\\
{}&{R}\pullback\ar[r]^-{}\ar[d]_-{}&{\catC(X,E)}\ar^-{p_*}[d]\\
&{\catC(Y,F)}\ar_-{j^*}[r]&{\catC(X,F)}}\]
\end{dfn}
This construction can be used to characterize lifting properties in the enriched setting:
\begin{lemma}\label{liftingpprodadj}
  In the situation of Definition~\ref{defliftchar}, the pair $(j,p)$ has the lifting property in $\catC_0$, if and only if the map of sets $\catV(\I,\catC(j^*,p_*)))$ is surjective.\end{lemma} 
\begin{proof}
Recall that morphisms $X \rightarrow Y$ in $\catC_0$ correspond to
elements of $\catV(\I,\catC(X,Y)) = \catC_0(X,Y)$ from Definition~\ref{underlyset}. Then the universal property of the pullback gives that elements of $\catV(\I,R)$ correspond exactly to commutative diagrams 
\[\xymatrix{{X}\ar[r]\ar_-{j}[d]&{E}\ar^-{p}[d]\\{Y}\ar[r]&{F}}\] in $\catC_0$. Then $\catV(\I,\catC(j^*,p_*)))$ sends maps $f\colon{} Y \rightarrow E$ in $\catC_0$ to the diagram with $f\circ j$ as the top and $p \circ f$ as the bottom horizontal arrow, so that surjectivity indeed corresponds exactly to the existence of the lift.
\end{proof}
Given that all of the three above constructions are defined, there is the following crucial relation between them:
\begin{lemma}\label{pushout.product.adjoints}
Let $(\catV, \smash, \I)$ be closed symmetric monoidal and have small limits. Let $\catC$ be enriched, tensored and cotensored over $\catV$ and have pullbacks and pushouts. Let $i: A \rightarrow B$ a morphism in $\catV$ and $j: X \rightarrow Y$ and $p: E\rightarrow F$ morphisms in the underlying category $\catC_0$ of $\catC$. Then the following maps in $\catV$ are naturally isomorphic:
\[\catC((i \square j)^*,p_*) \cong \catV(i^*,\catC(j^*,p_*)_*) \cong \catC(j^*,\catC(i^*,p_*))\]
\end{lemma}
\begin{proof}Note that for the middle map we considered $\catV$ as enriched over itself as in Example~\ref{selfenriched}. By careful use of the universal properties of pushouts and pullbacks as well as the defining adjunctions for tensors and cotensors of Definition~\ref{tenscotens}, one observes that all three maps are naturally isomorphic to the map from $\catV(B, \catC(Y,E))$ to the limit of 
\[\xymatrix{{\catV(A,\catC(Y,E))}\ar[d]\ar[dr]&{\catV(B,\catC(Y,F))}\ar[dl]\ar[dr]&{\catV(B,\catC(X,E))}\ar[dl]\ar[d]\\
{\catV(A,\catC(Y,E))}\ar[dr]&{\catV(A,\catC(X,E))}\ar[d]&{\catV(B,\catC(X,F))}\ar[dl]\\
&{\catV(A,\catC(X,F))}}\]
\end{proof}
These two lemmas allow us to characterize lifting properties in $\catC_0$ in terms of those in $\catV$, which is of course of particular interest when $\catC_0$ and $\catV$ are model categories (cf.~\ref{monoidalmodelcat}). \\
\section{Kan Extensions}\label{Vkanext}
The discussion about enriched Kan extensions in \cite[4]{Kel} is, due to its generality rather technical. As in the case of enriched (co-) limits, extra care has to be taken in several places. Since we do not need the full generality, we state a slightly simpler definition and list only the explicit properties we make use of, without going into much detail. We concentrate on the case of left Kan extensions, since the dual notion will not appear outside of pure existence statements.\\
Let $\catV$ be closed symmetric monoidal and consider the solid arrow diagram of $\catV$-categories and $\catV$-functors:
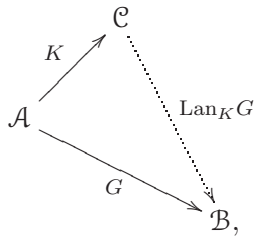
\[\xymatrix{
{}&{\catC}\ar^-{\Lan{K}{G}}@{.>}[ddr]&{}\\
{\catA}\ar_{G}[drr]\ar^-{K}[ur]&{}&{}\\
{}&{}&{\catB,}
}\]
where $\catA$ is equivalent to a small $\catV$-category and $\catB$ is cotensored over $\catV$.
\begin{dfn}\label{lkanext}
In the above situation, a \emph{left Kan extension $\Lan{K}{G}$ of $G$ along $K$}\index{left Kan extension $\Lan{K}{G}$}\index{LanKG@$\Lan{K}{G}$, left Kan extension } is a $\catV$-functor $\catC \rightarrow \catB$, together with a $\catV$-natural isomorphism
\[[\catC,\catB](\Lan{K}{G},S) \cong [\catA,\catB](G,S\circ K).\]
The image of the identity transformation for $S = \Lan{K}{G}$ is a $\catV$-natural transformation $\phi: G \rightarrow \Lan{K}{G} \circ K$ and is called the \emph{unit} of $\Lan{K}{G}$.
\end{dfn}
It is important to note, that in a situation where $\catB$ is not
cotensored, this definition is not adequate, in that it does not
describe the left Kan extension in the sense of Kelly, but rather a
weaker notion. For counterexamples see the discussion after
\cite[4.43]{Kel}.

The following proposition will give us the existence of left Kan extensions in all the cases that we will consider:
\begin{prop}{\cite[4.33]{Kel}}\label{kanexist}
A \(\catV\)-category $\catB$ admits all left Kan extensions of the form
$\Lan{K}{G}$, where $K: \catA \rightarrow \catC$ and $G: \catA
\rightarrow \catB$ and $\catA$ is equivalent to a small $\catV$-category, if and only if $\catB$ is enriched cocomplete.
\end{prop}
To check the required cocompleteness, we will generally be able to use the following characterization, which is a combination of several statements in \cite{Kel}:
\begin{thm}Let $\catB$ be enriched over $\catV$:
\begin{enumerate}
\item $\catB$ is cocomplete in the enriched sense, if and only if is tensored and admits all small conical (enriched) colimits.
\item $\catB$ is complete in the enriched sense, if and only if is cotensored and admits all small conical (enriched) limits.
\item Assuming $\catB$ is cotensored, $\catB$ admits all small conical colimits, if and only if its underlying ordinary category $\catB_0$ is cocomplete.
\item Assuming $\catB$ is tensored, $\catB$ admits all small conical limits, if and only if its underlying ordinary category $\catB_0$ is complete.
\end{enumerate}
For tensored and cotensored $\catB$, the conical (co-)limits are the ones created in $\catB_0$.
\end{thm}
\begin{proof}
The precise references in $\cite{Kel}$ are: Theorem 3.73 for (ii), dualize for (i). The discussion between 3.53 and 3.54 for conical (co-)-limits in $\catB$ or $\catB_0$, and the discussion between 3.33 and 3.34 for the connection to classical (co-)completeness,
\end{proof}
Since it is not always the enriched functor category from Remark~\ref{enrfunctorcat} that is of interest for us, we would also like a characterization of the left Kan extension in terms of the underlying category of enriched functors and enriched transformations. Luckily our assumption that $\catB$ is cotensored allows us to use the following universal property from \cite[4.43]{Kel} and the discussion that follows it:
\begin{thm}
  If $\catB$ is cotensored, then a $\catV$-functor $L$ is a left Kan extension of $G$ along $K$ if and only if there is a natural bijection of sets
  \[[\catC,\catB]_0(L,S) \cong [\catA,\catB]_0(G,S\circ K).\]
  In particular a $\catV$-functor $L$ equipped with a $\catV$-natural transformation $\phi: G \rightarrow L\circ K$ is a left Kan extension of $G$ along $K$ if and only if for any $\catV$-natural transformation
 $\alpha: G \rightarrow S \circ K$ there exist a unique $\beta\colon L\to S$ such that $\alpha=(\beta\circ\id_K)\phi$.
\end{thm}
Hence in the case of $\catB$ tensored, cotensored and cocomplete, the two characterizations together with the existence result Proposition~\ref{kanexist}, allow us to state the following:
\begin{prop}\label{kanisleftadj}
If $\catB$ is cotensored, then precomposition with $K\colon\catA\to\catC$ defines a $\catV$-functor $K^*\colon [\catC,\catB] \rightarrow [\catA,\catB]$ and the left Kan extension provides a left adjoint, both in the enriched sense, and on underlying ordinary categories.
\end{prop}
Finally, the following property helps to compute the Kan extensions in a lot of interesting special cases:
\begin{prop}{\cite[4.23]{Kel}}
 In the situation of Proposition~\ref{kanisleftadj}, the $\catV$-functor $K\colon\catA\to\catC$ is fully faithful if and only if for all $\catB$ the unit $\id_{[\catA,\catB]} \rightarrow K^*\Lan{K}{-}$ of the adjunction is a natural isomorphism.
\end{prop}

\section{Cofinality for coends}
\label{sec:coendcof}

\begin{dfn}\label{def:Fcofinal}
  Let \(i \colon \catD \to \catC\) be a full and faithful functor of
  \(\catT\)-categories and let \(F \colon \catC^{\op} \times 
  \catC \to \catT\) be a \(\catT\)-functor. We say that \(i\) is
  {\em \(F\)-cofinal}\index{cofinal, $F$-}\index{Fcofinal@$F$-cofinal} if for all morphisms
  \(\gamma \colon c_0
  \to c_1\) and \(f_0 \colon c_0 \to c_0'\) in \(\catC\) so that
  \(F(f_0,c)\) is a homeomorphism for every object \(c\) of \(\catC\),
  there  exists an object
  \(d\) of \(\catD\) and morphisms \(\delta \colon c_0' \to i(d)\) and
  \(f_1 \colon c_1 \to i(d)\) so that
  \begin{enumerate}
  \item 
    \(f_1 \gamma = \delta f_0\), that
  is, the following square commutes:
  \begin{displaymath}
    \xymatrix{
      c_0 \ar[r]^{\gamma} \ar[d]_{f_0} & c_1 \ar[d]^{f_1} \\
      c_0' \ar[r]^{\delta} &i(d).
    }
  \end{displaymath}
  \item The map \(F(f_1,c)\) is a homeomorphisms for
    every object \(c\) of \(\catC\). 
  \end{enumerate}
\end{dfn}
\begin{lem}\label{fromcond1tocofinal}
  Let \(i \colon \catD \to \catC\) be a full and faithful functor of
  \(\catT\)-categories and let \(F \colon \catC^{\op} \wedge 
  \catC \to \catT\) be a \(\catT\)-functor. Suppose that
  \begin{enumerate}
  \item for all morphisms
  \(\gamma \colon c_0
  \to c_1\) and \(f_0 \colon c_0 \to c_0'\) in \(\catC\) so that
  \(F(f_0,c)\) is a homeomorphism for every object \(c\) of \(\catC\),
  there  exists morphisms \(g' \colon c_0' \to c_1'\) and
  \(g_1 \colon c_1 \to c_1'\) in \(\catC\) so that \(g_1 \gamma =
  g' f_0\) and so that the map \(F(g_1,c)\) is a homeomorphisms for
    every object \(c\) of \(\catC\).
  \item For every object \(c_0\) of \(\catC\) there exists an object
    \(d\) of \(\catD\) and a morphism \(f \colon c_0 \to i(d)\) so that  
    the map \(F(f,c)\) is a homeomorphisms for
    every object \(c\) of \(\catC\).
  \end{enumerate}
  Then \(i\) is \(F\)-cofinal.
\end{lem}
\begin{proof}
  This is a direct consequence of the definition by letting \(\delta =
  fg'\) and \(f_1 = fg_1\). 
\end{proof}

In the rest of this subsection we fix
a \(\catT\)-functor \(F \colon \catC^{\op}
  \wedge \catC \to \catT\) and a functor \(i \colon
  \catD\to \catC\).

Note that if \(i\) is \(F\)-cofinal, then taking all of the maps in
the  \(F\)-cofinality condition 
be the identity on an object \(c\) of \(C\), we can for each \(c\)
choose a morphism \(f \colon c \to i(d)\) so that \(F(f,c')\) is a
homeomorphism for every object \(c'\) of \(\catC\).
\begin{prop}\label{Fcofinaliso}
  If \(i\) is \(F\)-cofinal, then the canonical map
  \(i \colon \int^\catD i^*F \to \int^{\catC} F\) induced by the
  functor \(i\) is a homeomorphism. 
\end{prop}
The proof of the above result occupies the rest of this section.
We prove it by constructing an inverse to the map
\(i\). Since \(\int^{\catC} F\) can be described as the
coequalizer of
\begin{displaymath}
  \bigvee_{c_0,c_1} F(c_1,c_0) \wedge \catC(c_0,c_1) \rightrightarrows
  \bigvee_{c} F(c,c), 
\end{displaymath}
where the pointed sums are indexed over objects and pairs of objects
of \(\catC\) respectively. The above maps take an element
\((c_0,c_1,x,\gamma)\), consisting of objects \(c_0\) and \(c_1\) of
\(\catC\), a point \(x \in F(c_1,c_0)\) and a morphism \(\gamma \colon
c_0 \to c_1\) in \(\catC\), to \(F(\gamma,c_0)x\) and
\(F(c_1,\gamma)x\) respectively. 

In the rest of this section we suppose that \(i\) is \(F\)-cofinal. 
\begin{dfn}\label{defvarphic}
  For each object \(c\) of \(\catC\)
  we construct a continuous map 
  \(\varphi_c \colon F(c,c) \to \int^{\catD} i^*F\)\index{phicf@\(\varphi_c \colon F(c,c) \to \int^{\catD} i^*F\)} as follows:
  Use \(F\)-cofinality to 
  to choose a morphism \(f \colon c \to i(d)\) with \(d\) an object of
  \(\catD\) so that \(F(f,c')\) is a homeomorphism for every object
  \(c'\) of \(\catC\). Then \(\varphi_c(x) \in \int^{\catD} i^*F\) is
  defined to be
  the element represented by \((d,F(i(d),f)F(f,c)^{-1}
  x) \in F(i(d),i(d))\).
\end{dfn}
\begin{lem}\label{cofinalityindepofrep}
  The map \(\varphi_c\) of Definition~\ref{defvarphic} is independent
  of the chosen \(f\).
\end{lem}
\begin{proof}
  Let \(f' \colon c \to i(d')\) be another morphism 
  with \(F(f',c')\) a homeomorphism for
  every object \(c'\) of \(\catC\). We need to explain why
  \begin{displaymath}
    (d,y) = (d,F(i(d),f)F(f,c)^{-1} x)
  \end{displaymath}
  and 
  \begin{displaymath}
    (d,y') = (d,F(i(d'),f')F(f',c)^{-1} x)
  \end{displaymath}
  represent the same point of \(\int^{\catD} i^*F\). Use
  \(F\)-cofinality to choose a commutative square of the form
  \begin{displaymath}
    \xymatrix{
      c \ar[r]^{f'} \ar[d]_{f} & i(d') \ar[d]^{g} \\
      i(d) \ar[r]^{g'} & i(d')'
    }
  \end{displaymath}
  where \(F(g,c')\) is a homeomorphism for all \(c'\) in \(\catC\).
  Since \(y = F(g',i(d)) F(g',i(d))^{-1}y\), the pairs \((d,y)\) and 
  \((d, F(i(d')',g')F(g',i(d))^{-1}y)\)  represent the same element in
  \(\int^{\catD} i^*F\). However,
  \begin{displaymath}
    F(i(d')',g')F(g',i(d))^{-1}y = F(i(d')',(g')f)F((g')f,c)^{-1} x,
  \end{displaymath}
  and since \(g'f = gf'\) this is equal to
  \begin{displaymath}
    F(i(d')',gf')F(gf',c)^{-1} x = F(i(d')',g)F(g,i(d'))^{-1} y'.
  \end{displaymath}
  Reasoning as above we see that this element
  represents the same element as \(y'\) in the coend \(\int^{\catD} i^*F\).
\end{proof}
 The maps \(\varphi_c\) assemble to a
continuous map \(\varphi \colon \bigvee_{c} F(c,c) \to \bigvee_{d}
F(i(d),i(d))\) with \(\varphi(c,x) = \varphi_c(x)\).
\begin{lem}\label{descendtocoend}
  Given a \(\catT\)-functor \(F \colon \catC^{\op} \wedge \catC \to
  \catT\) and an \(F\)-cofinal \(i \colon \catD \to \catC\), we have
  that, given \(x \in F(c_1,c_0)\) and \(\gamma \in \catC(c_0,c_1)\),
  the point
  \(\varphi(c_0,F(\gamma,c_0)x)\) is equal to the point
  \(\varphi(c_1,F(c_1,\gamma)x)\) 
  of \(\int^{\catD} i^*F\). 
\end{lem}
The above lemma says that
if \(c_0\) and \(c_1\) are objects of \(\catC\)
  and \(\gamma \in \catC(c_0,c_1)\),
  then the diagram
  \begin{displaymath}
    \xymatrix{
        F(c_1,c_0) \ar[r]^{F(c_1,\gamma)} \ar[d]_{F(f,c_0)} &
        F(c_1,c_1) \ar[d]^{\varphi_{c_1}}\\
  F(c_0,c_0) \ar[r]^{\varphi_{c_0}} & \int^{\catD} i^* F
    }
  \end{displaymath}
  commutes.
\begin{proof}
  First we use \(F\)-cofinality to obtain a morphism \(f_0 \colon
  c_0 \to i(d_0)\) with \(F(f_0,c)\) a
  homeomorphism for all objects \(c\) of \(\catC\). Next we use the
  \(F\)-cofinality condition to obtain a commutative square
  \begin{displaymath}
    \xymatrix{
      c_0 \ar[r]^{\gamma} \ar[d]_{f_0} & c_1 \ar[d]^{f_1} \\
      i(d_0) \ar[r]^{i \delta} & i(d_1)
    }
  \end{displaymath}
  with \(F(f_1,c)\) a homeomorphism for all objects \(c\) of
  \(\catC\). By Lemma~\ref{cofinalityindepofrep}
  \(\varphi(c_0,F(\gamma,c_0)x)\) is represented by
  \((d_0,F(i(d_0),f_0)F(f_0,c_0)^{-1}F(\gamma,c_0)x)\) and
  \(\varphi(c_1,F(c_1,\gamma)x)\) is represented by
  \((d_1,F(i(d_1),f_1)F(f_1,c_1)^{-1}F(c_1,\gamma)x)\).
  However the diagram
  \begin{displaymath}
    \xymatrix{
      F(c_1,c_1) \ar[d]^{F(i(d_1),f_1)F(f_1,c_1)^{-1}} && 
      \ar[d]^{F(i(d_1),f_0)F(f_1,c_0)^{-1}} \ar[ll]_{F(c_1,\gamma)} F(c_1,c_0)
      \ar[rr]^{F(\gamma,c_0)} 
      && \ar[d]^{F(i(d_0),f_0)F(f_0,c_0)^{-1}} F(c_0,c_0) \\
      F(i(d_1),i(d_1)) &&
      \ar[ll]_{F(i(d_1),i(\delta))} F(i(d_1),i(d_0))
      \ar[rr]^{F(i(\delta),i(d_0))} && F(i(d_0),i(d_0))
    }
  \end{displaymath}
  commutes.  That is,
  \begin{align*}
    F(i(d_0) ,f_0) F(f_0,c_0)^{-1}F(\gamma,c_0)x & = 
    F(i(d_0) ,f_0) F(f_0,c_0)^{-1}F(\gamma,c_0) F(f_1,c_0)F(f_1,c_0)^{-1}x \\
    & = 
    F(i(d_0),f_0) F(i\delta, c_0) F(f_1, c_0)^{-1} x\\
    & = 
    F(i\delta, i(d_0)) F(i(d_1),f_0)  F(f_1, c_0)^{-1} x,
  \end{align*}
  and this element represents the same element as
  \begin{align*}
    F(i(d_1),i\delta) F(i(d_1),f_0)  F(f_1, c_0)^{-1} x
    &=
    F(i(d_1),f_1)F(i(d_1),\gamma)F(f_1,c_0)^{-1} x \\
    &=F(f_1, i(d_1))^{-1} F(c_1 ,f_1)F(c_1 , \gamma) x \\
    &= F(i(d_1) ,f_1) F(f_1, c_1)^{-1} F(c_1 ,\gamma)x,
  \end{align*}
  that is, it represents \(\varphi(c_1,F(c_1,\gamma)x)\) in
  \(\int^{\catD} i^*F\).
\end{proof}

\begin{proof}[Proof of Proposition~\ref{Fcofinaliso}]
  By Lemma~\ref{descendtocoend} the continuous map \(\varphi \colon
  \bigvee_{c} F(c,c) \to \bigvee_{d} F(i(d),i(d))\) induces a unique map
  \(\overline \varphi \colon \int^{\catC} F \to \int^{\catD} i^*
  F\). Lemma~\ref{cofinalityindepofrep} implies that the composite \(\overline \varphi i\)
  is the identity on \(\int^{\catD} i^*F\) since if \(c = i(d)\), then
  we can choose \(f\) to be the identity on \(c\). Conversely, if
  \((c,x)\) represents a point in \(\int^{\catC} F\), then \(i
  \varphi(c,x)\) is represented by \((d,F(i(d),f)F(f,c)^{-1} x)\) and
  this element represents the same element as
  \((c,x)\) in \(\int^{\catC} F\). Thus \(i \overline \varphi\) is
  the identity on 
  \(\int^{C} F\).
\end{proof}


\section{Model Categories}\label{model_categories}
We assume that the reader is familiar with the basic theory of model categories, an introductory account can for example be found in \cite{DS}. A more exhaustive source is \cite{H} or \cite{Hir}.

Almost all of the model structures we will discuss are \emph{cofibrantly generated}, we recall the definition and state the main theorem we use to recognize such model structures from \cite[2.1.3]{H}: 
\begin{dfn}
Let $\catC$ be a model category. It is called \emph{cofibrantly generated}\index{cofibrantly generated model category}\index{model structure!cofibrantly generated} if there are sets $I$ and $J$ of maps, such that:
\begin{indentpar}{1cm}
\begin{enumerate}
\item The domains of the maps of $I$ are small with respect to $I\cell$,
\item The domains of the maps of $J$ are small with respect to $J\cell$,
\item The class of fibrations is $J\inj$,
\item The class of acyclic fibrations is $I\inj$.
\end{enumerate}
\end{indentpar}
\end{dfn}
\begin{thm}[Recognition Theorem {\cite[2.1.19]{H}}] \label{h2119}\index{recognition theorem}
Suppose $\catC$ is a category with all small colimits and limits. Suppose $W$ is a subcategory of $\catC$ and $I$ and $J$ are sets of maps of $\catC$. Then there is a cofibrantly generated model structure on $\catC$ with $I$ as the set of generating cofibrations, $J$ as the set of generating acyclic cofibrations, and $W$ as the subcategory of weak equivalences if and only if the following conditions are satisfied:
\begin{indentpar}{1cm}
\begin{enumerate}
\item The subcategory $W$ has the two out of three property and is closed under retracts.
\item The domains of $I$ are small relative to $I\cell$.
\item The domains of $J$ are small relative to $J\cell$.
\item $J\cell \subset W \cap I\cof$.
\item $I\inj \subset W \cap J\inj$.
\item Either $W\cap I\cof \subset J\cof$ or $W \cap J\inj \subset I\inj$.
\end{enumerate}
\end{indentpar}
\end{thm}
The following lemmas are applicable in any model category. These are well known and often used without further mention in the literature, but since they lie at the heart of the homotopy theory we need, we recall the exact statements.
Recall the following definition from \cite[II 8.5]{GJ}):
\begin{dfn}\label{catcofobj}
A \emph{category of cofibrant objects}\index{category!of cofibrant objects} is a category $\catD$ with all finite coproducts, with two classes of maps, called weak equivalences and cofibrations, such that the following axioms are satisfied:
\begin{enumerate}
\item The weak equivalences satisfy the 2 out of 3 property, that is,
  if \(f\) and \(g\) are morphisms of \(\catD\) so that \(gf\) is
  defined and two of \(f\), \(g\) and \(gf\) are weak equivalences,
  then so is the third.
\item The composite of two cofibrations is a cofibration. Any isomorphism is a cofibration.
\item Pushouts along cofibrations exist. Cobase changes of cofibrations (that are weak equivalences) are cofibrations (and weak equivalences).
\item All maps from the initial object are cofibrations.
\item Any object $X$ has a cylinder object $\Cyl(X)$,\index{CylX@$\Cyl(X)$} \ie a factorization of the fold map $\nabla: X \coprod X \rightarrow X$ as \[X\coprod X \stackrel{i}{\rightarrow} \Cyl(X) \stackrel{\sigma}{\rightarrow} X,\] with $i$ a cofibration and $\sigma$ a weak equivalence.
\end{enumerate}
\end{dfn}
In any model category, the cofibrant objects form a \emph{category of cofibrant objects}. This lets us apply the following two important lemmas:
\begin{lemma}[{Generalized Cobase Change Lemma (cf.~\cite[II.8.5]{GJ})}] \label{gen.cobch}\index{cobase change lemma!generalized}
Let $\catC$ be a category of cofibrant objects. Suppose
\labeleq{cobasechange}{\xymatrix{{A}\ar_-{i}[d]\ar^-{f}[r]&{X}\ar[d]\\{Y}\ar_-{g}[r]&{P}\pushout}}
is a pushout diagram in $\catC$, such that $i$ is a cofibration and $f$ is a weak equivalence. Then $g$ is also a weak equivalence.
\end{lemma}
\begin{lemma}[{Generalized Cube Lemma (cf.~\cite[II.8.8]{GJ})}] \label{gen.cube}\index{cube lemma!generalized}
Let $\catC$ be a category of cofibrant objects. Suppose given a commutative cube 
\labeleq{cube}{
\xymatrix{
&{A_0}\ar[rr]\ar'[d][dd]_-{f_A}\ar_-{i_0}[dl]&&{X_0}\ar_>>>>>>>>{f_X}[dd]\ar[dl]\\
{Y_0}\ar[rr]\ar_>>>>>>>>{f_Y}[dd]&{}\ar@{}[r]&{P_0}\ar_>>>>>>>>{f_P}[dd]\\
&{A_1}\ar@{}[r]\ar'[r][rr]\ar_-{i_1}[dl]&&{X_1}\ar[dl]\\
{Y_1}\ar[rr]&&{P_1}}
}in $\catC$. Suppose further that the top and bottom faces are pushouts, that $i_0$ and $i_1$ are cofibrations and that the vertical maps $f_A$, $f_X$ and $f_Y$ are weak equivalences. Then the induced map of pushouts $f_P$ is also a weak equivalence.
\end{lemma}
Note that not all examples of categories of cofibrant objects come
from model structures, in particular we will want to apply Lemma~\ref{gen.cube} to cases where the cofibrations and weak equivalences
come from different model structures on the same category in
\ref{rem_cofcat}. In the case of topological model categories
(cf.~Definition~\ref{monoidalmodelcat}), May and Sigurdsson propose a more
general treatment in \cite[5.4]{MS}, using so called well-grounded
categories of weak equivalences. 
We will handpick some of the statements of \cite{MS} in our Subsection~\ref{topmodelcat} on topological model categories.\\
\begin{dfn}\label{monoidalmodelcat}
Let $(\catV,\otimes,\I)$ closed symmetric monoidal category, that is also a model category. Let $\catC$ be a category enriched, tensored and cotensored over $\catV$. Further let the underlying category $\catC_0$ of $\catC$ (cf.~\ref{underlyingcat}) have a model structure. Then this model structure on $\catC$ is called \emph{enriched over \catV}, if the following two axioms hold:
\begin{indentpar}{1cm}
\begin{enumerate}
\item \emph{Pushout product axiom}\index{pushout product!axiom} Let, as in Definition~\ref{pushoutproduct}, $i$ be a cofibration in $\catV$ and let $j$ be a cofibration in $\catC_0$. Then the map $i\square j$ in $\catC_0$ is also a cofibration. If in addition either one of $i$ or $j$ is acyclic, so is $i \square j$.
\item \emph{Unit axiom}\index{unit axiom}  Let $q: \I^c \stackrel{\sim}{\rightarrow} \I$ be a cofibrant replacement of the unit object of $\catV$. Then for every cofibrant object $A$ in $\catC$, the morphism \[q\otimes \id: \I^c \otimes A \rightarrow \I \otimes A \cong A\] is a weak equivalence.
\end{enumerate}
\end{indentpar}
If $\catC$ is equal to $\catV$, \ie we consider $\catV$ as enriched over itself (cf.~\ref{selfenriched}), a model structure satisfying the above axioms is called \emph{monoidal}.\index{monoidal!model structure}\index{model structure!monoidal} 
\end{dfn}
Note that the unit axiom is redundant, if the unit object of $\catV$ is itself cofibrant, since it is then implied by the pushout product axiom. For monoidal model categories, there is an additional important axiom:
\begin{dfn}
A monoidal model category $(\catC,\otimes)$ satisfies the \emph{monoid axiom},\index{monoid!axiom} if every map in \[\{\text{acyclic cofibrations}\}\otimes \catC)\rm{-cell}\] is a weak equivalence.
\end{dfn}
The pushout product axiom has several adjoint formulations:
\begin{lemma}\label{adjointpushoutprod}
In the situation of Definition~\ref{monoidalmodelcat}, the pushout product axiom is equivalent to both of the following formulations:
\begin{itemize}
\item Let $p$ be a fibration in $\catC_0$ and let $j$ be a cofibration in $\catV$, then $\catC(j^*,p_*)$ is a fibration in $\catC_0$, which is acyclic if either of $p$ or $j$ was.
\item Let $p$ be a fibration in $\catC_0$ and let $i$ be a cofibration in $\catC_0$, then $\catC(i^*,p_*)$ is a fibration in $\catC$, which is acyclic if either of $p$ or $i$ was.
\end{itemize}
\end{lemma}
\begin{proof}
This is immediate from Lemmas~\ref{liftingpprodadj} and~\ref{pushout.product.adjoints}.
\end{proof}
\begin{ex}
Taking $\catV$ to be the categories of simplicial sets, spaces, symmetric spectra or $\G$-spaces, yields, under the choice of the usual model structures, the well known notions of simplicial, topological, spectral and $\G$-topological model categories. 
\end{ex}
In particular the example of topological and $G$-topological model categories will be very important for us. We discuss some of their distinct features in the following subsection.

\section{Topological Model Categories}\label{topmodelcat}
In this subsection we have to discuss two different categories of topological spaces. We distinguish between the category $\catU$ of compactly generated weak Hausdorff spaces, and the category $\catT$ of such spaces with a distinguished base point. Alternatively one can think of $\catT$ as the under-category $* \rightarrow \catU$ for $*$ any one-point object in $\catU$.\\
Let $I$ denote the unit interval in $\catU$, as usual it comes equipped with the two inclusions of the endpoints. For any category $\catC$ enriched and tensored over $\catU$, we can then form \emph{homotopies in $\catC_0$}\index{homotopy!in $\catC_0$} in terms of the tensor with $I$:
\[\xymatrix{{{\{0\}\otimes X}}\ar[d]\ar^-{\cong}[r]&{X}\ar^-{h_0}[dr]\\{I \otimes X}\ar^-{h}[rr]&{}&{Y}\\{{\{1\}\otimes X}}\ar_-{\cong}[r]\ar[u]&{X}\ar_-{h_1}[ur]}\] Analogously for $\catC$ enriched over $\catT$, we can add a disjoint base points and use the tensor with $I_+$ to define \emph{(based) homotopies}.\index{homotopy!in $\catC_0$, based}\\  
There are two classical model structures on $\catU$ that are important
for us, the Str\o m model structure and the Quillen- or $q$-model structure. Especially the cofibrations of the former have very favorable properties, the defining one being the homotopy extension property:
\begin{dfn}
\label{dfn:h-cof} Let $\catC$ be enriched and tensored over $\catU$. 
A map $f\colon X\to Y$  in $\catC_0$ is a {\em Hurewicz fibration}\index{Hurewicz fibration@\hur! fibration} if it has the right homotopy lifting property with respect to the inclusions $i_0\colon A\to A\times I$.

A map $i:A\rightarrow X$ in $\catC_0$ is a \emph{\hur cofibration
}\index{Hurewicz cofibration@\hur! cofibration}\index{hcofibration@$h$-cofibration! free} if it satisfies the \emph{free homotopy extension property}. That is, for every map $f\colon{}X\rightarrow Y$ and homotopy $h:I \otimes A \rightarrow Y$ such that $h_0=f \circ i$, there is a homotopy $H:I \otimes X \rightarrow Y$ such that $H_0 =f$ and $H\circ(i\otimes\id)=h$.
\end{dfn}
The universal test case for \hur cofibrations is the mapping cylinder $Y=Mi=X\cup_i(I \otimes A)$, with the obvious $f$ and $h$. The exact statement is the following lemma.
\begin{lemma}
\label{lem_h-cof-test} Let $\catC$ be enriched and tensored over $\catU$ and have pushouts.
A map $i:A\rightarrow X$ in $\catC_0$ is a Hurewicz cofibration
if and
only if the canonical map $Mi\rightarrow I\otimes X$ has a retraction. 
\end{lemma}
\begin{rem}\label{rem.hcofclos}
This implies a variety of closure properties of the class of Hurewicz cofibrations. In particular any functor that preserves pushouts and the tensor with the interval also preserves Hurewicz cofibrations, since any functor preserves retractions.
\end{rem}
\begin{thm}{\cite[Theorem 3]{Str}, cf.~\cite[4.4.4]{MS}}
The homotopy equivalences, Hurewicz fibrations and Hurewicz cofibrations give a proper model structure on $\catU$.
\end{thm}
Note that Str\o m originally works in the category of all topological spaces, but the intermediate objects for the factorizations he constructs are in $\catU$ if source and target were. Properness is not mentioned in the original article, but is implied by the fact that all objects are fibrant and cofibrant. 

\begin{dfn}
Let $f$ be a map in $\catU$. Then $f$ is a \emph{weak equivalence}\index{weak equivalence in $\catU$} if it induces isomorphisms on all homotopy groups. Call $f$ a \emph{$q$-cofibration}\index{qcofibration@$q$-cofibration} if it has the left lifting property with respect to all Serre fibrations that are weak equivalences.
\end{dfn}
\begin{rem}
Recall that every Hurewicz fibration is a Serre fibration and every homotopy equivalence is a weak equivalence. Hence in particular any $q$-cofibration is a Hurewicz cofibration.
\end{rem}
\begin{thm}{\cite[II.3.1]{Q}, cf.~\cite[2.4.25]{H}}
The weak equivalences, Serre fibrations and $q$-cofibrations give a proper model structure on $\catU$.\index{model structure!on $\catU$}
\end{thm}
Again note that Quillen also works with general topological spaces, the transition to $\catU$ is well documented in \cite[2.4]{H}. Properness is proved using that every object is fibrant as well as the following lemma:
\begin{lemma}\label{rem_cofcat} The category $\catU$ is a category of cofibrant objects (Definition~\ref{catcofobj}) with respect to the Hurewicz cofibrations and the weak equivalences. In particular the generalized cobase change (\ref{gen.cobch}) and cube lemma (\ref{gen.cube}) hold for these choices. 
\end{lemma}
Moving to the context of based spaces, we can for example follow the discussion after Remark 1.1.7 of \cite{H} to transport both model structures from $\catU$ to $\catT$. This proves satisfactory in case of the Quillen model structure:
\begin{thm}
  The category $\catT$ is a proper model category using those based maps that are $q$-cofibrations, Serre fibrations respectively weak equivalences in $\catU$, \ie when forgetting the base points.
  \index{model structure!on $\catT$}
\end{thm}
\begin{rem}\label{qgencoft}
We will often make use of the fact that the Quillen model structures on $\catU$ and $\catT$ are cofibrantly generated. Generating sets of cofibrations and acyclic cofibrations are given in the pointed case by:
\begin{align*}
  I\defas{}&\{i: S^{n-1}_+ \rightarrow D^{n}_+ , n\geq 0\}\index{I@$I=\{S^{n-1}_+ \subseteq D^{n}_+\}$}\\
J\defas{}&\{i_0: D^{n}_+\rightarrow (D^{n}\times [0,1])_+, n\geq 0.\}\index{J@$J=\{D^{n}_+\subseteq (D^{n}\times [0,1])_+\}$}
\end{align*}
\end{rem}
\begin{thm}\index{model structure!on $\catT$}
The category $\catT$ has a proper model structure, the {\em Hurewicz
  model structure},  with cofibrations,
fibrations and weak equivalences given as the maps whose underlying
maps in $\catU$ that are respectively Hurewicz cofibrations, Hurewicz
fibrations and homotopy 
equivalences. 
\end{thm}
\begin{rem}
Note that not all spaces in $\catT$ are cofibrant with respect to the
Hurewicz. In particular the theorem only implies pointed analogs to
the versions of the generalized cube and cobase change lemmas from
Lemma~\ref{rem_cofcat} above for so called well based spaces: 
\end{rem}
\begin{dfn}
An object $X$ of $\catT$ is called \emph{well based} or \emph{well pointed}\index{well based} if the inclusion of the base point is a \hur cofibration. 
\end{dfn}
We need a stronger version of the cube lemma when we work in the $\catT$-enriched setting:
\begin{dfn}
\label{dfn:h-cofb}Let $\catC$ be enriched and tensored over $\catT$. A map $i:A\rightarrow X$ in $\catC_0$ is a \emph{based $h$-cofibration}\index{hcofibration@$h$-cofibration!based} if it satisfies the \emph{based homotopy extension property}.\index{homotopy!based - extension property} That is, for every map $f\colon{}X\rightarrow Y$ and based homotopy $h:I_+\smash A \rightarrow Y$ such that $h_0=f \circ i$, there is a based homotopy $H:I_+\smash X \rightarrow Y$ such that $H_0 =f$ and $H\circ(\id\smash i)=h$. In cases where no confusion is possible, we will usually omit the adjective based.
\end{dfn}
\begin{rem}Again there is a recognition lemma analogous to Lemma~\ref{lem_h-cof-test} in terms of a reduced mapping cylinder, implying a similar closure property as in Lemma~\ref{lem_h-cof-test}. Also note that all (unbased or based) Hurewicz cofibrations are closed inclusions (cf.~\cite[$\S$ 6, Ex 1,]{Mayconcise}, \cite[5.2 ff.]{MMSS}).
\end{rem}
The following proposition is a combination of Proposition 9 in
\cite{Str} and the proposition on page 44 of \cite{Mayconcise}. Both are proved by explicitly constructing the required homotopies, respectively retractions.
\begin{prop}
Let $f\colon{} X \rightarrow Y$ be a map between well based spaces in
$\catT$. Then $f$ is a based homotopy equivalence if and only if it is
an (unbased) homotopy equivalence and it is a based Hurewicz cofibration if and only if it is a Hurewicz cofibration.
\end{prop}
Note that being a weak equivalence in $\catT$ and $\catU$ is always equivalent, so we have the following corollary:
\begin{cor}\label{hcofcubespaces}
If all involved spaces are well based, then the generalized cube lemma and the generalized cobase change lemma hold for based Hurewicz cofibrations and homotopy equivalences. Also, they hold for Hurewicz cofibrations and weak equivalences if all the spaces $A_i$ and $Y_i$ in the diagrams~\ref{cobasechange} and~\ref{cube} are well based. 
\end{cor}
Finally we record the following property from \cite[6.8(v)]{MMSS}:
\begin{lemma}
Transfinite composition of Hurewicz cofibrations that are weak equivalences are weak equivalences.
\end{lemma}
The following condition on sets of maps in a topological category has
proven very helpful in several contexts. We use the formulation from
\cite[5.3]{MMSS}, and hence use $\catT$ for the enrichment. Let
$\catA$ and $\catC$ be categories enriched over $\catT$ that are
(enriched) bicomplete and in particular tensored and cotensored. Let
$\catA$ be equipped with a continuous and faithful functor $F \colon \catA \rightarrow \catC$. 
\begin{cond2}{(Cofibration Hypothesis)}\label{cofhyp}
Let $I$ be a set of maps in $\catA$. We say that $I$ satisfies the \emph{cofibration hypothesis}\index{cofibration hypothesis} if it satisfies the following two conditions.
\begin{enumerate}
\item Let $i: A \rightarrow B$ be a coproduct of maps in $I$. Then
  \(F\) takes any cobase change of $i$ in $\catA$ to a Hurewicz cofibration in $\catC$.
\item The colimit of every sequence \(\catA\) that \(F\) takes to a
  sequence of Hurewicz cofibrations in \(\catC\) is preserved by \(F\).
\end{enumerate}
\end{cond2}
\begin{rem}
In particular \(F\) takes $I$-cell complexes in $\catA$ to sequential colimits along Hurewicz cofibrations in $\catC$.
\end{rem}
The smallness conditions in the definition of a cofibrantly generated model category are as lax as possible. In many of the topological examples, we can actually be more strict, in order to get around having to deal with transfinite inductions as much as possible. A convenient condition is the following, again taken from \cite[5.6, ff.]{MMSS}, with $\catA$ and $\catC$ as above:
\begin{dfn}
An object $X$ of $\catA$ is \emph{compact}\index{compact object} if \[\catA(X,\colim Y_n) \cong \colim \catA(X,Y_n),\] whenever $Y_n \rightarrow Y_{n+1}$ is a sequence of maps in $\catA$ that are $h$-cofibrations in $\catC$.
\end{dfn}
\begin{dfn}
Let $\catA$ be a model category. Then $\catA$ is \emph{compactly generated},\index{compactly generated} if it is cofibrantly generated with generating sets of (acyclic) cofibrations $I$ and $J$, such that the domains of all maps in $I$ or $J$ are compact, and $I$ and $J$ both satisfy the cofibration hypothesis~\ref{cofhyp}.
\end{dfn}
\section{Simplicial Objects in Topological Categories}\label{subsectsimplobj}
In this section, we recall some basic simplicial techniques. A convenient reference for a lot of the following discussion is \cite[VII.3]{GJ}, but we need some rather specific technical lemmas which to the author's knowledge have not been formulated similarly before. We start by reminding the reader of the basic definitions:
\begin{dfn}
 The \emph{simplicial category $\catDelta$}\index{simplicial!category, the}\index{Delta@$\catDelta$, the simplicial category} has the finite ordinal numbers as objects and order preserving maps as morphisms between them. 
\end{dfn}
To be more specific, we will denote objects of $\catDelta$ by $\obindelta{n}$, \ie \[{\obindelta{n}} \defas{} \{0 < 1 < \ldots < n\}.\]
Recall the generating morphisms $s_i$ and $d_i$ in $\catDelta$ and the relations between them from \cite[I.1.2]{GJ}.
\begin{dfn}\label{simplobject} Let $\catC$ be a category. The \emph{category $s\catC$ of simplicial objects in $\catC$}\index{simplicial!object in a category} is the functor category $[\catDelta^{\op},\catC]$. 
\end{dfn}
Let from now on $\catC$ be enriched, cocomplete and tensored over the category of simplicial sets.
\begin{dfn}\label{geometricrealiz}
 The \emph{geometric realization $\vert X \vert_\catC$}\index{geometric!realization} of a simplicial object $X \in s\catC$ is the coend
\[{\vert X \vert_\catC} \defas{} \int^{{\bf k} \in \Delta^{\op}} X_k \otimes \Delta^k,\] where $\Delta^k$ is the simplicial $n$-simplex given by $\Delta^k_n = \catDelta(\bf n,k)$. With the obvious extension on morphisms, this defines a functor $|\cdot|_\catC\colon s\catC \rightarrow \catC$.
\end{dfn}
We will often drop the subscript from $\vert \cdot \vert_\catC$ when the category is clear. Note that any functor $\catC \rightarrow \catC'$ that preserves colimits and tensors preserves the geometric realization.
\begin{dfn}\label{ralizationadj}
 If $\catC$ is also cotensored over simplicial sets, the geometric realization has a right adjoint given by the functor that assigns to an object $Y$ of $\catC$ the simplicial object $\catDelta_\catC Y=\catC(\catDelta,Y)$ which is given in level $\obindelta k$ by
\[(\catDelta_\catC Y)_k=\catC(\catDelta,Y)_k \defas{} \catC(\Delta^k, Y).\]
\end{dfn}
\begin{rem}
The most important special case for our applications will be when the category $\catC$ is actually enriched and tensored over $\catT$. In this case, we can first transport the enrichment to $\catU$ along the forgetful functor and then to simplicial sets via the singular set functor as in Construction~\ref{enrtransport} since both of these are (lax) monoidal. Then the defining adjunctions immediately give an isomorphism \[X_k \otimes_{s\catset} \Delta^k \cong X_k \otimes_\catU \vert\Delta^k\vert \cong \X_k \otimes_\catT \vert\Delta^k\vert_+,\] where the $\vert\Delta^k\vert$ denotes the topological $k$-simplex (with a disjoint base point on the right).\\
Classical realization of simplicial sets is then a special case of the above by viewing sets as discrete objects of $\catT$.
\end{rem}
We want to filter the geometric realization, in analogy to the classical case where the geometric realization admits a filtration via its structure as a $CW$-complex. For this purpose consider for a natural number $n$ the full subcategory $\catDelta_n$ of $\catDelta$ consisting of all objects $\obindelta{k}$ with $k \leq n$. 
\begin{dfn}
 For $X \in s\catC$ a simplicial object, define the \emph{$n$-skeleton}\index{skeleton} $\sk_n\vert X \vert_{\catC}$\index{sknX@$\sk_n\vert X \vert_{\catC}$, $n$-skeleton} as the coend
\[{\sk_n\vert X \vert_\catC \defas{} \int}^{{\bf k} \in \Delta_n^{\op}} X_k \otimes \Delta^k.\] 
\end{dfn}
Again, we will often simplify notation and just write $\sk X$ instead of $\sk \vert X \vert_\catC$ when the context does not allow confusion.
\begin{lemma}
The inclusions of categories $\catDelta_n \rightarrow \catDelta_{n+1} \rightarrow \ldots \rightarrow \catDelta$ induce morphisms of coends and we get \[\colim_n \sk_n  X \cong \vert X \vert.\]
\end{lemma}
We define an analog to the degenerate simplices, or rather the latching spaces in the classical setting:
\begin{dfn}\label{degsimplices}
Let $X \in s\catC$ be a simplicial object. The \emph{latching object $L_n X$}\index{latching object $L_n X$}\index{lnX@$L_n X$, latching object } comes together with a distinct map $L_n X \rightarrow X_n$ and is defined inductively as follows: Let $L_0 X$ be the initial object of $\catC$
. Assume that $L_{n} X$ and $L_{n} X \rightarrow X_n$ are already
defined, and that the following diagram commutes:
\labeleq{latching}{\xymatrix{
{L_n X}\ar[rrrr]\ar[rrrd]\ar[dddr]&{}&{}&{}&{X_n}\ar@{.>}^-{s_0}[dddr]\\
{}&{\cdots}&{}&{X_n}\ar@{.>}^-{s_1}_-{}[ddrr]&{}&{}\\
{}&{}&{}&{\cdots}&{}&{}\\
&{X_n}\ar@{.>}^-{s_n}[rrrr]&{}&{}&{}&{X_{n+1}}
}}
We let $L_{n+1} X$ be the colimit of the following solid part of the
above diagram.
The map $L_{n+1}X \rightarrow X_{n+1}$ is induced by the simplicial
degeneracy maps $s_i: X_n \rightarrow X_{n+1}$. 
\end{dfn}
The importance of the latching objects lies in the following proposition:
\begin{prop}{\cite[VII.3.8]{GJ}}\label{skeletonfiltration}
Let $X \in s\catC$ be a simplicial object. Then for all $n \geq 0$ there is a pushout diagram in $\catC$:
\[\xymatrix{
{X_n \otimes \partial\Delta^n \cup_{L_n X \otimes \partial\Delta^n} L_n X \otimes \Delta^n}\ar[d]\ar[r]&{\sk_{n-1} X}\ar[d]\\{X_n \otimes \Delta^n}\ar[r]&{\sk_n X,}\pushout
}\]
where the left vertical map is the pushout product of $L_n X \rightarrow X_n$ with the inclusion of the boundary $\partial\Delta^n \rightarrow \Delta^n$.
\end{prop}
\begin{dfn}\label{dfn:Cproper}
Let \(C\) be a class of morphisms in $\catC$. We say that a simplicial
object $X \in s\catC$ is {\em $C$-proper},\index{Cproper@$C$-proper}\index{proper} if all the maps $L_n X
\rightarrow X_n$ are in $C$.
\end{dfn}
We finally turn to the case of $\catC$ being a topological model category, \ie a model category enriched over $\catU$ in the sense of Definition~\ref{monoidalmodelcat}.
\begin{prop}\label{cofproperfiltr}
Let $\catC$ be enriched and tensored over $\catU$ and let \(C\) be a
class of morphisms in \(\catI\). Suppose that \(C\) is closed under
cobase change 
and satisfies the pushout product axiom with respect to the Quillen
model structure on $\catU$ (\eg if \(C\) is the class of cofibrations in a model category
enriched over $\catU$ in the sense of Definition~\ref{monoidalmodelcat}). Then
for any $C$-proper simplicial object $X$ in $s\catC$, the skeleton
filtration of $\vert X \vert_\catC$ consists morphisms in \(C\).
\[\xymatrix{
{X_n \otimes \vert\partial\Delta^n\vert \cup_{L_n X \otimes \vert\partial\Delta^n\vert} L_n X \otimes \vert\Delta^n\vert}\ar[d]\ar[r]&{\sk_{n-1} X}\ar^-{C\cof}[d]\\{X_n \otimes \vert\Delta^n\vert}\ar[r]&{\sk_n X,}\pushout
}\]
\end{prop}
The next proposition concerns interactions of simplicial objects with weak equivalences. 
\begin{prop}\label{cofproperfiltreq}
Let $\catC$ be enriched and tensored over $\catU$, with a
class \(C\) of cofibrations and a class of weak equivalences, such
that the cofibrant 
objects form a category of cofibrant objects in the sense of
Definition~\ref{catcofobj}. Assume that the cofibrations and weak equivalences
are compatible with the enrichment in the sense that the pushout
product axiom~\ref{monoidalmodelcat}$(i)$ is satisfied. 
Let $X$ and $Y$ in $s\catC$ be \(C\)-proper simplicial
objects such that $X_0$ and $Y_0$ are cofibrant. If $f\colon{} X
\rightarrow Y$ is a morphism of simplicial objects that is a weak
equivalence in each simplicial degree, then the induced map of
realizations  
\[\vert f\vert_\catC: \vert X \vert_\catC \rightarrow \vert Y \vert_\catC\]
is a weak equivalence.
\end{prop}
\begin{proof}
 We begin with showing that all the $X_n$, $Y_n$ and $L_n X$ and $L_n Y$ are cofibrant. $L_1 X = X_0$ is cofibrant by hypothesis, so assume inductively that $L_{n-1} X$ is cofibrant. Since $X$ is \(C\)-proper, the solid arrow part of diagram~\ref{latching} consists only of cofibrations, hence in particular $X_n$ is cofibrant. Since $L_{n+1} X$ is an iterated pushout of $X_n$ along cofibrations it is cofibrant itself. We continue by induction on the skeleton filtration of Proposition~\ref{cofproperfiltr} to show that the maps $sk_n X \rightarrow \sk_n Y$ are weak equivalences. Note that the tensor with a cofibrant space preserves weak equivalences between cofibrant objects by \cite[II.8.4]{GJ}. Hence by the generalized cube lemma we only need to show that the maps
\[X_n \otimes \vert\partial\Delta^n\vert \cup_{L_n X \otimes \vert\partial\Delta^n\vert} L_n X \otimes \vert\Delta^n\vert \longrightarrow Y_n \otimes \vert\partial\Delta^n\vert \cup_{L_n Y \otimes \vert\partial\Delta^n\vert} L_n Y \otimes \vert\Delta^n\vert\]
are weak equivalences between cofibrant objects. Again using the generalized cube lemma on the defining diagram for the pushout product of $L_n X \rightarrow X$ and $\partial \Delta^n \rightarrow \Delta^n$, this reduces to showing that $L_n X \rightarrow L_n Y$ is a weak equivalence. As above this is proven inductively, by comparing the diagrams~\ref{latching} for $X$ and $Y$ and applying the generalized cube lemma to each of the iterated pushouts.
\end{proof}
\begin{rem}
A very obvious example for categories $\catC$ which satisfy the requirements of the above proposition is given by a model category enriched over $\catU$ in the sense of Definition~\ref{monoidalmodelcat}. However, we will in particular want to apply the proposition to (levelwise) Hurewicz cofibrations and $\pi_*$-isomorphisms of orthogonal spectra, so the more general formulation is necessary.
\end{rem}
It can be hard to verify the properness of a simplicial object. Sometimes the following is easier to check:
\begin{dfn}\label{dfngood}
Fix a class \(C\) of morphisms called in $\catC$. We call a simplicial
object $X \in s\catC$ \emph{$C$-good},\index{Cgood@$C$-good} if  for all $n$ all the
degeneracy maps $s_i\colon X_n \rightarrow X_{n+1}$ are in $C$.
\end{dfn}
In particular in $\catT$ and $\catU$, there is Lillig's Union Theorem \cite{Lil}, which implies the following helpful statement:
\begin{lemma} \label{goodisproper}For simplicial objects in the categories $\catT$ or $\catU$, Hurewicz proper and Hurewicz good are equivalent notions. Since colimits and tensors are computed levelwise, the same is true for levelwise Hurewicz cofibrations of (equivariant) orthogonal spectra.
\end{lemma}
%
{The following is an important example of a simplicial object, and in particular comes up in the proofs of the convenience property for the $\Sp$-model structures:
\begin{dfn}
  \label{barconstr}
Let $(\catC,\smash,\I)$ be symmetric monoidal, let $(R, \phi, \eta)$ be a monoid in $\catC$, $(M,\mu)$ a right $R$-module, $(N,\nu)$ a left $R$-module. Define \emph{the bar construction} $\B_*(M,R,N) \in s\catC$ by setting
\[\B_p(M,R,N) = M \smash R^{\smash p} \smash N,\] where $R^{\smash 0} = \I$. The face and degeneracy operators on $\B_p(M,R,N)$ are 
\[d_i = \begin{cases}\mu \smash \id_R^{\smash (p-1)} \smash \id_N & \text{if } i = 0\\
\id_m \smash \id_R^{\smash (i-1)}\smash \phi \smash \id_R^{\smash (p-i-1)} \smash \id_N& \text{if } 0 < i < p\\
\id_M \smash \id_R^{\smash (p-1)} \smash \nu& \text{if } i = p\end{cases}\]
and $s_i = \id_m \smash \id_R^{\smash i} \smash \eta \smash \id_R^{p-i} \smash \id_N$ if $0 \leq i \leq p$.\\ Note that if $M$ was an $(R',R)$-bimodule, then $\B_*(M,R,N)$ is a simplicial $R'$-module.\\
In the case that $\catC$ is enriched and tensored over simplicial sets, so that geometric realization makes sense, we will usually denote the realization by
\[\B(M,R,N) \defas{} \vert \B_*(M,R,N)\vert.\]14
\end{dfn}
}

\section{Assembling Model Structures}\label{assemblingsect}

Given a model structure on a category $\catC$, one often wants to give corresponding structures to categories of functors $\catD \rightarrow \catC$ for some diagram category $\catD$. Theorems on the possibility and methods to do this are well studied in many cases, examples can be found in \cite[14.2.1]{Hir} for cases of cofibrantly generated structures on $\catC$, in \cite[Chapter 5]{H} for the case of $\catD$ a Reedy category. More recently Angeltveit has studied the Reedy approach in an enriched setting (\cite{A}). The result of this section is more in the direction of the former, in particular as a special case we will get an enriched version of Hirschhorn's Theorem \cite[11.6.1]{Hir}. However, the significant difference in our approach is, that we lift not just a single model structure on the target category, but rather assemble a new model structure from several given ones.

Hirschhorn's method uses the evaluation functors that any diagram category is equipped with; we give a short recollection:

Let $\catD$ be a small category. Consider the trivial category $\catTri$ with one
object $*$, and only one (identity) morphism. For each object $d$ of
$\catD$, there is an embedding \(\inc_d \colon * \to \catD\) sending
the object $*$ to $d$. Then the evaluation functor $\fev{}{d}$\index{evd@$\fev{}{d}$} assigns
to a functor $X\colon \catD \rightarrow \catC$ the precomposition
\(\fev{}d X = X \circ \inc_d\) with the inclusion \(\inc_d\).
We have adjoint pairs:
\[\Fr{}{d}: \catC \cong [\catTri,\catC] \leftrightarrows [\catD,\catC]
: \fev{}{d},\]
where $\Fr{}{d}(-)$\index{Fd@$\Fr{}{d}$} is the left Kan extension.
Then, given a cofibrantly generated model structure on $\catC$ with generating sets of (acyclic) cofibrations $I$ and $J$, we can form the sets \[\Fr{}{}I \defas{} \bigcup_{d\in\catD} \Fr{}{d}I
\quad\text{ and }\quad
\Fr{}{}J \defas{} \bigcup_{d\in\catD} \Fr{}{d}J.\]
\begin{thm}{\cite[11.6.1]{Hir}}\label{hirsch}
Let $\catD$ be a small category, and let $\catC$ be a cofibrantly generated model category with generating cofibrations $I$ and generating acyclic cofibrations $J$. Then the category $[\catD,\catC]=[\catD,\catC]_0$ of $\catD$-diagrams in $\catC$ is a cofibrantly generated model category in which a map $f\colon{} X \rightarrow Y$ is 
\begin{itemize}
\item a weak equivalence if $\fev{}{d} (f): X_d \rightarrow Y_d$ is a weak equivalence in $\catC$ for every object $d \in \catD$,
\item a fibration if $\fev{}{d} (f): X_d \rightarrow Y_d$ is a fibration in $\catC$ for every object $d \in \catD$, and
\item an (acyclic) cofibration if it is a retract of a transfinite composition of cobase changes of maps in $\Fr{}{}I$ ($\Fr{}{}J$).
\end{itemize}
\end{thm}
Let us now move to an enriched setting.
Let $(\catV, \smash, \I)$ be a
complete closed symmetric monoidal category, and let $\catC$ and
$\catD$ be enriched over $\catV$, such that $\catD$ is
$\catV$-equivalent to a small category, hence the enriched
functor category $[\catD, \catC]$ exists
(\ref{enrfunctorcat}). Consider $\catTri$ as the trivial
$\catV$-category, \ie as the $\catV$-category with one
object $*$ such that the morphism object $\catTri(*,*)$ is unit in
$\catD$. Then analogous to the discussion above, the inclusion of
$\catTri$ at any object of $\catD$ yields evaluation functors by
precomposition. Under favorable conditions on $\catC$, these have left
adjoints which we again denote by $\Fr{}{d}$ (\eg if $\catC$ is
tensored, cotensored and enriched cocomplete,
cf.~Proposition~\ref{kanexist}). However this time, we want to consider an
intermediate functor category: Given an object $d \in \catD$, denote
by $\End_d$ the \emph{full} subcategory containing only that object. Then the
inclusion of $\catTri$ at $d \in \catD$ factors in the following way 
\labeleq{incfact}{\xymatrix{{\catTri}\ar_-{\inc}[d]\ar@{.>}^-{\fev{}{\!d}X}[ddrr]\\{\End_d}\ar_-{\inc_d}[d]\ar_-{\fev{'}{d}X}[drr]\\{\catD}\ar_-{X}[rr]&{}&{\catC,}}}
and hence we have a factorization of evaluation functors
\labeleq{evfact}{\xymatrix{{\catC}& \ar_-{\cong}[l] {[\catTri,\catC]}&{[\End_d,\catC]}\ar_-{}[l]&{[\catD,\catC]\ar_-{\fev{'}{d}}[l]\ar@(u,u)_-{\fev{}{d}}[ll]}}.}
Each of the functors in this factorization has an (enriched) left adjoint if and only if the appropriate left Kan extensions exist (Proposition~\ref{kanexist}), and in that case we denote them in the following way:
\labeleq{semfreefact}{\xymatrix{{\catC} \ar^-{\cong}[r]& {[\catTri,\catC]}\ar@(u,u)^-{\Fr{}{d}}[rr]\ar^-{\End_d\otimes -}[r]&{[\End_d,\catC]}\ar^-{\Gr{}{d}}[r]&{[\catD,\catC]}}.}
We call objects of the form \emph{$\Gr{}{d}X$ semi-free}, in analogy to the term \emph{free} for objects $\Fr{}{d}Y$. Note that the notation $\End_d \otimes -$ is not accidental, as it is in fact given by the categorical tensor with the endomorphism $\catD$-object of $d$ if it exists.

Assume that for each $d \in \catD$, there is a cofibrantly generated model structure $\catM_d$ on the underlying ordinary category $[\End_d,\catC]_0$ of $[\End_d,\catC]$, with generating (acyclic) cofibrations $I_d$ and $J_d$, and classes of weak equivalences $\catW_d$, respectively. Assume further that each $[\End_d,\catC]$ is tensored and cotensored over $\catV$, so that the semi-free functors $\Gr{}{d}$ all exist. Define the sets of maps $\Gr{}{}I$ and $\Gr{}{}J$ in $[\catD,\catC]_0$ as
\labeleq{GIGJ}{\Gr{}{}I \defas{} \bigcup_{d\in\catD}\Gr{}{d}I_d\;\;\;\;\;\;\;\;\;\;\;\;\;\; \Gr{}{}J \defas{} \bigcup_{d\in\catD}\Gr{}{d}J_d.}
Define the class $\catW$ of maps in $[\catD,\catC]_0$ as 
\labeleq{catW}{\catW \defas{} \{f \in [\catD,\catC]_0,\;\; {\rm{s.t. }}\;\; \fev{'}{d}(f) \,\in\, \catW_d \;\forall\; d \in \catD\}.}
Then the Assembling Theorem is the following
\begin{thm}\label{puzzling}
Let $\catV$ be a complete closed symmetric monoidal category and let $\catC$ and $\catD$ be enriched over $\catV$ such that $\catD$ is equivalent to a small subcategory. Assume that each of the functor categories $[\End_d,\catC]$ is tensored and cotensored over $\catV$ and that we have a family of cofibrantly generated model structures $\{\catM_d\}
$ as above.

Assume that the domains of the maps in $\Gr{}{}I$ are small relative to $\Gr{}{}I$-cell, the domains of the maps in $\Gr{}{}J$ are small with respect to $\Gr{}{}J$-cell and that $\Gr{}{}J$-cell $\subset \catW$.

Then the underlying category $[\catD,\catC]_0$ of $[\catD,\catC]$ is a cofibrantly generated model category where a map $f\colon{} X \rightarrow Y$ is
\begin{itemize}
\item a fibration if and only if each $\fev{'}{d}f$ is a fibration in the model structure $\catM_d$ on $[\End_d,\catC]_0$, and 
\item a weak equivalence if and only if each $\fev{'}{d}f$ is a weak equivalence in the model structure $\catM_d$ on $[\End_d,\catC]_0$ (\ie if and only if $f$ is in $\catW$). 
\end{itemize}
The generating cofibrations are given by $\Gr{}{}I$ and the generating acyclic cofibrations are given by $\Gr{}{}J$.
\end{thm}
\begin{proof}
We check the conditions from the recognition theorem \cite[2.1.19]{H}. First of all, enriched limits and colimits in $[\catD,\catC]$ are calculated pointwise by \cite[3.3]{Kel}, \ie the {(co)}limit of a diagram exists if and only if it does so after evaluating to the $[\End_d,\catC]$ or equivalently to $[\catTri,\catC]$. Since all the $[\End_d,\catC]$ had model structures, they were in particular bicomplete. As they were also tensored and cotensored, they were enriched bicomplete hence so is $[\catD,\catC]$. The class $\catW$ is a subcategory satisfying the 2 out of 3 axiom since it is defined by a levelwise property. By assumption, $\Gr{}{}J$-cell is in $\catW$, and since as a left adjoint $\Gr{}{d}$ preserves retracts and cell complexes, $\Gr{}{d}J_d$-cell $\subset \Gr{}{d}I_d$-cof, hence $\Gr{}{}J$-cell $\subset \Gr{}{}I$-cof. Since $\Gr{}{d}$ is left adjoint to $\fev{'}{d}$, a map has the right lifting property with respect to $\Gr{}{}I$ if and only if for each $d$ its evaluation is an acyclic fibration, in particular if and only if it is in $\catW$ and has the lifting property with respect to $\Gr{}{}J$. 
\end{proof}
\begin{rem}
Similarly to the argument for the bicompleteness of $[\catD,\catC]$, \cite[3.3]{Kel} implies that the assumption, that each of the $[\End_d,\catC]$ is tensored and cotensored, is immediately satisfied if $\catC$ was so itself.
\end{rem}
\begin{prop}\label{assemblenriched}
In the situation of the Assembling Theorem~\ref{puzzling}, assume in addition that $[\catD,\catC]$ is itself tensored and cotensored over $\catV$. If each of the model structures $\catM_d$ satisfies the pushout product axiom (\ref{monoidalmodelcat}(i))
, then so does the assembled model structure on $[\catD,\catC]_0$.
\end{prop}
\begin{proof}
As in \cite[Prop. 5.3.4]{HSS}, by the adjoint formulations in Lemma~\ref{adjointpushoutprod}, it suffices to check the pushout product axiom for $i$ a generating cofibration. But $\Gr{}{d}$ commutes with tensors and pushouts, hence $j \square \Gr{}{d}i \cong \Gr{}{d}{(j \square i)}$. Since $\Gr{}{d}$ also preserves cell complexes and retracts, \(j \square \Gr{}{d}i\) is indeed a cofibration. The case of $i$ or $j$ being acyclic is exactly the same. 
\end{proof}
\begin{rem}
Hence if we can guarantee the analogous proposition for the Unit axiom, a family $\{M_d\}$ of enriched model assembles puzzles together to an enriched model structure on $[\catD,\catC]$. In particular if the unit object of $\catV$ is cofibrant this is trivial. A common other way to ensure this is demanding some sort of \emph{cofibration hypothesis}, cf.~\ref{cofhyp} and a sufficiently general version of the cube lemma~\ref{gen.cube}.
\end{rem}
Depending on the setting, the condition $\Gr{}{}J$-cell $\subset \catW$ in can be hard to verify. A way around this is using Schwede and Shipley's lifting lemma \cite[2.3]{SS} instead of the recognition Theorem \cite[2.1.19]{H}. However, for that result to be applicable in our case, we require another layer of constructions:\\
In the situation of Theorem~\ref{puzzling}, consider the subcategory $\End_\catD$ of $\catD$, consisting of all objects but only the endomorphisms.
More precisely, define $\End_\catD(d,d) \defas{} \catD(d,d)$ but let $\End(d,e)$ be initial in $\catD$ for $d \neq e$. 
Then the inclusions of~\ref{incfact} factor through $\End_\catD$ and hence we get further factorizations of the evaluation functors from~\ref{evfact}
\[\xymatrix{{\catC}&\ar_-{ \cong}[l] {[\catTri,\catC]}&{[\End_d,\catC]}\ar_-{}[l]&{[\End_\catD,\catC]}\ar^-{}[l]&{[\catD,\catC]\ar_-{\fev{''}{}}[l]\ar@(d,d)^-{\fev{'}{d}}[ll]\ar@(u,u)_-{\fev{}{d}}[lll]}}.\]
Let
\[\xymatrix{{\catC}\ar^-{ \cong}[r]&{
    [\catTri,\catC]}\ar@(u,u)^-{\Fr{}{d}}[rrr]\ar^-{\End_d\otimes
    -}[r]&{[\End_d,\catC]}\ar@(d,d)_-{\Gr{}{d}}[rr]\ar^-{\Gr{\End}{d}}[r]&{[\End_\catD,\catC]}\ar^-{}[r]^-{}&{[\catD,\catC]}}\]
be the corresponding diagram of left adjoints~\ref{semfreefact}.
The induced functor pair $ [\catD,\catC]_0 \leftrightarrows [\End_\catD,\catC]_0 $ induces a monad $T$ on $[\End_\catD,\catC]_0$ (cf.~\cite[IV.1]{McL}) and we claim that the associated category of $T$-algebras is isomorphic to $[\catD,\catC]_0$. To prove this we check the prerequisites of Beck's Theorem in its weak form from \cite[Theorem 1]{B} (cf.~\cite[Ex. VI.7. 1-3]{McL}). Indeed, since $[\catD,\catC]$ is enriched cocomplete with colimits calculated pointwise by \cite[3.3]{Kel}, $[\catD,\catC]_0$ has all coequalizers and they are preserved under evaluation to $[\End_\catD,\catC]_0$. Furthermore the evaluation \emph{reflects} isomorphisms, since a $\catV$-natural transformation $\{\alpha_d\}_{d\in\catD}$ is an isomorphism if and only if each $\alpha_d$ is.\\
Further note that $[\End_\catD,\catC]_0 \cong \prod_{d\in\catD} [\End_d,\catC]_0$ since the $\catV$-naturality condition~\ref{enrnattransdiag} is void when $\End_{\catD}(d,e)$ is initial. Hence given the family $\{\catM_d\}_{d\in\catD}$ we get the product model structure on $[\End_\catD,\catC]_0$:
\begin{prop}In the situation of Theorem~\ref{puzzling}, there is a cofibrantly generated model structure on $[\End_\catD,\catC]_0$, where a map is a fibration, cofibration or weak equivalence if and only if it is one in $\catM_d$, for all $d \in \catD$. The generating sets of cofibrations and acyclic cofibrations are given by the sets $\Gr{\End}{}I$ and $\Gr{\End}{}J$, respectively, which are defined analogous to~\ref{GIGJ}.
\end{prop}
\begin{dfn}
In the situation of Theorem~\ref{puzzling}, an object $P$ of $[\catD,\catC]$ is a \emph{path object}\index{path object} of an object $X$ of $[\catD,\catC]$ if there is a factorization of the diagonal map \[\xymatrix{{X}\ar@(u,u)^-{\Delta}[rr]\ar^-{w}[r]&{P}\ar^-{p}[r]&{X \coprod X,}}\]
with $w \in \catW$ and $p$ a pointwise fibration, \ie a fibration in $\catM_d$ after evaluating to $[\End_d,\catC]_0$ for all $d\in\catD$.\end{dfn}
Then hypothesis $(2)$ of \cite[2.3]{SS} allows the following variation of Theorem~\ref{puzzling}
\begin{thm}
The assembling Theorem~\ref{puzzling} still holds if we replace the assumption $\Gr{}{}J$-cell $\subset \catW$, with the following:\\
In each of the model structures $\catM_d$, every object is fibrant and every object of $X \in [\catD,\catC]$ has a path object.
\end{thm}
As promised we study an enriched version of Theorem~\ref{hirsch}:
\begin{prop} \label{enrhirsch}
Theorem~\ref{hirsch} holds in the case of categories enriched over $\catV$, if we additionally assume that the domains of the maps in $\Fr{}{}I$ are small relative to $\Fr{}{}I$-cell, the domains of the maps in $\Fr{}{}J$ are small with respect to $\Fr{}{}J$-cell and that $\Fr{}{}J$-cell consists of maps that are level weak equivalences.
\end{prop}
The proof works entirely analogous to the one of Theorem~\ref{puzzling}. Note that we can reformulate the extra assumptions slightly in the following way:
\begin{lemma}
If tensoring with morphism objects of $\catD$ preserves cofibrations and acyclic cofibrations, the extra assumptions in Proposition~\ref{enrhirsch} are satisfied. In particular this is true if there is a model structure on $\catV$, such that all the morphism objects of $\catD$ are cofibrant, and the model structure on $\catC_0$ satisfies the pushout product axiom.
\end{lemma}
\begin{proof}
This is immediate once one checks that for objects $d$ and $e$ in $\catD$, the composition $\fev{}{e}\circ \Fr{}{d}$ is isomorphic to tensoring with $\catD(d,e)$. Since colimits in $[\catD,\catC]$ are calculated pointwise (\cite[3.3]{Kel}), maps in $\Fr{}{}J$ cell are levelwise retracts of $J$-cell complexes, and the same for $I$. Then all three extra assumptions follow immediately from the axioms of a cofibrantly generated model category.
\end{proof}
\begin{cor}
Since in every model structure cofibrations and acyclic cofibrations are preserved under coproducts, in the case $\catV = \catset$ Proposition~\ref{enrhirsch} reduces to Theorem~\ref{hirsch}. 
\end{cor}

\addcontentsline{toc}{chapter}{Index}
\printindex
\end{document}